\documentclass{booktemplate}

\usepackage[english]{babel}
\usepackage{mathrsfs, amssymb}
\usepackage{mathtools}
\mathtoolsset{centercolon}
\usepackage{booktabs}
\usepackage[shortlabels]{enumitem}
\usepackage[colorlinks, citecolor = blue]{hyperref}

\usepackage{color}
\usepackage{graphicx}
\usepackage{tikz-cd}
\usetikzlibrary{arrows}

\newcommand{\N}{\ensuremath{\mathbb{N}}}
\newcommand{\Z}{\ensuremath{\mathbb{Z}}}
\newcommand{\Q}{\ensuremath{\mathbb{Q}}}
\newcommand{\R}{\ensuremath{\mathbb{R}}}
\newcommand{\C}{\ensuremath{\mathbb{C}}}

\newcommand{\E}{\ensuremath{\mathbb{E}}}
\renewcommand{\P}{\ensuremath{\mathbb{P}}}

\newcommand{\mbb}{\mathbb}
\newcommand{\mbs}{\boldsymbol}
\newcommand{\mb}{\mathbf}
\newcommand{\mc}{\mathcal}
\newcommand{\ms}{\mathscr}

\DeclarePairedDelimiter\abs{\lvert}{\rvert}

\DeclarePairedDelimiter\cbrace\{\}
\DeclarePairedDelimiter\ha()
\DeclarePairedDelimiter{\ip}\langle\rangle
\DeclarePairedDelimiter{\nrm}\lVert\rVert

\newcommand{\nrmb}[1]{\bigl\|#1\bigr\|}
\newcommand{\absb}[1]{\bigl|#1\bigr|}
\newcommand{\hab}[1]{\bigl(#1\bigr)}
\newcommand{\cbraceb}[1]{\bigl\{#1\bigr\}}
\newcommand{\ipb}[1]{\bigl\langle#1\bigr\rangle}

\newcommand{\nrms}[1]{\Bigl\|#1\Bigr\|}
\newcommand{\abss}[1]{\Bigl|#1\Bigr|}
\newcommand{\has}[1]{\Bigl(#1\Bigr)}
\newcommand{\cbraces}[1]{\Bigl\{#1\Bigr\}}
\newcommand{\ips}[1]{\Bigl\langle#1\Bigr\rangle}

\DeclareMathOperator{\re}{Re}
\DeclareMathOperator{\im}{Im}
\DeclareMathOperator{\tr}{tr}
\DeclareMathOperator{\sgn}{sgn}
\DeclareMathOperator{\diag}{diag}
\DeclareMathOperator{\diam}{diam}
\DeclareMathOperator{\spn}{span}

\DeclareMathOperator{\supp}{supp}
\DeclareMathOperator{\ind}{\mathbf{1}}
\DeclareMathOperator{\UMD}{UMD}
\DeclareMathOperator{\BIP}{BIP}
\DeclareMathOperator{\alphaBIP}{\alpha-BIP}
\DeclareMathOperator{\gammaBIP}{\gamma-BIP}
\DeclareMathOperator{\ellBIP}{\ell^2-BIP}
\DeclareMathOperator{\alphaH}{\alpha-H^\infty}
\DeclareMathOperator{\gammaH}{\gamma-H^\infty}
\DeclareMathOperator{\ellH}{\ell^2-H^\infty}

\DeclareMathOperator{\op}{op}
\DeclareMathOperator{\HS}{HS}
\DeclareMathOperator{\gl}{\ell^g}

\newcommand{\dd}{\hspace{2pt}\mathrm{d}}
\newcommand{\ddn}{\mathrm{d}}
\newcommand{\ee}{\mathrm{e}}

\newcommand*{\lcdot}{\cdot}
\renewcommand{\emptyset}{\varnothing}
\def\avint_#1{\mathchoice{\mathop{\kern 0.2em\vrule width 0.6em height 0.69678ex depth -0.58065ex \kern -0.8em \intop}\nolimits_{\kern -0.4em#1}}{\mathop{\kern 0.1em\vrule width 0.5em height 0.69678ex depth -0.60387ex \kern -0.6em \intop}\nolimits_{#1}} {\mathop{\kern 0.1em\vrule width 0.5em height 0.69678ex depth -0.60387ex \kern -0.6em \intop}\nolimits_{#1}} {\mathop{\kern 0.1em\vrule width 0.5em height 0.69678ex depth -0.60387ex \kern -0.6em \intop}\nolimits_{#1}}}

\DeclareFontFamily{U}{mathx}{\hyphenchar\font45}
\DeclareFontShape{U}{mathx}{m}{n}{<5> <6> <7> <8> <9> <10> <10.95> <12> <14.4> <17.28> <20.74> <24.88> mathx10}{}
\DeclareSymbolFont{mathx}{U}{mathx}{m}{n}
\DeclareFontSubstitution{U}{mathx}{m}{n}
\DeclareMathAccent{\widecheck}{0}{mathx}{"71}

\newtheorem{theorem}{Theorem}[chapter]
\newtheorem{lemma}[theorem]{Lemma}
\newtheorem{corollary}[theorem]{Corollary}
\newtheorem{proposition}[theorem]{Proposition}

\theoremstyle{definition}
\newtheorem{definition}[theorem]{Definition}
\newtheorem{example}[theorem]{Example}

\theoremstyle{remark}
\newtheorem{remark}[theorem]{Remark}

\numberwithin{section}{chapter}
\numberwithin{theorem}{section}
\numberwithin{equation}{chapter}

\makeindex

\begin{document}
\frontmatter

\title[Euclidean structures and operator theory in Banach spaces]{Euclidean structures and operator theory in Banach spaces}

\author{Nigel J. Kalton}
\address{Department of Mathematics \\ University of Missouri \\ Colombia \\ MO 65201\\ USA}

\author{Emiel Lorist}
\address{Delft Institute of Applied Mathematics \\ Delft University of Technology \\ P.O. Box 5031\\ 2600 GA Delft \\The Netherlands}
\email{e.lorist@tudelft.nl}

\author{Lutz Weis}
\address{Institute for Analysis \\ Karlsruhe Institute for Technology \\ Englerstrasse 2\\ 76128 Karlsruhe \\Germany}
\email{lutz.weis@kit.edu}

\thanks{The second author is supported by the VIDI subsidy 639.032.427 of the Netherlands Organisation for Scientific Research (NWO). The third author is supported by the Deutsche Forschungsgemeinschaft
(DFG) through CRC 1173}




\begin{abstract}
We present a general method to extend results on Hilbert space operators to the Banach space setting by representing certain sets  of Banach space operators $\Gamma$ on a Hilbert space. Our assumption on $\Gamma$ is expressed in terms of $\alpha$-boundedness for a Euclidean structure $\alpha$ on the underlying Banach space $X$. This notion is originally motivated by $\mathcal{R}$- or $\gamma$-boundedness of sets of operators, but for example any operator ideal from the Euclidean space $\ell^2_n$ to $X$  defines such a  structure. Therefore, our method is quite flexible. Conversely we show that $\Gamma$ has to be $\alpha$-bounded for some Euclidean structure $\alpha$ to be representable on a Hilbert space.

By choosing the Euclidean structure $\alpha$ accordingly, we get a unified and more general approach to the Kwapie\'n--Maurey factorization theorem and the factorization theory of Maurey, Niki{\v{s}}in and Rubio de Francia. This leads to an improved version of the Banach function space-valued extension theorem of Rubio de Francia and a quantitative proof of the boundedness of the lattice Hardy--Littlewood maximal operator.
Furthermore, we use these Euclidean structures to build vector-valued function spaces. These enjoy the nice property that any bounded operator on $L^2$ extends to a bounded operator on these vector-valued function spaces, which is in stark contrast to the extension problem for Bochner spaces.
With these spaces we define an interpolation method, which has formulations modelled after both the real and the complex interpolation method.

Using our representation theorem, we prove a  transference principle for sectorial operators on a Banach space, enabling us to extend Hilbert space results for sectorial operators to the Banach space setting. We for example extend and refine the known theory based on $\mathcal{R}$-boundedness for the joint and operator-valued $H^\infty$-calculus. Moreover, we extend the classical characterization of the boundedness of the $H^\infty$-calculus on Hilbert spaces in terms of $\BIP$, square functions and dilations to the Banach space setting.
Furthermore we establish, via the $H^\infty$-calculus, a version of Littlewood--Paley theory and associated spaces of fractional smoothness  for a rather large class of sectorial operators.
Our abstract setup allows us to reduce assumptions on the geometry of $X$, such as (co)type and UMD.
We conclude with some sophisticated counterexamples for sectorial operators, with as a highlight the construction of a sectorial operator of angle $0$ on a closed subspace of $L^p$ for $1<p<\infty$ with a bounded $H^\infty$-calculus with optimal angle $\omega_{H^\infty}(A) >0$.
\end{abstract}

\maketitle

{\hypersetup{linkcolor=black}\tableofcontents}

\mainmatter
\addtocontents{toc}{\protect\setcounter{tocdepth}{0}}
\chapter*{Introduction}
Hilbert spaces, with their inner product and orthogonal decompositions, are the natural framework for operator and spectral theory and many Hilbert space results fail in more general Banach spaces, even $L^p$-spaces for $p \neq 2$. However, one may be able to recover versions of Hilbert space results for Banach space operators that are in some sense ``close'' to Hilbert space operators. For example, for operators on an $L^p$-scale the Calderon--Zygmund theory, the $A_p$-extrapolation method of Rubio de Francia and Gaussian kernel estimates are well-known and successful techniques to extrapolate $L^2$-results to the $L^p$-scale.

A further approach to extend Hilbert space results to the Banach space setting is to replace uniform boundedness assumptions on certain families of operators by stronger boundedness assumptions such as $\gamma$-boundedness or $\mc{R}$-boundedness. Recall that a set $\Gamma$ of bounded operators on a Banach space $X$ is $\gamma$-bounded if there is a constant such that for all $(x_1,\cdots,x_n) \in X^n$, $T_1,\cdots,T_n \in \Gamma$ and $n \in \N$ we have
\begin{equation}\label{eq:introgamma}\tag{1}
  \nrmb{(T_1x_1,\cdots,T_nx_n)}_{\gamma} \leq C \, \nrmb{(x_1,\cdots,x_n)}_{\gamma},
\end{equation}
where
$
  \nrm{(x_k)_{k=1}^n}_{\gamma} :=\ha{\E \nrm{\sum_{k=1}^n \gamma_kx_k}_X^2}^\frac{1}{2}$
with $(\gamma_k)_{k=1}^n$ a sequence of independent standard Gaussian random variables. If $X$ has finite cotype, then $\gamma$-boundedness is equivalent to the better known $\mc{R}$-boundedness and in an $L^p$-space with $1\leq p<\infty$ $\gamma$-boundedness is equivalent to the discrete square function estimate
\begin{equation}\label{eq:introl2}\tag{2}
  \nrmb{(T_1x_1,\cdots,T_nx_n)}_{\ell^2} \leq C \, \nrmb{(x_1,\cdots,x_n)}_{\ell^2},
\end{equation}
where $\nrm{(x_k)_{k=1}^n}_{\ell^2} :=\nrm{ \ha{\sum_{k=1}^n \abs{x_k}^2}^{1/2}}_{L^p}$. Examples of the extension of Hilbert space results to the Banach space setting under $\gamma$-boundedness assumptions include:
\begin{enumerate}[(i)]
\item On a Hilbert space the generator of a bounded analytic semigroup $(T_z)_{z \in \Sigma_\sigma}$ has $L^p$-maximal regularity, whereas on a $\UMD$ Banach space this holds if and only if $(T_z)_{z \in \Sigma_\sigma}$ is $\gamma$-bounded (see \cite{We01b}).
\item If $A$ and $B$ are commuting sectorial operators on a Hilbert space $H$ with $\omega(A)+\omega(B)<\pi$, then $A+B$ is closed on $D(A)\cap D(B)$ and
    \begin{equation*}
      \nrm{Ax}_H+\nrm{Bx}_H \lesssim \nrm{Ax+Bx}_H, \qquad x \in D(A)\cap D(B).
    \end{equation*}
    On a $\UMD$ Banach space this is still true if $A$ is $\gamma$-sectorial and $B$ has a bounded $H^\infty$-calculus (see \cite{KW01}).
\item A sectorial operator $A$ on a Hilbert space $H$ has a bounded $H^\infty$-calculus if and only if it has bounded imaginary powers $(A^{it})_{t \in \R}$. On a Banach space $X$ with Pisier's contraction property, one can characterize the boundedness of the $H^\infty$-calculus of a sectorial operator $A$ on $X$ by the $\gamma$-boundedness of the set
    $\cbrace{A^{it}:t \in [-1,1]}$ (see \cite{KW16}).
\end{enumerate}
These results follow an active line of research, which lift Hilbert space results to the Banach space setting. Typically one has to find the
``right'' proof in the Hilbert space setting and combine it with $\gamma$-boundedness and Banach space geometry  assumptions in a nontrivial way.

\bigskip

In this memoir we will vastly extend these approaches  by introducing Euclidean structures as a more flexible way to check the enhanced boundedness assumptions such as \eqref{eq:introgamma} and \eqref{eq:introl2} and as a tool to transfer Hilbert space results to the Banach space setting without reworking the proof in the Hilbert space case. Our methods  reduce the need for assumptions on the geometry of the underlying Banach space $X$ such as (co)type and the $\UMD$ property and we also reach out to further applications of the method such as factorization and extension theorems.

We start from the observation that the family of norms $\nrm{\cdot}_{\gamma}$ (and $\nrm{\cdot}_{\ell^2}$) on $X^n$ for $n \in \N$   has the following basic properties:
  \begin{align}\label{eq:E1intro}\tag{3}
    \nrm{(x)}_\gamma &= \nrm{x}_X,  &&x \in X\\
    \nrm{\mb{A}\mb{x}}_\gamma &\leq \nrm{\mb{A}}\nrm{\mb{x}}_\gamma, &&\mb{x} \in X^n,\label{eq:E2intro}\tag{4}
  \end{align}
where the matrix $\mb{A}\colon\C^n \to \C^m$ acts on the vector $\mb{x}=(x_1,\cdots,x_n) \in X^n$ in the canonical way and $\nrm{\mb{A}}$ is the operator norm of $\mb{A}$ with respect to the Euclidean norm. A \emph{Euclidean structure} $\alpha$ on $X$ is now any family of norms $\nrm{\cdot}_{\alpha}$ on $X^n$ for $n \in \N$, satisfying \eqref{eq:E1intro} and \eqref{eq:E2intro} for $\nrm{\cdot}_{\alpha}$. A family of bounded operators $\Gamma$ on $X$ is called $\alpha$-bounded if an estimate similar to \eqref{eq:introgamma} and \eqref{eq:introl2} holds for $\nrm{\cdot}_{\alpha}$. This notion of $\alpha$-boundedness captures the essence of what is needed to represent $\Gamma$ on a Hilbert space. Indeed, denote by $\Gamma_0$ the absolute convex hull of the closure of $\Gamma$ in the strong operator topology and let $\mc{L}_\Gamma(X)$ be the linear span of $\Gamma_0$ normed by the Minkowski functional
\begin{equation*}
  \nrm{T}_\Gamma = \inf \cbrace*{\lambda > 0: \lambda^{-1}T \in \Gamma_0}.
\end{equation*}
Then $\Gamma$ is $\alpha$-bounded for some Euclidean structure $\alpha$ if and only if we have the following ``representation'' of $\Gamma$: there is a Hilbert space $H$, a closed subalgebra $\mc{B}$ of $\mc{L}(H)$, bounded algebra homomorphisms $\tau\colon\mc{L}_\Gamma(X) \to \mc{B}$ and $\rho:\mc{B} \to \mc{L}(X)$ such that $\rho\tau(T) = T$ for all $T \in \mc{L}_\Gamma(X)$, i.e.
\begin{center}
  \begin{tikzpicture}
 \node (LH) at (4,0) {$\mc{L}(H)$} ;
 \node (SE2) at (3.25,0) {$\subseteq$};
 \node (B) at (2.75,0) {$\mc{B}$};
 \node (G) at (0,1) {$\Gamma$};
 \node (SE1) at (0.55,1) {$\subseteq$};
 \node (LG) at (1.5,1) {$\mc{L}_\Gamma(X)$};
 \node (LX) at (4,1) {$\mc{L}(X)$};

   \draw[->] (LG) to node [shift={(0.15,0.15)}] {$\tau$} (B);
    \draw[->] (B) to node [shift={(-0.15,0.15)}] {$\rho$} (LX);
    \draw[right hook->, shorten <=5pt, shorten >=5pt] (LG) to (LX);
\end{tikzpicture}
\end{center}
This theorem (see Theorems \ref{theorem:euclideanrepresentation} and \ref{theorem:Cboundedalphabounded}) is one of our main results. It reveals the deeper reason why results for bounded sets of operators on a Hilbert space extend to results for $\alpha$-bounded sets of operators on a Banach space.

On the one hand $\alpha$-boundedness is a strong notion, since it allows one to represent $\alpha$-bounded sets of Banach space operators as Hilbert space operators, but on the other hand it is a minor miracle that large classes of operators, which are of interest in applications, are $\alpha$-bounded. Partially this is explained by the flexibility we have to create a Euclidean structure:
\begin{enumerate}[(i)]
  \item The choices $\nrm{\cdot}_\gamma$ and $\nrm{\cdot}_{\ell^2}$ that appeared in \eqref{eq:introgamma} and \eqref{eq:introl2} are the ``classical'' choices.
  \item Every operator ideal $\mc{A}\subseteq \mc{L}(\ell^2,X)$ defines a Euclidean structure $\nrm{\cdot}_{\mc{A}}$
      \begin{equation*}
        \nrm{(x_1,\cdots,x_n)}_{\mc{A}}:=\nrmb{\sum_{k=1}^n e_k\otimes x_k}_{\mc{A}}, \qquad x_1,\cdots,x_n \in X,
      \end{equation*}
      where $(e_k)_{k=1}^\infty$ is an orthonormal basis for $\ell^2$.
      \item Let $\mc{B}$ be a closed  unital  subalgebra of a $C^*$-algebra. If $\rho\colon\mc{B} \to \mc{L}(X)$ is a bounded algebra homomorphism, then one can construct a Euclidean structure $\alpha$ so that for every bounded subset $\Gamma\subseteq \mc{B}$ the set $\rho(\Gamma)\subseteq \mc{L}(X)$ is $\alpha$-bounded.
\end{enumerate}
The choice $\alpha = \gamma$ and the connection to $\mc{R}$-boundedness leads to the theory  presented e.g. in \cite{DHP03, KW04} and \cite[Chapter 8]{HNVW17}. The choice $\alpha = \ell^2$ connects us with square function estimates, essential in the theory of singular integral operators in harmonic analysis. With a little bit of additional work, boundedness theorems for such operators, e.g. Calder\'on--Zygmund operators or Fourier multiplier operators, show the $\ell^2$-boundedness of large classes of such operators. Moreover $\ell^2$-boundedness of a family of operators can be deduced from uniform weighted $L^p$-estimates using  Rubio de Francia's $A_p$-extrapolation theory. See e.g. \cite{CMP11,GR85} and \cite[Section 8.2]{HNVW17}.

\bigskip

After proving these abstract theorems in Chapter \ref{part:1}, we make them more concrete by recasting them as factorization theorems for specific choices of the Euclidean structure $\alpha$ in Chapter \ref{part:2}. In particular choosing $\alpha = \gamma$ we can show a $\gamma$-bounded generalization of the classical Kwapie\'n--Maurey factorization theorem (Theorem \ref{theorem:Rbddkwapienmaurey}) and taking $\alpha$ the Euclidean structure induced by the $2$-summing operator ideal we can characterize $\alpha$-boundedness in terms of factorization through, rather than representability on, a Hilbert space (Theorem \ref{theorem:hilbertfactorization}).
 Zooming in on the case that $X$ is a Banach function space on some measure space $(S,\mu)$, we show that the $\ell^2$-structure is the canonical structure to consider and that we can actually factor an $\ell^2$-bounded family $\Gamma \subseteq \mc{L}(X)$ through the Hilbert space $L^2(S,w)$ for some weight $w$ (Theorem \ref{theorem:BFSfactorization}). Important to observe is that this is our first result where we actually have control over the Hilbert space $H$. Moreover it resembles the work of Maurey, Niki{\v{s}}in and Rubio de Francia \cite{Ma73, Ni70, Ru82b} on weighted versus vector-valued inequalities, but has the key advantage that no geometric properties  of the Banach function space are used. Capitalizing on these observations we deduce
 a Banach function space-valued extension theorem (Theorem \ref{theorem:BFSextrapolation}) with milder assumptions than the one in the work of Rubio de Francia \cite{Ru86}. This extension theorem implies the following new results related to the so-called $\UMD$ property for a Banach function space $X$:
  \begin{itemize}
      \item A quantitative proof of the boundedness of the lattice Hardy-Littlewood maximal function if $X$ has the $\UMD$ property.
  \item The equivalence of the dyadic $\UMD^+$ property and the $\UMD$ property.
  \item The necessity of the $\UMD$ property for the $\ell^2$-sectoriality of certain differentiation operators on $L^p(\R^d;X)$.
  \end{itemize}

\bigskip

Besides the discrete $\alpha$-boundedness estimates as in \eqref{eq:introgamma}
and  \eqref{eq:introl2} for a sequence of operators $(T_k)_{k=1}^n$, we also introduce continuous estimates for functions of operators $T\colon \R \to \mc{L}(X)$ with $\alpha$-bounded range, generalizing the well-known square function estimates for $\alpha=\ell^2$ on $X=L^p$ given by
\begin{equation*}
  \nrms{\has{\int_{\R}\abs{T(t)f(t)}^2 \dd t}^{1/2}}_{L^p} \leq C\,   \nrms{\has{\int_{\R}\abs{f(t)}^2 \dd t}^{1/2}}_{L^p}.
\end{equation*}
 To this end we introduce ``function spaces'' $\alpha(\R;X)$ and study their properties in Chapter \ref{part:3}. The space $\alpha(\R;X)$ can be thought of as the completion of the step functions
\begin{equation*}
  f(t) = \sum_{k=1}^n x_k \ind_{(a_{k-1},a_k)}(t),
\end{equation*}
for $x_1,\cdots,x_n \in X$ and $a_0<\cdots<a_n$, with respect to the norm
\begin{equation*}
  \nrm{f}_{\alpha} = \nrmb{\hab{(a_k-a_{k-1})^{-1/2}x_k}_{k=1}^n}_\alpha.
\end{equation*}
The most striking property of these spaces is that any bounded operator $T\colon L^2(\R) \to L^2(\R)$  can be extended to a bounded operator $\widetilde{T}\colon \alpha(\R;X) \to  \alpha(\R;X)$ with the same norm as $T$. As the Fourier transform is bounded on $L^2(\R)$ one can therefore quite easily develop Fourier analysis for $X$-valued functions without assumptions on $X$. For example boundedness of Fourier multiplier operators simplifies to the study of pointwise multipliers, for which  we establish boundedness in Theorem \ref{theorem:pointwisemultipliers1} under an $\alpha$-boundedness assumption.
This  is in stark contrast to the Bochner space case, as the extension problem for bounded operators $T\colon L^2(\R) \to L^2(\R)$ to the Bochner spaces $L^p(\R;X)$ is precisely the reason for limiting assumptions such as (co)type, Fourier type and $\UMD$. We  bypass these assumptions by working in $\alpha(\R;X)$.

With these vector-valued function spaces we define an interpolation method based on a Euclidean structure, the so-called $\alpha$-interpolation method. A charming feature of this $\alpha$-interpolation method is that its formulations modelled after the real  and the complex interpolation method turn out to be equivalent. For the $\gamma$- and $\ell^2$-structures this new interpolation method can be related to the real and complex interpolation methods under geometric assumptions on the interpolation couple of Banach spaces, see Theorem \ref{theorem:compareinterpolation}.

\bigskip

In Chapter \ref{part:4} and \ref{part:5} we apply Euclidean structures to the $H^\infty$-calculus of a sectorial operator $A$. This is feasible since a bounded $H^\infty$-calculus for $A$ defines a bounded algebra homomorphism $$\rho\colon H^\infty(\Sigma_\sigma) \to \mc{L}(X)$$ given by $f \mapsto f(A)$. Therefore our theory yields the $\alpha$-boundedness of $$\cbrace{f(A):\nrm{f}_{H^\infty(\Sigma_\sigma)} \leq 1}$$ for some Euclidean  structure $\alpha$, which provides a wealth of $\alpha$-bounded sets. Conversely, $\alpha$-bounded variants of notions like sectoriality and $\BIP$ allow us to transfer Hilbert space results to the Banach space setting, at the heart of which lies a transference result (Theorem \ref{theorem:transference}) based on our representation theorems.  With our techniques we generalize and refine the known results on the operator-valued and joint $H^\infty$-calculus and the ``sum of operators'' theorem for commuting sectorial operators on a Banach space.
We also extend the classical characterization of the boundedness of the $H^\infty$-calculus in Hilbert spaces to the Banach space setting. Recall that for a
sectorial operator $A$ on a Hilbert space $H$ the following are equivalent (see  \cite{Mc86,AMN97,Le98})
\begin{enumerate}[(i)]
  \item $A$ has a bounded $H^\infty$-calculus.
  \item $A$ has bounded imaginary powers $(A^{it})_{t \in \R}$.
  \item For one (all) $0 \neq \psi \in H^1(\Sigma_\sigma)$ with $\omega(A)<\sigma<\pi$ we have
      \begin{equation*}
        \nrm{x}_H \simeq \int_0^\infty \has{\nrm{\psi(tA)x}_H^2 \frac{\ddn t}{t}}^{1/2}, \qquad x \in D(A)\cap R(A).
      \end{equation*}
  \item $[X,D(A)]_{1/2} = D(A^{1/2})$ with equivalence of norms, where $[\cdot,\cdot]_{\theta}$ denotes the complex interpolation method.
  \item $A$ has a dilation to a normal operator on a larger Hilbert space $\widetilde{H}$.
\end{enumerate}
Now let $A$ be a sectorial operator on a general Banach space $X$. If $\alpha$ is a Euclidean structure on $X$ satisfying some mild assumptions and $A$ is \emph{almost $\alpha$-sectorial}, i.e.
\begin{equation*}
  \cbrace{\lambda AR(\lambda,A)^2:\lambda \in \C\setminus \overline{\Sigma}_\sigma}
\end{equation*}
is $\alpha$-bounded for some $\omega(A)<\sigma<\pi$, then the following are equivalent (see Theorems \ref{theorem:BIPHinfty}, \ref{theorem:squaretocalculus}, \ref{theorem:calculustosuqare}, \ref{theorem:dilationresolvent} and Corollary \ref{corollary:interptoHinfty})
\begin{enumerate}[(i)]
  \item $A$ has a bounded $H^\infty$-calculus.
  \item $A$ has $\alpha$-$\BIP$, i.e. $\cbrace{A^{it}:t \in [-1,1]}$ is $\alpha$-bounded.
  \item For one (all) $0 \neq \psi \in H^1(\Sigma_\nu)$ with $\sigma<\nu<\pi$ we have the      generalized square function estimates
\begin{equation}\tag{5}\label{eq:introsquarefun}
        \nrm{x}_X \simeq \nrm{t \mapsto \psi(tA)x}_{\alpha(\R_+,\frac{\ddn t}{t};X)}, \qquad x \in D(A) \cap R(A).
\end{equation}
\item $
(X,D(A))^\alpha_{1/2} = D(A^{1/2})
$
with equivalence of norms, where we use the $\alpha$-interpolation method from Chapter \ref{part:3}.
\item $A$ has a dilation to the ``multiplication operator'' $\mc{M}_s$ with $\sigma<s<\pi$ on $\alpha(\R;X)$ given by
 \begin{equation*}
   \mc{M}g(t) := (it)^{\frac{2}{\pi}s}\cdot g(t), \qquad t \in \R.
 \end{equation*}
\end{enumerate}
  For these results we could also use the stronger notion of \emph{$\alpha$-sectoriality}, i.e. the $\alpha$-boundedness of
\begin{equation*}
  \cbrace{\lambda R(\lambda,A):\lambda \in \C\setminus \overline{\Sigma}_\sigma}
\end{equation*}
 for some $\omega(A)<\sigma<\pi$,  which is thoroughly studied for the $\gamma$- and $\ell^2$-structure through the equivalence with $\mc{R}$-sectoriality. However, we opt for the weaker notion of almost  $\alpha$-sectoriality to avoid additional assumptions on both $\alpha$ and $X$.

We note that the generalized square function estimates as in \eqref{eq:introsquarefun} and their discrete counterparts
 \begin{equation*}
        \nrm{x}_X \simeq \sup_{t \in [1,2]}\nrmb{(\psi(2^ntA)x)_{n\in \Z}}_{\alpha(\Z;X)}, \qquad x \in X,
\end{equation*}
provide a version of Littlewood--Paley theory, which allows us to carry ideas from harmonic analysis to quite general situations. This idea is developed in Section \ref{section:scalespaces}, where we introduce a scale of intermediate spaces, which are close to the homogeneous fractional domain spaces $\dot{D}(A^\theta)$ for $\theta \in \R$ and are defined in terms of the generalized square functions
\begin{equation*}
  \nrm{x}_{H^{\alpha}_{\theta,A}}:= \nrm{t \mapsto \psi(tA)A^\theta x}_{\alpha(\R_+,\frac{\ddn t}{t};X)}.
\end{equation*}
If $A$ is almost $\alpha$-sectorial, we show that $A$ always has a bounded $H^\infty$-calculus on the spaces $H^{\alpha}_{\theta,A}$ and that $A$ has a bounded $H^\infty$-calculus on $X$ if and only if $\dot{D}(A^\theta) = H^{\alpha}_{\theta,A}$ with equivalence of norms (Theorem \ref{theorem:Hproperties}).
If $A$ is not almost $\alpha$-bounded, then our results on the generalized square function spaces break down. We analyse this situation carefully in Section \ref{section:Hspacenotalmostalpha} as a preparation for the final chapter.

The final chapter, Chapter \ref{part:6}, is devoted to some counterexamples related to the notions studied in Chapter \ref{part:4} and \ref{part:5}. In particular we use Schauder multiplier operators to show that almost $\alpha$-sectoriality does not come for free for a sectorial operator $A$, i.e. that  almost $\alpha$-sectoriality is not a consequence of the sectoriality of $A$ for any reasonable Euclidean structure $\alpha$. This result is modelled after a similar statement for $\mc{R}$-sectoriality by Lancien and the first author \cite{KL00}.
 Furthermore, in Section \ref{section:almostalphanotalpha} we show that almost $\alpha$-sectoriality is
 strictly weaker than $\alpha$-sectoriality, i.e. that there exists an almost $\alpha$-sectorial operator $A$ which is not $\alpha$-sectorial.

Throughout Chapter \ref{part:4} and \ref{part:5} we prove that the angles related to the various properties of a sectorial operator, like the angle of (almost) $\alpha$-sectoriality, ($\alpha$-)bounded $H^\infty$-calculus and ($\alpha$-)$\BIP$, are equal. Strikingly absent in that list is the angle of sectoriality of $A$. By an example of Haase it is known that it is possible to have $\omega_{\BIP}(A) \geq \pi$ and thus $\omega_{\BIP}(A) > \omega(A)$, see \cite[Corollary 5.3]{Ha03}. Moreover in \cite{Ka03} it was shown by the first author that it is also possible to have $\omega_{H^\infty}(A)>\omega(A)$. Using the generalized square function spaces and their unruly behaviour if $A$ is not almost $\alpha$-sectorial, we provide a more natural example of this situation, i.e. we construct a sectorial operator on a closed subspace of $L^p$ such that $\omega_{H^\infty}(A)>\omega(A)$ in Section \ref{section:angleexample}.

\section*{The history of Euclidean structures}
The $\gamma$-structure was first  introduced by Linde and Pietsch  \cite{LP74} and discovered for the theory of Banach spaces by Figiel and Tomczak--Jaegermann in \cite{FT79}, where it was used in the context of estimates for the projection constants of finite dimensional Euclidean subspaces of a Banach space. In \cite{FT79} the norms $\nrm{\cdot}_\gamma$ were called $\ell$-norms.

Our definition of a Euclidean structure is partially inspired by the similar idea of a lattice structure on a Banach space studied by Marcolino Nhani \cite{Ma01}, following ideas of Pisier. In his work $c_0$ plays the role of $\ell^2$.
Other, related research building upon the work of Marcolino Nhani includes:
\begin{itemize}
  \item Lambert, Neufang and Runde introduced \emph{operator sequence spaces}  in \cite{LNR04}, which use norms satisfying the basic properties of a Euclidean structure and an additional $2$-convexity assumption. They use these operator sequence spaces to study Fig\'a--Talamanca--Herz algebras from an operator-theoretic viewpoint.
  \item Dales, Laustsen, Oikhberg and Troitsky \cite{DLOT17} introduced \emph{$p$-multi\-norms}, building upon the work by Dales and Polyakov \cite{DP12} on $1$- and $\infty$-multinorms. They show that a strongly $p$-multinormed Banach space which is $p$-convex can be represented as a closed subspace of a Banach lattice. This representation was subsequently generalized by Oikhberg \cite{Oi18}.  The definition of a $2$-multinorm is exactly the same as our definition of a Euclidean structure.
\end{itemize}
 Further inspiration for the constructions in Section \ref{sec:cbounded} comes from the theory of operator spaces and completely bounded maps, see e.g. \cite{BL04,ER00,Pa02,Pi03}.

In the article by Giannopoulos and Milman \cite{GM01} the term ``Euclidean structure'' is used to indicate the appearance of the Euclidean space $\R^n$ in the Grassmannian manifold of finite dimensional subspaces of a Banach space, as e.g.~spelled out in Dvoretzky's theorem. This article strongly emphasizes the connection with convex geometry and the so-called ``local theory'' of Banach spaces and does not treat operator theoretic questions.
 For further results in this direction see \cite{MS86,Pi89,To89}.

Our project started as early as 2003 as a joint effort of N.J. Kalton and L. Weis and since then a partial draft-manuscript called ``Euclidean structures'' was circulated privately. The project suffered many delays, one of them caused by the untimely death of N.J. Kalton. Only when E. Lorist injected new results and energy the project was revived and finally completed. E. Lorist and L. Weis would like to dedicate this expanded version to N.J. Kalton, in thankful memory.
Some results concerning generalized square function estimates with respect to the $\gamma$-structure have in the mean time been published in \cite{KW16}.

\section*{Structure of the memoir}
This memoir is structured as follows: In Chapter 1 we give the definitions, a few examples and prove the basic properties of a Euclidean structure $\alpha$. Moreover we prove our  main representation results for $\alpha$-bounded families of operators, which will play an important role in the rest of the memoir.
Afterwards, Chapters \ref{part:2}-\ref{part:4} can be read (mostly) independent of each other:
\begin{itemize}
  \item In Chapter \ref{part:2} we highlight some special cases in which the representation results of Chapter \ref{part:1} can be made more explicit in the form of factorization theorems.
  \item In Chapter \ref{part:3} we introduce vector-valued function spaces and interpolation with respect to a Euclidean structure.
  \item In Chapter  \ref{part:4} we study the relation between Euclidean structures and the $H^\infty$-calculus for a sectorial operator.
\end{itemize}
Chapter \ref{part:5} treats generalized square function estimates and spaces and relies heavily on the theory developed in Chapter \ref{part:3} and \ref{part:4}. Finally in Chapter \ref{part:6} we treat counterexamples related to sectorial operators, which use the theory from Chapter \ref{part:4} and \ref{part:5}.

\section*{Notation and conventions}
Throughout this memoir $X$ will be a complex Banach space. For $n \in \N$ we let $X^n$ be the space of $n$-column vectors with entries in $X$.  For $m,n \in \N$ we denote the space $m\times n$ matrices with complex entries by $M_{m,n}(\C)$ and endow it with the operator norm.
We will often denote elements of $X^n$ by $\mb{x}$ and use $x_k$ for $1 \leq k \leq n$ to refer to the $k$th-coordinate of $\mb{x}$. We use the same convention for a matrix in $M_{m,n}(\C)$ and its entries.

The space of bounded linear operators on $X$ will be denoted by $\mc{L}(X)$ and we will write $\nrm{\lcdot}$ for the operator norm $\nrm{\lcdot}_{\mc{L}(X)}$.
For a Hilbert space $H$ we will always let its dual $H^*$ be its Banach space dual, i.e. using a bilinear pairing instead of the usual sesquilinear pairing.

By $\lesssim$ we mean that there is a constant $C>0$ such that inequality holds and by $\simeq$ we mean that both $\lesssim$ and $\gtrsim$  hold.

\section*{Acknowledgements}
The authors would like to thank Martijn Caspers, Christoph Kriegler, Jan van Neerven and Mark Veraar for their helpful comments on the draft version of this memoir and Tuomas Hyto\"nen for allowing us to include Theorem \ref{theorem:UMD+UMD}, which he had previously shown in unpublished work.
The authors would also like to thank Mitchell Taylor for bringing the recent developments regarding $p$-multinorms under our attention. Moreover the authors express their deep gratitude to the anonymous referee, who read this memoir very carefully and provided numerous insightful comments.
Finally the authors would like to thank everyone who provided feedback on earlier versions of this manuscript to Nigel J. Kalton.

\addtocontents{toc}{\protect\setcounter{tocdepth}{1}} 
\chapter{Euclidean structures and \texorpdfstring{$\alpha$}{a}-bounded operator families}\label{part:1}
In this first chapter we will start with the definition, a few examples and some  basic properties of a Euclidean structure $\alpha$. Afterwards  will study the boundedness of families of bounded operators on a Banach space with respect to a Euclidean structure in Section \ref{section:alphabounded}. The second halve of this chapter is devoted to one of our main theorems, which characterizes which families of bounded operators on a Banach space can be represented on a Hilbert space. In particular, in Section \ref{section:representation} we prove
  a representation theorem for $\alpha$-bounded families of operators.
Then, given a family of operators $\Gamma$ that is representable on a Hilbert space, we construct a Euclidean structure $\alpha$ such that  $\Gamma$ is $\alpha$-bounded  in Section \ref{sec:cbounded}.

\bigskip
\subsection{Random sums in Banach spaces}
Before we start, let us  introduce random sums in Banach spaces.
A random variable $\varepsilon$ on a probability space $(\Omega,\P)$ is called a {\em Rademacher} if it is uniformly distributed in $\cbrace{z \in \C: \abs{z} = 1}$. A random variable $\gamma$ on $(\Omega,\P)$ is called a {\em Gaussian} if its distribution has density
$$
    f(z) = \frac{1}{\pi} \ee^{-{\abs{z}^2}}, \qquad z \in \C,
$$
with respect to the Lebesgue measure on $\C$. A {\em{Rademacher sequence}} (respectively \emph{Gaussian sequence}) is a sequence of independent Rademachers (respectively Gaussians). For all our purposes we could equivalently use real-valued Rademacher and Gaussians, see e.g. \cite[Section 6.1.c]{HNVW17}.

Two important notions in Banach space geometry are type and cotype. Let $p \in [1,2]$ and $q \in [2,\infty]$ and let $(\varepsilon_k)_{k=1}^\infty$  be a Rademacher sequence. The space $X$ has {\em type $p$} if there exists a constant $C >0$ such that
  \begin{equation*}
    \nrms{\sum_{k=1}^n \varepsilon_n x_n}_{L^p(\Omega;X)} \leq C \,\has{\sum_{k=1}^n \nrm{x_k}_X^p}^\frac{1}{p}, \qquad \mb{x} \in X^n.
  \end{equation*}
The space $X$ has {\em cotype $q$} if there exists a constant $C >0$ such that
  \begin{equation*}
   \has{\sum_{k=1}^n \nrm{x_k}_X^q}^\frac{1}{q} \leq C\, \nrms{\sum_{k=1}^n \varepsilon_n x_n}_{L^q(\Omega;X)}, \qquad \mb{x} \in X^n,
  \end{equation*}
  with the obvious modification for $q= \infty$. We say $X$ has \emph{nontrivial type} if it has type $p>1$ and we say $X$ has \emph{finite cotype} if it has cotype $q< \infty$.   Any Banach space has type $1$ and cotype $\infty$. Moreover nontrivial type implies finite cotype (see \cite[Theorem 7.1.14]{HNVW17}).

 We can compare Rademachers sums with Gaussians sums and if $X$ is a Banach lattice with $\ell^2$-sums.
\begin{proposition}\label{proposition:gaussianradermacherl2comparison}
Let $1\leq p,q \leq \infty$, let $(\varepsilon_k)_{k=1}^\infty $ be a Rademacher sequence and let $(\gamma_k)_{k=1}^\infty $ be a  Gaussian sequence. Then for all  $\mb{x} \in X^n$ we have
\begin{equation*}
  \nrms{\has{\sum_{k=1}^n \abs{x_k}^2}^{1/2}}_X \lesssim\nrms{\sum_{k=1}^n \varepsilon_kx_k}_{L^p(\Omega;X)} \lesssim\nrms{\sum_{k=1}^n \gamma_kx_k}_{L^q(\Omega;X)},
\end{equation*}
where the first expression is only valid if $X$ is a Banach lattice. If $X$ has finite cotype, then for all  $\mb{x} \in X^n$ we have
\begin{equation*}
   \nrms{\sum_{k=1}^n \gamma_kx_k}_{L^p(\Omega;X)}\lesssim\nrms{\sum_{k=1}^n \varepsilon_kx_k}_{L^p(\Omega;X)} \lesssim \nrms{\has{\sum_{k=1}^n \abs{x_k}^2}^{1/2}}_X
\end{equation*}
where the last expression is only valid if $X$ is a Banach lattice.
\end{proposition}

For the proof we refer to \cite[Theorem 6.2.4, Corollary 7.2.10 and Theorem 7.2.13]{HNVW17}.

\section{Euclidean structures} \label{section:euclidean}
A {\em Euclidean structure} on $X$ is a family of norms $\nrm{\cdot}_\alpha$ on $X^n$ for all $n \in \N$ such that
  \begin{align}
 \label{eq:E1x}
    \nrm{(x)}_\alpha &= \nrm{x}_X,  &&x \in X\\
\label{eq:E2x}
    \nrm{\mb{A}\mb{x}}_\alpha &\leq \nrm{\mb{A}}\nrm{\mb{x}}_\alpha, &&\mb{x} \in X^n,\quad \mb{A} \in M_{m,n}(\C), \quad m\in \N.
  \end{align}
It will be notationally convenient to define $\nrm{\mb{x}}_\alpha := \nrm{\mb{x}^{T}}_\alpha$ for a row vector $\mb{x}$ with entries in $X$.
Alternatively a Euclidean structure can be defined as a norm on the space of finite rank operators from $\ell^2$ to $X$, which we denote by $\mc{F}(\ell^2,X)$. For $e \in \ell^2$ and $x \in X$ we write $e \otimes x$ for the {\em rank-one} operator $f \mapsto \ip{f, e}x$. Clearly we have $\nrm{e\otimes x} = \nrm{e}_{\ell^2}\nrm{x}_X$ and any element $T \in \mc{F}(\ell^2,X)$ can be represented as
\begin{equation*}
  T = \sum_{k=1}^n e_k \otimes x_k
\end{equation*}
with $(e_k)_{k=1}^n$ an orthonormal sequence in $\ell^2$ and $\mb{x} \in X^n$. If $\alpha$ is a Euclidean structure on $X$ and $T \in \mc{F}(\ell^2,X)$ we define
\begin{align*}
  \nrm{T}_\alpha &:= \nrm{\mb{x}}_\alpha,
\end{align*}
where $\mb{x}$ is such that $T$ is representable in this form.
This definition is independent of the chosen orthonormal sequence by \eqref{eq:E2x} and this norm satisfies
  \begin{align}
  \tag{\ref*{eq:E1x}$'$} \label{eq:E1}
    \nrm{f \otimes x}_\alpha &= \nrm{f}_{\ell^2}\nrm{x}_X &&\qquad f \in \ell^2,\quad x \in X, \\
  \tag{\ref*{eq:E2x}$'$}\label{eq:E2}
    \nrm{TA}_{\alpha} &\leq \nrm{T}_\alpha\nrm{A} &&\qquad T \in \mc{F}(\ell^2,X),\quad A \in \mc{L}(\ell^2).
\end{align}
Conversely a norm $\alpha$ on $\mc{F}(\ell^2,X)$ satisfying \eqref{eq:E1} and \eqref{eq:E2} induces a unique Euclidean structure by
\begin{equation*}
  \nrm{\mb{x}}_\alpha := \nrms{f \mapsto \sum_{k=1}^n \ip{f, e_k} x_k}_\alpha, \qquad \mb{x} \in X^n,
\end{equation*}
where $(e_k)_{k=1}^n$ is a orthonormal system in $\ell^2$.

Conditions \eqref{eq:E2x} and \eqref{eq:E2} express the right-ideal property of a Euclidean structure. We will call a Euclidean structure $\alpha$ {\em ideal} if it also has the left-ideal condition
\begin{align}
 \label{eq:E3x}
    \nrm{(Sx_1, \ldots, Sx_n)}_\alpha
     &\leq C \,\nrm{S} \nrm{\mb{x}}_\alpha, &&\mb{x} \in X^n, \quad S \in \mc{L}(X),
\intertext{which in terms of the induced norm on $\mc{F}(\ell^2,X)$ is given by}
\tag{\ref*{eq:E3x}$'$}   \label{eq:E3}  \nrm{ST}_{\alpha} &\leq C \,\nrm{S}\nrm{T}_\alpha &&T \in \mc{F}(\ell^2,X), \quad S \in \mc{L}(X).
\end{align} If we can take $C = 1$ we will call $\alpha$ {\em isometrically ideal}.

  A {\em global Euclidean structure} $\alpha$ is an assignment of a Euclidean structure $\alpha_X$ to any Banach space $X$. If it can cause no confusion we will denote the induced structure $\alpha_X$ by $\alpha$. A global Euclidean structure is called \emph{ideal} if we have
  \begin{align}
  \label{eq:E4x} \nrm{(Sx_1, \ldots, Sx_n)}_{\alpha_Y}
     &\leq \nrm{S} \nrm{\mb{x}}_{\alpha_X}, && \mb{x} \in X^n, \quad S \in \mc{L}(X,Y)
  \intertext{for any Banach spaces $X$ and $Y$. In terms of the induced norm on $\mc{F}(\ell^2,X)$ this assumption is given by}
\tag{\ref*{eq:E4x}$'$}   \label{eq:E4}  \nrm{ST}_{\alpha_Y} &\leq \nrm{S}\, \nrm{T}_{\alpha_X} &&T \in \mc{F}(\ell^2,X), \quad S \in \mc{L}(X,Y).
\end{align}
 Note that if $\alpha$ is an ideal global Euclidean structure then $\alpha_X$ is isometrically ideal, which can be seen by taking $Y=X$ in the definition. Many natural examples of Euclidean structures are in fact isometrically ideal and are inspired by the theory of operator ideals, see \cite{Pi80}.

For two Euclidean structures $\alpha$ and $\beta$ we write $\alpha \lesssim \beta$ if there is a constant $C>0$ such that $\nrm{\mb{x}}_\alpha \leq C \nrm{\mb{x}}_\beta$ for all $\mb{x} \in X^n$. If $C$ can be taken equal to $1$ we write $\alpha \leq \beta$.

\begin{proposition}\label{proposition:constructglobal}
  Let $\beta$ be an ideal Euclidean structure on a Banach space $X$. Then there exists an ideal global Euclidean structure $\alpha$ such that $\alpha_X \simeq \beta$. Moreover, if $\beta$ is isometrically ideal, then $\alpha_X = \beta$.
\end{proposition}

\begin{proof}
Define $\alpha_Y$ on a Banach space $Y$ as
\begin{equation*}
  \nrm{\mb{y}}_{\alpha_Y} = \sup\cbrace*{\nrm{(Ty_1,\ldots, Ty_n)}_\beta: T \in \mc{L}(Y,X), \nrm{T} \leq 1}, \quad \mb{y} \in Y^n.
\end{equation*}
Then \eqref{eq:E1x} and \eqref{eq:E2x} for $\alpha_Y$ follow directly from the same properties of $\beta$ and \eqref{eq:E4x} is trivial, so $\alpha$ is an ideal global Euclidean structure. Furthermore, by the ideal property of $\beta$, we have
\begin{align*}
  \nrm{\mb{x}}_{\alpha_X} &\leq C\, \nrm{\mb{x}}_{\beta}  \leq C\,\nrm{\mb{x}}_{\alpha_X}, \qquad \mb{x} \in X^n
\end{align*}
so $\alpha_X$ and $\beta$ are equivalent. Moreover, they are equal if $\beta$ is isometrically ideal.
\end{proof}
Although our definition of a Euclidean structure is isometric in nature, we will mostly be interested in results stable under isomorphisms. If $\alpha$ is a Euclidean structure on a Banach space $X$ and we equip $X$ with an equivalent norm $\nrm{\lcdot}_1$, then $\alpha$ is not necessarily a Euclidean structure on $(X,\nrm{\lcdot}_1)$. However, this is easily fixed. Indeed, if $C^{-1}\,\nrm{\lcdot}_X \leq \nrm{\lcdot}_1 \leq C\, \nrm{\lcdot}_X$, we define
\begin{equation*}
  \nrm{\mb{x}}_{\alpha_1} := \max \cbrace{\nrm{\mb{x}}_{\op_1},C^{-1}\, \nrm{\mb{x}}_\alpha},\qquad \mb{x} \in X^n,
\end{equation*}
where $\op_1$ denotes the Euclidean structure on $(X,\nrm{\lcdot}_1)$ induced by the operator norm on $\mc{F}(\ell^2,X)$. Then $\alpha_1$ is a Euclidean structure on $(X,\nrm{\lcdot}_1)$ such that $\alpha \simeq \alpha_1$.

\subsection*{Examples of Euclidean structures} As already noted in the previous section, the operator norm induces an ideal global Euclidean structure, as it trivially satisfies \eqref{eq:E1}, \eqref{eq:E2} and \eqref{eq:E4}. For $\mb{x} \in X^n$ the induced Euclidean structure is given by
\begin{align*}
  \nrm{\mb{x}}_{\op} &= \sup \cbraces{\nrmb{\sum_{k=1}^n a_k x_k}_X: \sum_{k=1}^n \abs{a_k}^2 \leq 1} =\sup_{\nrm{x^*}_{X^*}\leq 1} \has{\sum_{k=1}^n \abs{x^*(x_k)}^2}^{1/2}.
\end{align*}
Another example is induced by the {\em nuclear norm} on $\mc{F}(X,Y)$, which for $T \in \mc{F}(X,Y)$ is defined by
\begin{equation*}
    \nrm{T}_\nu := \inf\cbraces{\sum_{k=1}^n\nrm{e_k}\nrm{x_k}_X: T = \sum_{k=1}^ne_k \otimes x_k}
\end{equation*}
in which the infimum is taken over all finite representations of $T$, see e.g. \cite[Chapter 1]{Ja87}. Again this norm satisfies \eqref{eq:E1}, \eqref{eq:E2} and \eqref{eq:E4} and for $\mb{x} \in X^n$ the induced Euclidean structure is given by
\begin{equation*}
  \nrm{\mb{x}}_{\nu} = \inf \cbraces{\sum_{j=1}^m\nrm{y_j}_X: \mb{x} = \mb{A} \mb{y},  \mb{A} \in M_{n,m}(\C),\max_{1\leq j\leq m}\sum_{k=1}^n \abs{A_{kj}}^2 \leq 1,}.
\end{equation*}
The operator and nuclear Euclidean structures are actually the maximal and minimal Euclidean structures.

\begin{proposition}\label{proposition:operatornuclear}
For any Euclidean structure $\alpha$ on $X$ we have $$\op \leq \alpha \leq \nu.$$
\end{proposition}
\begin{proof}
Fix $\mb{x} \in X^n$. For the operator norm structure we have
  \begin{equation*}
  \nrm{\mb{x}}_{\op} = \sup_{\substack{\mb{A} \in M_{1,n}(\C)\\ \nrm{\mb{A}} \leq 1}} \nrm{\mb{A}\mb{x}}_X = \sup_{\substack{\mb{A} \in M_{1,n}(\C)\\ \nrm{\mb{A}} \leq 1}} \nrm{\mb{A}\mb{x}}_\alpha \leq \nrm{\mb{x}}_\alpha.
\end{equation*}
For the nuclear structure take $\mb{y} \in X^m$ such that $\mb{x} = \mb{A} \mb{y}$ with $\mb{A} \in M_{n,m}(\C)$ and $\sum_{k=1}^n \abs{A_{kj}}^2 \leq 1$ for $1\leq j \leq m$. Then we have
\begin{align*}
  \nrm{\mb{x}}_\alpha = \nrm{\mb{A}\mb{y}}_\alpha \leq \sum_{j=1}^m \nrmb{\hab{A_{1j}y_j,\ldots,A_{nj}y_j}}_\alpha \leq \sum_{j=1}^m\nrmb{(y_j)}_\alpha =  \sum_{j=1}^m\nrm{y_j}_X,
\end{align*}
so taking the infimum over all such $\mb{y}$ gives $\nrm{\mb{x}}_\alpha \leq \nrm{\mb{x}}_{\nu}$.
\end{proof}

The most important Euclidean structure for our purposes is the \emph{Gaussian structure}, induced by a norm on $\mc{F}(\ell^2,X)$ first introduced by Linde and Pietsch  \cite{LP74} and discovered for the theory of Banach spaces by Figiel and Tomczak--Jaegermann \cite{FT79}. It is defined  by
\begin{equation*}
  \nrm{T}_{\gamma} := \sup \has{\E \nrmb{\sum_{k=1}^n \gamma_kTe_k}_X^2}^{1/2}, \qquad T \in \mc{F}(\ell^2,X),
\end{equation*}
where the supremum is taken over all finite orthonormal sequences $(e_k)_{k=1}^n$ in $\ell^2$. For $\mb{x} \in X^n$ the induced Euclidean structure is given by
\begin{equation*}
  \nrm{\mb{x}}_{\gamma} :=\has{\E \nrmb{\sum_{k=1}^n \gamma_kx_k}_X^2}^\frac{1}{2}, \qquad \mb{x} \in X^n,
\end{equation*}
where $(\gamma_k)_{k=1}^n$ is  Gaussian sequence (see e.g. \cite[Proposition 9.1.3]{HNVW17}). Properties \eqref{eq:E1} and \eqref{eq:E4} are trivial, and \eqref{eq:E2} is proven in \cite[Theorem 9.1.10]{HNVW17}. Therefore the Gaussian structure is an ideal global Euclidean structure.

\bigskip

Another structure of importance is the \emph{$\pi_2$-structure} induced by the $2$-summing operator ideal, which will be studied more thoroughly in Section \ref{section:hilbertfactorization}.  The $\pi_2$-norm is defined for $T \in \mc{F}(\ell^2,X)$ as
\begin{equation*}
  \nrm{T}_{\pi_2}:= \sup \cbraces{\hab{\sum_{k=1}^n  \nrm{TAe_k}_X^2}^{1/2}: A \in \mc{L}(\ell^2), \nrm{{A}}\leq 1},
\end{equation*}
where $(e_k)_{k=1}^\infty$ is an orthonormal basis for $\ell^2$.
The induced Euclidean structure for $\mb{x} \in X^n$ is
\begin{equation*}
  \nrm{\mb{x}}_{\pi_2} := \sup \cbraces{\hab{\sum_{j=1}^m\nrm{y_j}_X^2}^{1/2}: \mb{y} = \mb{A} \mb{x}, \mb{A} \in M_{m,n}(\C), \nrm{\mb{A}} \leq 1}.
\end{equation*}
Properties \eqref{eq:E1x}, \eqref{eq:E2x} and \eqref{eq:E4x} are easily checked, so the $\pi_2$-structure is an ideal global Euclidean structure as well. If $X$ is a Hilbert space, the $\pi_2$-summing norm coincides with the Hilbert-Schmidt norm, which is given by
\begin{equation*}
  \nrm{T}_{\HS} := \has{\sum_{k=1}^\infty \nrm{Te_k}^2}^{1/2}, \qquad T \in \mc{F}(\ell^2,X)
\end{equation*}
for any orthonormal basis $(e_k)_{k =1}^\infty$ of $\ell^2$. For an introduction to the theory of $p$-summing operators we refer to \cite{DJT95}.

\bigskip

If $X$ is a Banach lattice, there is an additional important Euclidean structure, the \emph{$\ell^2$-structure}. It is given by
\begin{equation*}
    \nrm{\mb{x}}_{\ell^2} := \nrms{\hab{\sum_{k=1}^n \abs{x_k}^2}^{1/2}}_X, \qquad \mb{x} \in X^n.
\end{equation*}
  Again \eqref{eq:E1x} is trivial and \eqref{eq:E2x} follows directly from
\begin{equation}\label{eq:latticel2sup}
  \has{\sum_{k=1}^n \abs{x_k}^2}^{1/2} = \sup\cbraces{\absb{ \sum_{k=1}^n a_kx_k}: \sum_{k=1}^n \abs{a_k}^2 \leq 1},
\end{equation}
where the supremum is taken in the lattice sense, see \cite[Section 1.d]{LT79}.

 By the Krivine-Grothendieck theorem \cite[Theorem~1.f.14]{LT79} we get for $S \in \mc{L}(X)$ and $\mb{x} \in X^n$ that
  \begin{equation*}
    \nrm{(Sx_1, \ldots, Sx_n)}_{\ell^2} \leq K_G \nrm{S} \nrm{\mb{x}}_{\ell^2},
  \end{equation*}
  where $K_G$ is the complex Grothendieck constant. Therefore the $\ell^2$-structure is ideal. The Krivine-Grothendieck theorem also implies that if $X$ is a Banach space that can be represented as a Banach lattice in different ways, then the corresponding $\ell^2$-structures are equivalent. This follows directly by taking $T$ the identity operator on $X$. An example of such a situation is $L^p(\R)$ for $p \in (1,\infty)$, for which the Haar basis is unconditional and induces a lattice structure different from the canonical one.

\bigskip

The $\ell^2$-structure is not a global Euclidean structure, as it is only defined for Banach lattices. However, starting from the $\ell^2$-structure on some Banach lattice $X$,  Proposition \ref{proposition:constructglobal} says that there is an ideal global Euclidean structure, which is equivalent to the $\ell^2$-structure on $X$. We define the $\gl$-structure as the structure obtained in this way starting from the lattice $L^1$. So for $\mb{x} \in X^n$ we define
  \begin{equation*}
    \nrm{\mb{x}}_{\gl} := \sup_T{\nrm{(Tx_1,\ldots,Tx_n)}_{\ell^2}},
  \end{equation*}
  where the supremum is taken over all $T\colon X \to L^1(S)$ with $\nrm{T} \leq 1$ for any measure space $(S,\mu)$. By definition this is a global, ideal Euclidean structure.

  Let us compare the Euclidean structures we have introduced.
\begin{proposition} \label{proposition:compareEuclidean} We have on $X$
\begin{enumerate}[(i)]
\item \label{it:compareEuclidean1} $\gamma \leq \pi_2$. Moreover $\pi_2 \lesssim \gamma$  if and only if $X$ has cotype 2.
\end{enumerate}
Suppose that $X$ is a Banach lattice, then we have on $X$
\begin{enumerate}[(i)]\setcounter{enumi}{1}
\item \label{it:compareEuclidean2} $\ell^2\lesssim \gamma$. Moreover $\gamma \lesssim \ell^2$  if and only if $X$ has finite cotype.
\item \label{it:compareEuclidean3} $\ell^2 \leq \gl \lesssim \ell^2$.
\end{enumerate}
\end{proposition}

\begin{proof}
For \ref{it:compareEuclidean1} let $(\gamma_k)_{k=1}^n$ be a Gaussian sequence on a probability space $(\Omega,\ms{F},\P)$. Let $f_1,\ldots,f_n \in L^2(\Omega)$ be simple functions of the form $f_k = \sum_{j=1}^m t_{jk}\ind_{A_j}$ with $t_{jk} \in \C$ and $A_j \in \ms{F}$ for $1\leq j\leq m$ and $1 \leq k \leq n$.  Define $$\mb{A} := \hab{\P(A_j)^{1/2} t_{jk}}_{j,k=1}^{m,n}.$$ Then we have for $\mb{x} \in X^n$ and $\mb{y} := \mb{A}\mb{x}$
\begin{equation*}
  \nrms{\sum_{k=1}^n f_k x_k}_{L^2(\Omega;X)} = \has{\sum_{j=1}^m \nrm{y_j}^2}^{1/2} \leq \nrm{\mb{x}}_{\pi_2} \nrm{\mb{A}}
\end{equation*}
and
\begin{equation*}
  \nrm{\mb{A}} = \sup_{\nrm{\mb{b}}_{\ell^2_m}\leq 1} \has{\sum_{j=1}^m\absb{\sum_{k=1}^n \P(A_j)^{1/2} t_{jk} b_k}^2}^{1/2} = \sup_{\nrm{\mb{b}}_{\ell^2_m}\leq 1} \nrms{\sum_{k=1}^n b_k f_k}_{L^2(\Omega)}.
\end{equation*}
Thus approximating $(\gamma_k)_{k=1}^n$ by such simple functions in $L^2(\Omega)$, we deduce
\begin{align*}
  \nrm{\mb{x}}_\gamma \leq \nrm{\mb{x}}_{\pi_2} \sup_{\nrm{\mb{b}}_{\ell^2_m}\leq 1} \nrms{\sum_{k=1}^n b_k \gamma_k}_{L^2(\Omega)} = \nrm{\mb{x}}_{\pi_2}.
\end{align*}
Suppose that $X$ has cotype $2$. By \cite[Corollary 7.2.11]{HNVW17} and the right ideal property of the $\gamma$-structure, we have for all $\mb{x} \in X^n$, $\mb{A} \in M_{m,n}(\C)$ with $\nrm{\mb{A}} \leq 1$ and $\mb{y} = \mb{A} \mb{x}$ that
\begin{equation*}
  \has{\sum_{k=1}^n\nrm{y_k}_X^2}^{1/2} \lesssim \has{\E \nrmb{\sum_{k=1}^n \gamma_ky_k}_X^2}^{1/2} =  \nrm{\mb{A} \mb{x}}_{\gamma}\leq  \nrm{\mb{x}}_{\gamma},
\end{equation*}
which implies that $\nrm{\mb{x}}_{\pi_2}\lesssim \nrm{\mb{x}}_\gamma$. Conversely, suppose that the $\gamma$-structure is equivalent to $\pi_2$-structure, then we have for all $\mb{x} \in X^n$
\begin{equation*}
  \has{\sum_{k=1}^n\nrm{x_k}_X^2}^{1/2} \leq \nrm{\mb{x}}_{\pi_2} \lesssim \nrm{\mb{x}}_{\gamma } =  \has{\E \nrmb{\sum_{k=1}^n \gamma_kx_k}_X^2}^{1/2}.
\end{equation*}
So by \cite[Corollary 7.2.11]{HNVW17} we know that $X$ has cotype $2$.

For \ref{it:compareEuclidean2} assume that $X$ is a Banach lattice.   By Proposition \ref{proposition:gaussianradermacherl2comparison} we have $\nrm{\mb{x}}_{\ell^2} \lesssim \nrm{\mb{x}}_{\gamma}$. If $X$ has finite cotype we also have $\nrm{\mb{x}}_{\gamma} \lesssim \nrm{\mb{x}}_{\ell^2}$. Conversely, if the $\ell^2$-structure is equivalent to $\gamma$-structure, then we have again by Proposition \ref{proposition:gaussianradermacherl2comparison} that
\begin{equation*}
  \has{\E \nrmb{\sum_{k=1}^n \varepsilon_kx_k}_X^2}^\frac{1}{2} \gtrsim \nrms{\hab{\sum_{k=1}^n \abs{x_k}^2}^{1/2}}_X \gtrsim \has{\E \nrmb{\sum_{k=1}^n \gamma_kx_k}_X^2}^\frac{1}{2},
\end{equation*}
where $(\varepsilon_k)_{k=1}^n$ is a Rademacher sequence.
This implies that $X$ has finite cotype by \cite[Corollary 7.3.10]{HNVW17}.

For \ref{it:compareEuclidean3} note that by the Krivine-Grothendieck theorem \cite[Theorem~1.f.14]{LT79} we have $\nrm{\mb{x}}_{\gl} \leq K_G\,\nrm{\mb{x}}_{\ell^2}$. Conversely take a positive $x^* \in X^*$ of norm one such that
\begin{equation*}
  \ips{\has{\sum_{k=1}^n\abs{x_k}^2}^{1/2},x^*} = \nrm{\mb{x}}_{\ell^2}.
\end{equation*}
Let $L$ be the completion of $X$ under the seminorm $\nrm{x}_L:= x^*(\abs{x})$. Then $L$ is an abstract $L^1$-space and is therefore order isometric to $L^1(S)$ for some measure space $(S,\mu)$, see \cite[Theorem 1.b.2]{LT79}. Let $T:X \to L$ be the natural norm one lattice homomorphism. Then we have
\begin{equation*}
  \nrm{\mb{x}}_{\ell^2} = \nrms{\hab{\sum_{k=1}^n \abs{Tx_k}^2}^{1/2}}_L \leq \nrm{\mb{x}}_{\gl}. \qedhere
\end{equation*}
\end{proof}

\subsection*{Duality of Euclidean structures}
We will now consider duality for Euclidean structures. If $\alpha$ is a Euclidean structure on a Banach space $X$, then there is a natural \emph{dual Euclidean structure} $\alpha^*$ on $X^*$ defined by
\begin{equation*}
  \nrm{\mb{x}^*}_{\alpha^*} := \sup\cbraces{\sum_{k=1}^n \abs{x_k^*(x_k)}:\mb{x} \in X^n,\nrm{\mb{x}}_\alpha \leq 1}, \qquad \mb{x}^* \in (X^*)^n.
\end{equation*}
This is indeed a Euclidean structure, as \eqref{eq:E1x} and \eqref{eq:E2x} for $\alpha^*$ follow readily from their respective counterparts for $\alpha$. We can then also induce a structure $\alpha^{**}$ on $X^{**}$, and the restriction of $\alpha^{**}$ to $X$ coincides with $\alpha$.
If $\alpha$ is ideal,  then the analogue of \eqref{eq:E3x} holds for weak$^*$-continuous operators, i.e. we have
\begin{equation*}
  \nrm{(S^*x_1^*,\ldots,S^*x_n^*)}_{\alpha^*} \leq C \, \nrm{S} \nrm{\mb{x}^*}_{\alpha^*} , \qquad \mb{x}^* \in (X^*)^n,\quad S \in \mc{L}(X).
\end{equation*}
In particular, $\alpha^*$ is ideal if $\alpha$ is ideal and $X$ is reflexive.

If we prefer to express the dual Euclidean structure in terms of a norm on $\mc{F}(\ell^2,X)$, we can employ trace duality. If $T\in \mc{F}(X)$ and we have two representations of $T$, i.e.
\begin{equation*}
    T = \sum_{k=1}^n x^*_k \otimes x_k = \sum_{j=1}^m \bar{x}^*_j \otimes \bar{x}_j,
  \end{equation*}
where $x_k,\bar{x}_j \in X$ and $x_k^*,\bar{x}_j^* \in X^*$, then $\sum_{k=1}^n \ip{x_k,x^*_k} = \sum_{j=1}^m \ip{\bar{x}_j,\bar{x}^*_j}$ (see \cite[Proposition 1.3]{Ja87}). Therefore we can define the {\em trace} of $T$ as
\begin{equation*}
  \tr(T) = \sum_{k=1}^n \ip{x_k, x_k^*}
\end{equation*}
for any finite representation of $T$.
We define the norm $\alpha^*$ on $\mc{F}(\ell^2,X^*)$ as
\begin{equation*}
  \nrm{T}_{\alpha^*} := \sup \cbrace*{\abs{\tr(S^*T)}: S \in \mc{F}(\ell^2,X) ,\alpha(S) \leq 1}, \qquad T \in \mc{F}(\ell^2,X^*)
\end{equation*}
This definition coincides with the definition in terms of vectors in $X^n$. Indeed, for $\mb{x}^* \in (X^*)^n$ and $T \in \mc{F}(\ell^2,X^*)$ defined as $T = \sum_{k=1}^n e_k \otimes x_k^*$ for some orthonormal sequence $(e_k)_{k=1}^n$ in $\ell^2$, we have that
\begin{align*}
  \nrm{\mb{x}^*}_{\alpha^*} &=
  \sup\cbraces{\sum_{k=1}^n \abs{x_k^*(x_k)}:\mb{x} \in X^n,\nrm{\mb{x}}_\alpha \leq 1}\\
  &= \sup \cbraces{\absb{\sum_{k=1}^n\ip{Se_k,Te_k}}: S \in \mc{F}(\ell^2,X) , \nrm{S}_\alpha \leq 1}\\
  &= \sup \cbrace*{\abs{\tr(S^*T)}: S \in \mc{F}(\ell^2,X) ,\nrm{S}_\alpha} =\nrm{T}_{\alpha^*}.
\end{align*}

Note that if for two Euclidean structures $\alpha$ and $\beta$ on $X$ we have $\alpha \lesssim\beta$, then $\beta^* \lesssim\alpha^*$ on $X^*$. Part of the reason why the $\gamma$- and the $\ell^2$-structure work well in practice, is the fact that they are self-dual under certain assumptions on $X$. This is contained in the following proposition, along with a few other relations between dual Euclidean structures.

\begin{proposition} \label{proposition:dualeuclidean} On $X^*$ we have
\begin{enumerate}[(i)]
\item \label{it:dualEuclidean0}$\op^*=\nu$ and $\nu^*=\op$.
  \item \label{it:dualEuclidean1} $\gamma^*\leq \gamma$. Moreover $\gamma \lesssim \gamma^*$ if and only if $X$ has nontrivial type.
  \item \label{it:dualEuclidean2} If $X$ is a Banach lattice, $(\ell^2)^*=\ell^2$
\end{enumerate}
\end{proposition}

\begin{proof}
Fix $\mb{x}^* \in (X^*)^n$. For \ref{it:dualEuclidean0} let $\mb{y}^* \in (X^*)^m$ be such that $\mb{x}^* = \mb{A} \mb{y}^*$ with $\mb{A} \in M_{n,m}(\C)$ and $\sum_{k=1}^n \abs{A_{kj}}^2 \leq 1$ for $1\leq j \leq m$. Then we have
\begin{align*}
  \nrm{\mb{x}^*}_{\op^*} &= \sup\cbraces{\sum_{k=1}^n \abs{x_k^*(x_k)}:\mb{x} \in X^n, \nrmb{\sum_{k=1}^n b_kx_k}_X\leq 1, \sum_{k=1}^n\abs{b_k}^2 \leq 1}\\
  &\leq \sup\cbraces{\sum_{j=1}^m \sum_{k=1}^n \abs{y_j^*(A_{kj}x_k)}:\mb{x} \in X^n, \nrmb{\sum_{k=1}^n b_kx_k}_X\leq 1, \sum_{k=1}^n\abs{b_k}^2 \leq 1}\\
  &\leq \sum_{j=1}^m \nrm{y_j^*}_X,
\end{align*}
so taking the infimum over all such $\mb{y}$ shows $\nrm{\mb{x}^*}_{\op^*} = \nrm{\mb{x}^*}_{\nu}$. This also implies that $\nrm{\mb{x}^{**}}_{\op^*} = \nrm{\mb{x}^{**}}_{\nu}$ for all $\mb{x}^{**} \in (X^{**})^n$. Dualizing and restricting to $X^*$ we obtain that $\nu^*=\op$ on $X^*$.

  For \ref{it:dualEuclidean1}  we have for   a Gaussian sequence $(\gamma_k)_{k=1}^n$ by H\"older's inequality that
  \begin{align*}
    \nrm{\mb{x}^*}_{\gamma^*} &= \sup \cbraces{\abss{\E \sum_{k=1}^n \ip{\gamma_k x_k, \gamma_k x_k^*}}:\mb{x} \in X^n, \nrm{\mb{x}}_\gamma \leq 1} \leq \nrm{\mb{x}^*}_{\gamma}.
  \end{align*}
 The converse estimate defines the notion of Gaussian $K$-convexity of $X$, which is equivalent to $K$-convexity of $X$ by \cite[Corollary 7.4.20]{HNVW17}. It is a deep result of Pisier \cite{Pi82} that $K$-convexity is equivalent to nontrivial type, see  \cite[Theorem 7.4.15]{HNVW17}.

For \ref{it:dualEuclidean2} we note that since $X(\ell^2_n)^* = X^*(\ell^2_n)$ by \cite[Section 1.d]{LT79}, we have
\begin{align*}
  \nrms{\hab{\sum_{k=1}^n \abs{x_k^*}^2}^{1/2}}_X = \sup \cbraces{\sum_{k=1}^n\abs{x_k^*(x_k)}: \mb{x}  \in X^n, \nrms{\hab{\sum_{k=1}^n \abs{x_k}^2}^{1/2}}_X\leq 1},
\end{align*}
so indeed $\ell^2=(\ell^2)^*$.
\end{proof}

Using  a duality argument we can compare the $\ell^2_n(X)$-norm and the $\alpha$-norm of a vector in a finite dimensional subspace of $X^n$.

\begin{proposition}\label{proposition:finitedimensionalalpha}
Let $E$ be a finite dimensional subspace of $X$. Then for $\mb{x} \in E^n$ we have
\begin{equation*}
  (\dim (E))^{-1}\has{\sum_{k=1}^n \nrm{x_k}_X^2}^{1/2} \leq \nrm{\mb{x}}_\alpha \leq \dim (E)\has{\sum_{k=1}^n \nrm{x_k}_X^2}^{1/2}
\end{equation*}
\end{proposition}

\begin{proof}
  For $\mb{x} \in E^n$ we have by Proposition \ref{proposition:operatornuclear} that
  \begin{equation*}
    \nrm{\mb{x}}_{\alpha} \leq \nrm{\mb{x}}_{\nu}  \leq \dim (E)\nrm{\mb{x}}_{\op} \leq \dim (E)\has{\sum_{k=1}^n \nrm{x_k}_X^2}^{1/2}.
  \end{equation*}
  Conversely take $\mb{x}^* \in (E^*)^n$ with $\nrm{\mb{x}^*}_{\alpha^*} \leq 1$ and $
    \nrm{\mb{x}}_\alpha = \sum_{k=1}^n x^*_k(x_k)$.
  Then
  \begin{equation*}
    \nrm{\mb{x}^*}_{\alpha^*} \leq \dim(E) \has{\sum_{k=1}^n \nrm{x_k^*}_{X^*}^2}^{1/2}
  \end{equation*}
  and therefore
  \begin{equation*}
    \has{\sum_{k=1}^n \nrm{x_k}_X^2}^{1/2} \leq \dim (E) \nrm{\mb{x}}_\alpha. \qedhere
  \end{equation*}
\end{proof}

\subsection*{Unconditionally stable Euclidean structures}
We end this section with an additional property of a Euclidean structure that will play an important role in Chapter \ref{part:4}-\ref{part:6}. We will say that a Euclidean structure $\alpha$ on $X$ is \emph{unconditionally stable} if there is a $C>0$ such that
\begin{align}
  \label{eq:unconditional1} \nrm{\mb{x}}_\alpha &\leq C \sup_{\abs{\epsilon_{k}}=1}\nrmb{\sum_{k=1}^n \epsilon_k x_k}_X,&& \mb{x} \in X^n\\
  \label{eq:unconditional2} \nrm{\mb{x}^*}_{\alpha^*} &\leq C \sup_{\abs{\epsilon_{k}}=1}\nrmb{\sum_{k=1}^n \epsilon_k x_k^*}_{X^*}&& \mb{x}^* \in (X^*)^n.
\end{align}
The next proposition gives some examples of unconditionally stable structures.

\begin{proposition}\label{proposition:unconditionallystable}
~
\begin{enumerate}[(i)]
\item \label{it1:unconditional} The $\gl$-structure on $X$ is unconditionally stable.
\item \label{it2:unconditional} If $X$ has finite cotype, the $\gamma$-structure on $X$ is unconditionally stable.
\item \label{it3:unconditional} If $X$ is a Banach lattice, the $\ell^2$-structure on $X$ is unconditionally stable.
\end{enumerate}
\end{proposition}

\begin{proof}
  Fix $\mb{x} \in X^n$ and $\mb{x}^* \in (X^*)^n$. For \ref{it1:unconditional} let $V\colon X \to L^1(S)$ be a norm-one operator. Then by Proposition \ref{proposition:gaussianradermacherl2comparison} we have
  \begin{equation*}
    \nrm{(Vx_1,\ldots,Vx_n)}_{\ell^2} \lesssim  \E\nrmb{\sum_{k=1}^n\varepsilon_k Vx_k}_{L^1(S)} \leq \sup_{\abs{\epsilon_{k}}=1}\nrmb{\sum_{k=1}^n \epsilon_k x_k}_X,
  \end{equation*}
  where $(\varepsilon_k)_{k \geq 1}$ is a Rademacher sequence. So taking the supremum over all such $V$ yields \eqref{eq:unconditional1}. Now suppose that
  \begin{equation*}
    \sup_{\abs{\epsilon_{k}}=1}\nrmb{\sum_{k=1}^n \epsilon_k x_k^*}_{X^*} = 1.
  \end{equation*}
  Define $V:X \to \ell^1_n$ by $Vx = \hab{x_1^*(x),\ldots,x_n^*(x)}$, for which we have $\nrm{V}\leq 1$. Suppose that $\nrm{\mb{x}}_{\gl} \leq 1$. Then $\nrm{(Vx_1,\ldots,Vx_n)}_{\ell^2} \leq 1$, i.e.
  \begin{equation*}
    \sum_{j=1}^n \has{\sum_{k=1}^n \abs{x_j^*(x_k)}^2}^{1/2}\leq 1
  \end{equation*}
  and hence
  \begin{equation*}
    \sum_{j=1}^n  \abs{x_j^*(x_j)}\leq 1.
  \end{equation*}
  This means that $\nrm{\mb{x}}_{(\gl)^*} \leq 1$, so \eqref{eq:unconditional2} follows.

  For \ref{it2:unconditional}, we have by Proposition \ref{proposition:gaussianradermacherl2comparison} that
  \begin{align*}
    \nrm{\mb{x}}_{\gamma} &\lesssim \has{\E\nrmb{\sum_{k=1}^n\varepsilon_k x_k}_{X}^2}^{1/2}\leq  \sup_{\abs{\epsilon_{k}}=1}\nrmb{\sum_{k=1}^n \epsilon_k x_k}_X,
\end{align*}
where $(\varepsilon_k)_{k \geq 1}$ is a Rademacher sequence. For \eqref{eq:unconditional2} assume that $\nrm{\mb{x}}_\gamma \leq 1$. Then again by Proposition \ref{proposition:gaussianradermacherl2comparison} we have
\begin{align*}
  \abss{\sum_{k=1}^n \ip{x_k,x_k^*}} &= \abss{\E  \ipb{\sum_{k=1}^n\varepsilon_kx_k,\sum_{k=1}^n\varepsilon_kx_k^*}}\\
  &\leq \has{\E\nrmb{\sum_{k=1}^n\varepsilon_k x_k}_{X}^2}^{1/2} \has{\E\nrmb{\sum_{k=1}^n\varepsilon_k x_k^*}_{X^*}^2}^{1/2}\\
  &\lesssim \sup_{\abs{\epsilon_{k}}=1}\nrmb{\sum_{k=1}^n \epsilon_k x_k^*}_{X^*}.
\end{align*}

Finally \ref{it3:unconditional} follows from \ref{it1:unconditional} and the equivalence of the $\ell^2$-structure and the $\gl$-structure, see Proposition \ref{proposition:compareEuclidean}.
\end{proof}

\section{\texorpdfstring{$\alpha$}{a}-bounded operator families}\label{section:alphabounded}
Having introduced Euclidean structures in the preceding section, we will now connect Euclidean structures to operator theory.
  \begin{definition}
    Let $\alpha$ be a Euclidean structure on $X$. A family of operators $\Gamma \subseteq \mc{L}(X)$ is called {\em $\alpha$-bounded} if
\begin{equation*}
  \nrm{\Gamma}_\alpha:= \sup \cbraceb{\nrm{(T_1x_1,\ldots,T_nx_n)}_{\alpha}:T_k \in \Gamma, \mb{x} \in X^n,\nrm{\mb{x}}_\alpha \leq 1}
\end{equation*}
is finite. If $\alpha$ is a global Euclidean structure, this definition can analogously be given for $\Gamma \subseteq \mc{L}(X,Y)$, where $Y$ is another Banach space.
  \end{definition}
 We allow repetitions of the operators in the definition of $\alpha$-boundedness. In the case that $\alpha = \ell^2$ it is known that it is equivalent to test the definition only for distinct operators, see \cite[Lemma 4.3]{KVW16}. For $\alpha=\gamma$ this is an open problem.

 Closely related to $\gamma$ and $\ell^2$-boundedness is the notion of $\mc{R}$-boundedness. We say that $\Gamma\subseteq \mc{L}(X)$ is  \emph{$\mc{R}$-bounded} if there is a $C>0$ such that for all $\mb{x} \in X^n$
\begin{equation*}
  \has{\E\nrmb{\sum_{k=1}^n \varepsilon_kT_kx_k}^2}^{1/2} \leq C \, \has{\E\nrmb{\sum_{k=1}^n \varepsilon_kx_k}^2}^{1/2}, \qquad T_k \in \Gamma,
\end{equation*}
where $(\varepsilon_k)_{k=1}^\infty$ is a Rademacher sequence.
Note that the involved $\mc{R}$-norms do not form a Euclidean structure, as they do not satisfy \eqref{eq:E2x}. However, we have the following connections (see \cite{KVW16}):
\begin{itemize}
  \item $\mc{R}$-boundedness implies $\gamma$-boundedness. Moreover $\gamma$-boundedness and $\mc{R}$-boundedness are equivalent on $X$ if and only if $X$ has finite cotype.
  \item  $\ell^2$-boundedness, $\gamma$-boundedness and $\mc{R}$-boundedness are equivalent on a Banach lattice $X$ if and only if $X$ has finite cotype.
\end{itemize}
Following the breakthrough papers \cite{CDSW00,We01b}, $\gamma$- and $\ell^2$- and $\mc{R}$-bound\-ed\-ness have played a major role in the development of vector-valued analysis over the past decades (see e.g. \cite[Chapter 8]{HNVW17}).

\bigskip

We call an operator $T \in \mc{L}(X)$ $\alpha$-bounded if $\{T\}$ is $\alpha$-bounded. It is not always the case that any $T \in \mc{L}(X)$ is $\alpha$-bounded. In fact we have the following characterization:

\begin{proposition}\label{proposition:idealsingleoperator}
  Let $\alpha$ be a Euclidean structure on $X$. Then every $T \in \mc{L}(X)$ is $\alpha$-bounded with $\nrm{\cbrace{T}}_\alpha \leq C\, \nrm{T}$ if and only if $\alpha$ is ideal with constant $C$.
\end{proposition}
\begin{proof}
  First assume that $\alpha$ is ideal with constant $C$. Then we have for all $T \in \mc{L}(X)$
  \begin{align*}
    \nrm{\{T\}}_\alpha &= \sup\cbrace{\nrm{(Tx_1, \ldots Tx_n)}_\alpha:\mb{x} \in X^n,\nrm{\mb{x}}_\alpha \leq 1} \leq C\, \nrm{T},
  \end{align*}
  where $C$ is the ideal constant of $\alpha$.
  Now suppose that for all $T \in \mc{L}(X)$ we have $\nrm{\cbrace{T}}_\alpha \leq C\, \nrm{T}$. Then for $\mb{x} \in X^n$ and $T \in \mc{L}(X)$ we have
  \begin{align*}
    \nrm{(Tx_1, \ldots Tx_n)}_\alpha \leq \nrm{\{T\}}_\alpha \nrm{\mb{x}}_\alpha \leq C \,\nrm{T} \nrm{\mb{x}}_\alpha,
  \end{align*}
  so $\alpha$ is ideal with constant $C$.
\end{proof}

Next we establish some basic properties of $\alpha$-bounded families of operators.
\begin{proposition}
\label{proposition:alphaproperties}
   Let $\alpha$ be a Euclidean structure on  $X$ and let $\Gamma, \Gamma' \subseteq \mc{L}(X)$ be $\alpha$-bounded.
  \begin{enumerate}[(i)]
    \item \label{it:propalha1} For $\Gamma'' = \cbrace{TT':T \in \Gamma, T' \in \Gamma'}$ we have $\nrm{\Gamma''}_\alpha \leq \nrm{\Gamma}_\alpha\nrm{\Gamma'}_\alpha$.
    \item \label{it:propalha2} For $\Gamma^* = \cbrace{T^*:T \in \Gamma}$ we have $\nrm{\Gamma^*}_{\alpha^*} = \nrm{\Gamma}_{\alpha}$.
    \item \label{it:propalha3}For $\alpha$-bounded $\Gamma_k \subseteq \mc{L}(X)$ for $k \in \N$ we have $\nrm{\bigcup_{k=1}^\infty \Gamma_k}_\alpha \leq \sum_{k=1}^\infty \nrm{\Gamma_k}_\alpha $.
    \item \label{it:propalha4} For the absolutely convex hull $\tilde\Gamma$  of $\Gamma$ we have $\nrm{\tilde\Gamma}_\alpha = \nrm{\Gamma}_\alpha$.
    \item \label{it:propalha5} For the closure of $\Gamma$ in the strong operator topology $\tilde\Gamma$ we have $\nrm{\tilde\Gamma}_\alpha = \nrm{\Gamma}_\alpha$.
  \end{enumerate}
\end{proposition}
\begin{proof}
\ref{it:propalha1}  is immediate from the definition, \ref{it:propalha2} is a consequence of our definition of duality, \ref{it:propalha3} follows from the triangle inequality and \ref{it:propalha5} is clear from the definition of an $\alpha$-bounded family of operators.

  For \ref{it:propalha4} we first note that $\nrm{\cup_{0\leq \theta \leq 2 \pi} e^{i\theta}\Gamma}_\alpha = \nrm{\Gamma}_\alpha$. It remains to check that $\nrm{\text{conv}(\Gamma)}_\alpha = \nrm{\Gamma}_\alpha$. Suppose that for $1\leq j \leq m$ we have $S_j = \sum_{k=1}^n a_{jk} T_k$ where $T_1,\ldots,T_n \in \Gamma$, $a_{jk} \geq 0$ and $\sum_{k=1}^n a_{jk}=1$ for $1\leq j \leq m$. Let $(\xi_j)_{j=1}^m$ be a sequence of independent random variables with $\P(\xi_j=k)=a_{jk}$ for $1 \leq k \leq n$. Then
  \begin{align*}
    \nrm{(S_1x_1,\ldots,S_nx_n)}_\alpha &= \nrm{\E(T_{\xi_1}x_1,\ldots,T_{\xi_n}x_n)}_\alpha\\
    &\leq \E\nrmb{(T_{\xi_1}x_1,\ldots,T_{\xi_n}x_n)}_\alpha\\
    &\leq \nrm{\Gamma}_\alpha
  \end{align*}
  for all $\mb{x} \in X^n$ with $\nrm{\mb{x}}_\alpha \leq 1$, so $\nrm{\text{conv}(\Gamma)}_\alpha = \nrm{\Gamma}_\alpha$.
\end{proof}

As a corollary of Proposition \ref{proposition:alphaproperties}\ref{it:propalha4} and \ref{it:propalha5} we also have the $\alpha$-boundedness of $L^1$-integral means of $\alpha$-bounded sets. Moreover from the triangle inequality for $\nrm{\cdot}_{\alpha}$ we obtain boundedness of $L^\infty$-integral means. If $\alpha=\gamma$, there is a scale of results under type and cotype assumptions (see \cite{HV09}).

\begin{corollary}\label{corollary:L1mean}
  Let $\alpha$ be a Euclidean structure on a Banach space $X$, let $(S,\mu)$ be a measure space and let  $f:S \to \Gamma$ be strongly measurable.
  \begin{enumerate}[(i)]
    \item If $\Gamma:= \cbrace{f(s):s \in S}$ is $\alpha$ bounded, then the set
    \begin{equation*}
      \Gamma_f^1 := \cbraces{\int_S\varphi(s)f(s)\dd s:\nrm{\varphi}_{L^1(S)}\leq 1}
    \end{equation*}
    is $\alpha$-bounded with $\nrm{\Gamma_f^1}_{\alpha} \leq \nrm{\Gamma}$.
    \item If $\int_S\nrm{f}\dd \mu <\infty$ and $\alpha$ is ideal, then the set
    \begin{equation*}
      \Gamma_f^\infty := \cbraces{\int_S\varphi(s)f(s)\dd s:\nrm{\varphi}_{L^\infty(S)}\leq 1}
    \end{equation*}
    is $\alpha$-bounded with $\nrm{\Gamma_f^\infty}_{\alpha} \lesssim\int_S\nrm{f}\dd \mu$.
  \end{enumerate}
\end{corollary}

\subsection*{A technical lemma}
We end this section with a technical lemma that will be crucial in the representation theorems in this chapter, as well as in the more concrete factorization theorems in Chapter \ref{part:2}. The proof of this lemma (in the case $\Gamma=\varnothing$) is a variation of the proof of \cite[Theorem 7.3.4]{AK16}, which is the key ingredient to prove the Maurey-Kwapie\'n theorem on factorization of an operator $T:X \to Y$ through a Hilbert space (see \cite{Kw72,Ma74}). We will make the connection to the Maurey-Kwapie\'n factorization theorem clear in Section \ref{section:hilbertfactorization}, where we will prove a generalization of that theorem.

\begin{lemma}
\label{lemma:technical}
  Let $\alpha$ be a Euclidean structure on $X$ and let $Y\subseteq X$ be a subspace. Suppose that $F:X \to [0,\infty)$ and $G:Y \to [0,\infty)$ are two positive homogeneous functions such that
  \begin{align}
    \label{eq:F}
    \has{\sum_{k=1}^nF(x_k)^2}^\frac{1}{2} &\leq \nrm{\mb{x}}_\alpha, &&\mb{x} \in X^n,\\
    \label{eq:G}
    \nrm{\mb{y}}_\alpha &\leq \has{\sum_{k=1}^nG(y_k)^2}^\frac{1}{2},  &&\mb{y} \in Y^n.
  \end{align}
   Let  $\Gamma \subseteq \mc{L}(X)$ an be $\alpha$-bounded family of operators.
  Then there exists a $\Gamma$-invariant subspace $Y\subseteq X_0 \subseteq X$ and a Hilbertian seminorm $\nrm{\lcdot}_0$ on $X_0$ such that
  \begin{align}
      \nrm{Tx}_0 &\leq 2 \, \nrm{\Gamma}_\alpha \nrm{x}_0 &&\qquad x \in X_0, T \in \Gamma, \label{eq:F3}\\
    \nrm{x}_0 &\geq F(x) &&\qquad x \in X_0 \label{eq:F1},\\
    \nrm{x}_0 &\leq 4G(x) &&\qquad x \in Y \label{eq:F2}.
  \end{align}
\end{lemma}

\begin{proof}
Let $X_0$ be the smallest $\Gamma$-invariant subspace of $X$ containing $Y$, i.e. set $Y_0 := Y$, define for $N \geq 1$
  \begin{equation*}
    Y_N := \cbraces{T x: T  \in \Gamma, x \in Y_{N-1}}.
  \end{equation*}
  and take $X_0 := \bigcup_{N\geq 0} Y_N$. We will prove the lemma in three steps.

  \bigskip

  \textbf{Step 1:} We will first show that $G$ can be extended to a function $G_0$ on $X_0$, such that $2G_0$ satisfies \eqref{eq:G} for all $\mb{y} \in X_0^n$. For this pick a sequence of real numbers $(a_N)_{N=1}^\infty$ such that $a_N>1$ and $\prod_{N=1}^\infty a_N=2$ and define $b_M := \prod_{N=1}^M a_N$. For $y \in Y$ we set $G_0(y) = G(y)$ and proceed by induction. Suppose that $G_0$ is defined on $\bigcup_{N=0}^M  Y_N$ for some $M \in \N$ with
  \begin{equation}\label{eq:partialG}
    \nrm{\mb{y}}_\alpha \leq b_M \has{\sum_{k=1}^nG_0(y_k)^2}^{1/2}
  \end{equation}
   for any $\mb{y} \in \hab{\bigcup_{N=0}^{M}Y_N}^n$.

   For $y \in Y_{M+1}\setminus \bigcup_{N=0}^M \limits Y_N$ pick a  $T \in \Gamma$ and an $x \in Y_M$ such that $Tx =y$ and define
  \begin{equation*}
    G_0(y) := {\frac{1}{a_{M+1}-1}} \nrm{\Gamma}_\alpha \cdot G_0(x).
  \end{equation*}
  For $\mb{y} \in \hab{\bigcup_{N=0}^{M+1}Y_N}^n$ we let $\mc{I} =\cbrace{k: y_k \in \bigcup_{N=0}^{M}Y_N}$. For $k \notin \mc{I}$ we let $T_k \in \Gamma$ and $x_k \in \bigcup_{N=0}^{M}Y_N$ be such that $T_kx_k = y_k$.
  Then by our definition of $G_0$ we have
  \begin{align*}
     \nrm{\mb{y}}_{\alpha}&\leq\nrmb{(\ind_{k \in \mc{I}}y_k)_{k=1}^n}_\alpha + \nrmb{(\ind_{k \notin \mc{I}}y_k)_{k=1}^n}_\alpha\\
    &\leq b_M \has{\sum_{k \in \mc{I}} G_0(y_k)^{2}}^{1/2}+ b_M  \nrm{\Gamma}_\alpha \has{\sum_{k \notin \mc{I}} G_0(x_k)^{2}}^{1/2}\\
    &\leq b_M \has{\sum_{k =1}^n G_0(y_k)^{2}}^{1/2}+ b_M \ha{{a_{M+1}-1}}  \has{\sum_{k =1}^n G_0(y_k)^{2}}^{1/2}\\
 &= b_{M+1} \has{\sum_{k =1}^n G_0(y_k)^{2}}^{1/2}
  \end{align*}
 So $G_0$ satisfies \eqref{eq:partialG} for $M+1$. Therefore, by induction we can define $G_0$ on $X_0$, such that $2G_0$ satisfies \eqref{eq:G} for all $\mb{y} \in X_0^n$.

  \bigskip

\textbf{Step 2:} For $x \in X$ define the function $\phi_x: X^* \to \R_+$ by $\phi_x(x^*) := \abs{x^*(x)}^2$. We will construct a sublinear functional on the space
\begin{equation*}
  \mbb{V} := \spn\cbrace{\phi_x:x \in X_0}.
\end{equation*}
For this note that every $\psi \in \mbb{V}$ has a representation of the form
\begin{equation}
	\label{eq:pdef}
    \psi = \sum_{k=1}^{n_u} \phi_{u_k} - \sum_{k=1}^{n_v} \phi_{v_k} + \sum_{k=1}^{n_{x}}\ha*{\phi_{T_{k}x_{k}} - \phi_{2\nrm{\Gamma}_\alpha x_{k}}}
\end{equation}
  with $u_k \in X_0$, $v_k, x_{k} \in X$ and $T_{k} \in \Gamma$. Define $p:\mbb{V} \to [-\infty,\infty)$ by
  \begin{equation*}
     p(\psi) = \inf \cbrace*{16 \sum_{k=1}^{n_u} G_0(u_k)^2 - \sum_{k=1}^{n_v} F(v_k)^2},
  \end{equation*}
  where the infimum is taken over all representations of $\psi$ in the form of \eqref{eq:pdef}. This functional clearly has the following properties
  \begin{align}
	 p(a \psi) &= a p(\psi), &&\quad \psi \in \mbb{V}, a > 0, \label{eq:p1}\\
	 p(\psi_1 + \psi_2) &\leq p(\psi_1) + p(\psi_2), &&\quad \psi_1,\psi_2 \in \mbb{V}, \label{eq:p2}\\
	 p(\phi_{Tx} - \phi_{2\nrm{\Gamma}_\alpha x}) &\leq 0, &&\quad x \in X_0, T \in \Gamma,\label{eq:p3}\\
p(-\phi_x) &\leq -F(x)^2, &&\quad x\in X_0,\label{eq:p4}\\
	 p(\phi_x) &\leq 16 G_0(x)^2, &&\quad x \in X_0.\label{eq:p5}
  \end{align}
  We will check that $p(0)=0$. It is clear that $p(0) \leq 0$. Let
  \begin{equation*}
	 0 = \sum_{k=1}^{n_u} \phi_{u_k} - \sum_{k=1}^{n_v} \phi_{v_k} +  \sum_{k=1}^{n_{x}}\ha*{\phi_{T_{k}x_{k}} - \phi_{2\nrm{\Gamma}_\alpha x_{k}}}
  \end{equation*}
be a representation of the form of \eqref{eq:pdef}. So for any $x^* \in X^*$ we have
\begin{equation}
\label{eq:repeq}
\begin{aligned}
  	\sum_{k=1}^{n_u} \abs{x^*(u_k)}^2 +  \sum_{k=1}^{n_{x}}\abs{x^*(T_{k}x_{k})}^2 = \sum_{k=1}^{n_v}\abs{x^*(v_k)}^2  + \sum_{k=1}^{n_{x}} \abs{x^*(2\nrm{\Gamma}_\alpha x_{k})}^2.
\end{aligned}
\end{equation}
Let
\begin{align*}
	\mb{u} &:= (u_k)_{k=1}^{n_u} \in X_0^{n_u},&\mb{v} &:= (v_k)_{k=1}^{n_v} \in X^{n_v},\\
 \mb{x} &:= (x_{k})_{k=1}^{n_x} \in X^{n_x}, & \mb{y}&:= (T_{k}{x_{k}})_{k=1}^{n_x} \in X^{n}
\end{align*}
and define the vectors
\begin{equation*}
	\mb{\bar{u}} = \begin{pmatrix}	\mb{u} \\ \mb{y} \end{pmatrix}, \qquad \mb{\bar{v}} = \begin{pmatrix}	 \mb{v} \\ 2 \nrm{\Gamma}_\alpha \mb{x} \end{pmatrix}.
\end{equation*}
Note that \eqref{eq:repeq} implies, by the Hahn-Banach theorem, that
\begin{equation*}
  v_1,\cdots,v_{n_v}, x_{1},\cdots,x_{n_x} \in \spn \cbraceb{u_1,\cdots,u_{n_u}, T_{1}x_{1},\cdots,T_{n_x}x_{n_x}}.
\end{equation*}
Therefore there exists a matrix $\mb{A}$ with $\nrm{\mb{A}}=1$ such that $\mb{\bar{v}} = \mb{A} \mb{\bar{u}}$. Thus by property \eqref{eq:E2x} of a Euclidean structure we get that $\nrm{\mb{\bar{v}}}_\alpha \leq \nrm{\mb{\bar{u}}}_\alpha$. Now we have, using the triangle inequality, that
\begin{equation*}
	\nrm*{\mb{\bar{v}}}_\alpha\leq \nrm{\mb{\bar{u}}}_\alpha
	 \leq \nrm{\mb{u}}_\alpha + \frac{1}{2} \nrmb{2 \nrm{\Gamma}_\alpha \mb{x}}_\alpha
	\leq \nrm{\mb{u}}_\alpha + \frac{1}{2} \nrm*{\mb{\bar{v}}}_\alpha	,
\end{equation*}
which means that
\begin{equation*}
	\nrm{\mb{v}}_\alpha \leq	\nrm*{\mb{\bar{v}}}_\alpha \leq 2 \nrm{\mb{u}}_\alpha.
\end{equation*}
By assumption \eqref{eq:F} on $F$ and \eqref{eq:G} on $2G_0$ we have
\begin{align*}
\sum_{k=1}^{n_v}  F(v_k)^2 \leq \nrm{\mb{v}}^2_\alpha \leq 4 \nrm{\mb{u}}^2_\alpha \leq 16 \sum_{k=1}^{n_u}  G_0(u_k)^2,
\end{align*}
which means that $p(0) \geq 0$ and thus $p(0) =0$. Now with property \eqref{eq:p2} of $p$ we have
\begin{equation*}
	p(\psi) + p(-\psi) \geq p(0) = 0,
\end{equation*}
so $p(\psi) > - \infty$ for all $\psi \in \mbb{V}$. Combined with properties \eqref{eq:p1} and \eqref{eq:p2} this means that $p$ is a sublinear functional.

\bigskip

\textbf{Step 3.} To complete the prove of the lemma, we construct a semi-inner product from our sublinear functional $p$ using Hahn--Banach. Indeed, by applying the Hahn-Banach theorem \cite[Theorem~3.2]{Ru91}, we obtain a linear function $f$ on $\mbb{V}$ such that $f(\psi) \leq p(\psi)$ for all $\psi \in \mbb{V}$. By property \eqref{eq:p4} we know that $p(-\phi_x) \leq 0$ and thus $f(\phi_x) \geq 0$ for all $x \in X_0$.
We take the complexification of $\mbb{V}$
\begin{equation*}
  \mbb{V}^\C = \cbrace*{v_1+iv_2:v_1,v_2 \in \mbb{V}}
\end{equation*}
with addition and scalar multiplication defined as usual.
We extend $f$ to a complex linear functional on this space by $f(v_1+iv_2) = f(v_1)+if(v_2)$ and define a pseudo-inner product on $X_0$ by $\ip{x,y} = f(\rho_{x,y})$ with $\rho_{x,y}:X^* \to \C$ defined as $\rho_{x,y}(x^*) = x^*(x) \overline{x^*(y)}$ for all $x^* \in X^*$. This is well-defined since
\begin{equation*}
  \rho_{x,y} = \frac{1}{4} \ha*{\phi_{x+y}- \phi_{x-y} + i \phi_{x+iy} -i\phi_{x-iy}} \in \mbb{V}^\C.
\end{equation*}

On $X_0$ we define $\nrm{\lcdot}_0$ as the seminorm induced by this semi-inner product, i.e. $\nrm{x}_0 := \sqrt{\ip{x,x}} = \sqrt{f(\phi_x)}$. Then for $x \in X_0$ and $T \in \Gamma$ we have by property \eqref{eq:p3} of $p$
\begin{equation*}
  \nrm{Tx}^2_0 \leq p(\phi_{Tx}-\phi_{2 \nrm{\Gamma}_\alpha x}) + f(\phi_{2 \nrm{\Gamma}_\alpha x}) \leq 4\nrm{\Gamma}_\alpha^2\nrm{x}^2_0.
\end{equation*}
 By property \eqref{eq:p4} of $p$  we have
\begin{equation*}
  \nrm{x}^2_0 = f(\phi_x) \geq -p(-\phi_x) \geq F(x)^2, \qquad x \in X_0,
\end{equation*}
and by property \eqref{eq:p5} of $p$ we have
\begin{equation*}
  \nrm{y}^2_0 = f(\phi_y) \leq p(\phi_y) \leq 16G_0(y)^2 = 16G(y)^2, \qquad y \in Y.
\end{equation*}
So $\nrm{\lcdot}_0$ satisfies \eqref{eq:F3}-\eqref{eq:F2}.
\end{proof}

\section[The representation of \texorpdfstring{$\alpha$}{a}-bounded operator families]{The representation of \texorpdfstring{$\alpha$}{a}-bounded operator families on a Hilbert space}\label{section:representation}
We will now represent an $\alpha$-bounded family of operators $\Gamma$ as a corresponding family of operators $\widetilde{\Gamma}\subseteq \mc{L}(H)$ for some Hilbert space $H$. As a preparation we record an important special case of  Lemma \ref{lemma:technical}.
\begin{lemma}
\label{lemma:technicalcase}
  Let $\alpha$ be a Euclidean structure on $X$ and let $\Gamma \subseteq \mc{L}(X)$ be $\alpha$-bounded. Then for any $\eta = (y_0,y_1) \in X\times X$ there exists a $\Gamma$-invariant subspace $X_\eta\subseteq X$ with $y_0 \in X_\eta$ and a Hilbertian seminorm $\nrm{\lcdot}_{\eta}$ on $X_{\eta}$ such that
  \begin{align}
      \nrm{Tx}_\eta &\leq 2\nrm{\Gamma}_\alpha \nrm{x}_\eta &&\qquad x \in X_\eta, \, T \in \Gamma, \label{eq:eta3}.\\
          \nrm{y_0}_\eta & \leq 4 \nrm{y_0}_X\label{eq:eta2}\\
    \nrm{y_1}_\eta & \geq \nrm{y_1}_X &&\qquad \text{if }y_1 \in X_\eta \label{eq:eta1}
  \end{align}
\end{lemma}

\begin{proof}
  Define $F_\eta:X \to [0,\infty)$ as
  \begin{equation*}
    F_\eta(x) = \begin{cases}
   \nrm{x}_X & \text{if } x \in \spn{\cbrace{y_1}}, \\
   0       & \text{otherwise},
    \end{cases}
  \end{equation*}
  let $Y = \spn\{y_0\}$ and define $G_\eta:Y \to [0,\infty)$ as $G_\eta(x) = \nrm{x}_X$.
   Then $F_\eta$ and $G_\eta$ satisfy \eqref{eq:F} and \eqref{eq:G} by Proposition \ref{proposition:finitedimensionalalpha}, so by Lemma \ref{lemma:technical} we can find a $\Gamma$-invariant subspace $X_\eta$ of $X$ containing $y_0$ and a seminorm $\nrm{\lcdot}_\eta$ on $X_\eta$ induced by a semi-inner product for which \eqref{eq:F3}-\eqref{eq:F2} hold, from which \eqref{eq:eta3}-\eqref{eq:eta1} directly follow.
\end{proof}

With Lemma \ref{lemma:technicalcase} we can now represent a $\alpha$-bounded family of Banach space operators on a Hilbert space. Note that by Proposition \ref{proposition:alphaproperties} we know that without loss of generality we can restrict to families of operators that are absolutely convex and closed in the strong operator topology.

\begin{theorem}
\label{theorem:euclideanrepresentation}
  Let $\alpha$ be a Euclidean structure on $X$ and let $\Gamma \subseteq \mc{L}(X)$ be absolutely convex, closed in the strong operator topology and $\alpha$-bounded. Define $\nrm{T}_\Gamma = \inf \cbrace*{\lambda > 0: \lambda^{-1}T \in \Gamma}$ on the linear span of $\Gamma$ denoted by $\mc{L}_\Gamma(X)$. Then there is a Hilbert space $H$, a closed subalgebra $\mc{B}$ of $\mc{L}(H)$, a bounded algebra homomorphism $\rho\colon \mc{B} \to \mc{L}(X)$ and a bounded linear operator $\tau\colon\mc{L}_\Gamma(X) \to \mc{B}$ such that
  \begin{align*}
    \rho\tau(T) &= T, \qquad T \in \mc{L}_\Gamma(X),\\
    \nrm{\rho} &\leq 4,\\
    \nrm{\tau} &\leq 2 \nrm{\Gamma}_\alpha.
  \end{align*}
  Furthermore, if $\mc{A}$ is the algebra generated by $\Gamma$, $\tau$ extends to an algebra homomorphism of $\mc{A}$ into $\mc{B}$ such that $\rho\tau(S)=S$ for all $S \in \mc{A}$.
\end{theorem}

\begin{proof}
  Let $\mc{A}$ be the algebra generated by $\Gamma$. For any $\eta \in X \times X$ we let $(X_\eta,\nrm{\lcdot}_\eta)$ be as in Lemma \ref{lemma:technicalcase} and take $N_\eta =  \cbrace{x \in X_\eta: \nrm{x}_\eta = 0}$. Let $H_\eta$ be the completion of the quotient space $X_\eta / N_\eta$, which is a Hilbert space. Let $\pi_\eta:\mc{A} \to \mc{L}(H_\eta)$ be the algebra homomorphism mapping elements of $\mc{A}$ to their representation on $H_\eta$, which is well-defined since $X_\eta$ is $\mc{A}$ invariant.

  Define $E = \cbrace{(x,Sx):x \in X, S \in \mc{A}} \subseteq X \times X$. We define the Hilbert space $H$ by the direct sum
  $ H = \oplus_{\eta \in E} H_{\eta} $
  with norm $\nrm{\lcdot}_H$ given by
  \begin{equation*}
    \nrm{h}_H = \has{\sum_{\eta \in E} \nrm{h_\eta}_{\eta}^2}^\frac{1}{2}
  \end{equation*}
  for $h \in H$ with $h = (h_\eta)_{\eta \in E}$. Furthermore we define the algebra homomorphism $\tau:\mc{A} \to \mc{L}(H)$ by $\tau = \oplus_{\eta \in E} \pi_{\eta}$.

  For all $T \in \mc{L}_\Gamma(X)$ we then have
  \begin{align*}
    \nrm{\tau(T)}
    &= \nrm{T}_\Gamma \sup \cbraces{\has{\sum_{\eta \in E}\nrms{\pi_\eta\has{\frac{T}{\nrm{T}_\Gamma}} (x_\eta)}^2_\eta}^\frac{1}{2}: \nrmb{(x_\eta)_{\eta \in E}} \leq 1}
    \\&\leq2\nrm{\Gamma}_{\alpha}\nrm{T}_\Gamma.
  \end{align*}
   Therefore the restriction $\tau|_{\Gamma}:\mc{L}_\Gamma(X) \to \mc{L}(H)$ is a bounded linear operator with $\nrm{\tau|_\Gamma} \leq 2\nrm{\Gamma}_{\alpha}$.
   Now for $S \in \mc{A}$ and $x \in X$ with $\nrm{x}_X \leq 1$ define $\zeta = (x,Sx)$. We have
  \begin{align*}
    \nrm{\tau(S)} &\geq \sup_{\eta \in E}\nrm{\pi_\eta (S)} \geq \nrm{\pi_{\zeta} (S)(x)}_{\zeta} \nrm{x}_{\zeta}^{-1} \geq \nrm{Sx}_X\cdot\ha*{4\nrm{x}_X}^{-1}
  \end{align*}
  using \eqref{eq:eta2} and $\eqref{eq:eta1}$. So $\nrm{\tau(S)} \geq \frac{1}{4}\nrm{S}$, which means that $\tau$ is injective.
  If we now define $\mc{B}$ as the closure of $\tau(\mc{A})$ in $\mc{L}(H)$, we can extend $\rho = \tau^{-1}$ to an algebra homomorphism $\rho:\mc{B} \to \mc{L}(X)$ with $\nrm{\rho} \leq 4$ since $\nrm{\tau} \geq \frac{1}{4}$. This proves the theorem.
\end{proof}

\section{The equivalence of \texorpdfstring{$\alpha$}{a}-boundedness and \texorpdfstring{$C^*$}{C*}-boundedness }\label{sec:cbounded}
There is also a converse to Theorem \ref{theorem:euclideanrepresentation}, for which we will have to make a detour into operator theory. We will introduce  matricial algebra norms in order to connect $\alpha$-boundedness of a family of operators to the theory of completely bounded maps. For background on the theory developed in this section we refer to  \cite{BL04,ER00,Pa02,Pi03}.

\subsection*{Matricial algebra norms}
 Denote the space of $m \times n$-matrices with entries in a complex algebra $\mc{A}$ by $M_{m,n}(\mc{A})$. A \emph{matricial algebra norm} on $\mc{A}$ is a norm $\nrm{\lcdot}_{\mc{A}}$ defined on each $M_{m,n}(\mc{A})$ such that
  \begin{align*}
    \nrm{\mb{S}\mb{T}}_{\mc{A}} &\leq \nrm{\mb{S}}_{\mc{A}}\nrm{\mb{T}}_{\mc{A}}, && \mb{S} \in M_{m,k}(\mc{A}),\mb{T} \in M_{k,n}(\mc{A})\\
    \nrm{\mb{A}\mb{T}\mb{B}}_{\mc{A}} &\leq \nrm{\mb{A}}\nrm{\mb{T}}_{\mc{A}}\nrm{\mb{B}}, && \mb{A} \in M_{m,j}(\C),\mb{T} \in M_{j,k}(\mc{A}),\mb{B} \in M_{k,n}(\C).
  \end{align*}
The algebra $\mc{A}$ with an associated matricial algebra norm will be called a {\em matricial normed algebra}. In the case that $\mc{A}\subseteq \mc{L}(X)$, we call a matricial algebra norm {\em coherent} if the norm of a $1 \times 1$-matrix is the operator norm of its entry, i.e. if $\nrm{(T)}_{\mc{A}} = \nrm{T}$ for all $T \in \mc{A}$.

The following example shows that any Euclidean structure induces a  matricial algebra norm on $\mc{L}(X)$.
\begin{example}
\label{example:hatalpha}
Let $\alpha$ be a Euclidean structure on $X$. For $T \in M_{m,n}(\mc{L}(X))$ we define
\begin{equation*}
  \nrm{\mb{T}}_{\hat{\alpha}} = \sup\cbrace{\nrm{\mb{T}\mb{x}}_\alpha: \mb{x} \in X^n, \nrm{\mb{x}}_\alpha \leq 1}.
\end{equation*}
Then $\nrm{\lcdot}_{\hat\alpha}$ is a coherent matricial algebra norm on $\mc{L}(X)$.
\end{example}
\begin{proof}
Take $\mb{S} \in M_{m,k}(\mc{L}(X))$ and $\mb{T} \in M_{k,n}(\mc{L}(X))$. We have
\begin{align*}
  \nrm{\mb{S}\mb{T}}_{\hat{\alpha}} &= \sup \cbraces{\frac{\nrm{\mb{S}\mb{y}}_\alpha}{\nrm{\mb{y}}_\alpha}\frac{\nrm{\mb{y}}_\alpha}{\nrm{\mb{x}}_\alpha}: \mb{x} \in X^n,  \mb{y} = \mb{T}\mb{x}}\leq \nrm{\mb{S}}_{\hat{\alpha}}\nrm{\mb{T}}_{\hat{\alpha}}.
\end{align*}
Moreover for any $\mb{A} \in M_{m,j}(\C)$, $\mb{T} \in M_{j,k}(\mc{L}(X))$ and $\mb{B} \in M_{k,n}(\C)$ we have by property $\eqref{eq:E2x}$ of the Euclidean structure that
\begin{align*}
  \nrm{\mb{A}\mb{T}\mb{B}}_{\hat{\alpha}} &\leq \sup \cbrace*{\nrm{\mb{A}}\nrm{\mb{T}\mb{B}\mb{x}}_\alpha: \mb{x} \in X^n,\nrm{\mb{B}\mb{x}}_\alpha \leq \nrm{\mb{B}}}\leq \nrm{\mb{A}}\nrm{\mb{T}}_{\hat{\alpha}}\nrm{\mb{B}},
\end{align*}
so $\nrm{\lcdot}_{\hat\alpha}$ is a matricial algebra norm. Its coherence follows from
\begin{align*}
  \nrm{(T)}_{\hat\alpha} &= \sup\cbrace{\nrm{Tx}: x \in X ,\nrm{x} \leq 1} = \nrm{T}
\end{align*}
for $T \in \mc{L}(X)$, where we used property $\eqref{eq:E1x}$ of the Euclidean structure.
\end{proof}

Using this induced matricial algebra norm we can reformulate $\alpha$-boundedness. Indeed, for a family of operators $\Gamma \subseteq \mc{L}(X)$ we have
\begin{equation}\label{eq:reformalphabounded}
\nrm{\Gamma}_\alpha = \sup\cbraceb{\nrm{\mb{T}}_{\hat{\alpha}}:\mb{T} = \diag(T_1,\ldots,T_n), \quad T_1,\ldots,T_n \in \Gamma}.
\end{equation}

This reformulation allows us to characterize those Banach spaces on which $\alpha$-boundedness is equivalent to uniform boundedness, using a result of Blecher, Ruan and Sinclair \cite{BRS90}.
\begin{proposition}
  Let $\alpha$ be a Euclidean structure on $X$ such that for any family of operators $\Gamma \subseteq \mc{L}(X)$ we have
  \begin{equation*}
    \nrm{\Gamma}_{\alpha} = \, \sup_{T \in \Gamma}\,\nrm{T}.
  \end{equation*}
  Then $X$ is isomorphic to a Hilbert space.
\end{proposition}

\begin{proof}
  Take $T_1,\ldots,T_n \in \mc{L}(X)$ and let $\Gamma = \cbrace{T_k:1\leq k\leq n}$. We have
  \begin{equation*}
    \sup_{1\leq k\leq n}\nrm{T_k} \leq \nrm{\diag(T_1,\ldots,T_n)}_{\hat{\alpha}}\leq \nrm{\Gamma}_\alpha =  \sup_{1\leq k\leq n}\nrm{T_k},
  \end{equation*}
  which implies by \cite{BRS90} that $\mc{L}(X)$ is isomorphic to an operator algebra and that therefore $X$ is isomorphic to a Hilbert space by \cite{Ei40}.
\end{proof}

\subsection*{$C^*$-boundedness}
Now let us turn to the converse of Theorem  \ref{theorem:euclideanrepresentation}, for which we need to reformulate its conclusion. By an {\em operator algebra} $\mc{A}$ we shall mean a closed unital subalgebra of a $C^*$-algebra. By the Gelfand-Naimark theorem we may assume without loss of generality that $\mc{A}$ consists of bounded linear operators on a Hilbert space $H$.
We say that $\Gamma \subseteq \mc{L}(X)$ is {\em $C^*$-bounded}  if there exists a $C>0$, an operator algebra $\mc{A}$ and a bounded algebra homomorphism $\rho: \mc{A} \to \mc{L}(X)$ such that
 \begin{equation*}
   \Gamma \subseteq \cbraces{\rho(T): T \in \mc{A}, \nrm{T} \leq \frac{C}{\nrm{\rho}}}.
 \end{equation*}
The least admissible $C$ is denoted by $\nrm{\Gamma}_{C^*}$.

From Theorem \ref{theorem:euclideanrepresentation} we can directly deduce that any $\alpha$-bounded family of operators is $C^*$-bounded. We will show that any $C^*$-bounded family of operators is $\alpha$-bounded for some Euclidean structure $\alpha$. As a first step we will prove a converse to Example \ref{example:hatalpha}, i.e. we will show that a matricial algebra norm on a subalgebra of $\mc{L}(X)$ gives rise to a Euclidean structure.

\begin{proposition}
\label{proposition:matricialEuclidean}
  Let $\mc{A}$ be a subalgebra of $\mc{L}(X)$ and let $\nrm{\lcdot}_{\mc{A}}$ be a matricial algebra norm on $\mc{A}$ such that $\nrm{(T)}_{\mc{A}} \geq \nrm{T}$ for all $T \in \mc{A}$. Then there is a Euclidean structure $\alpha$ on $X$ such that $\nrm{\mb{T}}_{\hat\alpha} \leq \nrm{\mb{T}}_{\mc{A}}$ for all $\mb{T} \in M_{m,n}(\mc{A})$.
\end{proposition}

\begin{proof}
  Define the $\alpha$-norm of a column vector $\mb{x} \in X^n$ by
  \begin{equation*}
    \nrm{\mb{x}}_\alpha = \max\cbraceb{\nrm{\mb{x}}_\beta, \nrm{\mb{x}}_{\op}}
  \end{equation*}
  with
  \begin{equation*}
    \nrm{\mb{x}}_\beta = \sup \cbraceb{\nrm{\mb{S}\mb{x}}: \mb{S} \in M_{1,n}(\mc{A}), \nrm{\mb{S}}_{\mc{A}} \leq 1}.
  \end{equation*}
  Then $\nrm{\lcdot}_\alpha$ is a Euclidean structure, since we already know $\op$ is a Euclidean structure and for $\beta$ we have
  \begin{equation*}
    \nrm{(x)}_\beta \leq \sup\cbrace{\nrm{Sx}_X: S\in \mc{A}, \nrm{S} \leq 1} = \nrm{x}_X
  \end{equation*}
  for any $x \in X$, so \eqref{eq:E1x} holds. Moreover, if $\mb{A} \in M_{m,n}(\C)$ and $\mb{x} \in X^n$, we have
  \begin{equation*}
    \nrm{\mb{A}\mb{x}}_\beta = \sup\cbrace*{\nrm{\mb{S}\mb{A}\mb{x}}: \mb{S} \in M_{1,m}(\mc{A}), \nrm{\mb{S}}_{\mc{A}} \leq 1} \leq \nrm{\mb{A}}\nrm{\mb{x}}_\beta,
  \end{equation*}
so $\beta$ satisfies \eqref{eq:E2x}.

  Now suppose that $\mb{T} \in M_{m,n}(\mc{A})$, $\mb{x} \in X^n$ with $\nrm{\mb{x}}_\alpha \leq 1$ and $\mb{y} = \mb{T} \mb{x}$, then
  \begin{align*}
    \nrm{\mb{y}}_\beta &= \sup\cbrace*{\nrm{\mb{S}\mb{T}\mb{x}}: \mb{S} \in M_{1,m}(\mc{A}), \nrm{\mb{S}}_{\mc{A}} \leq 1}\\
    & \leq \sup\cbrace*{\nrm{\mb{S}\mb{x}}: \mb{S} \in M_{1,n}(\mc{A}), \nrm{\mb{S}}_{\mc{A}} \leq \nrm{\mb{T}}_{\mc{A}}} = \nrm{\mb{T}}_{\mc{A}}\nrm{\mb{x}}_\beta
  \end{align*}
  and
  \begin{align*}
    \nrm{\mb{y}}_{\op} &= \sup\cbrace{\nrm{\mb{A}\mb{T}\mb{x}}: \mb{A} \in M_{1,m}(\C), \nrm{\mb{A}} \leq 1}\\
    &\leq \sup\cbrace{\nrm{\mb{S}\mb{x}}: \mb{S} \in M_{1,n}(\mc{A}), \nrm{\mb{S}}_{\mc{A}} \leq \nrm{\mb{T}}_{\mc{A}}} = \nrm{\mb{T}}_{\mc{A}} \nrm{\mb{x}}_\beta.
  \end{align*}
  From this we immediately get
  \begin{align*}
    \nrm{\mb{T}}_{\hat\alpha} &= \sup\cbrace{\nrm{\mb{y}}_\alpha: \mb{y} = \mb{T}\mb{x},\mb{x} \in X^n, \nrm{\mb{x}}_\alpha \leq 1}\leq \nrm{\mb{T}}_{\mc{A}},
  \end{align*}
  which proves the proposition.
\end{proof}

If $\mc{A}$ and $\mc{B}$ are two matricial normed algebras, then an algebra homomorphism $\rho: \mc{A} \to \mc{B}$ naturally induces a map $\rho:M_{m,n}(\mc{A}) \to M_{m,n}(\mc{B})$ by setting $\rho(\mb{T}) = (\rho(T_{jk}))_{j,k=1}^{m,n}$ for $\mb{T} \in M_{m,n}(\mc{A})$. The algebra homomorphism $\rho$ is called \emph{completely bounded} if these maps are uniformly bounded for $m,n \in \N$.

 We will  use Proposition \ref{proposition:matricialEuclidean} to prove that the bounded algebra homomorphism $\rho$ in the definition of $C^*$-boundedness can be used to construct a Euclidean structure on $X$ such that $\rho$ is completely bounded if we equip the operator algebra $\mc{A}$ with its natural matricial algebra norm  given by
\begin{equation*}
  \nrm{\mb{T}}_{\mc{A}} = \nrm{\mb{T}}_{\mc{L}(\ell_n^2(H),\ell_m^2(H))}, \qquad T \in M_{m,n}(\mc{A})
\end{equation*}
 and we equip $\mc{L}(X)$ with the  matricial algebra norm $\hat{\alpha}$ induced by $\alpha$.

\begin{proposition}
\label{proposition:completelybounded}
  Let $H$ be a Hilbert space and suppose that $\mc{A}\subseteq \mc{L}(H)$ is an operator algebra. Let $\rho\colon\mc{A} \to \mc{L}(X)$ be a bounded algebra homomorphism. Then there exists a Euclidean structure $\alpha$ on $X$ such that
  \begin{equation*}
    \nrm{\rho(\mb{T})}_{\hat{\alpha}} \leq \nrm{\rho} \nrm{\mb{T}}_{\mc{L}(\ell^2_n(H), \ell^2_m(H))}
  \end{equation*}
  for all $\mb{T} \in M_{m,n}(\mc{A})$, i.e. $\rho$ is completely bounded.
\end{proposition}

\begin{proof}
  We induce a  matricial algebra norm $\beta$ on $\rho(\mc{A})$ by setting for $\mb{S} \in M_{m,n}(\rho(\mc{A}))$
  \begin{equation*}
    \nrm{\mb{S}}_\beta = \nrm{\rho} \inf\{\nrm{\mb{T}}_{\mc{L}(\ell^2_n(H),\ell^2_m(H))}: \mb{T} \in M_{m,n}(\mc{A}),\rho(\mb{T}) = \mb{S}\}.
  \end{equation*}
  This is indeed a matricial algebra norm since  for  $\mb{S} \in M_{m,k}(\rho(\mc{A}))$ and $\mb{T} \in M_{k,n}(\rho(\mc{A}))$ we have that
  \begin{align*}
    \nrm{\mb{S}\mb{T}}_\beta &= \nrm{\rho} \inf \cbrace{\nrm{\mb{U}} : \mb{U} \in M_{m,n}(\mc{A}), \rho(\mb{U}) = \mb{S}\mb{T}} \\
    &\leq \nrm{\rho} \inf \cbrace{\nrm{\mb{U}}\nrm{\mb{V}} : \rho(\mb{U}) = \mb{S}, \rho(\mb{V}) = \mb{T}}\\
    &\leq \nrm{\mb{S}}_\beta \nrm{\mb{T}}_\beta
  \end{align*}
  as $\nrm{\rho} \geq 1$. Moreover for any ${S} \in \rho(\mc{A})$ we have
  \begin{equation*}
    \nrm{(S)}_\beta = \nrm{\rho} \inf\cbrace{\nrm{T}:T \in \mc{A},\rho(T) = S} \geq \nrm{S}.
  \end{equation*}
Hence, by Proposition \ref{proposition:matricialEuclidean}, there exists a Euclidean structure $\alpha$ such that $\nrm{\mb{S}}_{\hat\alpha} \leq \nrm{\mb{S}}_\beta$ for all $\mb{S} \in M_{m,n}(\rho(\mc{A}))$. This means
  \begin{equation*}
    \nrm{\rho(\mb{T})}_{\hat\alpha} \leq \nrm{\rho(\mb{T})}_{\beta} \leq \nrm{\rho} \nrm{\mb{T}}
  \end{equation*}
  for all $\mb{T} \in M_{m,n}(\mc{A})$, proving the proposition.
\end{proof}

\begin{remark}
  If $\mc{A} = C(K)$ for $K$ compact and $X$ has Pisier's contraction property, then one can take $\alpha=\gamma$ in Proposition \ref{proposition:completelybounded} (cf. \cite{PR07,KL10}). For further results on $\gamma$-bounded representations of groups, see \cite{Le10b}.
\end{remark}

We now have all the necessary preparations to turn Theorem \ref{theorem:euclideanrepresentation} into an `if and only if' statement.

\begin{theorem}
\label{theorem:Cboundedalphabounded}
  Let $\Gamma \subseteq \mc{L}(X)$. Then $\Gamma$ is $C^*$-bounded if and only if there exists a Euclidean structure $\alpha$ on X such that $\Gamma$ is $\alpha$-bounded. Moreover $\nrm{\Gamma}_{\alpha} \simeq \nrm{\Gamma}_{C^*}$.
\end{theorem}

\begin{proof}
  First suppose that $\alpha$ is a Euclidean structure on $X$ such that $\Gamma$ is $\alpha$-bounded. Let $\tilde{\Gamma}$ be the closure in the strong operator topology of the absolutely convex hull of $\Gamma \cup \ha{\nrm{\Gamma}_\alpha\cdot I_X}$, where $I_X$ is the identity operator on $X$. By Proposition \ref{proposition:alphaproperties} we know that $\tilde{\Gamma}$ is $\alpha$-bounded with $\nrm{\tilde{\Gamma}}_\alpha \leq 2 \, \nrm{\Gamma}_\alpha$. Then, by Theorem \ref{theorem:euclideanrepresentation}, we can find a closed subalgebra $\mc{A}$ of a $C^*$-algebra and a bounded algebra homomorphism $\rho:\mc{A}  \to \mc{L}(X)$ such that
  \begin{align*}
    \Gamma \subseteq \tilde\Gamma &\subseteq \cbrace{\rho\tau(T):T \in \mc{L}_{\tilde\Gamma}(X), \nrm{T}_{\tilde\Gamma}\leq 1}\\
    &\subseteq \cbraces{\rho(T):T \in \mc{A}, \nrm{T}\leq \frac{16\,\nrm{\Gamma}_\alpha}{\nrm{\rho}}}.
  \end{align*}
So $\Gamma$ is $C^*$-bounded with constant $\nrm{\Gamma}_{C^*} \leq 16\, \nrm{\Gamma}_\alpha$.

  Now assume that $\Gamma$ is $C^*$-bounded. Let $\mc{A}$ be an operator algebra over a Hilbert space $H$ and let $\rho:\mc{A} \to \mc{L}(X)$ a bounded algebra homomorphism such that
  \begin{equation*}
    \Gamma \subseteq \cbrace*{\rho(S):S \in \mc{A}, \nrm{S}\leq \frac{\nrm{\Gamma}_{C^*}}{\nrm{\rho}}}.
  \end{equation*}
By Proposition \ref{proposition:completelybounded} there is a Euclidean structure $\alpha$ such that
  \begin{equation*}
    \nrm{\rho(\mb{T})}_{\hat{\alpha}} \leq \nrm{\rho} \nrm{\mb{T}}_{\mc{L}(\ell^2_n(H))}
  \end{equation*}
  for all $\mb{T} \in M_{m,n}(\mc{A})$. Take $T_1,\ldots,T_n \in \Gamma$ and let $S_1,\ldots,S_n \in \mc{A}$ be such that $\rho(S_k)=T_k$ and $\nrm{S_k}\leq \frac{\nrm{\Gamma}_{C^*}}{\nrm{\rho}}$ for $1 \leq k \leq n$. Then we have, using the Hilbert space structure of $H$, that
  \begin{align*}
    \nrm{\diag(T_1,\ldots,T_n)}_{\hat{\alpha}}    &=  \nrmb{\rho\hab{\diag(S_1,\ldots,S_n)}}_{\hat{\alpha}}\\
    &\leq \nrm{\rho} \nrm{\diag(S_1,\ldots,S_n)}_{\mc{L}(\ell^2_n(H))} \\&\leq \nrm{\Gamma}_{C^*}.
  \end{align*}
  So by \eqref{eq:reformalphabounded} we obtain that $\Gamma$ is $\alpha$-bounded with $\nrm{\Gamma}_\alpha \leq \nrm{\Gamma}_{C^*}$.
\end{proof}

\chapter{Factorization of \texorpdfstring{$\alpha$}{a}-bounded operator families}\label{part:2}
In Chapter \ref{part:1} we have seen that the $\alpha$-boundedness of a family of operators is inherently tied up with a Hilbert space hiding in the background. However, we did not obtain information on the structure of this Hilbert space, nor on the form of the algebra homomorphism connecting the Hilbert and Banach space settings. In this chapter we will highlight some special cases in which the representation can be made more explicit.

 The first special case that we will treat is the case where $\alpha$ is either the $\gamma$- or the $\pi_2$-structure. In this case Lemma \ref{lemma:technical} implies a $\gamma$-bounded version of the Kwapie\'n--Maurey factorization theorem. Moreover we will show that $\pi_2$-boundedness can be characterized in terms of factorization through a Hilbert space, i.e. yielding control over the algebra homomorphism.

 Afterwards we turn our attention to the case in which $X$ is a Banach function space. In Section \ref{section:BFSell2boundedness} we will show that, under a mild additional assumption on the Euclidean structure $\alpha$, an $\alpha$-bounded family of operators on a Banach function space is actually $\ell^2$-bounded. This implies that the $\ell^2$-structure is the canonical structure to consider on Banach function spaces.

 In Section
 \ref{section:BFSfactorization} we prove a version of Lemma \ref{lemma:technicalcase} for Banach function spaces, in which the abstract Hilbert space is replaced by a weighted $L^2$-space. This is remarkable, since this is gives us crucial information on the Hilbert space $H$. This formulation resembles the work of Maurey, Niki{\v{s}}in and Rubio de Francia on weighted versus vector-valued inequalities, but has the key advantage that no geometric properties  of the Banach function space are used.
 In Section \ref{section:vectorextension} we use this version of Lemma \ref{lemma:technicalcase} to prove a vector-valued extension theorem with weaker assumptions than the one in the work of Rubio de Francia. This has applications in vector-valued harmonic analysis, of which we will give a few examples.

\section[Factorization of \texorpdfstring{$\gamma$}{y}- and \texorpdfstring{$\pi_2$}{pi2}-bounded operator families]{Factorization of \texorpdfstring{$\gamma$}{y}- and \texorpdfstring{$\pi_2$}{pi2}-bounded operator families}\label{section:hilbertfactorization}
In this section we will consider the special case where $\alpha$ is either the $\gamma$- or the $\pi_2$-structure. For these Euclidean structures we will show that $\alpha$-bounded families of operators can be factorized through a Hilbert space under certain geometric conditions on the underlying Banach spaces.
 All results in this section will be based on the following lemma, which is a special case of Lemma \ref{lemma:technical}.

\begin{lemma}\label{lemma:euclideankwapienmaurey}
  Let $X$ and $Y$ be Banach spaces. Let $\Gamma_1\subseteq \mc{L}(X,Y)$ and suppose that there is a $C>0$ such  that for all $\mb{x} \in X^n$ and $S_1,\ldots,S_n \in \Gamma_1$  we have
  \begin{equation*}
    \nrm{(S_1x_1,\ldots,S_nx_n)}_{\pi_2} \leq C \, \has{\sum_{k=1}^n\nrm{x_k}^2}^{1/2}.
  \end{equation*}
  Let $\Gamma_2 \subseteq \mc{L}(Y)$ be a $\pi_2$-bounded family of operators.
  Then there is a Hilbert space $H$, a contractive embedding $U\colon H \to Y$, a $\widetilde{S} \in \mc{L}(X,H)$ for every $S \in \Gamma_1$ and a $\widetilde{T} \in \mc{L}(H)$ for every $T \in \Gamma_2$   such that
  \begin{align*}
  \nrm{\widetilde{S}} &\leq 4\,C,& &S \in \Gamma_1,\\
   \nrm{\widetilde{T}} &\leq 2\,\nrm{\Gamma}_{\pi_2}, & & T \in \Gamma_2 .
\end{align*}
and the following diagram commutes:
\begin{center}
\begin{tikzcd}
X\arrow{rd}{\widetilde{S}} \arrow{r}{S}                                                      & Y \arrow{r}{T} & Y              \\
& H \arrow{u}{U} \arrow{r}{\widetilde{T}} &H \arrow{u}{U}
\end{tikzcd}
\end{center}
\end{lemma}

\begin{proof}
Define $F:Y \to [0,\infty)$ as $F(y) = \nrm{y}_Y$. Then we have, by the definition of the $\pi_2$-structure, for any $\mb{y} \in Y^n$
  \begin{equation*}
    \has{\sum_{k=1}^n F(y_k)^2}^{1/2} = \has{\sum_{k=1}^n \nrm{y_k}_X^2}^{1/2} \leq \nrm{\mb{y}}_{\pi_2}.
  \end{equation*}
      Let
      $$\widetilde{Y} = \cbrace{Sx: S \in \Gamma_1,\, x \in X}\subseteq Y$$
       and define $G\colon\widetilde{Y} \to [0,\infty)$ by
      $$G(y) := C \cdot \inf\cbraceb{\nrm{x}_X:x \in X,\, Sx=y,\,S \in \Gamma_1}.$$
      Fix $\mb{y} \in \widetilde{Y}$, then for any $S_1,\ldots,S_n \in \Gamma_1$ and $\mb{x} \in X^n$ such that $y_k = Sx_k$ we have
      \begin{equation*}
        \nrm{\mb{y}}_{\pi_2}\leq C\cdot\has{\sum_{k=1}^n\nrm{x_k}_X^2}^{1/2}.
      \end{equation*}
       Thus, taking the infimum over all such $S_k$ and $\mb{x}$, we obtain $$\nrm{\mb{y}}_{\pi_2}\leq \hab{\sum_{k=1}^nG(y_k)^2}^{1/2}.$$ Hence by Lemma \ref{lemma:technical} there is a Hilbertian seminorm $\nrm{\lcdot}_0$ on a $\Gamma_2$-invariant subspace $Y_0$ of $Y$ which contains $\widetilde{Y}$ and satisfies \eqref{eq:F3}-\eqref{eq:F2}. In particular, for $y \in Y_0$ we have $\nrm{y}_Y = F(y)\leq \nrm{y}_0$, so $\nrm{\lcdot}_0$ is a norm.

      Let $H$ be the completion of $(Y_0,\nrm{\lcdot}_0)$ and let $U\colon H \to Y$ be the inclusion mapping. For every $S \in \Gamma_1$ let $\widetilde{S}:X \to H$ be the mapping $x \mapsto Sx \in \tilde{Y} \subseteq Y_0$. Then we have for any $x \in X$ that
      \begin{equation*}
        \nrm{Sx}_0 \leq 4\, G(Sx) \leq 4\,C\,\nrm{x}_X,
      \end{equation*}
      so $\nrm{\widetilde{S}} \leq 4\,C$. Moreover we have $S=U\widetilde{S}$. Finally  let $\widetilde{T}$ be the canonical extension of $T \in \Gamma_2$ to $H$. Then we have $\nrm{\widetilde{T}} \leq 2\nrm{\Gamma_2}_{\pi_2}$ and $TU = U\widetilde{T}$, which proves the lemma.
\end{proof}

\subsection*{A $\gamma$-bounded Kwapie\'n--Maurey factorization theorem}
As a first application of Lemma \ref{lemma:euclideankwapienmaurey} we prove a $\gamma$-bounded version of the Kwapie\'n-Maurey factorization theorem (see \cite{Kw72, Ma74} and \cite[Theorem 7.4.2]{AK16}).

\begin{theorem}\label{theorem:Rbddkwapienmaurey}
Let $X$ be a Banach space with type $2$ and $Y$ a Banach space with cotype $2$. Let $\Gamma_1\subseteq \mc{L}(X,Y)$ and  $\Gamma_2 \subseteq\mc{L}(Y)$ be $\gamma$-bounded families of operators. Then there is a Hilbert space $H$, a contractive embedding $U:H \to Y$, a $\widetilde{S} \in \mc{L}(X,H)$ for every $S \in \Gamma_1$ and a $\widetilde{T} \in \mc{L}(H)$ for every $T \in \Gamma_2$   such that
  \begin{align*}
   \nrm{\widetilde{S}} &\lesssim  \nrm{\Gamma_1}_{\gamma}& &S\in \Gamma_1\\
  \nrm{\widetilde{T}} &\lesssim \nrm{\Gamma_2}_{\gamma}, & & T \in \Gamma_2 .
\end{align*}
and the following diagram commutes:
\begin{center}
\begin{tikzcd}
X\arrow{rd}{\widetilde{S}} \arrow{r}{S}                                                      & Y \arrow{r}{T} & Y              \\
& H \arrow{u}{U} \arrow{r}{\widetilde{T}} &H \arrow{u}{U}
\end{tikzcd}
\end{center}
\end{theorem}

Note that the Kwapie\'n-Maurey factorization theorem follows from Theorem \ref{theorem:Rbddkwapienmaurey} by taking $\Gamma_1 = \cbrace{S}$ for some $S\in \mc{L}(X, Y)$ and taking $\Gamma_2=\varnothing$. In particular the fact that any Banach space with type $2$ and cotype $2$ is isomorphic to a Hilbert space follows by taking $X=Y$, $\Gamma_1=\cbrace{I_X}$ and $\Gamma_2=\varnothing$.

\begin{proof}[Proof of Theorem \ref{theorem:Rbddkwapienmaurey}]
Note that $\gamma$-boundedness and $\pi_2$-boundedness are equivalent on a space with cotype $2$ by Proposition \ref{proposition:compareEuclidean}. Thus $\Gamma_2$ is $\pi_2$-bounded on $Y$. Furthermore, using Proposition \ref{proposition:compareEuclidean}, the $\gamma$-boundedness of $\Gamma_1$ and Proposition \ref{proposition:gaussianradermacherl2comparison}, we have for $\mb{x} \in X^n$ and $T_1,\ldots,T_n \in \Gamma$
  \begin{align*}
    \nrm{(T_1x_1,\ldots,T_nx_n)}_{\pi_2} &\lesssim \nrm{(T_1x_1,\ldots,T_nx_n)}_{\gamma}\\
    &\leq\nrm{\Gamma}_{\gamma} \nrm{(x_1,\ldots,x_n)}_{\gamma}\\&\lesssim\nrm{\Gamma}_{\gamma}
    \has{\sum_{k=1}^n\nrm{x_k}^2}^{1/2}.
  \end{align*}
Therefore the theorem follows from Lemma \ref{lemma:euclideankwapienmaurey}.
\end{proof}

\subsection*{Factorization of $\pi_2$-bounded operator families through a Hilbert space}
If we let $X$ be a Hilbert space in Lemma \ref{lemma:euclideankwapienmaurey}, we can actually characterize the $\pi_2$-boundedness of a family of operators on $Y$ by a factorization property. In order to prove this will need the $\pi_2$-summing norm for operators $T \in \mc{L}(Y,Z)$, where $Y$ and $Z$ are Banach spaces. It is defined as
\begin{equation*}
  \nrm{T}_{\pi_2} := \sup\cbraces{\hab{\sum_{k=1}^n \nrm{Ty_k}_Z^2}^{1/2}: \mb{y} \in Y^n, \sup_{\nrm{y^*}_{Y^*} \leq 1} \hab{\sum_{k=1}^n\abs{\ip{y_k,y^*}}^2}^{1/2} \leq 1}
\end{equation*}
Clearly $\nrm{T} \leq \nrm{T}_{\pi_2}$ and $T$ is called $2$-summing if $\nrm{T}_{\pi_2} <\infty$. For a connection between $p$-summing operators and factorization through $L^p$ we refer to \cite{To89,DJT95} and the references therein. If $Y = \ell^2$ this definition coincides with the definition given in Section \ref{section:euclidean}, which follows from the fact that $\mc{L}(\ell^2)$ is isometrically isomorphic to $\ell^2_{\mathrm{weak}}(\ell^2$), the space of all sequences $(y_n)_{n\geq 1}$ in $\ell^2$ for which
\begin{equation*}
  \nrm{(y_n)_{n\geq 1}}_{\ell^2_{\mathrm{weak}}(\ell^2)}:=\sup_{\nrm{y^*}_{\ell^2} \leq 1} \hab{\sum_{n=1}^\infty\abs{\ip{y_n,y^*}}^2}^{1/2},
\end{equation*}
is finite, see e.g. \cite[Proposition 2.2]{DJT95}.

\begin{theorem}\label{theorem:hilbertfactorization}
Let $Y$ be a Banach space and let $\Gamma \subseteq \mc{L}(Y)$. Then $\Gamma$ is $\pi_2$-bounded if and only if there is a $C>0$ such that for any  Hilbert space $X$ and $S \in \mc{L}(X,Y)$, there is a Hilbert space $H$, a contractive embedding $U:H \to Y$, a $\widetilde{S} \in \mc{L}(X,H)$ and a $\widetilde{T} \in \mc{L}(H)$ for every $T \in \Gamma$ such that
\begin{align*}
  \nrm{\widetilde{S}} &\leq 4 \nrm{S}\\
  \nrm{\widetilde{T}} &\leq C,  &T \in \Gamma
\end{align*}
and the following diagram commutes
\begin{center}
\begin{tikzcd}
X\arrow{rd}{\widetilde{S}} \arrow{r}{S}                                                      & Y \arrow{r}{T} & Y              \\
& H \arrow{u}{U} \arrow{r}{\widetilde{T}} &H \arrow{u}{U}
\end{tikzcd}
\end{center}
Moreover $C>0$ can be chosen such that $\nrm{\Gamma}_{\pi_2} \simeq C$.
\end{theorem}

\begin{proof}
For the `only if' statement let $X$ be a Hilbert space and $S \in \mc{L}(X,Y)$. Note that by the ideal property of the $\pi_2$-structure and the coincidence of the $\pi_2$-norm and the Hilbert-Schmidt norm on Hilbert spaces we have for all $\mb{x} \in X^n$
  \begin{equation*}
    \nrm{(Sx_1,\ldots,Sx_n)}_{\pi_2} \leq \nrm{S}  \nrm{(x_1,\ldots,x_n)}_{\pi_2} = \has{\sum_{k=1}^n\nrm{x_k}^2}^{1/2}.
  \end{equation*}
  Therefore the `only if' statement follows directly from Lemma \ref{lemma:euclideankwapienmaurey} using $\Gamma_1 = \cbrace{S}$.

For the `if' statement let $T_1,\ldots,T_n \in \Gamma$. Let $\mb{y} \in Y^n$ with $\nrm{\mb{y}}_{\pi_2}\leq 1$  and let $V$ be the finite rank operator associated to $\mb{y}$, i.e.
      \begin{equation*}
        Vf := \sum_{k=1}^n \ip{f,e_k}y_k, \qquad f \in \ell^2
      \end{equation*}
      for some orthonormal sequence $(e_k)_{k=1}^n$ in $\ell^2$. We will combine the given Hilbert space factorization with Pietsch factorization theorem to factorize $V$ and $T_1,\ldots,T_n$. In particular we will construct operators such that the following diagram commutes:
\begin{center}
\begin{tikzcd}
\ell^2 \arrow{d}{\widetilde{V}} \arrow{rr}{V} &                                                          & Y \arrow{r}{T_k} & Y              \\
L^\infty(\Omega) \arrow{r}{J} & L^2(\Omega)=X \arrow{r}{\widetilde{S}} \arrow{ru}{S} & H \arrow{u}{U} \arrow{r}{\widetilde{T}_k} &H \arrow{u}{U}
\end{tikzcd}
\end{center}
As $\nrm{V}_{\pi_2} \leq 1$, by the Pietsch factorization theorem \cite[p.48]{DJT95} there is a probability space $(\Omega,\P)$ and operators $\widetilde{V} \colon \ell^2 \to L^\infty(\Omega)$ and $S \colon L^2(\Omega) \to Y$, such that $\nrm{\widetilde{V}}\leq 1$, $\nrm{S}\leq 1$ and $V=SJ\widetilde{V}$, where $J\colon L^\infty(\Omega)\to L^2(\Omega)$ is the canonical inclusion.

      We now use the assumption with $X = L^2(\Omega)$ and $S$ to construct $H$, $U$, $\widetilde{S}$ and $\widetilde{T}_k$ for $1 \leq k \leq n$ with the prescribed properties. Define $R \in \mc{F}(\ell^2,X)$ by $Re_k = T_kVe_k$ and $\widetilde{R} \in \mc{F}(\ell^2,L^2(\Omega))$ by $\widetilde{R}e_k = \widetilde{T}_k\widetilde{S}J\widetilde{V}e_k$ for $1 \leq k \leq n$. Then $R = U\widetilde{R}$ (see the diagram above) and therefore we have
      \begin{align*}
          \nrm{R}_{\pi_2}\leq \nrm{U} \, \nrmb{\widetilde{R}}_{\pi_2}= \nrmb{\widetilde{R}}_{\HS} &=  \has{\sum_{k=1}^n \nrmb{\widetilde{T}_k\widetilde{S}J\widetilde{V}e_k}^2}^{1/2}\\ &\leq 4C \has{\sum_{k=1}^n\nrmb{J\widetilde{V}e_k}^2}^{1/2}.
      \end{align*}
      Since $\nrm{J}_{\pi_2} = 1$ by \cite[Example 2.9(d)]{DJT95}, we have $\nrm{J\widetilde{V}}_{\pi_2}\leq 1$ by the ideal property of the $\pi_2$-summing norm. Moreover $\nrm{(T_1y_1,\ldots,T_ny_n)}_{\pi_2} = \nrm{R}_{\pi_2}$, so we can conclude
       \begin{equation*}
         \nrm{(T_1y_1,\ldots,T_ny_n)}_{\pi_2} \leq 4C,
       \end{equation*}
       i.e. $\Gamma$ is $\pi_2$-bounded with $\nrm{\Gamma}_{\pi_2}  \leq 4C$.
\end{proof}

In Theorem \ref{theorem:hilbertfactorization} it suffices to consider the case where $S$ and $\widetilde{S}$ are injective, which allows us to restate the theorem in terms of Hilbert spaces embedded in $Y$.

\begin{corollary}\label{corollary:hilbertfact1}
 Let $Y$ be a Banach space and let $\Gamma \subseteq \mc{L}(Y)$. Then $\Gamma$ is $\pi_2$-bounded  if and only if there is a $C> 0$ such that: \\
  \begin{minipage}[t]{0.06\textwidth}
\vspace{0.7\baselineskip}\begin{center}
  $(\star)$
\end{center}
\end{minipage}
\begin{minipage}[t]{0.94\textwidth}
For any Hilbert space $X$ contractively embedded in $Y$ there is a Hilbert space $H$ with $X \subseteq H \subseteq Y$ such that $H$ is contractively embedded in $Y$, the embedding $X \hookrightarrow H$ has norm at most $4$, and such that $T$ is an operator on $H$ with $\nrm{T}_{\mc{L}(H)} \leq C$ for all $T \in \Gamma$.
\end{minipage} Moreover $C>0$ can be chosen such that $\nrm{\Gamma}_{\pi_2} \simeq C$.
\end{corollary}

\begin{proof}
For the `if' statement note that in the proof of Theorem \ref{theorem:hilbertfactorization} the orthonormal sequence can be chosen such that $V$ is injective, and thus $S$ can be made injective by restricting to $J\widetilde{V}(\ell^2) \subseteq L^2(\Omega)$.
For the converse note that if $S$ is injective in the proof of Lemma \ref{lemma:euclideankwapienmaurey}, then the constructed $\widetilde{S}$ is as well.
\end{proof}

Since the $\pi_2$-structure is equivalent to the $\gamma$-structure if $Y$ has cotype $2$ by Proposition \ref{proposition:compareEuclidean}, we also have:

\begin{corollary}\label{corollary:hilbertfact2}
 Let  $Y$ be a Banach space with cotype $2$ and let $\Gamma \subseteq \mc{L}(Y)$. Then $\Gamma$ is $\gamma$-bounded if and only if there is a $C>0$ such that $(\star)$ of Corollary \ref{corollary:hilbertfact1} holds. Moreover $C>0$ can be chosen such that $\nrm{\Gamma}_{\gamma} \simeq C$.
\end{corollary}

Finally we note that we can dualize Theorem \ref{theorem:hilbertfactorization}, Corollary \ref{corollary:hilbertfact1} and  \ref{corollary:hilbertfact2}. For example we have:

\begin{corollary}\label{corollary:hilbertfact3}
 Let  $Y$ be a Banach space with cotype $2$ and let $\Gamma \subseteq \mc{L}(Y)$. Then $\Gamma$ is $\gamma$-bounded if and only if there is a $C> 0$ such that: \\
  \begin{minipage}[t]{0.06\textwidth}
\vspace{1\baselineskip}\begin{center}
  $(\star\star)$
\end{center}
\end{minipage}
\begin{minipage}[t]{0.94\textwidth}
For any Hilbert space $X$ in which $Y$ is contractively embedded there is a Hilbert space $H$ with $Y \subseteq H \subseteq X$ such that $Y$ is contractively embedded in $H$, the embedding $H \hookrightarrow X$ has norm at most $4$, and such that $T$ extends boundedly to $H$ with $\nrm{T}_{\mc{L}(H)} \leq C$ for all $T \in \Gamma$.
\end{minipage} Moreover $C>0$ can be chosen such that $\nrm{\Gamma}_{\gamma} \simeq C$.
\end{corollary}

\begin{proof}
  Note that since $Y$ has type $2$, $Y^*$ has non-trivial type and  cotype $2$. Therefore by Proposition \ref{proposition:dualeuclidean} the $\gamma^*$-structure is equivalent to the $\gamma$-structure on $X^*$. Moreover by Proposition \ref{proposition:alphaproperties} we know that $\Gamma^*$ is $\gamma^*$-bounded on $Y^*$ with $\nrm{\Gamma^*}_{\gamma^*}\leq C$. So the corollary follows by dualizing Corollary \ref{corollary:hilbertfact2}.
\end{proof}

\section{\texorpdfstring{$\alpha$}{a}-bounded operator families on Banach function spaces} \label{section:BFSell2boundedness}
For the remainder of this chapter, we will study Euclidean structures and factorization in the case that $X$ is a Banach function space.

\begin{definition}
  Let $(S,\mu)$ be a $\sigma$-finite measure space. A subspace $X$ of the space of measurable functions on $S$, denoted by $L^0(S)$, equipped with a norm $\nrm{\cdot }_X$ is called a \emph{Banach function space} if it satisfies the following properties:
\begin{enumerate}[(i)]
\item If $x\in L^0(S)$ and $y\in X$ with $\abs{x} \leq \abs{y}$, then $x\in X$ and $\nrm{x}_X \leq \nrm{y}_X$.
\item There is an $x\in X$ with $x>0$ a.e.
\item If $0
\leq x_n \uparrow x$ for $(x_n)_{n=1}^\infty$  in $X$, $x \in L^0(S)$ and $\sup_{n \in \N}\nrm{x_n}_X < \infty$, then $x\in X$ and $\nrm{x}_X = \sup_{n\in \N}\nrm{x_n}_X$.
\end{enumerate}
A Banach function space $X$ is called \textit{order-continuous} if additionally
\begin{enumerate}[(i)]
\setcounter{enumi}{3}
\item If $0 \leq x_n \uparrow x \in X$ with $(x_n)_{n=1}^\infty$ a sequence in $X$ and $x \in X$, then $\nrm{x_n -x}_X \to 0$.
\end{enumerate}
\end{definition} Order-continuity of a Banach function space $X$ ensures that the dual $X^*$ is a Banach function space (see \cite[Section 1.b]{LT79}) and that the Bochner space $L^p(S';X)$ is a Banach function space on $(S \times S',\mu \times \mu')$ for any $\sigma$-finite measure space $(S',\mu')$. As an example we note that any Banach function space which is reflexive or has finite cotype is order-continuous.

Since a Banach function space $X$ is in particular a Banach lattice, it admits the $\ell^2$-structure.
The main result of this section will be that the $\ell^2$-structure is actually the canonical structure to study on Banach function spaces. Indeed, we will show that, under mild assumptions on the Euclidean structure $\alpha$, $\alpha$-boundedness implies $\ell^2$-boundedness. We start by noting the following property of a Hilbertian seminorm on a space of functions.

\begin{lemma}\label{lemma:tychonoff}
Let $(S,\mu)$ be a measure space and let $X \subseteq L^0(S)$ be a vector space with a Hilbertian seminorm $\nrm{\cdot}_0$. Suppose that there is a $C>0$ such that for $x \in L^0(S)$ and $y \in X$
  \begin{align*}
     \abs{x} \leq \abs{y} \Rightarrow x \in X \text{ and }\nrm{x}_0 \leq C \, \nrm{y}_0.
  \end{align*}
Then there exists a seminorm $\nrm{\cdot}_1$ on $X$ such that
\begin{align*}
\tfrac1C\nrm{x}_0 &\leq \nrm{x}_1 \leq C\, \nrm{x}_0 &&x \in X,\\
  \nrm{x+y}^2_1 &= \nrm{x}^2_1+\nrm{y}^2_1, && x,y\in X: x \wedge y=0.
\end{align*}
\end{lemma}

\begin{proof}
Let $\Pi$ be the collection of all finite measurable partitions of $S$, partially ordered by refinement. We define
\begin{equation*}
  \nrm{x}_1 =  \inf_{\pi \in \Pi} \sup_{\pi' \geq \pi} \has{\sum_{E \in \pi'} \nrm{x\ind_E}_0^2}^{1/2}, \qquad x \in X,
\end{equation*}
which is clearly a seminorm.
 For a $\pi \in \Pi$, write $\pi = \cbrace{E_1,\cdots,E_n}$ and let $(\varepsilon_k)_{k=1}^n$ be a Rademacher sequence.  Then we have for all $x \in X$ that
  \begin{equation*}
     \sum_{k=1}^n \nrm{x\ind_{E_k}}_0^2 = \E\sum_{j=1}^n \sum_{k=1}^n  \varepsilon_j \varepsilon_k \ip{x \ind_{E_j}, x \ind_{E_k}} = \E \nrms{\sum_{k=1}^n \varepsilon_k\cdot x \ind_{E_k}}_0^2 \leq C^2\, \nrm{x}_0^2
  \end{equation*}
  and, since $\bigcup_{k=1}^n E_k = S$, we deduce in the same fashion
  \begin{equation*}
    \nrm{x}_0^2 \leq C^2\,  \E \nrms{\sum_{k=1}^n \varepsilon_k\cdot x \ind_{E_k}}_0^2 =  C^2\, \sum_{k=1}^n \nrm{x \ind_{E_k}}^2_0 = C^2\,\sum_{k=1}^n \nrm{x\ind_{E_k}}_0^2 .
  \end{equation*}
  Therefore we have  $\tfrac1C\nrm{x}_0 \leq \nrm{x}_{1} \leq C\, \nrm{x}_0$ for all  $x \in X$ and $\pi \in \Pi$.
 Furthermore if $x,y \in X$ with $x\wedge y=0$, then for $\pi \geq \cbrace{\supp x,S \setminus \supp x}$ we have
 \begin{equation*}
   \sum_{E \in \pi} \nrm{(x+y)\ind_E}_0^2 = \sum_{E \in \pi} \nrm{x\ind_E}_0^2+\sum_{E \in \pi} \nrm{y\ind_E}_0^2.
 \end{equation*}
  So we also get $\nrm{x+y}^2_{1} = \nrm{x}^2_{1} + \nrm{y}^2_{1}$,
which proves the lemma.
\end{proof}

Let $X$ be a Banach function space. For $m \in L^\infty(S)$ we define the pointwise multiplication operator $T_m\colon X \to X$  by $T_mx = m\cdot x$ and denote the collection of pointwise multiplication operators on $X$ by
\begin{equation}\label{eq:mcM}
  \mc{M} = \cbrace{T_m: m \in L^\infty(S), \nrm{m}_{L^\infty(S)} \leq 1} \subseteq \mc{L}(X).
\end{equation}
It turns out that if $\alpha$ is a Euclidean structure on $X$ such that $\mc{M}$ is $\alpha$-bounded, then $\alpha$-boundedness implies $\ell^2$-boundedness. This will follow from the fact that an $\mc{M}$-invariant subspace of $X$ satisfies the assumptions of   Lemma \ref{lemma:tychonoff}.

\begin{theorem}
  \label{theorem:BFSlatticestructure}
  Let $X$ be a Banach function space on $(S,\mu)$, let $\alpha$ be a Euclidean structure on $X$ and assume that $\mc{M}$ is $\alpha$-bounded. If
   $\Gamma\subseteq \mc{L}(X)$ is $\alpha$-bounded, then $\Gamma$ is $\ell^2$-bounded with $\nrm{\Gamma}_{\ell^2} \lesssim \nrm{\Gamma}_\alpha$, where the implicit constant only depends on $\nrm{\mc{M}}_{\alpha}$.
\end{theorem}

\begin{proof}

Let $\mb{x} \in X^n$ and $T_1,\ldots,T_n \in \Gamma$. We will first reduce the desired estimate to an estimate for simple functions. Define $z_0 := \ha{\sum_{k=1}^n \abs{x_k}^2}^\frac{1}{2}$. Then  we have $x_kz_0^{-1} \in L^\infty(S)$, which means that we can find simple functions $\mb{u} \in X^n$ such that $\nrm{u_k - x_kz_0^{-1}}_{L^\infty(S)} \leq \frac{1}{n}$ for $1 \leq k \leq n$. Defining
$y_0:= \ha{\sum_{k=1}^n \abs{u_kz_0}^2}^{1/2}$, we have
\begin{equation*}
     \nrm{z_0-y_0}_X
 \leq \sum_{k=1}^n \nrm{x_k-u_kz_0}_X
  \leq \sum_{k=1}^n \nrm{z_0}_X \nrm{u_k-x_kz_0^{-1}}_{L^\infty(S)} \leq \nrm{z_0}_X.
\end{equation*}
Define $ z_1 := \ha{\sum_{k=1}^n\abs{T_k(u_kz_0)}^2}^{1/2}$. Then similarly, using $\nrm{T_k} \leq \nrm{\Gamma}_\alpha$, we have
\begin{equation*}\begin{aligned}
     \nrms{\hab{\sum_{k=1}^n \abs{T_kx_k}^2}^{1/2}- z_1}_X  \leq \nrm{z_0}_X \nrm{\Gamma}_\alpha
\end{aligned}
\end{equation*}
For $1 \leq k \leq n$ we have $T(u_kz_0)z_1^{-1} \in L^\infty(S)$, which means that we can find simple functions $\mb{v} \in X^n$ such that $\abs{v_k} \leq \abs{T(u_kz_0)z_1^{-1}}$
and
$$\nrm{v_k - T(u_kz_0)z_1^{-1}}_{L^\infty(S)} \leq \nrm{\Gamma}_\alpha \frac{\nrm{z_0}_X}{\nrm{z_1}_X} \frac{1}{n}.$$
It follows that
\begin{align}\label{eq:vunderu}
  \abs{v_kz_1} \leq \abs{T(u_kz_0)}
\end{align}
and, defining $y_1:= \ha{\sum_{k=1}^n \abs{v_kz_1}^2}$, that
\begin{align*}
  \nrm{ z_1 - y_1}_X \leq \sum_{k=1}^n \nrm{T_k(u_kz_0)-v_kz_1}_X\leq \nrm{z_0}_X \nrm{\Gamma}_\alpha.
\end{align*}
Thus, combining the various estimates, we have
\begin{align*}
\nrms{\has{\sum_{k=1}^n \abs{T_k(x_k)}^2}^\frac{1}{2}}_X &\leq  \nrm{y_1}_X + 2 \, \nrm{\Gamma}_\alpha \nrms{\has{\sum_{k=1}^n \abs{x_k}^2}^\frac{1}{2}}_X \\
\nrm{y_0} &\leq 2\, \nrms{\has{\sum_{k=1}^n \abs{x_k}^2}^\frac{1}{2}}_X,
\end{align*}
so it suffices to prove $\nrm{y_1}_X \lesssim \nrm{\Gamma}_\alpha \nrm{y_0}_X$, which is the announced reduction to simple functions.

Define $C_\Gamma := 4 \nrm{\Gamma}_\alpha$, $C_{\mc{M}} := 4 \nrm{\mc{M}}_\alpha$ and set
  \begin{equation*}
    \Gamma_0 := \hab{\tfrac{1}{2\nrm{\Gamma}_\alpha} \cdot \Gamma} \cup \hab{\tfrac{1}{2\nrm{\mc{M}}_\alpha} \cdot  \mc{M}}.
  \end{equation*}
  Then $\Gamma_0$ is $\alpha$-bounded with $\nrm{\Gamma_0}_\alpha \leq 1$ by Proposition \ref{proposition:alphaproperties}. Thus, applying Lemma \ref{lemma:technicalcase} to $\Gamma_0$ and $\eta=(y_0,y_1)$, we can find a $\Gamma$- and $\mc{M}$-invariant subspace $X_\eta \subseteq X$ with $y_0 \in X_\eta$ and a Hilbertian seminorm $\nrm{\lcdot}_\eta$ on $X_\eta$ such that \eqref{eq:eta2} and \eqref{eq:eta1} hold and
  \begin{align}
  \label{eq:gammainequality}
    \nrm{Tx}_\eta &\leq C_\Gamma \nrm{x}_\eta, &&x \in X_\eta, \,T \in \Gamma.\\
    \nrm{Tx}_\eta &\leq C_{\mc{M}} \nrm{x}_\eta, &&x \in X_\eta, \,T \in \mc{M}.
  \label{eq:deltainequality}
  \end{align}
  In particular, \eqref{eq:deltainequality} implies if $x \in L^0(S)$,  $\tilde{x} \in X_\eta$,   and $\abs{x} \leq \abs{\tilde{x}}$, then $x \in X_\eta$ and
  \begin{equation}\label{eq:modulusinequality}
    \nrm{{x}}_\eta \leq C_{\mc{M}} \nrm{\tilde{x}}_\eta.
  \end{equation}
  Therefore we deduce that $u_ky_0, T_k(u_ky_0), v_k z_1 \in X_\eta$ for $1 \leq k \leq n$ and $y_1,z_1 \in X_\eta$.
  Moreover, by Lemma \ref{lemma:tychonoff} there is a seminorm $\nrm{\cdot}_\nu$ on $X_\eta$ such that
  \begin{align} \label{eq:etaHilbert1}
  \tfrac1{C_\mc{M}}\nrm{x}_\eta &\leq \nrm{x}_\nu \leq C_{\mc{M}}\, \nrm{x}_\eta &&x \in X,\\
  \label{eq:etaHilbert2}
  \nrm{x_1+x_2}^2_\nu &= \nrm{x_1}^2_\nu+\nrm{x_2}^2_\nu, && x_1,x_2 \in X: x_1 \wedge x_2=0.
\end{align}

  Let $\Sigma'$ be a coarsest $\sigma$-algebra such that $\mb{u}$ and $\mb{v}$ are measurable and let $E_1,\ldots,E_m \in \Sigma'$ be the atoms of this finite $\sigma$-algebra.
   By applying \eqref{eq:vunderu}-\eqref{eq:etaHilbert2}, we get
   \begingroup
   \allowdisplaybreaks
   \begin{align*}
     \nrm{y_1}^2_\eta 
     &\leq C_{\mc{M}}^2 \sum_{j=1}^m \nrms{\has{\sum_{k=1}^n \abs{v_k}^2 }^{1/2}z_1 \ind_{E_j} }_\nu^2 && \qquad \text{by }\eqref{eq:etaHilbert1}+\eqref{eq:etaHilbert2}\\
     &=  C_{\mc{M}}^2 \sum_{k=1}^n \sum_{j=1}^m\nrm{v_k z_1 \ind_{E_j} }_\nu^2 && \qquad  \text{since } v_k \text{ is constant on } E_j\\
     &\leq C_{\mc{M}}^4 \sum_{k=1}^n\nrm{v_k z_1}_\eta^2 && \qquad \text{by }\eqref{eq:etaHilbert1}+\eqref{eq:etaHilbert2}\\
     &\leq C_{\mc{M}}^6 \sum_{k=1}^n\nrm{T_k(u_k z_0)}_\eta^2 && \qquad \text{by } \eqref{eq:vunderu}+\eqref{eq:modulusinequality} \\
     &\leq C_{\mc{M}}^{6}\,C_\Gamma^2 \sum_{k=1}^n\nrm{u_k z_0}_\eta^2 && \qquad \text{by } \eqref{eq:gammainequality} \\
     &\leq C_{\mc{M}}^{8}\,C_\Gamma^2 \sum_{k=1}^n \sum_{j=1}^m \nrm{u_k z_0 \ind_{E_j}}_\nu^2 && \qquad \text{by } \eqref{eq:etaHilbert1}+\eqref{eq:etaHilbert2} \\
     &= C_{\mc{M}}^{8}\,C_\Gamma^2 \sum_{j=1}^m  \nrms{\Big(\sum_{k=1}^n \abs{u_k}^2\Big)^{1/2} z_0 \ind_{E_j}}_\nu^2 && \qquad \text{since } u_k \text{ is constant on } E_j\\
     & \leq C_{\mc{M}}^{10}\,C_\Gamma^2 \nrm{y_0 }_\eta^2 && \qquad \text{by }\eqref{eq:etaHilbert1}+\eqref{eq:etaHilbert2}
   \end{align*}
   \endgroup
   Hence, by combining this estimate with \eqref{eq:eta2} and \eqref{eq:eta1}, we get
   \begin{equation*}
     \nrm{y_1}_X \leq \nrm{y_1}_\eta \leq C_{\mc{M}}^5\,C_\Gamma\nrm{y_0}_\eta \leq 4\,C_{\mc{M}}^5\,C_\Gamma \nrm{y_0}_X,
   \end{equation*}
which concludes the proof.
\end{proof}

In view of Theorem \ref{theorem:BFSlatticestructure} it would be interesting to investigate sufficient conditions on a general Banach space $X$ such that $\alpha$-boundedness of a family of operators on $X$ implies e.g. $\gamma$-boundedness.

\begin{remark}~
  \begin{enumerate}[(i)]
  \item By Theorem \ref{theorem:Cboundedalphabounded} one could replace the assumption on $\Gamma$ and $\mc{M}$ in Theorem \ref{theorem:BFSlatticestructure} by the assumption that $\Gamma \cup {\mc{M}}$ is $C^*$-bounded.
  \item \label{it:Deltagammabdd} ${\mc{M}}$ is $\gamma$-bounded if and only if $X$ has finite cotype. Indeed, the `if' statement follows from Proposition \ref{proposition:compareEuclidean} and the fact that $\nrm{{\mc{M}}}_{\ell^2} =1$. The `only if' part follows from  a variant of \cite[Example 8.1.9]{HNVW17} and the fact that if $X$ does not have finite cotype, then $\ell^\infty_n$ is $(1+\varepsilon)$-lattice finitely representable in $X$ for any $n \in \N$  (see \cite[Theorem 1.f.12]{LT79}).
  \item The assumption that ${\mc{M}}$ is $\alpha$-bounded is not only sufficient, but also necessary in Theorem \ref{theorem:BFSlatticestructure} if $\alpha = \gamma$. Indeed, for the $\gamma$-structure we know that $\gamma$-boundedness implies $\ell^2$-boundedness if and only if $X$ has finite cotype, see \cite[Theorem 4.7]{KVW16}. Therefore if $\gamma$-boundedness implies $\ell^2$-boundedness on $X$, then $X$ has finite cotype. This implies that ${\mc{M}}$ is $\gamma$-bounded.
  \end{enumerate}
\end{remark}

\begin{remark}\label{remark:ellq}
On a Banach function space $X$ one can also define for $q \in [1,\infty)$
\begin{equation*}
  \nrm{\mb{x}}_{\ell^q} := \nrms{\hab{\sum_{k=1}^n \abs{x_k}^q}^{1/q}}_X, \qquad \mb{x} \in X^n.
\end{equation*}
and study the $\ell^q$-boundedness of operators, which was initiated in \cite{We01} and done systematically in \cite{KU14}. Our representation results of Chapter \ref{part:1} rely heavily on the Hilbert structure of $\ell^2$ and therefore a generalization of our representation results to an ``$\ell^q$-Euclidean structure'' setting seems out of reach.
\end{remark}

\section{Factorization of \texorpdfstring{$\ell^2$}{l2}-bounded operator families through \texorpdfstring{$L^2(S,w)$}{L2(S,w)}}\label{section:BFSfactorization}  As we have seen in the previous section, the $\ell^2$-structure is the canonical structure to consider on a Banach function space $X$. In this section we prove a version of Lemma \ref{lemma:technicalcase} for the $\ell^2$-boundedness of a family of operators on a Banach function space, in which we have control over the the space $X_\eta$ and the Hilbertian seminorm $\nrm{\cdot}_{\eta}$. Indeed, we will see that an $\ell^2$-bounded family of operators on a Banach function space $X$ can be factorized through a weighted $L^2$-space. In fact, this actually characterizes $\ell^2$-boundedness on $X$.

By a \emph{weight} on a measure space $(S,\mu)$ we mean a measurable function $w:S \to [0,\infty)$. For $p \in [1,\infty)$ we let $L^p(S,w)$ be the space of all $f \in L^0(S)$ such that
\begin{equation*}
  \nrm{f}_{L^p(S,w)}:=\has{\int_S \abs{f}^pw\dd \mu}^{1/p}<\infty.
\end{equation*}
Our main result is as follows. For the special case $X = L^p(S)$ this result can be found in the work of Le Merdy and Simard \cite[Theorem 2.1]{LS02}. See also Johnson and Jones \cite{JJ78} and Simard \cite{Si99}.

\begin{theorem}
\label{theorem:BFSfactorization}
  Let $X$ be an order-continuous Banach function space on $(S,\Sigma, \mu)$ and let $\Gamma \subseteq \mc{L}(X)$. $\Gamma$ is $\ell^2$-bounded if and only if there exists a constant $C>0$ such that for all $y_0,y_1 \in X$ there exists a weight $w$ such that $y_0,y_1 \in L^2(S,w)$ and
  \begin{align}
  \label{eq:main1}
  \nrm{Tx}_{L^2(S,w)} &\leq C \, \nrm{x}_{L^2(S,w)}, \quad x \in X \cap L^2(S,w),\, T \in \Gamma\\
    \label{eq:main2}
    \nrm{y_0}_{L^2(S,w)}&\leq c \,\nrm{y_0}_X,\\
    \label{eq:main3}
    \nrm{y_1}_{L^2(S,w)} &\geq \tfrac1c \, \nrm{y_1}_X,
  \end{align}
  where $c$ is a numerical constant. Moreover $C$ can be chosen such that $\nrm{\Gamma}_{\ell^2} \simeq C$.
\end{theorem}

\begin{proof}
We will first prove the `if' part. Let $\mb{x} \in X^n$ and $T_1,\ldots,T_n \in \Gamma$. Define $y_0 = \ha{\sum_{k=1}^n\abs{x_k}^2}^\frac{1}{2}$ and $y_1 = \ha{\sum_{k=1}^n\abs{T_k x_k}^2}^\frac{1}{2}$. Then we have by applying \eqref{eq:main1}-\eqref{eq:main3}
\begin{equation*}
  \nrm{y_1}_X^2 \leq c^2 \sum_{k=1}^n \int_S \abs{T_kx_k}^2w\dd \mu \leq c^2C^2 \sum_{k=1}^n \int_S \abs{x_k}^2w\dd\mu \leq c^4C^2 \nrm{y_0}_X^2,
\end{equation*}
so $\nrm{\Gamma}_{\ell^2} \leq c^2 C$.

\bigskip

Now for the converse take $y_0,y_1 \in X$ arbitrary and let $\tilde{u} \in X$ with $\tilde{u}>0$ a.e. Assume without loss of generality that $\nrm{y_0}_X=\nrm{y_1}_X=\nrm{\tilde{u}}_X=1$ and define
\begin{equation*}
  u =  \frac{1}{3}\hab{\abs{y_0} \vee \abs{y_1} \vee \tilde{u}}.
\end{equation*}
Then $\nrm{u}_X \leq 1$, $u>0$ a.e. and
\begin{equation}\label{eq:yjnorm}
 \nrm{y_j^2 u^{-1}}_X \leq \nrm{y_j}_X \nrm{y_ju^{-1}}_{L^\infty(S)} \leq 3 \,\nrm{y_j}_X, \qquad j=0,1.
\end{equation}
Let $Y = \cbrace{x \in X: x^2u^{-1} \in X}$ with norm $\nrm{x}_Y := \nrm{x^2u^{-1}}_X^{1/2}$.
Then $Y$ is  an order-continuous Banach function space and for $\mb{v}\in Y^n$ we have
\begin{align*}
    \nrms{\has{\sum_{k=1}^n \abs{v_k}^2}^{1/2}}_Y &=   \nrms{\sum_{k=1}^n \abs{v_k}^2u^{-1}}_X^{1/2} \\&\leq \has{\sum_{k=1}^n  \nrmb{\abs{v_k}^2u^{-1}}_X}^{1/2} = \has{\sum_{k=1}^n  \nrm{v_k}_Y^2}^{1/2},
\end{align*}
i.e. $Y$ is $2$-convex. Moreover by H\"olders inequality for Banach function spaces (\cite[Proposition 1.d.2(i)]{LT79}), we have
\begin{equation*}
  \nrm{x}_X \leq \nrm{x^2u^{-1}}_X^{1/2}\nrm{u}_X^{1/2} = \nrm{x}_Y,
\end{equation*}
so $Y$ is contractively embedded in $X$.
By \eqref{eq:yjnorm} we have $u,y_0,y_1 \in Y$.

We will now apply Lemma \ref{lemma:technical}. Define $F:X \to [0,\infty)$ by
\begin{equation*}
  F(x) = \begin{cases}
    \nrm{x}_X &\text{if } x \in \spn\{y_1\}\\
    0 &\text{otherwise}
  \end{cases}
\end{equation*}
and $G:Y \to [0,\infty)$ by $G(x) = \nrm{x}_Y$. Then \eqref{eq:F}  holds by Proposition \ref{proposition:finitedimensionalalpha} and \eqref{eq:G} follows from the $2$-convexity of $Y$.
Let $\mc{M}$ be the pointwise multiplication operators as in \eqref{eq:mcM} and define
  \begin{equation*}
    \Gamma_0 := \hab{\tfrac{1}{2\nrm{\Gamma}_{\ell^2}} \cdot \Gamma}  \cup \hab{\tfrac12 \cdot  \mc{M}}.
  \end{equation*}
Then $\Gamma_0$ is $\alpha$-bounded with $\nrm{\Gamma_0}_{\ell^2}\leq 1$ by Proposition \ref{proposition:alphaproperties}. Applying Lemma \ref{lemma:technical} to $\Gamma_0$, we can find a $\Gamma$- and $\mc{M}$-invariant subspace $Y\subseteq X_0 \subseteq X$ and a Hilbertian seminorm $\nrm{\lcdot}_0$  such that
\begin{align*}
  \nrm{Tx}_0 &\leq 4 \,\nrm{\Gamma}_{\ell^2} \nrm{x}_0, && x \in X_0,\, T \in \Gamma,\\
  \nrm{Tx}_0 &\leq 4 \,\nrm{x}_0, && x \in X_0,\, T \in \mc{M},\\
    \nrm{x}_0 &\leq 4 \,\nrm{x}_Y, &&x \in Y,\\
        \nrm{y_1}_0 &\geq \nrm{y_1}_X.
\end{align*}
The second property implies that if $x \in L^0(S)$ and  $\tilde{x} \in X_0$ with $\abs{{x}} \leq \abs{\tilde{x}}$, then ${x} \in X_0$ and $\nrm{{x}}_0 \leq 4\,\nrm{\tilde{x}}_0$. Thus we may, at the
the loss of a numerical constant, furthermore assume
\begin{align}\label{eq:squarenormidentity}
  \nrm{x_1+x_2}^2_0 &= \nrm{x_1}^2_0+\nrm{x_2}^2_0, && x_1,x_2 \in X: x_1 \wedge x_2=0
\end{align}
by Lemma \ref{lemma:tychonoff}.

Define a measure $\lambda(E) = \nrm{u\ind_E}_0^2$ for all $E \in \Sigma$. Using \eqref{eq:squarenormidentity}, the $\sigma$-additivity of this measure follows from
\begin{align*}
  \lambda\has{\bigcup_{k=1}^\infty E_k} &= \lim_{n \to \infty} \nrms{\sum_{k=1}^n u\ind_{E_k}}^2_0= \sum_{k=1}^\infty \lambda(E_k)
\end{align*}
for $E_1,E_2,\ldots \in \Sigma$ pairwise disjoint, since $u \ind_E \in Y$ for any $E \in \Sigma$ and $Y$ is order-continuous. Moreover we have for any $E \in \Sigma$ with $\mu(E) = 0$ that
\begin{equation*}
  \lambda(E) = \nrm{u\ind_E}_0^2 \lesssim  \nrm{u\ind_E}_Y^2 = \nrm{\ind_\varnothing}_Y^2 =0
\end{equation*}
so $\lambda$ is absolutely continuous with respect to $\mu$. Thus, by the Radon-Nikodym theorem, we can find an $f \in L^1(S)$ such that
\begin{equation*}
  \nrm{u\ind_E}_0^2 = \lambda(E) =  \int_E f \dd\mu
\end{equation*}
for all $E \in \Sigma$. Define the weight $w := u^{-2}f$.

Take $x \in Y$ and let $(v_n)_{n=1}^\infty$ be a sequence functions of the form
\begin{equation*}
  v_n = u\sum_{j=1}^{m_n} a^n_j \ind_{E^n_j},\qquad a_j^n \in \C, \, E^n_j \in \Sigma,
\end{equation*}
such that $\abs{v_n} \uparrow \abs{x}$. Then $\lim_{n \to \infty} \nrm{v_n - x}_0 =0$ by the order-continuity of $Y$. So by \eqref{eq:squarenormidentity} and the monotone convergence theorem
\begin{align*}
  \nrm{x}_0^2 &=
  \lim_{n \to \infty} \sum_{j=1}^{m_n} \abs{a^n_j}^2\nrm{u\ind_{E^n_j}}_0^2 = \lim_{n \to \infty} \sum_{j=1}^{m_n} \int_{E^n_j} \abs{a^n_j}^2 u^2 w \dd \mu = \int_S \abs{x}^2w \dd\mu.
\end{align*}
In particular, $y_0,y_1 \in L^2(S,w)$ and, using  \eqref{eq:yjnorm}, we have
\begin{align*}
  \has{\int_S\abs{y_0}^2w\dd\mu}^\frac{1}{2} &= \nrm{y_0}_0 \leq 16\, \nrm{y_0}_Y\leq 48 \,\nrm{y_0}_X,\\
  \has{\int_S\abs{y_1}^2w\dd\mu}^\frac{1}{2} &= \nrm{y_1}_0 \geq \frac14\, \nrm{y_1}_X,
\end{align*}
so we can take $c=48$.
Take $T \in \Gamma$ and $x \in Y$ and define $m_n = \min(1, nu \cdot \abs{Tx}^{-1})$ for $n \in \N$. Then $m_n\cdot Tx \in Y$ and  $\abs{m_n\cdot Tx} \uparrow \abs{Tx}$. So by the monotone convergence theorem we have
\begin{equation*}
    \begin{aligned}
  \has{\int_S\abs{Tx}^2w\dd\mu}^\frac{1}{2} &= \lim_{n \to \infty} \has{\int_S\abs{m_n\cdot Tx}^2w\dd\mu}^\frac{1}{2}
   \\
   &= \lim_{n \to \infty} \nrm{m_n\cdot Tx}_0 \\&\leq  4^6\, \nrm{\Gamma}_{\ell^2}\,\has{\int_S\abs{x}^2w\dd\mu}^\frac{1}{2}.
   \end{aligned}
\end{equation*}
To conclude, note that $Y$ is dense in $X \cap L^2(S,w)$ by order-continuity. Therefore, since $T$ is bounded on $X$ as well, this estimate extends to all $x \in X \cap L^2(S,w)$. This means that \eqref{eq:main1}-\eqref{eq:main3} hold with $C \leq 4^6 \,\nrm{\Gamma}_{\ell^2}$.
\end{proof}

\begin{remark}
  In the view of Theorem \ref{theorem:Cboundedalphabounded} and Theorem \ref{theorem:BFSlatticestructure} we may replace the assumption that $\Gamma$ is $\ell^2$-bounded by the assumption that $\Gamma \cup \mc{M}$ is $C^*$-bounded in Theorem \ref{theorem:BFSfactorization}.
\end{remark}

\subsection*{The role of $2$-convexity}
If the Banach function space $X$ is $2$-convex, i.e. if
\begin{equation*}
  \nrms{\has{\sum_{k=1}^n \abs{x_k}^2}^{1/2}}_X \leq \has{\sum_{k=1}^n \nrm{x_k}_X^2}^{1/2}, \qquad \mb{x} \in X^n,
\end{equation*}
we do not have to construct a $2$-convex Banach function space $Y$ as we did in the proof of Theorem \ref{theorem:BFSfactorization}. Instead, we can just use $X$ in place of $Y$, which yields more stringent conditions on the weight in Theorem \ref{theorem:BFSfactorization}.

\begin{theorem}
\label{theorem:BFSfactorization2convex}
 Let $X$ be an order-continuous, $2$-convex Banach function space on $(S,\Sigma,\mu)$ and let $\Gamma\subseteq \mc{L}(X)$. Then $\Gamma$ is $\ell^2$-bounded if and only if there exists a constant $C>0$ such that for any weight $w$ with
  \begin{align*}
    \nrm{x}_{L^2(S,w)} &\leq \nrm{x}_X \qquad x \in X,
  \intertext{there exists a weight $v \geq w$ such that}
     \nrm{Tx}_{L^2(S,v)} &\leq C\, \nrm{x}_{L^2(S,v)} &&\qquad x \in X, \, T \in \Gamma\\
    \nrm{x}_{L^2(S,v)} &\leq c \,  \nrm{x}_X &&\qquad x \in X,
  \end{align*}
  where $c$ is a numerical constant. Moreover $C>0$ can be chosen such that $\nrm{\Gamma}_{\ell^2} \simeq C$.
\end{theorem}

\begin{proof}
  The proof is similar to, but simpler than, the proof of Theorem \ref{theorem:BFSfactorization}. The a priori given weight $w$ allows us to define $F:X \to [0,\infty)$  as
  $$F(x) = \has{\int_S\abs{x}^2w\dd \mu}^{1/2}$$
  and the $2$-convexity allows us to use $Y=X$ and define $G:X \to [0,\infty)$ as $G(x) = \nrm{x}_X$. For more details, see \cite[Theorem 4.6.3]{Lo16}
\end{proof}

\begin{remark}
  Theorem \ref{theorem:BFSfactorization2convex} is closely related to the work of Rubio de Francia, which was preceded by the factorization theory of Niki{\v{s}}in \cite{Ni70} and Maurey \cite{Ma73}. In his work Rubio de Francia proved Theorem \ref{theorem:BFSfactorization2convex} with all $2$'s replaced by any $q \in [1,\infty)$ for the following special cases:
   \begin{itemize}
     \item For $X=L^p(S)$ in \cite{Ru82b},
     \item For $\Gamma =\cbrace{T}$ with $T \in \mc{L}(X)$ in \cite[III Lemma 1]{Ru86},
   \end{itemize}
   see also \cite{GR85}. These results have been combined in \cite[Lemma 3.4]{ALV17}, yielding Theorem \ref{theorem:BFSfactorization2convex} with all $2$'s replaced by any $q \in [1,\infty)$. These results are proven using  different techniques and for $q \neq 2$, as discussed in Remark \ref{remark:ellq}, seem out of reach using our approach.
\end{remark}

\section{Banach function space-valued extensions of operators} \label{section:vectorextension}
In this final section on the $\ell^2$-structure on Banach function spaces  we will apply Theorem \ref{theorem:BFSfactorization} to obtain an extension theorem in the spirit of Rubio de Francia's extension theorem for Banach function space-valued functions (see \cite[Theorem 5]{Ru86}). We will apply this theorem to deduce the following results related to the $\UMD$ property for Banach function spaces:
\begin{itemize}
  \item We will provide a quantitative proof of the boundedness of the lattice Hardy-Littlewood maximal function on $\UMD$ Banach function spaces.
  \item We will show that the so-called dyadic $\UMD^+$ property is equivalent to the $\UMD$ property for Banach function spaces.
  \item We will show that the $\UMD$ property is necessary for the $\ell^2$-sectoriality of certain differentiation operators on $L^p(\R^d;X)$, where $X$ is a Banach function space.
\end{itemize}

\subsection*{Tensor extensions and Muckenhoupt weights}
Let us first introduce the notions we need to state the main theorem of this section. Let $p \in [1,\infty)$, $w$ a weight on $\R^d$ and suppose that $T$ is a bounded linear operator on $L^p(\R^d,w)$. We may define a linear operator $\widetilde{T}$ on the tensor product $L^p(\R^d,w) \otimes X$ by setting
\begin{equation*}
  \widetilde{T}(f\otimes x) : = Tf \otimes x, \qquad f \in L^p(\R^d,w),\, x \in X,
\end{equation*}
and extending by linearity. For $p \in [1,\infty)$ the space $L^p(\R^d,w) \otimes X$ is dense in the Bochner space $L^p(\R^d,w;X)$ and it thus makes sense to ask whether $\widetilde{T}$ extends to a bounded operator on $L^p(\R^d,w;X)$. If this is the case, we will denote this operator again by $\widetilde{T}$. For a family of operators $\Gamma \subseteq \mc{L}(L^p(\R^d,w))$ we write $$\widetilde{\Gamma}:= \cbrace{\widetilde{T}:T \in \Gamma}.$$

We denote the Lebesgue measure $\lambda$ on $\R^d$ of a measurable set $E \subseteq \R^d$ by $\abs{E}$. For $p \in (1,\infty)$ we will say that a weight $w$ on $\R^d$ is in the \emph{Muckenhoupt class $A_p$} and write $w \in A_p$ if the weight characteristic
\begin{equation*}
  [w]_{A_p} := \sup_{Q} \frac{1}{\abs{Q}} \int_Q w\dd \lambda \cdot\has{\frac{1}{\abs{Q}}\int_Q w^{-\frac{1}{p-1}}\dd \lambda}^{p-1}
\end{equation*}
is finite, where the supremum is taken over all cubes $Q \subseteq \R^d$ with sides parallel to the axes.

\subsection*{An abstract extension theorem}
We can now state the main theorem of this section.

\begin{theorem}\label{theorem:BFSextrapolation}
  Let $X$ be an order-continuous Banach function space on $(S,\Sigma,\mu)$, let $p \in (1,\infty)$ and $w \in A_p$. Assume that there is a family of operators $\Gamma \subseteq \mc{L}(L^p(\R^d,w))$ and an increasing function $\phi\colon\R_+ \to \R_+$ such that
  \begin{itemize}
    \item For all weights $v\colon \R^d \to (0,\infty)$ we have
    \begin{equation*}
      [v]_{A_2} \leq \phi\has{\sup_{T \in \Gamma}\, \nrm{T}_{\mc{L}(L^2(\R^d,v))}}.
    \end{equation*}
    \item $\widetilde{\Gamma}$ is $\ell^2$-bounded on $L^p(\R^d,w;X)$.
  \end{itemize}
  Let $f,g \in L^p(\R^d,w;X)$ and suppose that there is an increasing function $\psi\colon \R_+ \to \R_+$ such that for all $v \in A_2$ we have
  \begin{align*}
    \nrm{f(\cdot,s)}_{L^2(\R^d,v) }&\leq \psi( [v]_{A_2}) \nrm{g(\cdot,s)}_{L^2(\R^d,v) }, \qquad  s \in S.
  \intertext{Then there is a numerical constant $c>0$ such that}
    \nrm{f}_{L^p(\R^d,w;X)} &\leq c \cdot \psi \circ \phi\hab{c\, \nrm{\widetilde{\Gamma}}_{\ell^2} }\nrm{g}_{L^p(\R^d,w;X)}.
  \end{align*}
\end{theorem}

One needs to take care when considering $f(\cdot,s)$ for $f \in L^p(\R^d,w;X)$ and $s \in S$ in Theorem \ref{theorem:BFSextrapolation}, as this is not necessarily a function in $L^2(\R^d,v)$. This technicality can in applications be circumvented by  only using e.g. simple functions or smooth functions with compact support and a density argument.

\begin{proof}[Proof of Theorem \ref{theorem:BFSextrapolation}]
  Let $u \in L^p(\R^d,w)$ be such that there is a $c_K>0$ with $u \geq c_K \ind_K$ for every compact $K \subseteq \R^d$. Let $x \in X$  be such that $x>0$ a.e. and $$\nrm{u \otimes x}_{L^p(\R^d,w;X)}\leq \nrm{g}_{L^p(\R^d,w;X)}.$$
  Since $X$ is order-continuous, $L^p(\R^d,w;X)$ is an order-continuous Banach function space over the measure space $(\R^d \times S,w\dd \lambda\dd \mu)$, so by Theorem \ref{theorem:BFSfactorization} we can find a weight $v$ on $\R^d \times S$ and a numerical constant $c>0$ such that for all $T \in \widetilde{\Gamma}$ and $h \in L^p(\R^d,w;X) \cap L^2(\R^d \times S,v\cdot w)$
  \begin{align}
  \label{eq:fact1} \nrm{T h}_{L^2(\R^d \times S,v\cdot w)} &\leq c \, \nrm{\widetilde{\Gamma}}_{\ell^2} \,\nrm{h}_{L^2(\R^d \times S,v\cdot w)}\\
   \label{eq:fact2} \nrmb{\abs{g} + u \otimes x}_{L^2(\R^d \times S,v\cdot w)} &\leq c \, \nrmb{\abs{g} + u \otimes x}_{L^p(\R^d,w;X)},\\
    \label{eq:fact3}\nrm{f}_{L^2(\R^d \times S,v\cdot w)} &\geq \frac{1}{c}\, \nrm{f}_{L^p(\R^d,w;X)}.
  \end{align}
  Note that \eqref{eq:fact2} and the definition of $x$ imply
   \begin{equation}
     \label{eq:fact2alt} \nrm{g}_{L^2(\R^d \times S,v\cdot w)} \leq 2c\, \nrm{g}_{L^p(\R^d,w;X)}
   \end{equation}
     and $u \in L^2(\R^d,v(\cdot,s)\cdot w)$ for $\mu$-a.e. $s \in S$. Therefore, by the definition of $u$, we know that $v(\cdot,s)\cdot w$ is locally integrable on $\R^d$ for $\mu$-a.e. $s \in S$. Let $\mc{A}$ be the $\Q$-linear span of indicator functions of rectangles with rational corners, which is
 a countable, dense subset of both $L^p(\R^d,w)$ and $L^2(\R^d,v(\cdot,s)\cdot w)$ for $\mu$-a.e. $s \in S$. Define
  \begin{align*}
    \mc{B} &= \cbraceb{\psi \otimes (x\ind_E): \psi \in \mc{A}, \, E \in \Sigma }.
  \end{align*}
  Then $\mc{B} \subseteq L^p(\R^d,w;X) \cap L^2(\R^d \times S,v\cdot w)$ since $u \otimes x\in L^2(\R^d \times S,v\cdot w)$. Testing \eqref{eq:fact1} on all $h \in \mc{B}$ we find that for all $T \in \Gamma$ and $\psi \in \mc{A}$
  \begin{equation*}
    \nrm{T\psi}_{L^2(\R^d,v(\cdot,s)\cdot w)} \leq c\,\nrm{\widetilde{\Gamma}}_{\ell^2} \,\nrm{\psi}_{L^2(\R^d,v(\cdot,s)\cdot w)}, \qquad s \in S.
  \end{equation*}
Since $\mc{A}$ is countable and dense in $L^2(\R^d,v(\cdot,s)\cdot w)$, we have by assumption that $v(\cdot,s)\cdot w \in A_2$ with $[v(\cdot,s)\cdot w]_{A_2} \leq \phi(c\,\nrm{\widetilde{\Gamma}}_{\ell^2})$ for $\mu$-a.e. $s\in S$. Therefore, using Fubini's theorem, our assumption, \eqref{eq:fact3} and \eqref{eq:fact2alt}, we obtain
\begin{align*}
  \nrm{f}_{L^p(\R^d,w;X)} &\leq c\, \has{\int_{S} \int_{\R^d} \abs{f}^2 v\cdot w \dd \lambda \dd \mu}^{1/2}
  \\&\leq c\cdot \psi\circ \phi\hab{c\,\nrm{\widetilde{\Gamma}}_{\ell^2}} \has{\int_{S} \int_{\R^d} \abs{g}^2 v\cdot w \dd \lambda \dd \mu}^{1/2}
  \\&\leq 2c^2 \cdot \psi\circ \phi\hab{c\,\nrm{\widetilde{\Gamma}}_{\ell^2}}\nrm{g}_{L^p(\R^d,w;X)},
\end{align*}
proving the statement.
\end{proof}

We say that a Banach space $X$ has the {$\UMD$ property} if the martingale difference sequence of any finite martingale in $L^p(S;X)$ on a $\sigma$-finite measure space $(S,\mu)$ is unconditional for some (equivalently all) $p \in (1,\infty)$. That is, if
  for all finite martingales $(f_k)_{k=1}^n$ in $L^p(S;X)$ and scalars $|\epsilon_k|=1$ we have
 \begin{equation}\label{eq:UMD}
   \nrms{\sum_{k=1}^n \epsilon_kdf_k}_{L^p(S;X)} \leq C\, \nrms{\sum_{k=1}^n df_k}_{L^p(S;X)}.
 \end{equation}
 The least admissible constant $C>0$ in \eqref{eq:UMD} will be denoted by $\beta_{p,X}$.  For a detailed account of the theory $\UMD$ Banach spaces we refer the reader to \cite[Chapter 4]{HNVW16} and \cite{Pi16}.

Let us point out some choices of $\Gamma \subseteq \mc{L}(L^p(\R^d,w))$ that satisfy the assumptions Theorem \ref{theorem:BFSextrapolation} when $X$ has $\UMD$ property:
\begin{itemize}
\item $\Gamma = \cbrace{H}$, where $H$ is the Hilbert transform.
\item  $\Gamma = \cbrace{R_k:k=1,\ldots,d}$ where $R_k$ is the $k$-th Riesz projection.
\item $\Gamma:= \cbrace{T_Q:Q \text{ a cube in }\R^d}$, where $T_Q\colon L^p(\R^d) \to L^p(\R^d)$ is the averaging operator
  \begin{equation*}
    T_Qf(t) := \has{\frac{1}{\abs{Q}}\int_Q f \dd \lambda} \ind_Q(t), \qquad t \in \R^d.
  \end{equation*}
\end{itemize}
We will encounter these choices of $\Gamma$ in the upcoming applications of Theorem \ref{theorem:BFSextrapolation}. For these choices of $\Gamma$ one obtains an extension theorem for $\UMD$ Banach function spaces in the spirit of \cite[Theorem 5]{Ru86}.

\begin{corollary}\label{corollary:extrapolationUMD}
Let $X$ be a $\UMD$ Banach function space and let $T$ be a bounded linear operator on $L^{p_0}(\R^d,v)$ for some $p_0 \in (1,\infty)$ and all $v \in A_{p_0}$. Suppose that there is an increasing function $\phi\colon\R_+\to \R_+$ such that
  \begin{equation*}
    \nrm{T}_{L^{p_0}(\R^d,v)\to L^{p_0}(\R^d,v)} \leq \phi([v]_{A_{p_0}}), \qquad v \in A_{p_0}.
  \end{equation*}
Then $\widetilde{T}$ extends uniquely to a bounded linear operator on $L^p(\R^d,w;X)$ for all $p \in (1,\infty)$ and $w \in A_p$.
\end{corollary}

\begin{proof}
For $k=1,\cdots,d$ let $R_k$ denote the $k$-th Riesz projection on $L^p(\R^d,w)$ and set $\Gamma = \cbrace{R_k:k=1,\ldots,d}$. Then we have for any weight $v$ on $\R^d$ that
    \begin{equation*}
      [v]_{A_2} \lesssim_d \has{\sup_{T \in \Gamma}\, \nrm{T}_{L^2(\R^d,v)\to L^2(\R^d,v)}}^4.
    \end{equation*}
by \cite[Theorem 7.4.7]{Gr14a}. Moreover by the triangle inequality, the ideal property of the $\ell^2$-structure, \cite[Theorem 5.5.1]{HNVW16} and \cite[Corollary 2.11]{HH14} we have
\begin{equation*}
  \nrm{\widetilde{\Gamma}}_{\ell^2} \lesssim \sum_{k=1}^d \nrm{R_k}_{L^p(\R^d,w;X)\to L^p(\R^d,w;X)} \lesssim_{X,p,d} [w]_{A_p}^{\max\cbrace{\frac{1}{p-1},1}}
\end{equation*}
Thus $\Gamma$ satisfies the assumptions of Theorem \ref{theorem:BFSextrapolation}. Let $f \in L^p(\R^d,w)\otimes X$.  By Rubio de Francia extrapolation (see \cite[Theorem 3.9]{CMP12}) there is an increasing function $\psi\colon\R_+\to \R_+$, depending on $\phi,p,p_0,d$, such that
  for all $v \in A_{2}$ we have
  \begin{equation*}
    \nrm{Tf(\cdot,s)}_{L^2(\R^d,v)} \leq \psi([v]_{A_{2}}) \nrm{f(\cdot,s)}_{L^2(\R^d,v)}, \qquad s \in S.
  \end{equation*}
  Therefore by Theorem \ref{theorem:BFSextrapolation} we obtain
  \begin{equation*}
       \nrm{Tf}_{L^p(\R^d,w;X)} \leq c\cdot  \psi\hab{C_{X,p,d}\cdot [w]_{A_p}^{4\cdot \max\cbrace{\frac{1}{p-1},1}} }\nrm{f}_{L^p(\R^d,w;X)}
  \end{equation*}
 which yields the desired result by density.
\end{proof}

The advantages of Theorem \ref{theorem:BFSextrapolation} and Corollary \ref{corollary:extrapolationUMD} over \cite[Theorem 5]{Ru86} are as follows
\begin{itemize}
  \item Theorem \ref{theorem:BFSextrapolation} yields a quantitative estimate of the involved constants, whereas this dependence is hard to track in \cite[Theorem 5]{Ru86}.
  \item Theorem \ref{theorem:BFSextrapolation} and Corollary \ref{corollary:extrapolationUMD} allow weights in the conclusion, whereas  \cite[Theorem 5]{Ru86} only yields an unweighted extension.
  \item \cite[Theorem 5]{Ru86} relies upon the boundedness of the lattice Hardy-Littlewood maximal operator on $L^p(\R^d;X)$,  whereas this is not used in the proof of Theorem \ref{theorem:BFSextrapolation}. Therefore, we can use Theorem \ref{theorem:BFSextrapolation} to give a quantitative proof of the boundedness of the lattice Hardy-Littlewood maximal operator on $\UMD$ Banach function spaces, see Theorem \ref{theorem:latticemaximal}.
  \item Instead of assuming the $\UMD$ property of $X$, the assumptions of Theorem \ref{theorem:BFSextrapolation} are flexible enough to allow one to deduce the $\UMD$ property of $X$ from $\ell^2$-boundedness of other operators, see Theorem \ref{theoreml2:sectoriality}.
\end{itemize}

\begin{remark}
  Rubio de Francia's extension theorem for $\UMD$ Banach function spaces  has also been generalized  in \cite{ALV17,LN19,LN20}:
\begin{itemize}
\item  In \cite[Corollary 3.6]{ALV17} a rescaled version of Corollary \ref{corollary:extrapolationUMD} has been obtained by adapting the original proof of Rubio de Francia.
\item In \cite{LN19} the proof of \cite[Theorem 5]{Ru86} has been generalized to allow for a multilinear limited range variant. The proof of Theorem \ref{theorem:BFSextrapolation} does not lend itself for such a generalization.
\item Using the stronger assumption of sparse domination, the result in \cite{LN19} has been made quantitative and has been extended to multilinear weight and $\UMD$ classes in \cite{LN20}.
\end{itemize}
\end{remark}

\subsection*{The lattice Hardy--Littlewood maximal operator}\label{subsection:latticemaximal}
As a first application of  Theorem \ref{theorem:BFSextrapolation}, we will show the boundedness of the lattice Hardy--Littlewood maximal operator on $\UMD$ Banach function spaces. Let $X$ be an order-continuous Banach function space and $p \in (1,\infty)$.
For $f \in L^p(\R^d;X)$ the lattice Hardy--Littlewood maximal operator is defined as
\begin{equation*}
  \widetilde{M}f(t) := \sup_{Q\ni t} \,\has{\frac{1}{\abs{Q}}\int_Q \abs{f} \dd \lambda}, \qquad t \in \R^d,
\end{equation*}
where the supremum is taken in the lattice sense over cubes $Q \subseteq \R^d$ containing $t$ (see \cite{GMT93} or \cite[Section 5]{HL17} for the details).
 The boundedness of the lattice Hardy-Littlewood maximal operator for $\UMD$ Banach function spaces $X$ is a deep result shown by Bourgain \cite{Bo84} and Rubio de Francia \cite{Ru86}. Using this result, the following generalizations were subsequently shown on $\UMD$ Banach function spaces:
 \begin{itemize}
   \item Garc\'ia-Cuerva, Mac\'ias and Torrea showed in \cite{GMT93} that $\widetilde{M}$ is bounded on $L^p(\R^d,w;X)$ for all Muckenhoupt weights $w \in A_p$. Sharp dependence on the weight characteristic was obtained in \cite{HL17} by H\"anninen and the second author.
   \item Deleaval, Kemppainen and Kriegler showed in \cite{DKK18} that $\widetilde{M}$ is bounded on $L^p(S;X)$ for any space of homogeneous type $S$.
   \item Deleaval and Kriegler obtained dimension free estimates for a centered version of $\widetilde{M}$ on $L^p(\R^d;X)$ in \cite{DK17b}.
 \end{itemize}
With Theorem \ref{theorem:BFSextrapolation} we can reprove the result of Bourgain and Rubio de Francia and obtain an explicit estimate of the operator norm of $\widetilde{M}$ in terms of the $\UMD$ constant of $X$. Tracking this dependence in the proof of Bourgain and Rubio de Francia would be hard, as it involves the weight characteristic dependence of the inequality \cite[(a.5)]{Ru86}.

\begin{theorem}\label{theorem:latticemaximal}
  Let $X$ be a $\UMD$ Banach function space with cotype $q \in (1,\infty)$ with constant $c_{q,X}$. The lattice Hardy-Littlewood maximal operator $\widetilde{M}$ is bounded on $L^p(\R^d;X)$ for all $p \in (1,\infty)$ with
  \begin{equation*}
    \nrm{\widetilde{M}}_{\mc{L}(L^p(\R^d;X))} \lesssim \,q \bigl(c_{q,X}\beta_{p,X}\bigr)^2,
  \end{equation*}
  where the implicit constant only depends on $p$ and $d$.
\end{theorem}

\begin{proof}
   Let $p \in (1,\infty)$ and $f \in L^p(\R^d)$. Define for any cube $Q \subseteq \R^d$  the averaging operator
  \begin{equation*}
    T_Qf(t) := \has{\frac{1}{\abs{Q}}\int_Q f \dd \lambda} \ind_Q(t), \qquad t \in \R^d
  \end{equation*}
and set $\Gamma:= \cbrace{T_Q:Q \text{ a cube in }\R^d}.$ Then we know that $\widetilde{\Gamma}$ is $\ell^2$-bounded on $L^p(\R^d;X)$ with
  \begin{equation*}
    \nrm{\widetilde{\Gamma}}_{\ell^2} \lesssim \sqrt{q}c_{q,X} \beta_{p,X}
  \end{equation*}
  by \cite[Theorem 7.2.13 and Proposition 8.1.13]{HNVW17}, where the implicit constant depends on $p$ and $d$.

  Let $w\colon \R^d \to (0,\infty)$ and set $C:= \sup_{T \in \Gamma}\, \nrm{T}_{\mc{L}(L^2(\R^d,w))}$. Fix a cube $Q \subseteq \R^d$. Applying $T_Q$ to the function $(w+\varepsilon)^{-1} \ind_Q$ for some $\varepsilon>0$ we obtain
  \begin{equation*}
    \int_{Q} \has{\frac{1}{\abs{Q}} \int_Q (w(t)+\varepsilon)^{-1} \dd t}^2  w(s)\dd s  \leq C^2 \int_Q \frac{w(t)}{(w(t)+\varepsilon)^2} \dd t
  \end{equation*}
  which implies
  \begin{equation*}
    \has{\frac{1}{\abs{Q}}\int_Q w(t) \dd t} \has{\frac{1}{\abs{Q}} \int_Q (w(t)+\varepsilon)^{-1} \dd t } \leq C^2
  \end{equation*}
  So by letting $\varepsilon \to 0$ with the monotone convergence theorem, we obtain $w \in A_2$ with $[w]_{A_2} \leq C^2$. So $\Gamma$ satisfies the assumptions of Theorem \ref{theorem:BFSextrapolation} with $\phi(t) = t^2$.

By Theorem \ref{theorem:BFSextrapolation}, using the weighted boundedness of the scalar-valued Hardy-Littlewood maximal operator from \cite[Theorem 7.1.9]{Gr14a}, we know that for any simple function $f\colon \R^d \to X$
  \begin{equation*}
    \nrmb{\widetilde{M}f}_{L^p(\R^d;X)} \lesssim q \bigl(c_{q,X}\beta_{p,X}\bigr)^2\nrm{f}_{L^p(\R^d;X)}.
  \end{equation*}
  where the implicit constant depends on $p$ and $d$.  So, by the density of the simple functions in $L^p(\R^d;X)$, we obtain the desired result.
  \end{proof}

\begin{remark}\label{remark:maximalbdd}~
\begin{itemize}
\item One could also use $\Gamma = \cbrace{H}$ or $\Gamma = \cbrace{R_k:k=1,\cdots, d}$ in the proof of Theorem \ref{theorem:latticemaximal}, where $H$ is the Hilbert transform and $R_k$ is the $k$-th Riesz projection. Then the first assumption on $\Gamma$ in Theorem \ref{theorem:BFSextrapolation} follows from \cite[Theorem 7.4.7]{Gr14a} and the second from \cite[Theorem 5.1.1 and 5.5.1]{HNVW16} and the ideal property of the $\ell^2$-structure.
  \item In Theorem \ref{theorem:latticemaximal} the assumption that $X$ has finite cotype may be omitted, since the $\UMD$ property implies that there exists a constant $C_p>0$ such that $X$ has cotype $C_p \beta_{p,X}$ with constant less than $C_p$ (see \cite[Lemma 32]{HLN16}). This yields the bound $ \nrm{\widetilde{M}}_{\mc{L}(L^p(\R^d;X))} \lesssim \beta_{p,X}^3$  in the conclusion of Theorem \ref{theorem:latticemaximal}.
  \item One would be able to avoid the cotype constant in the conclusion of Theorem \ref{theorem:latticemaximal} if one can find a single operator $T$ that both characterizes $v \in A_p$ with $\phi(t) = t^2$ and is bounded on $L^p(\R^d;X)$ with $\nrm{T}_{\mc{L}(L^p(\R^d;X)} \lesssim \beta_{p,X}$.
\end{itemize}
\end{remark}

\subsection*{Randomized \texorpdfstring{$\UMD$}{UMD} properties}
As a second application of Theorem \ref{theorem:BFSextrapolation} we will prove the equivalence of the $\UMD$ property and the dyadic $\UMD^+$ property. Let us start by introducing the randomized $\UMD$ properties for a Banach space $X$.

We say that $X$ has the $\UMD^+$ (respectively $\UMD^-$) property if for some (equivalently all) $p \in (1,\infty)$ there exists a constant $\beta^+>0$ (respectively $\beta^->0$) such that for all finite martingales $(f_k)_{k=1}^n$ in $L^p(S;X)$ on a  $\sigma$-finite measure space $(S,\mu)$ we have
  \begin{equation}\label{eq:UMD+-}
  \frac{1}{\beta^-}\nrms{\sum_{k=1}^n df_k}_{L^p(S;X)} \leq \nrms{\sum_{k=1}^n \varepsilon_k df_k}_{L^p(S\times \Omega;X)} \leq \beta^+ \nrms{\sum_{k=1}^n df_k}_{L^p(S;X)},
\end{equation}
where $(\varepsilon_k)_{k=1}^n$ is a Rademacher sequence on $(\Omega,\P)$. The least admissible constants in \eqref{eq:UMD+-} will be denoted by $\beta_{p,X}^{+}$ and $\beta_{p,X}^-$.
If \eqref{eq:UMD+-} holds for all Paley-Walsh martingales on a probability space $(S,\mu)$ we say that $X$ has the dyadic $\UMD^+$ or $\UMD^-$ property respectively and denote the least admissible constants by $\beta_{p,X}^{\Delta,+}$ and $\beta_{p,X}^{\Delta,-}$.

As for the $\UMD$ property, the $\UMD^+$ and $\UMD^-$ properties are independent of $p \in (1,\infty)$ (see \cite{Ga90}). We trivially have $\beta_{p,X}^{\Delta,-} \leq \beta_{p,X}^{-}$ and $\beta_{p,X}^{\Delta,+} \leq \beta_{p,X}^{+}$. Furthermore we know that $X$ has the $\UMD$ property if and only if it has the $\UMD^+$ and $\UMD^-$ properties with
\begin{equation*}
  \max\cbrace{\beta^-_{p,X},\beta^+_{p,X}} \leq \beta_{p,X} \leq \beta_{p,X}^-\beta_{p,X}^+,
\end{equation*}
see e.g. \cite[Proposition 4.1.16]{HNVW16}. The relation between the norm of the Hilbert transform on $L^p(\mathbb{T};X)$ and $\beta_{p,X}^{\Delta,+}$ and $\beta_{p,X}^{\Delta,-}$ has recently been investigated in \cite{OY18} and the $\UMD^+$ property was recently shown to be equivalent to a recoupling property for tangent martingales in \cite{Ya20}. We refer to \cite{HNVW16, Ve07} for further information on these randomized $\UMD$ properties.

Two natural questions regarding these randomized $\UMD$ properties are the following:
\begin{itemize}
  \item Does either the $\UMD^-$ property or the $\UMD^+$ property imply the $\UMD$ property? For the $\UMD^-$ property it turns out that this is not the case, as any $L^1$-space has it, see \cite{Ga90}. For the $\UMD^+$ property this is an open problem. For general Banach spaces it is known that one cannot expect a better than quadratic bound relating $\beta_{p,X}$ and $\beta_{p,X}^+$ (see \cite[Corollary 5]{Ge99}).
  \item The dyadic $\UMD$ property implies its non-dyadic counterpart. Does the same hold for the dyadic $\UMD^+$ and $\UMD^-$ properties? For the $\UMD^-$ property it is known that the constants $\beta_{p,X}^-$ and $\beta_{p,X}^{\Delta,-}$ are not the same in general, as explained in \cite{CV11}.
\end{itemize}
Using Theorem \ref{theorem:BFSextrapolation}, we will show that on Banach function spaces the dyadic $\UMD^+$ property implies the $\UMD$ property (and thus also the $\UMD^+$ property), with a quadratic estimate of the respective constants.

The equivalence of the $\UMD^+$ property and the $\UMD$ property on Banach function spaces has previously been shown in unpublished work of T.P. Hyt\"{o}nen, using Stein's inequality to deduce the $\ell^2$-boundedness of the Poisson semigroup on $L^p(\R^d;X)$, from which the boundedness of the Hilbert transform on $L^p(\R^d;X)$ was concluded using Theorem \ref{theorem:BFSfactorization}.

\begin{theorem}\label{theorem:UMD+UMD}
  Let $X$ be a Banach function space on $(S,\Sigma,\mu)$. Assume that $X$ has the dyadic $\UMD^+$ property and cotype $q\in(1,\infty)$ with constant $c_{q,X}$. Then $X$ has the $\UMD$ property with
  \begin{equation*}
    \beta_{p,X} \lesssim  q\,\hab{c_{q,X} \, \beta_{p,X}^{\Delta,+}}^2,
    \end{equation*}
  where the implicit constant only depends on $p \in (1,\infty)$.
\end{theorem}

\begin{proof}
Denote the standard dyadic system on $[0,1)$ by $\ms{D}$, i.e.
\begin{equation*}
  \ms{D} := \bigcup_{k \in \N} \ms{D}_k,\qquad \ms{D}_k := \cbrace{2^{-k}([0,1)+j) : j=0,\ldots, 2^k-1}.
\end{equation*}
Then $(\ms{D}_k)_{k=1}^n$ is a Paley-Walsh filtration on $[0,1)$ for all $n \in \N$. Let $p \in (1,\infty)$ and define
$$\Gamma := \cbraceb{\E(\cdot|\ms{D}_k): k \in \N}$$
 on $L^p(0,1)$. By a dyadic version of Stein's inequality, which can be proven analogously to \cite[Theorem 4.2.23]{HNVW16}, we have
\begin{equation*}
   \nrms{\sum_{k=1}^n \varepsilon_k \E(f_k|\ms{D}_k)}_{L^p([0,1)\times \Omega;X)} \leq \beta_{p,X}^{\Delta,+}\, \nrms{\sum_{k=1}^n \varepsilon_k f_k}_{L^p([0,1)\times \Omega;X)},
\end{equation*}
where $(\varepsilon_k)_{k=1}^n$ is a Rademacher sequence.
So by \cite[Theorem 7.2.13]{HNVW17} and the ideal property of the $\ell^2$-structure, we know that $\widetilde{\Gamma}$ is $\ell^2$-bounded with
  \begin{equation}\label{eq:UMD+UMDl2bounded}
    \nrmb{\widetilde{\Gamma}}_{\ell^2}\leq C \,\sqrt{q} \, c_{q,X} \, \beta_{p,X}^{\Delta,+},
  \end{equation}
where $C>0$ only depends on $p$.

Define the dyadic weight class $A_2^{\ms{D}}$ as all weights $w$ on $[0,1)$ such that
\begin{equation*}
 [w]_{A_2^{\ms{D}}} := \sup_{I \in \ms{D}} \frac{1}{\abs{I}}\int_I w\dd \lambda \cdot \frac{1}{\abs{I}}\int_I w^{-1} \dd \lambda<\infty.
\end{equation*}
Let $w:[0,1) \to (0,\infty)$ be a weight.  Arguing as in Theorem \ref{theorem:latticemaximal}, we know that $$[w]_{A_2^{\ms{D}}} \leq \has{\sup_{T \in \Gamma}\nrm{T}_{\mc{L}(L^2(w))}}^2.$$
Furthermore note that, with a completely analogous proof, Theorem \ref{theorem:BFSextrapolation} is also valid for the interval $[0,1)$ instead of $\R^d$ and using weights $v \in A_2^{\ms{D}}$ instead of weights $v \in A_2$. Therefore we know that if $f,g \in L^p([0,1);X)$ are such that for all $v \in A_2^{\ms{D}}$ we have
  \begin{align}
    \label{eq:dyadicextrapass} \nrm{f(\cdot,s)}_{L^2([0,1),v) }&\leq  C\cdot [v]_{A_2^{\ms{D}}} \nrm{g(\cdot,s)}_{L^2([0,1),v) }, \qquad  s \in S,
  \intertext{then it follows that}
   \label{eq:dyadicextracon} \nrm{f}_{L^p([0,1);X)} &\leq c\cdot C\cdot\nrm{\widetilde{\Gamma}}_{\ell^2}^2 \, \nrm{g}_{L^p([0,1);X)},
  \end{align}
  for some numerical constant $c$.

   Define for every interval $I \in \ms{D}$ the Haar function $h_I$ by
\begin{equation*}
  h_I := \abs{I}^{\frac{1}{2}}(\ind_{I_-} - \ind_{I_+}),
\end{equation*}
where $I_+$ and  $I_-$ are the left and right halve of $I$. For $f \in L^p([0,1);X)$ define the Haar projection $D_I$ by
\begin{equation*}
  D_I f(t) := h_I(t) \int_{0}^1 f(s) h_I(s) \dd s
\end{equation*}
Let $\mc{A}$ be the set of all simple functions $f \in L^p([0,1);X)$ such that $D_I f \neq 0$ for only finitely many $I \in \ms{D}$. Then for all $f \in \mc{A}$, $w \in A_2^{\ms{D}}$ and $\epsilon_I \in \cbrace{-1,1}$ we have
\begin{equation*}
      \nrms{\sum_{I \in \ms{D}} \epsilon_I D_I f(\cdot,\omega)}_{L^2([0,1),w) }\lesssim [w]_{{A}_2^{\ms{D}}} \, \nrm{f(\cdot,s)}_{L^2([0,1),w) }, \qquad  s \in S
\end{equation*}
by \cite{Ja00}, so \eqref{eq:dyadicextrapass} is satisfied.  Therefore, using \eqref{eq:UMD+UMDl2bounded} and \eqref{eq:dyadicextracon}, we obtain that
\begin{equation}\label{eq:haaruncond}
  \nrms{\sum_{I \in \ms{D}} \epsilon_I D_I f}_{L^p([0,1);X)} \leq C\, q\,\bigl(c_{q,X} \,\beta_{p,X}^{\Delta,+}\bigr)^2 \nrm{f}_{L^p([0,1);X)}
\end{equation}
for all $f \in \mc{A}$ and $\epsilon_I \in \cbrace{-1,1}$. Note that $\mc{A}$ is dense in $L^p([0,1);X)$ by \cite[Lemma 4.2.12]{HNVW16} and we may take $\epsilon_I \in \C$ with $\abs{\epsilon_I} = 1$ by the triangle inequality. So
\begin{equation*}
    \beta_{p,X} \leq C\, q\,\bigl(c_{q,X} \,\beta_{p,X}^{\Delta,+}\bigr)^2
  \end{equation*}
 as \eqref{eq:haaruncond} characterizes the $\UMD$ property of $X$ by \cite[Theorem 4.2.13]{HNVW16}.
 \end{proof}

\begin{remark}~
\begin{itemize}
  \item  As in Remark \ref{remark:maximalbdd},  the assumption that $X$ has finite cotype may be omitted in Theorem \ref{theorem:UMD+UMD}.  This would yield the bound $\beta_{p,X} \leq C_p\,\bigl(\beta_{p,X}^{\Delta,+}\bigr)^3$ for all $p \in (1,\infty)$ in the conclusion of Theorem \ref{theorem:UMD+UMD}.
  \item A similar argument as in the proof of Theorem \ref{theorem:UMD+UMD} can be used to show Theorem \ref{theorem:latticemaximal} with the sharper estimate
  \begin{equation*}
    \nrm{\widetilde{M}}_{\mc{L}(L^p(\R^d;X))} \lesssim \,q \bigl(c_{q,X}\beta_{p,X}^{+,\Delta}\bigr)^2.
  \end{equation*}
\end{itemize}
\end{remark}

\subsection*{\texorpdfstring{$\ell^2$}{l2}-sectoriality  and the \texorpdfstring{$\UMD$}{UMD} property}\label{subsection:ell2vsUMD}
For the $\ell^2$-structure on a Banach function space $X$ we say that a sectorial operator $A$ on $X$ is $\ell^2$-sectorial if the resolvent set $$\cbrace{\lambda R(\lambda,A):\lambda \neq 0, \abs{\arg \lambda}>\sigma }$$ is $\ell^2$-bounded for some $\sigma \in (0,\pi)$. We will introduce $\alpha$-sectorial operators properly in Chapter \ref{part:4}.

It is well-known that both the differentiation operator $Df := f'$ with domain $W^{1,p}(\R;X)$ and the Laplacian $-\Delta$ with domain $W^{2,p}(\R^d;X)$ are $\mc{R}$-sectorial, and thus $\ell^2$-sectorial,
if $X$ has the $\UMD$ property (see \cite[Example 10.2]{KW04} and \cite[Theorem 10.3.4]{HNVW17}).
 Using Theorem \ref{theorem:BFSextrapolation} we can turn this into an `if and only if' statement for order-continuous Banach function spaces.

\begin{lemma} \label{lemma:A2testing}
  Let $0\neq\varphi \in L^1(\R^d)\cap L^2(\R^d)$ be real-valued and let $w$ be a weight on $\R^d$. Suppose that there is a $C>0$ such that for all $f \in L^2(\R^d,w)$ and $\lambda \in \R$ we have
  \begin{equation*}
    \nrm{\varphi_\lambda * f}_{L^2(\R^d,w)} \leq C\, \nrm{f}_{L^2(\R^d,w)}
  \end{equation*}
  where $\varphi_\lambda(t) := \abs{\lambda}^d \varphi(\lambda t)$ for $t \in \R^d$. Then $w \in A_2$ and $[w]_{A_2} \lesssim C^4$, where the implicit constant depends on $\varphi$ and $d$
\end{lemma}

\begin{proof}
  Let $\psi = \varphi_{-1}*\varphi$. Then $\psi(-t)=\psi(t)$ for all $t \in \R^d$ and $\psi(0) = \nrm{\varphi}_{L^2{(\R^d)}}^2>0$. Moreover
  \begin{equation*}
    \nrm{\psi}_{L^\infty(\R^d)} \leq \nrm{\varphi}_{L^2(\R^d)}^2,
  \end{equation*}
  so $\psi$ is continuous by the density of $C_c(\R^d)$ in $L^2(\R^d)$. Therefore we can find a $\delta>0$ such that $\psi(t) >\delta$ for all $\abs{t} < \delta$. Define $\psi_\lambda(t) := \lambda^d \, \psi(\lambda t)$ for $\lambda >0$. Then we have for all $f \in L^1(\R^d)\cap L^2(\R^d,w)$ that
  \begin{equation*}
    \nrm{\psi_\lambda *f}_{L^2(\R^d,w)} = \nrm{\varphi_{-\lambda} * \varphi_{\lambda} *f}_{L^2(\R^d,w)} \leq C^2 \nrm{f}_{L^2(\R^d,w)}
  \end{equation*}
  Now let $Q$ be a cube in $\R^d$ and let $f \in L^1(\R^d)\cap L^2(\R^d,w)$  be nonnegative and supported on $Q$. Take $\lambda = \frac{\delta}{\diam(Q)}$, then for $t \in Q$
  \begin{equation*}
    \psi_\lambda*f(t) = \lambda^d \int_Q\psi\hab{\lambda (t-s)}f(s) \dd s
    \geq \frac{\delta^{d+1}}{\abs{Q}} \int_Q f(s)\dd s.
  \end{equation*}
So by the same reasoning as in the proof of Theorem \ref{theorem:latticemaximal}, we have $w \in A_2$ with $[w]_{A_2} \lesssim C^4$ with an implicit constant depending on $\varphi,d$.
\end{proof}

Using Lemma \ref{lemma:A2testing} to check the weight condition of Theorem \ref{theorem:BFSextrapolation},
the announced theorem follows readily.

\begin{theorem}\label{theoreml2:sectoriality}
  Let $X$ be an order-continuous Banach function space and let $p \in (1,\infty)$. The following are equivalent:
  \begin{enumerate}[(i)]
    \item \label{it:sectorial1} $X$ has the $\UMD$ property.
    \item \label{it:sectorial2} The differentiation operator $D$ on $L^p(\R;X)$ is $\ell^2$-sectorial.
    \item \label{it:sectorial3} The Laplacian $-\Delta$  on $L^p(\R^d;X)$ is $\ell^2$-sectorial.
  \end{enumerate}
\end{theorem}

\begin{proof}
We have already discussed the implications \ref{it:sectorial1} $\Rightarrow$ \ref{it:sectorial2} and \ref{it:sectorial1} $\Rightarrow$ \ref{it:sectorial3}. We will  prove \ref{it:sectorial3} $\Rightarrow$ \ref{it:sectorial1}, the proof of \ref{it:sectorial2} $\Rightarrow$ \ref{it:sectorial1} being similar. Take $\lambda \in \R$ and define the operators
\begin{equation*}
  T_\lambda := -\lambda^2 \Delta (1- \lambda^2 \Delta)^{-2} = -\Delta R\has{-\frac{1}{\lambda^2},-\Delta} \cdot  \frac{1}{\lambda^2}R\has{-\frac{1}{\lambda^2},-\Delta}.
\end{equation*}
Since $-\Delta$ is $\ell^2$-sectorial on $L^p(\R^d;X)$, we know that the family of operators $\widetilde{\Gamma} = \cbraceb{\widetilde{T}_\lambda: \lambda \in \R}$ is $\ell^2$-bounded on $L^p(\R^d;X)$. Furthermore we have for $f \in L^2(\R^d)$ that $T_1f = \varphi * f $ with $\varphi  \in L^1(\R^d)\cap L^2(\R^d)$ such that
\begin{equation*}
  \hat{\varphi}(\xi) = \frac{(2\pi\abs{\xi})^2}{\bigl(1+(2\pi\abs{\xi})^2\bigr)^{2}}, \qquad \xi \in \R^d.
\end{equation*}
Moreover $T_\lambda f = \varphi_\lambda * f $ for $\varphi_\lambda(x) = \lambda^d \varphi(\lambda x)$ and $\lambda \in \R$.  Using Lemma \ref{lemma:A2testing} this implies that the assumptions of Theorem \ref{theorem:BFSextrapolation} are satisfied.

 Now by Theorem \ref{theorem:BFSextrapolation} and the boundedness of the Riesz projections on $L^2(\R^d,w)$ for all $w\in A_2$ (see \cite{Pe08}), we find that for all $f \in C_c^\infty(\R^d)\otimes X$
\begin{equation*}
  \nrm{R_k f}_{L^p(\R^d;X)} \lesssim \nrm{\widetilde{\Gamma}}_{\ell^2}^4 \nrm{f}_{L^p(\R^d;X)}, \qquad k=1,\ldots,d.
\end{equation*}
 So, by the density of $ C_c^\infty(\R^d) \otimes X$ in $L^p(\R^d;X)$, the Riesz projections are bounded on $L^p(\R^d;X)$, which means that $X$ has the $\UMD$ property by \cite[Theorem 5.5.1]{HNVW16}.
\end{proof}

The proof scheme of Theorem \ref{theoreml2:sectoriality} can be adapted to various other operators. We mention two examples:
\begin{itemize}
  \item In \cite{Lo18} it was shown that the $\UMD$ property is sufficient for the $\ell^2$-boundedness of a quite broad class of convolution operators on $L^p(\R^d;X)$. Using a similar proof as the one presented in Theorem \ref{theoreml2:sectoriality}, one can show that  the $\UMD$ property of the Banach function space $X$ is necessary for the $\ell^2$-boundedness of these operators.
  \item On general Banach spaces $X$ we know by a result of Coulhon and Lamberton \cite{CL86} (quantified by Hyt\"onen \cite{Hy15}), that the maximal $L^p$-regularity of $(-\Delta)^{1/2}$ implies that $X$ has the $\UMD$ property. Maximal $L^p$-regularity implies the $\mc{R}$-sectoriality of $(-\Delta)^{1/2}$ on $L^p(\R^d;X)$ by a result of Cl\'ement and Pr\"uss \cite{CP01} and the converse holds if $X$ has the $\UMD$ property by \cite{We01b}. It is therefore a natural question to ask whether the $\mc{R}$-sectoriality of $(-\Delta)^{1/2}$ on $L^p(\R^d;X)$ also implies that $X$ has the $\UMD$ property. By the equivalence of $\mc{R}$-sectoriality and $\ell^2$-sectoriality on Banach lattices with finite cotype, we can show that this is indeed the case for Banach function spaces with finite cotype, using a similar proof as in the proof of Theorem \ref{theoreml2:sectoriality}. The question for general Banach spaces remains open. This is also the case for the question whether the $\mc{R}$-sectoriality of $-\Delta$ on $L^p(\R^d;X)$ implies that $X$ has the $\UMD$ property, see \cite[Problem 7]{HNVW17}.
\end{itemize}

\chapter{Vector-valued function spaces and interpolation}\label{part:3}
In Chapter \ref{part:1} we treated Euclidean structures as a norm on the space of functions from $\cbrace{1,\ldots,n}$ to $X$ or as a norm on the space of operators from $\ell^2_n$ to $X$ for each $n \in \N$. In this chapter we will extend this norm to include functions from an arbitrary measure space $(S,\mu)$ to $X$ and to operators from an arbitrary Hilbert space $H$ to $X$. After introducing the relevant concepts, we will study the properties of the so-defined function spaces $\alpha(S;X)$ and operator spaces $\alpha(H,X)$. Their most important property is that every bounded operator on $L^2(S)$, e.g. the Fourier transform or a singular integral operator on $L^2(\R^d)$, extends automatically to a bounded operator on the $X$-valued function space $\alpha(S;X)$ for any Banach space $X$. This is in stark contrast  to the situation for the Bochner spaces $L^2(S;X)$ and greatly simplifies analysis for vector-valued functions in these spaces.

 In the second halve of this chapter we will develop an interpolation method based on these vector-valued function spaces. A charming feature of this $\alpha$-inter\-polation method is that its formulations modelled after the real and the complex interpolation methods are equivalent. The $\alpha$-interpolation method can therefore be seen as a way to keep strong interpolation properties of Hilbert spaces in a Banach space context.

As a standing assumption throughout this chapter and the subsequent chapters we suppose that $\alpha$ is a Euclidean structure on $X$.

\section{The spaces \texorpdfstring{$\alpha(H,X)$}{a(H,X)} and \texorpdfstring{$\alpha(S;X)$}{a(S;X)}}\label{section:alphaspaces}  Our first step is to extend the definition of the $\alpha$-norm to infinite vectors. For an infinite vector $\mb{x}$ with entries in $X$ we define
\begin{equation*}
  \nrm{\mb{x}}_{\alpha} =\sup_{n\in \N} \,\nrm{(x_1,\ldots,x_n)}_\alpha.
\end{equation*}
 We then define $\alpha_+(\N;X)$ as the space of all infinite column vectors $\mb{x}$ such that $\nrm{\mb{x}}_\alpha<\infty$ and let  $\alpha(\N;X)$ be  the subspace of $\alpha_+(\N;X)$ consisting of all $\mb{x} \in \alpha_+(\N;X)$ such that
\begin{equation*}
  \lim_{n \to \infty} \nrm{(0,\ldots,0,x_{n+1},x_{n+2}, \ldots)}_\alpha = 0.
\end{equation*}
Proposition \ref{proposition:finitedimensionalalpha} shows that if $\mb{x} \in \alpha_{+}(\N;X)$ has finite dimensional range, then $\mb{x} \in \alpha(\N;X)$. This leads to following characterization of $\alpha(\N;X)$.

\begin{proposition}\label{proposition:alphadef2}
  Let $\mb{x} \in \alpha_+(\N;X)$. Then $\mb{x}\in \alpha(\N;X)$ if and only if there exists an sequence $(\mb{x}^k)_{k=1}^\infty $ with finite dimensional range such that  $\lim_{k\to \infty} \nrm{\mb{x} - \mb{x}^k}_{\alpha} = 0$.
\end{proposition}

From Proposition \ref{proposition:alphadef2} and Property \eqref{eq:E2x} of a Euclidean structure we obtain directly the important fact that every bounded operator on $\ell^2$ extends to a bounded operator on $\alpha(\N;X)$ and $\alpha_+(\N;X)$.

\begin{proposition}\label{proposition:infinitematrixalphafunctions}
  If $\mb{x} \in \alpha_+(\N;X)$ and $\mb{A}$ is an infinite matrix representing a bounded operator on $\ell^2$, then $\mb{A}\mb{x} \in \alpha_+(\N;X)$ with
  \begin{equation*}
    \nrm{\mb{A}\mb{x}}_\alpha \leq \nrm{\mb{A}}\nrm{\mb{x}}_\alpha
  \end{equation*}
   If either $\mb{x} \in \alpha(\N;X)$ or $\mb{A}$ represents a compact operator on $\ell^2$, then $\mb{A}\mb{x} \in \alpha(\N;X)$.
\end{proposition}

\subsection*{The space $\alpha(H,X)$}
As announced we wish to extend the definition of the $\alpha$-norms to functions on a measure space different from $\N$ and to operators from a Hilbert space $H$ to $X$ for $H$ different from $\ell^2$.

\begin{definition}
  Let $H$ be a Hilbert space. We let $\alpha(H,X)$  (resp. $\alpha_+(H,X)$) be the space of all $T \in \mc{L}(H,X)$ such that
$(Te_k)_{k=1}^\infty \in \alpha(\N;X)$ (resp. $(Te_k)_{k=1}^\infty \in \alpha_+(\N;X)$) for all orthonormal systems $(e_k)_{k=1}^\infty$ in $H$. We then set
\begin{equation*}
  \nrm{T}_{\alpha(H;X)}=\nrm{T}_{\alpha_+(H;X)} :=  \sup\, \nrm{(Te_k)_{k=1}^\infty}_\alpha,
\end{equation*}
where the supremum is taken over all orthonormal systems $(e_k)_{k=1}^\infty$ in $H$.
\end{definition}
If $H$ is separable, then it suffices to compute $\nrm{(Te_k)_{k=1}^\infty}_\alpha$ for a fixed orthonormal basis $(e_k)_{k=1}^\infty$ of $H$ by Proposition \ref{proposition:infinitematrixalphafunctions}.

For $\alpha = \gamma$ the spaces $\gamma_+(H,X)$ and $\gamma(H,X)$ are already well-studied in literature (see for example  \cite{KW16}, \cite[Chapter 9]{HNVW17} and the references therein). Since many of the basic properties of $\alpha(H,X)$ have proofs similar to the ones for $\gamma(H,X)$, we can be brief here  and refer to \cite[Chapter 9]{HNVW17} for inspiration.
In particular:
\begin{itemize}
  \item Both $\alpha_+(H,X)$ and $\alpha(H,X)$ are Banach spaces.
  \item $\alpha(H,X)^*$ can be canonically identified with $\alpha^*_+(H^*,X^*)$ through trace duality. Note that in this duality one should not identify $H$ with its Hilbert space dual, see \cite[Section 9.1.b]{HNVW17} for a discussion.
  \item In many cases $\alpha(H,X)$ and $\alpha_+(H,X)$ coincide. For the Gaussian structure this is the case if and only if $X$ does not contain a closed subspace isomorphic to $c_0$.
\end{itemize}
It follows readily from Proposition \ref{proposition:alphadef2} that $\alpha(H,X)$ is the closure of the finite rank operators in $\alpha_+(H,X)$. This can be used to show that every $T \in \alpha(H,X)$ is supported on a separable closed subspace of $H$:

\begin{proposition}\label{proposition:separablereduction}
 Let $H$ be a Hilbert space and $T \in \alpha(H,X)$. Then there is a separable closed subspace $H_0$ of $H$ such that $T\varphi = 0$ for all $\varphi \in H_0^{\bot}$.
\end{proposition}

\begin{proof}
  Let $T = \lim_{k \to \infty} T_k$ in $\alpha(H,X)$ where each $T_k$ is of the form
  \begin{equation*}
    T_k\varphi = \sum_{j=1}^{m_k} \ip{\varphi,\psi_{jk}} x_{jk}
  \end{equation*}
  with $\psi_{jk} \in H^*$, $x_{jk} \in X$ for $1 \leq j \leq m_k$ and $k \in \N$. Let $H_0$ be the closure of the linear span of $\cbrace{\psi_{jk}:1\leq j\leq m_k, k\in \N}$ in $H$. Then $H_0$ is separable and $T\varphi = 0$ for all $\varphi \in H_0^{\bot}$.
\end{proof}

As we already noted, $\alpha(H,X)^*$ can be identified with $\alpha^*_+(H^*,X^*)$ through trace duality. In the converse direction we have the following proposition.

\begin{proposition}\label{proposition:normingalpha}
Let $H$ be a Hilbert space, let $Y \subseteq X^*$ be norming for $X$ and let $T \in \mc{L}(H,X)$. If there is a $C>0$ such that for all finite rank operators $S:H^*\to Y$ we have
\begin{equation*}
   \abs{\tr(S^*T)} \leq C \,\nrm{S}_{\alpha^*(H^*,X^*)}
\end{equation*}
Then $T \in \alpha_+(H,X)$ with $\nrm{T}_{\alpha_+(H,X)} \leq C$.
\end{proposition}

\begin{proof}
  Let $(e_k)_{k=1}^n$ be an orthonormal sequence in $H$ and $\varepsilon>0$. Define $x_k = Te_k$ and let $(x_k^*)_{k=1}^n$ be a sequence in $Y$ with $\nrm{(x_k^*)_{k=1}^n}_{\alpha^*} \leq 1$ and
  $$\nrm{(x_k)_{k=1}^n}_\alpha \leq \sum_{k=1}^n\abs{x_k^*(x_k)}+\varepsilon.$$
  Then, for the finite rank operator $S = \sum_{k=1}^n e_k \otimes x_k^*$, we have
  \begin{equation*}
    \nrm{(T_ke_k)_{k=1}^n}_\alpha \leq \sum_{k=1}^n\abs{x_k^*(x_k)}+\varepsilon =\abs{\tr(S^*T)}+\varepsilon    \leq C+\varepsilon.
  \end{equation*}
  Taking the supremum over all orthonormal sequences in $H$ finishes the proof.
\end{proof}

\subsection*{The space $\alpha(S;X)$}
We will mostly be using $H=L^2(S)$ for a measure space $(S,\mu)$. We abbreviate
\begin{align*}
  \alpha(S;X)&:=\alpha(L^2(S),X)\\
  \alpha_+(S;X)&:=\alpha_+(L^2(S),X)
\end{align*}
 For an operator $T \in \mc{L}(L^2(S),X)$ we say that  $T$ is \emph{representable} if there exists a strongly measurable $f \colon S \to X$ with $x^*\circ f \in L^2(S)$ for all $x^* \in X$ such that
\begin{equation}\label{eq:operatorfromrepresentation}
    T\varphi = \int_S \varphi f \dd \mu, \qquad \varphi \in L^2(S).
\end{equation}
Here the integral is well defined by Pettis' theorem \cite[Theorem 1.2.37]{HNVW16}. Equivalently $T$ is representable if there exists a strongly measurable $f \colon S \to X$ such that
for all $x^* \in X^*$ we have
\begin{equation}\label{eq:Trepresentable}
  x^*\circ f = T^*(x^*).
\end{equation}

Conversely, if we start from a strongly measurable function $f\colon S \to X$ with $x^*\circ f \in L^2(S)$ for all $x^* \in X$, we can define the operator $T_f:L^2(S) \to X$ as in \eqref{eq:operatorfromrepresentation}, which is again well defined by Pettis' theorem. If $T_f \in \alpha(S;X)$ (resp. $\alpha_+(S;X)$) we can identify $f$ and $T_f$, since $f$ is the unique representation of $T_f$. In this case we write $f \in \alpha(S;X)$ (resp. $f \in \alpha_+(S;X)$)  and assign to $f$ the $\alpha$-norm
\begin{align*}
\nrm{f}_{\alpha(S;X) }&:= \nrm{T_f}_{\alpha(S;X) },\\
  \nrm{f}_{\alpha_+(S;X) }&:= \nrm{T_f}_{\alpha_+(S;X) }.
\end{align*}

In Proposition \ref{proposition:alphaspace=bochnerspace}, we will see that
\begin{equation*}
  {\alpha}_{\bullet}(S;X):= \cbraceb{T \in \alpha(S;X): T \text{ is representable by a function }  f\colon S \to X}
\end{equation*}
is usually not all of $\alpha(S,X)$. However it is often useful to think of the space $({\alpha}_{\bullet}(S;X),\nrm{\lcdot}_{\alpha(S;X)})$ as a normed function space and of $\alpha(S;X)$ as its completion, where the elements of $\alpha(S;X) \setminus {\alpha}_{\bullet}(S;X)$ are interpreted as operators $T:L^2(S) \to X$. If $S = \R^d$ we have $C_0^\infty(\R^d) \subseteq L^2(\R^d) \overset{T}{\longrightarrow}  X$ and we may also think of $\alpha(S;X)$ as a space of $X$-valued distributions. Then \eqref{eq:operatorfromrepresentation} conforms with the usual interpretation of a locally integrable function $f$ as a distribution $T$.

\begin{proposition}\label{proposition:densealphaspace}
  Let $(S,\mu)$ be a measure space and let $\mc{A}$  be a dense subset of $L^2(S)$. Then $$\spn\cbrace{f\otimes x:f \in \mc{A}, x \in X}$$ is dense in $\alpha(S;X)$.
\end{proposition}

\begin{proof}
Since the finite rank operators are dense in $\alpha(S;X)$, it suffices to show that every rank one operator $T = g \otimes x$ with $g \in L^2(S)$ and $x \in X$ can be approximated by operators $T_{f_n}$ with $f_n \in \spn\cbrace{h\otimes x:h \in \mc{A}}$. For this let $(h_n)_{n=1}^\infty$ be such that $h_n \to g$ in $L^2(S)$ and define $f_n = h_n\otimes x$. Then we have, using Proposition \ref{proposition:infinitematrixalphafunctions},
\begin{equation*}
  \nrm{T-T_{f_n}}_{\alpha(S;X)} = \nrm{(g-h_n)\otimes x}_{\alpha(S;X)} = \nrm{g-h_n}_{L^2(S)} \nrm{x}_X \qedhere
\end{equation*}
\end{proof}

Proposition \ref{proposition:densealphaspace} allows us
to work with work with functions rather than operators in $\alpha(S;X)$. The following lemma sometimes allows us to reduce considerations even further to bounded functions on sets of finite measure.

\begin{lemma}\label{lemma:partitions}
  Let $(S,\mu)$ be a measure space and let $f\colon S \to X$ be strongly measurable. Then there exists a partition $\Pi = \cbrace{E_{n}}_{n=1}^\infty$ of $S$ such that
  $E_{n}$ has positive finite measure and $f$ is bounded on $E_{n}$ for all $n \in \N$.
  Furthermore, there exists a sequence of such partitions $\Pi_m = \cbrace{E_{nm}}_{n=1}^\infty$ such that for the associated averaging projections
    \begin{equation*}
      P_mf(s):= \sum_{n=1}^\infty \ind_{E_{nm}}(s) \frac{1}{\mu(E_{nm})} \int_{E_{nm}}f \dd\mu, \qquad s \in S,\, m \in N,
    \end{equation*}
    we have $P_mf(s) \to f(s)$ for $\mu$-a.e. $s\in S$.
\end{lemma}

\begin{proof}
  By \cite[Proposition 1.1.15]{HNVW16} we know that $f$ vanishes off a $\sigma$-finite subset of $S$, so without loss of generality we may assume that $(S,\mu)$ is $\sigma$-finite. Let $(S_n)_{n=1}^\infty$ be a sequence of disjoint measurable sets of finite measure such that $S = \bigcup_{n=1}^\infty S_n$. For $n,k \in \N$ set
  \begin{equation*}
    A_{n,k} :=  \cbrace{s \in S_n:k-1 \leq \nrm{f(s)}_X < k}.
  \end{equation*}
  The sets $(A_{n,k})_{n,k=1}^\infty$ are pairwise disjoint, have finite measure and $f$ is bounded on $A_{n,k}$ for all $n,k \in \N$. Relabelling and leaving out all sets with measure zero proves the first part of the lemma.

  For the second part note that by the first part we may assume that $S$ has finite measure and $f$ is bounded. By \cite[Lemma 1.2.19]{HNVW16} there exists a sequence of simple functions $(f_m)_{m=1}^\infty$ and a sequence of measurable sets $(B_m)_{m=1}^\infty$ such that
  \begin{equation*}
    \sup_{s \in B_m}\nrm{f_m(s)-f(s)}_X <\frac{1}{m} \qquad \text{and} \qquad \mu(S\setminus B_m) <2^{-m-1}.
  \end{equation*}
  Upon replacing $B_m$ by $\bigcap_{j\geq m} B_j$ the sequence $(B_m)_{m=1}^\infty$ can be taken to be increasing with $\mu(S\setminus B_m) \to 0$ for $m \to \infty$. For each $m \in \N$ let $\Pi_m=\cbrace{E_{nm}}_{n=1}^{M_m}$ be the partition of $S$ consisting of the atoms of the finite $\sigma$-algebra generated by $B_1,\ldots,B_m$ and the simple functions $f_1,\ldots,f_m$.

  Take $s \in S$, fix $\varepsilon>0$ and let $N\in \N$ such that $N>\varepsilon/2$ and $s \in B_N$, which is possible for a.e. $s\in S$ since $\mu\hab{\textstyle\bigcup_{m=1}^\infty B_m} = \mu(S).$ Then we have for all $m \geq N$ that $\nrm{f_m - f}_X<1/m$ on $B_m$, and thus in particular on $E_{jm}$ for $j \in N$ such that $s \in E_{jm}$. Therefore $$\nrm{P_mf(s)-P_mf_m(s)}_X<\tfrac1m$$ and since $P_mf_m = f_m$ we conclude
  \begin{equation*}
    \nrm{P_mf(s)-f(s)}_X \leq \nrm{P_mf(s)-P_mf_m(s)}_X+\nrm{f_m(s)-f(s)}_X<\tfrac2m<\epsilon.
  \end{equation*}
  Therefore $P_mf(s) \to f(s)$ for $\mu$-a.e. $s\in S$, which concludes the proof.
\end{proof}

\subsection*{Representability of operators in $\alpha(S;X)$}
We will now study the representability of elements of $\alpha(S;X)$ with the aim of characterizing
 when all elements of $\alpha(S;X)$ are representable by a function $f:S \to X$.  If $(S,\mu)$ is atomic, then it is clear that every element of $\alpha(S;X)$ is  representable by a function. All elements of $\alpha(S;X)$ are also  representable by a function if $\alpha = \pi_2$ and $X$ is a Hilbert space, since the Hilbert-Schmidt norm coincides with the $\pi_2$-norm in this case and we have the following, well-known lemma.

\begin{lemma}\label{lemma:HSrepresentable}
Let $H$ be a Hilbert space, $(S,\mu)$ a measure space and suppose that $T:L^2(S) \to H$ is a Hilbert-Schmidt operator. Then there is a strongly measurable $f\colon S \to H$ such that $T\varphi = \int_S\varphi f \dd\mu$ for all $\varphi \in L^2(S)$.
\end{lemma}

\begin{proof}
  We can represent $T$ in the form
  \begin{equation*}
    T\varphi  = \sum_{k=1}^\infty a_k\ip{\varphi,e_k}h_k, \qquad \varphi \in L^2(S),
  \end{equation*}
  where $(a_k)_{k=1}^\infty \in \ell^2$, $(e_k)_{k=1}^\infty$ is an orthonormal sequence in $L^2(S)$ and $(h_k)_{k=1}^\infty$ is an orthonormal sequence in $H$. Let
  \begin{equation*}
    f(s):=\sum_{k=1}^\infty a_k e_k(s) h_k, \qquad \mu\text{-a.e. }s \in S.
  \end{equation*}
  This defines a strongly measurable map $f:S \to H$ since
  \begin{equation*}
    \int_S\sum_{k=1}^\infty \abs{a_k}^2\abs{e_k}^2 \dd \mu < \infty.
  \end{equation*}
  Moreover $T\varphi = \int_S\varphi f \dd\mu$ for all $\varphi \in L^2(S)$.
\end{proof}

It turns out that the two discussed cases, i.e. $(S,\mu)$ atomic or $X$ a Hilbert space and $\alpha = \pi_2$, are in a certain sense the only occasions in which all elements of $\alpha(S;X)$ are representable by a function. This will be a consequence of the following proposition.

\begin{proposition}\label{proposition:alphaspace=bochnerspace}
  Let $(S,\Sigma,\mu)$ be a non-atomic measure space. Every operator in $\alpha(S,X)$ is representable if and only if $\pi_2 \lesssim \alpha$.
\end{proposition}

\begin{proof}
  Let us first show that if $T \in \alpha(S;X) \subseteq \pi_2(S;X)$, then $T$ is representable. Note that $T$ is $2$-summing, so by the Pietsch factorization theorem \cite[p.48]{DJT95}, we know that $T$ has a factorization $T=UJV$:
  \begin{center}
\begin{tikzcd}
L^2(S)\arrow{r}{T} \arrow{d}{V}                                                      & X               \\
L^\infty(S') \arrow{r}{J}& L^2(S') \arrow{u}{U} \arrow{u}{U}
\end{tikzcd}
\end{center}
  where $(S',\mu')$ is a finite measure space and $J$ is the inclusion map. Since $J$ is $2$-summing by Grothendiek's theorem (see e.g. \cite[Theorem 3.7]{DJT95}), $VJ$ is also $2$-summing and thus a Hilbert-Schmidt operator. Therefore $T$ is representable by Lemma \ref{lemma:HSrepresentable}.

  Conversely suppose that every $T \in \alpha(S;X)$ is representable. By restricting to a subset of $S$ we may assume $\mu(S)<\infty$ and then by rescaling we may assume $\mu(S) = 1$. We define a map $$J:\alpha(S;X) \to L^0(S;X)$$ such that $JT$ is a representing function for $T \in \alpha(S;X)$. This map is well-defined since the representing function is unique up to $\mu$-a.e. equality. Let us consider the topology of convergence in measure on $L^0(S;X)$. If $T_n \to T$ in $\alpha(S;X)$ and $f_n:=JT_n \to f$ in $L^0(S;X)$, then it is clear by the dominated convergence theorem that
  \begin{equation*}
    T\varphi =\int_S \varphi f \dd \mu, \qquad \varphi \in \mc{A},
  \end{equation*}
  where
  \begin{equation*}
    \mc{A}:=\cbrace{\varphi \in L^2(S): \int_S \abs{\varphi} \sup_{n \in \N}\nrm{f_n}_X\dd \mu<\infty}.
  \end{equation*}
  Since  $\mc{A}$ is dense in $L^2(S)$ by Lemma \ref{lemma:partitions}, this shows that $f = JT$. Hence $J$ has a closed graph and is therefore continuous. In particular it follows that there is a constant $C>0$ so that if $\nrm{JT(s)}_X\geq 1$ for $\mu$-a.e. $s \in S$, then $\nrm{T}_{\alpha(S;X)} \geq C^{-1}$. Now take $\mb{x} \in X^n$  such that $\sum_{k=1}^n \nrm{x_k}^2 =1$ and partition $S$ into sets $E_1,\ldots,E_n$ with measure $\nrm{x_1}^2,\ldots,\nrm{x_n}^2$, which is possible since $(S,\mu)$ is non-atomic. Define
   \begin{align*}
     e_k &= \ind_{E_k} \nrm{x_k}^{-1},\\
     f(s) &= \sum_{k=1}^n x_k e_k(s), \qquad s \in S
   \end{align*}
 for $1 \leq k \leq n$. Then $(e_k)_{k=1}^n$ is an orthonormal sequence in $L^2(S)$ and $\nrm{f(s)}_X = 1$ for $s \in S$, so
  \begin{equation*}
    \nrm{\mb{x}}_\alpha = \nrmb{\sum_{k=1}^n e_k \otimes x_k}_{\alpha(S;X)} \geq C^{-1}.
  \end{equation*}
  This implies that $\hab{\sum_{k=1}^n \nrm{x_k}^2}^{1/2} \leq C \,\nrm{\mb{x}}_\alpha$ for all $\mb{x} \in X^n$. Thus for any
 $\mb{A} \in M_{m,n}(\C)$ with $\nrm{\mb{A}} \leq 1$ we have
  \begin{equation*}
    \has{\sum_{j=1}^m \nrm{\mb{A}\mb{x}}^2}^{1/2} \leq C \, \nrm{\mb{A}\mb{x}}_\alpha \leq C  \, \nrm{\mb{x}}_\alpha,
  \end{equation*}
  which shows that $\pi_2 \lesssim \alpha$.
\end{proof}

\begin{corollary}
  Let $(S,\mu)$ be a non-atomic measure space. All $S \in \alpha(S;X)$ and $T \in \alpha^*(S;X^*)$ are representable if and only if $\alpha$ and $\alpha^*$ are equivalent to the $\pi_2$-structure and $X$ is isomorphic to a Hilbert space.
\end{corollary}

\begin{proof}
  The `if' part follows directly from Lemma \ref{lemma:HSrepresentable}. For the `only if' part note that, by Proposition \ref{proposition:alphaspace=bochnerspace}, Proposition \ref{proposition:compareEuclidean} and Proposition \ref{proposition:dualeuclidean}, we have on $X^*$
  \begin{equation}\label{eq:equivalencestructures}
    \pi_2 \lesssim \alpha^* \lesssim \pi_2^* \leq  \gamma^* \leq \gamma \leq \pi_2.
  \end{equation}
  This implies that $\alpha^*$ is equivalent to the $\pi_2$-structure on $X^*$. A similar argument on $X^{**}$ implies that $\alpha^{**}$ is equivalent to the $\pi_2$-structure on $X^{**}$, so $\alpha$ is equivalent to the $\pi_2$-structure on $X$. By \eqref{eq:equivalencestructures} and Propositions \ref{proposition:compareEuclidean} and \ref{proposition:dualeuclidean}, we also have that $X^*$ has nontrivial type and cotype 2. Therefore by \cite[Proposition 7.4.10]{HNVW17} we know that $X^{**}$, and thus $X$, has type $2$. A similar chain of inequalities on $X^{**}$ shows that $X^{**}$, and thus $X$, has cotype $2$. So by Theorem \ref{theorem:Rbddkwapienmaurey} we know that $X$ is isomorphic to a Hilbert space.
\end{proof}

We end this section with a representation result for the $\ell^2$-structure on a Banach function space $X$ or a $C_0(K)$ space. Note that by $\ell^2(S;X)$ we mean the space $\alpha(S;X)$ where $\alpha$ is the $\ell^2$-structure, not the sequence space $\ell^2$ indexed by $S$ with values in $X$.

\begin{proposition}\label{proposition:l2representation}
  Let $(S,\mu)$ be a measure space and suppose that $X$ is either an order-continuous Banach function space or $C_0(K)$ for some locally compact $K$. Then for any strongly measurable $f\colon S\to X$
  we have $f \in \ell^2(S;X)$ if and only if $\hab{\int_S \abs{f}^2 \dd \mu}^{1/2} \in X$ with
  \begin{equation*}
    \nrm{f}_{\ell^2(S;X)} = \nrms{\has{\int_S \abs{f}^2 \dd \mu}^{1/2}}_X.
  \end{equation*}
\end{proposition}

\begin{proof}
  We will prove the `only if' statement, the `if' statement being similar, but simpler. Let $f\colon S \to X$ be strongly measurable. By \cite[Proposition 1.1.15]{HNVW16} we may assume that $S$ is $\sigma$-finite and by Proposition \ref{proposition:separablereduction} we may assume that $L^2(S)$ is separable. Suppose that $(e_k)_{k=1}^\infty$ is an orthonormal basis of $L^2(S)$ such that $\int_S\abs{e_k} \nrm{f}_X<\infty$ for all $k \in \N$. Such a basis can for example be constructed by partitioning $S$ into sets of finite measure where $f$ is bounded as in Lemma \ref{lemma:partitions}. Let $x_k := \int_S e_k f\dd \mu$ and $f_n := \sum_{k=1}^n \overline{e}_k \otimes x_k$. Then $f \in \ell^2(S;X)$ if and only if $(x_k)_{k=1}^\infty \in \ell^2(\N;X)$. This occurs if and only if
  $$\lim_{n\to \infty} \nrms{\hab{\sum_{k=n+1}^\infty \abs{x_k}^2}^{1/2}}_X =0.$$
  By order-continuity or Dini's theorem respectively, this occurs if and only if we have $\hab{\sum_{k=1}^\infty \abs{x_k}^2}^{1/2} \in X$. Since
  \begin{equation*}
    \has{\sum_{k=1}^\infty \abs{x_k}^2}^{1/2} =  \has{\int_S \abs{f}^2 \dd \mu}^{1/2},
  \end{equation*}
  the result follows.
\end{proof}
If for example $X = L^p(\R)$, then a measurable $f: \R \to L^p(\R)$ belongs to $\ell^2(\R;X)$ if and only if
\begin{equation*}
  \nrm{f}_{\ell^2(\R;X)} = \has{ \int_{\R}\has{\int_{\R} \abs{f(t,s)}^2 \dd t}^{p/2}\dd s}^{1/p} < \infty.
\end{equation*}
For a Banach function space with finite cotype we also have that
\begin{equation*}
  \nrm{f}_{\gamma(S;X)} \simeq  \nrm{f}_{\ell^2(S;X)} =\nrms{\has{\int_S\abs{f}^2\dd \mu}^{1/2}}_X
\end{equation*}
which follows from Proposition \ref{proposition:compareEuclidean} (see also  \cite[Theorem 9.3.8]{HNVW17}). This equation suggests to think of the norms $\nrm{\lcdot}_{\gamma(S;X)}$ and $\nrm{\lcdot}_{\alpha(S;X)}$ as generalizations of the classical square functions in $L^p$-spaces to the Banach space setting. We will support this heuristic in the next section by showing that $\alpha$-norms have properties quite similar to the usual function space properties of $L^p(S';L^2(S))$. In Chapter \ref{part:5} we will use this heuristic to generalize the classical $L^p$-square functions for sectorial operators to arbitrary Banach spaces.

\section{Function space properties of \texorpdfstring{$\alpha(S;X)$}{a(S;X)}}
We will now take a closer look at the space $\alpha(S;X)$ as the completion of a function space over the measure space $(S,\mu)$. We start with some embedding between these spaces and the more classical Bochner spaces $L^2(S;X)$.
If $E$ is a finite-dimensional subspace of $X$ and $f\colon S \to E$ is strongly measurable, then $f \in \alpha(S,X)$ if and only if $f \in L^2(S;X)$. In fact, by Proposition \ref{proposition:finitedimensionalalpha}, we have
\begin{equation}\label{eq:finitedimesnionalE}
  (\dim(E))^{-1} \nrm{f}_{L^2(S;X)} \leq \nrm{f}_{\alpha(S;X)} \leq \dim(E) \nrm{f}_{L^2(S;X)}.
\end{equation}
Moreover if $\dim(L^2(S)) = \infty$, it is known that for the $\gamma$-structure we have
\begin{align}\label{eq:typeembedding}
 \nrm{f}_{\gamma(S;X)} &\lesssim \nrm{f}_{L^2(S;X)}, && f \in L^2(S;X),
\intertext{if and only if $X$ has type $2$ and}
 \nrm{f}_{L^2(S;X)} &\lesssim \nrm{f}_{\gamma(S;X)}, && f \in \gamma(S;X)\label{eq:cotypeembedding}
\end{align}
if and only if $X$ has cotype $2$, see \cite[Section 9.2.b]{HNVW17}. Further embeddings under smoothness conditions can be found in \cite[Section 9.7]{HNVW17}. We leave the generalization of these embeddings to a general Euclidean structure $\alpha$ to the interested reader.

\subsection*{Extension of bounded operators on $L^2(S)$}
    One of the main advantages the spaces $\alpha(S;X)$ have over the Bochner spaces $L^p(S;X)$ is the fact that any operator $T \in \mc{L}(L^2(S_1),L^2(S_2))$ can be extended to a bounded operator $\widetilde{T}\colon \alpha(S_1;X) \to \alpha(S_2;X)$. Indeed, putting $\widetilde{T}U := U\circ T^*$ for $U \in \alpha(S_1;X)$, we have that $\widetilde{T}$ is bounded by Proposition \ref{proposition:infinitematrixalphafunctions}.
     For functions this read as follows:

\begin{proposition}\label{proposition:operatoronfunctions}
   Let $(S_1,\mu_1)$ and $(S_2,\mu_2)$ be measure spaces and let $f\colon S_1 \to X$ be a strongly measurable function in $\alpha(S_1;X)$. Take $T \in \mc{L}(L^2(S_1), L^2(S_2))$ and suppose that there exists a strongly measurable $g:S_2 \to X$ such that for every $x^* \in X^*$ we have
  \begin{equation*}
    x^* \circ g = T(x^* \circ f)
  \end{equation*}
  or equivalently $x^* \circ g \in L^2(S_2)$ and
  \begin{equation*}
    \int_{S_2} \varphi g \dd \mu_2 = \int_{S_1} (T^*\varphi) f \dd \mu_1, \qquad \varphi \in L^2(S_2).
  \end{equation*}
  Then $g \in \alpha(S_2;X)$ and
  \begin{equation*}
    \nrm{g}_{\alpha(S_2;X)} \leq \nrm{T} \nrm{f}_{\alpha(S_1;X)}.
  \end{equation*}
\end{proposition}

 In the setting of Proposition \ref{proposition:operatoronfunctions} we write $Tf=g$.
 As typical examples, we note that multiplication by an $L^\infty$-function is a bounded operation on $\alpha(\R;X)$ and  the Fourier transform can be extended from an isometry on $L^2(\R)$ to an isometry on $\alpha(\R;X)$. Combining these examples we would obtain a Fourier multiplier theorem, which we will treat more generally in Corollary \ref{corollary:fouriermulitplier}.

\begin{example}\label{example:pointwisemulti}
  Let $(S,\mu)$ be a measure space and suppose that $f \in \alpha(S;X)$. For any $m \in L^\infty(S)$ we have $m f \in \alpha(S;X)$ with $$\nrm{m f}_{\alpha(S;X)} \leq \nrm{m}_{L^\infty(S)}\nrm{f}_{\alpha(S;X)}.$$
\end{example}

\begin{example}\label{example:fourier}
 Suppose that $f \in L^1(\R;X)$ with $f \in \alpha(\R;X)$. Define
  \begin{align*}
   \mc{F}f(\xi)&:= \widehat{f}(\xi):=\int_{\R}f(t)\ee ^{-2 \pi i t\xi}\dd t,&&\xi \in \R,\\
   \mc{F}^{-1}f(\xi)&:= \widecheck{f}(\xi):=\int_{\R}f(t)\ee ^{2 \pi i t\xi}\dd t&&\xi \in \R.
  \end{align*}
  Then $\widehat{f},\widecheck{f} \in \alpha(\R;X)$ with $\nrm{\widehat{f}}_{\alpha(\R;X)} = \nrm{\widecheck{f}}_{\alpha(\R;X)} =\nrm{f}_{\alpha(\R;X)}$
\end{example}

\subsection*{The $\alpha$-H\"older inequality}
Next we will prove H\"older's inequality for $\alpha$-spaces, which is a realisation of the duality pairing between $\alpha_+(S;X)$ and $\alpha_+^*(S;X^*)$ for representable elements. Conversely, we will show that the representable elements of a subspace of $\alpha_+^*(S;X^*)$ are norming for $\alpha_+(S;X)$ using Proposition \ref{proposition:normingalpha}.

\begin{proposition}\label{proposition:alphaholder}Let $(S,\mu)$ be a measure space.
\begin{enumerate}[(i)]
  \item \label{it:alphaholderineq} Suppose that $f \colon S\to X$ and $g\colon S \to X^*$ are in $\alpha_+(S;X)$ and $\alpha^*_+(S;X^*)$ respectively. Then $\ip{f,g} \in L^1(S)$ and
  \begin{equation*}
    \int_S \abs{\ip{f,g}}\dd \mu \leq \nrm{f}_{\alpha_+(S;X)} \nrm{g}_{\alpha^*_+(S;X^*)}
  \end{equation*}
  \item \label{it:alphaholdernorming} Let $Y\subseteq X^*$ be norming  for $X$ and let $f:S \to X$ be strongly measurable. If there is a $C>0$ such that for all $g \in L^2(S)\otimes Y$ we have
  \begin{equation*}
     \int_S \abs{\ip{f,g}}\dd \mu \leq  C\, \nrm{g}_{\alpha^*(S;X^*)},
  \end{equation*}
  then $f \in \alpha_+(S;X)$ with $\nrm{f}_{\alpha_+(S;X)} \leq C$.
\end{enumerate}

\end{proposition}

\begin{proof}
We will only prove \ref{it:alphaholderineq}, as \ref{it:alphaholdernorming} follows directly from of Proposition \ref{proposition:normingalpha} and the fact that any finite rank operator is representable as an element of $L^2(S)\otimes Y$.
Let
$\Pi_m=\{E_{km}\}_{k=1}^{\infty}$ be a sequence of partitions  of $S$ as in Lemma \ref{lemma:partitions}
and let $P_m$ be the associated averaging projections. Then
\begin{equation*}
  \int_S \abs{\ip{f,g}} d\mu \leq \sup_{m \in \N}
\int_S \abs{\ip{P_mf,g}} \dd\mu .
\end{equation*}
For each $m \in \N$ we define a measurable function $h_m$ with $\abs{h_m}=1$
$\mu$-a.e. and $h_m{\ip{P_m f,g}}=\abs{\ip{P_mf,g}}$. Define $Q_{nm},R_{nm}:L^2(S) \to L^2(S)$ by
\begin{align*}
  Q_{nm}\varphi &:= \ind_{\cup_{k=1}^n E_{km}} P_m\varphi, &&\varphi \in L^2(S),\\
  R_{nm}\varphi &:= h_m \ind_{\cup_{k=1}^n E_{km}} P_m\varphi, &&\varphi \in L^2(S).
\end{align*}
 and extend these operators to bounded operators on $\alpha_+(S;X)$ and $\alpha_+(S;X^*)$ using Proposition \ref{proposition:operatoronfunctions}. Define
\begin{align*}
  x_{km} &:= \hab{\mu(E_{km})}^{-1/2} \int_{E_{km}}f \dd \mu,\\
  x^*_{km} &:= \hab{\mu(E_{km})}^{-1/2} \int_{E_{km}}h_mg \dd \mu.
\end{align*}
and note that
\begin{align*}
  Q_{nm}f &= \sum_{k=1}^n \hab{\mu(E_{km})}^{-1/2} x_{km} \ind_{E_{km}}, \\
  R_{nm}g &= \sum_{k=1}^n \hab{\mu(E_{km})}^{-1/2} x_{km}^* \ind_{E_{km}}.
\end{align*}
It follows that
\begin{equation*}
  \int_{\bigcup_{k=1}^n E_{km}} \abs{\ip{P_mf,g}} = \sum_{k=1}^n \ip{x_{km},x^*_{km}} \leq \nrm{(x_{km})_{k=1}^n}_{\alpha}\nrm{(x^*_{km})_{k=1}^n}_{\alpha^*}.
\end{equation*}
Since we have
\begin{align*}
  \nrm{(x_{km})_{k=1}^n}_{\alpha} &= \nrm{Q_{nm}f}_{\alpha_+(S;X)} \leq \nrm{f}_{\alpha_+(S;X)},\\
  \nrm{(x^*_{km})_{k=1}^n}_{\alpha} &= \nrm{R_{nm}g}_{\alpha^*_+(S;X^*)} \leq \nrm{g}_{\alpha^*_+(S;X^*)},
\end{align*}
the result follows by first letting $n \to \infty$ and then $m \to \infty$ using Fatou's lemma.
\end{proof}

\subsection*{Convergence properties}
In the function spaces $L^p(S;X)$ we have convergence theorems like Fatou's lemma and the dominated convergence theorem.
In the next proposition we summarize some convergence properties of the $\alpha$-norms. For example  \ref{it:convergenceproperty1} can be seen as an $\alpha$-version of Fatou's lemma. It is important to note that even if all $f_n$'s are in $\alpha(S;X)$, we can only deduce that $f$ is in $\alpha_+(S;X)$.

\begin{proposition}~\label{proposition:alphaspaceconvergence} Let $f:S \to X$ be a strongly measurable function.
  \begin{enumerate}[(i)]
    \item \label{it:convergenceproperty1} Suppose that $f_n\colon S \to X$ are functions in $\alpha_{+}(S;X)$ such that $$\sup_{n \in \N}\,\nrm{f_n}_{\alpha_+(S;X)}<\infty.$$ If $f_n(s)$ converges weakly to $f(s)$ $\mu$-a.e, then $f \in \alpha_+(S;X)$ with  $$\nrm{f}_{\alpha_+(S;X)} \leq \liminf_{n \to \infty}\, \nrm{f_n}_{\alpha_+(S;X)}.$$
  \end{enumerate}
  Now suppose that $f\in \alpha(S;X)$.
  \begin{enumerate}[(i)]\setcounter{enumi}{1}
    \item \label{it:convergenceproperty4} Let $(g_n)_{n=1}^\infty$ be a sequence in $L^\infty(S)$ with $\abs{g_n} \leq 1$ and $g_n(s) \to 0$  $\mu$-a.e. Then $\lim_{n\to \infty}\nrm{g_n\cdot f}_{\alpha(S;X)} = 0$.
    \item \label{it:convergenceproperty5} If $\alpha$ is ideal and $T_n,T \in \mc{L}(X)$ with $\lim_{n\to \infty} T_nx=Tx$ for $x \in X$, then $\lim_{n\to \infty} T_n \circ f \to T \circ f$ in $\alpha(S;X)$.
  \end{enumerate}
\end{proposition}

\begin{proof}
For \ref{it:convergenceproperty1} note that for all $x^* \in X^*$ we have
\begin{equation*}
  \sup_{n \in \N} \,\nrmb{x^*\circ f_n}_{L^2(S)} \leq \sup_{n \in \N}\, \nrm{f_n}_{\alpha_+(S;X)} \nrm{x^*}_{X^*} <\infty.
\end{equation*}
 Let $(e_m)_{m=1}^\infty$ be an orthonormal sequence in $L^2(S)$, set $x_{nm} = \int_S e_m f_n \dd \mu$ and $x_m = \int_S e_m f \dd \mu$.
   Then by the dominated convergence theorem we have for all $x^* \in X^*$
   \begin{equation*}
     \lim_{n \to \infty} \ip{x_{nm},x^*} = \lim_{n \to \infty} \int_S e_m \ip{f_n,x^*} \dd \mu = \int_S e_m \ip{f,x^*} \dd \mu=\ip{x_m,x^*}.
   \end{equation*}
 Thus by $\alpha$-duality we have for each $m \in \N$
  \begin{equation*}
    \nrm{(x_1,\ldots, x_m)}_\alpha \leq \liminf_{n \to \infty} \nrm{(x_{n1},\ldots,x_{nm})}_\alpha \leq \liminf_{n \to \infty} \nrm{f_n}_{\alpha_+(S;X)},
  \end{equation*}
  so \ref{it:convergenceproperty1} follows by taking the supremum over all orthonormal sequences in $L^2(S)$. For  \ref{it:convergenceproperty4} let $\varepsilon>0$. By Proposition \ref{proposition:densealphaspace} we can find a finite dimensional subspace $E\subseteq X$ and an $h \in L^2(S;E)$ such that $\nrm{f-g}_{\alpha(S;X)}<\varepsilon$. Then by \eqref{eq:finitedimesnionalE} and the dominated convergence theorem we have
  \begin{equation*}
 \lim_{n\to \infty}\nrm{g_n\cdot f}_{\alpha(S;X)} \leq \dim(E)\lim_{n\to \infty}\nrm{g_n\cdot h}_{L^2(S;X)}+\varepsilon = \varepsilon.
  \end{equation*}
  The proof of \ref{it:convergenceproperty5} is similar.
\end{proof}

\subsection*{The $\alpha$-multiplier theorem}
We now come to one of the main theorems of this section, which characterize $\alpha$-boundedness of a family of operators in terms of the boundedness of a pointwise multiplier on $\alpha(S;X)$. This will be very useful later. We say that a function $T\colon S \to \mc{L}(X)$  is strongly measurable in the strong operator topology if $Tx\colon S \to X$ is strongly measurable for all $x \in X$.
For $f \colon S \to X$ we define $Tf\colon S \to X$ by
\begin{equation*}
  Tf(s) := T(s)f(s), \qquad s \in S.
\end{equation*}

\begin{theorem}\label{theorem:pointwisemultipliers1}
Let $(S,\mu)$ be a measure space, let $T:S \to \mc{L}(X)$ be strongly measurable in the strong operator topology and set $\Gamma = \cbrace{T(s):s \in S}$. If $\Gamma$ is $\alpha$-bounded, then $Tf \in \alpha_+(S;X)$ with
  \begin{equation*}
    \nrm{Tf}_{\alpha_+(S;X)} \leq \nrm{\Gamma}_{\alpha} \nrm{f}_{\alpha(S;X)}
  \end{equation*}
  for all $f \in \alpha(S;X)$.
\end{theorem}

\begin{proof}
Let $\Pi=\{E_{m}\}_{m=1}^{\infty}$ be a partition of $S$ with associated averaging projection $P$ for $f$ as in Lemma
\ref{lemma:partitions} and let $\Pi'=\{E_{n}'\}_{n=1}^{\infty}$ be a partition of $S$ with associated averaging projection $P'$ for $TPf$ as in Lemma
\ref{lemma:partitions}. Assume without loss of generality that $\Pi'$ is a finer partition than $\Pi$, i.e. for any $n\in \N$ there is an $m_n \in \N$ such that $E'_n\subseteq E_{m_n}$.
Then
  \begin{equation*}
    P'TPf = \sum_{n=1}^\infty S_{n}x_{m_n} \ind_{E'_n}
  \end{equation*}
  where
  \begin{align*}
    x_m &= \frac{1}{\mu(E_m)}  \int_{E_m} f \dd\mu, &&m \in \N,\\
    S_{n} x &= \frac{1}{\mu(E_{n}')} \int_{E_{n}'} Tx\dd \mu, && n \in \N,\, x \in X.
  \end{align*}
  So we obtain
  \begin{align*}
  \nrm{Pf}_{\alpha(S;X)} &= \nrms{\sum_{n=1}^\infty x_m \ind_{E_m}}_{\alpha(S;X)}  = \sup_{n \in \N} \nrmb{\hab{x_{m_k} \mu(E_{k}')^{1/2}}_{k=1}^{n}}_\alpha, \\
    \nrm{P'TPf}_{\alpha_+(S;X)} &= \nrms{\sum_{n=1}^{\infty} S_{n}x_{m_n} \ind_{E_n'}}_{\alpha_+(S;X)} = \sup_{n \in \N} \nrmb{\hab{S_{k}x_{m_k} \mu(E_{k}')^{1/2}}_{k=1}^{n}}_\alpha.
  \end{align*}
  Since $S_n$ belongs to the strong operator topology closure of the convex hull of $\Gamma$ for all $n \in \N$, it follows from Proposition \ref{proposition:alphaproperties} and Proposition \ref{proposition:operatoronfunctions} that
  $$\nrm{P'TPf}_{\alpha_+(S;X)} \leq \nrm{\Gamma}_{\alpha} \nrm{Pf}_{\alpha(S;X)}\leq \nrm{\Gamma}_{\alpha} \nrm{f}_{\alpha(S;X)}.$$ Now let $P_m$ be a sequence of such averaging projections for $f$ as in Lemma \ref{lemma:partitions} and, for every $m \in \N$, let $P_{m'}'$ be a sequence of such averaging projections for $TP_mf$ as in Lemma \ref{lemma:partitions}. Then we have
  \begin{equation*}
    \lim_{m \to \infty} \lim_{m' \to \infty} P'_{m'}TP_mf(s)  = Tf(s), \qquad s \in S,
  \end{equation*}
so the conclusion follows by applying Proposition \ref{proposition:alphaspaceconvergence}\ref{it:convergenceproperty1} twice.
\end{proof}

\begin{remark} Since we use Proposition \ref{proposition:alphaspaceconvergence}\ref{it:convergenceproperty1} in the proof of Theorem \ref{theorem:pointwisemultipliers1}, we do not know whether $Tf \in \alpha(S;X)$. We refer to \cite[Section 9.5]{HNVW17} for a discussion on sufficient conditions such that one can conclude $Tf \in \alpha(S;X)$ in the case $\alpha=\gamma$.
\end{remark}

We also have a converse of Theorem \ref{theorem:pointwisemultipliers1}, for which we need to assume that the measure space $(S,\mu)$ has more structure. A \emph{metric measure space} $(S,d,\mu)$ is  a complete separable metric space $(S,d)$ with a locally finite Borel measure $\mu$. We  denote by $\supp(\mu)$ the smallest closed set with the property that its complement has measure zero.

\begin{theorem}\label{theorem:pointwisemultipliers2}
  Let $(S,d,\mu)$ be a metric measure space, let $T:S \to \mc{L}(X)$ be continuous in the strong operator topology and set $$\Gamma = \cbrace{T(s):s \in \supp(\mu)}.$$ If we have $Tf \in \alpha_+(S;X)$ for all $f \in \alpha(S;X)$ with
  \begin{equation*}
    \nrm{Tf}_{\alpha_+(S;X)} \leq C \nrm{f}_{\alpha(S;X)},
  \end{equation*}
then $\Gamma$ is $\alpha$-bounded with $\nrm{\Gamma}_{\alpha} \leq C$.
\end{theorem}

\begin{proof}
 Take $T_1,\ldots,T_n \in \Gamma$ and $\mb{x} \in X^n$. Let $s_1,\ldots,s_n \in \supp(\mu)$ be such that $T_k = T(s_k)$ for $1 \leq k \leq n$ and let $\varepsilon>0$. For $1 \leq k \leq n$, using the continuity of $T$ and the fact that $s_k \in \supp(\mu)$, we can select an open ball $O_k \subseteq \supp(\mu)$ with finite positive measure such that $s_k \in O_k$ and
 \begin{equation}\label{eq:pointwisemultieq}
   \nrm{T(s)x_k -T(s_k)x_k} \leq n^{-1}\varepsilon,\qquad s \in O_k.
 \end{equation}
  If  $O_{k_1}\cap O_{k_2}\neq \varnothing$ for $1 \leq k_1 \neq k_2 \leq n$, then $\mu(O_{k_1}\cap O_{k_2})>0$. Since $\mu$ is non-atomic, there are disjoint $E_1,E_2$ with positive measure such that $O_{k_1}\cap O_{k_2}=E_1\cup E_2$. Iteratively replacing $O_{k_1}$ by $O_{k_1}\setminus E_1$ and $O_{k_2}$ by $O_{k_2}\setminus E_2$ for all pairs $1 \leq k_1 \neq k_2 \leq n$, we obtain pairwise disjoint sets $O_1,\ldots,O_n$ of positive finite measure such that \eqref{eq:pointwisemultieq} holds.

Let $P$ be the averaging projection associated to $O_1,\ldots,O_n$ and define $f = \sum_{k=1}^n \mu(O_k)^{-1/2}x_k\ind_{O_k}$. Then
  \begin{equation*}
    PTf = \sum_{k=1}^n \mu(O_k)^{-1/2}y_k\ind_{O_k}
  \end{equation*}
  for
  $$
  y_k = \frac{1}{\mu(O_k)} \int_{O_k}Tx_k \dd \mu, \qquad 1 \leq k \leq n.
  $$
  Note that $\nrm{y_k - T_kx_k} \leq n^{-1} \varepsilon$, so we have by Proposition \ref{proposition:operatoronfunctions}, the fact that $(\mu(O_k)^{-1/2}\ind_{O_k})_{k=1}^n$ is an orthonormal system in $L^2(S)$ and our assumption, that
  \begin{align*}
     \nrm{\mb{y}}_\alpha = \nrm{PTf}_{\alpha(S;X)} \leq C\,\nrm{f}_{\alpha(S;X)} = C\,\nrm{\mb{x}}_\alpha.
  \end{align*}
  Therefore $\nrm{(T_1x_1,\ldots,T_nx_n)}_\alpha \leq C \nrm{\mb{x}}_\alpha +\varepsilon$, which proves  the theorem.
\end{proof}

We conclude this section by combining Theorem \ref{theorem:pointwisemultipliers1} and Example \ref{example:fourier} into the following Fourier multiplier theorem.

\begin{corollary}\label{corollary:fouriermulitplier}
     Suppose that $m:\R \to \mc{L}(X)$ is strongly measurable in the strong operator topology and $ \cbrace{m(s):s \in \R}$ is $\alpha$-bounded. For $f \in L^1(\R;X)$ such that $\widehat{f} \in L^1(\R;X)$ we define
    $$T_mf(s) = \mc{F}^{-1}\hab{m(s)\widehat{f}(s)}, \qquad s \in S.$$ If $f \in \alpha(\R;X)$, then  $Tf \in \alpha_+(\R;X)$ with
  \begin{equation*}
    \nrm{T_mf}_{\alpha_+(\R;X)} \leq \nrmb{\cbrace{m(s):s \in \R}}_\alpha \nrm{f}_{\alpha(\R;X)}.
  \end{equation*}
\end{corollary}

\section{The \texorpdfstring{$\alpha$}{a}-interpolation method}\label{section:interpolation}
In this section we will develop a theory of interpolation using Euclidean structures.  This method seems especially well-adapted to the study of sectorial operators and semigroups, which we will explore further in Chapter \ref{part:5}. Although we develop this interpolation method in more generality, the most important example is the Gaussian structure, which gives rise to the Gaussian  method of interpolation. A discrete version of the Gaussian method was already considered in \cite{KKW06}, where it is used to the study the $H^\infty$-calculus of various differential operators. The continuous version of the Gaussian method was studied in \cite{SW06,SW09}, where Gaussian interpolation of Bochner spaces $L^p(S;X)$ and square function spaces $\gamma(S;X)$, as well as a Gaussian version of abstract Stein interpolation, was treated. Furthermore, for Banach function spaces, an $\ell^q$-version of this interpolation method was developed in \cite{Ku15}. An abstract framework covering these interpolation methods, as well as the real and complex interpolation methods, is developed in \cite{LL21}.

The results in \cite{KKW06,SW06,SW09} were based on a draft version of this memoir, which explains why some of these papers omit various proofs with a reference to this memoir, see e.g.  \cite[Proposition 7.3]{KKW06} and \cite[Section 2]{SW06}.

\bigskip

Throughout this section we let $\alpha$ be a global Euclidean structure, $(X_0,X_1)$ a compatible pair of Banach spaces and $\theta \in (0,1)$.  We will define interpolation spaces $(X_0,X_1)_{\theta}^{\alpha}$ and $(X_0,X_1)_{\theta}^{\alpha_+}$ and refer to these methods of interpolation as the $\alpha$-method and the $\alpha_+$-method. Note that we will only use the Euclidean structures $\alpha_0$ on $X_0$ and $\alpha_1$ on $X_1$ for our construction, so the assumption that $\alpha$ is a global Euclidean structure is only for notational convenience.

Let us consider the space
$$L^2(\R)+ L^2(\R,\ee^{-2t}\ddn t) = L^2(\R,\min\cbrace{1,\ee^{-2t}}\ddn t).$$
We call an operator
\begin{equation*}
  T:L^2(\R)+L^2(\R,\ee^{-2t}\ddn t)\to  X_0+X_1.
\end{equation*}
admissible and write
 $T\in \mc{A}$ (respectively $T \in \mc{A}_+$) if  $T\in   \alpha(\R,e^{-2jt}\ddn t;X_j)$ (respectively $T \in \alpha_{+}(\R,\ee^{-2jt}\ddn t;X_j)$) for $j=0,1$. We define
\begin{align*}
  \nrm{T}_{\mc{A}}&:= \max_{j=0,1}\,\nrm{T_j}_{\alpha(\R,\ee^{-2jt}\ddn t;X_j)},\\
  \nrm{T}_{\mc{A}_+}&:= \max_{j=0,1}\nrm{T_j}_{\alpha_+(\R,\ee^{-2jt}\ddn t;X_j)},
\end{align*} where $T_j$ denotes the operator $T$ from $L^2(\R,\ee^{-2jt}\ddn t)$ into  $X_j$. Both $\mc{A}$ and $\mc{A}_+$ are complete with respect to their norm.

 Denote by
$e_{\theta}$ the function $t \mapsto \ee^{\theta t}$. We define $(X_0,X_1)^{\alpha}_{\theta}$ as the space of all $x\in
X_0+X_1$ such that
\begin{align*}
  \nrm{x}_{(X_0,X_1)^{\alpha}_{\theta}}:= \inf\{\nrm{T}_{\mc{A}}: \  T\in\mc{A},\
T(\ee_{\theta})=x\} <\infty.
\intertext{The space   $(X_0,X_1)^{\alpha}_{\theta,+}$ is defined similarly as the space of all $x \in X_0+X_1$ such that}
  \nrm{x}_{(X_0,X_1)^{\alpha_+}_{\theta}}:= \inf\{\nrm{T}_{\mc{A}_+}: \  T\in\mc{A}_+,\
T(\ee_{\theta})=x\} <\infty.
\end{align*}
Then $(X_0,X_1)^{\alpha}_{\theta}$ and $(X_0,X_1)^{\alpha_+}_{\theta}$ are quotient spaces of $\mc{A}$ and $\mc{A}_+$ respectively and thus Banach spaces.
For brevity we will sometimes write $X_{\theta}:=(X_0,X_1)^{\alpha}_{\theta}$ and $X_{\theta,+}:=(X_0,X_1)^{\alpha_+}_{\theta}$.

\begin{proposition}[$\alpha$-Interpolation of operators]\label{proposition:interpolation}  Suppose that $(X_0,X_1)$ and $(Y_0,Y_1)$ are compatible pairs of Banach spaces and $\alpha$ is ideal.  Assume  that $S\colon X_0+X_1\to Y_0+Y_1$ is a bounded operator such that $S(X_0)\subset Y_0$ and $S(X_1)\subset Y_1$. Then $S\colon X_\theta \to Y_\theta$ is bounded with
\begin{equation*}
  \nrm{S}_{X_{\theta}\to Y_{\theta}}
\leq \nrm{S}_{X_0\to Y_0}^{1-\theta}\nrm{S}_{X_1\to Y_1}^{\theta}.
\end{equation*}
\end{proposition}

A similar statement  holds for $S_+\colon X_{\theta,+} \to Y_{\theta,+}$.

\begin{proof}
Suppose $T\in\mc{A}$.  Fix $\tau$ so that $\nrm{S}_{X_1\to Y_1}
=\ee^{\tau}\nrm{S}_{X_0\to Y_0}$ and let $U_\tau$ be the shift operator given by $U\varphi = \varphi(\cdot-\tau)$, which satisfies
\begin{equation}\label{eq:Utau}
  \nrm{U_\tau}_{\mc{L}(L^2(\R,\ee^{-2jt}\ddn t))} \leq \ee^{-j\tau}, \qquad j=0,1.
\end{equation}
The ideal property of $\alpha$ means that
$ST U_{\tau}$ is admissible and
\begin{equation*}
  \nrm{STU_{\tau}}_{\mc{A}} \leq \max_{j=0,1}\cbraceb{\nrm{S}_{X_j \to Y_j} \nrm{T}_{\alpha(\R,\ee^{-2jt}\ddn t;X_j)} \ee^{-j\tau}} \leq \nrm{S}_{X_0 \to Y_0} \nrm{T}_{\mc{A}}.
\end{equation*}
Now if $T(\ee_\theta) = x$, then $\ee^{\theta \tau} \cdot STU_\tau(\ee_{\theta}) = Sx$ and therefore
\begin{equation*}
  \nrm{S}_{X_{\theta}\to Y_{\theta}} \leq \ee^{\theta \tau} \nrm{S}_{X_0 \to Y_0} = \nrm{S}_{X_0\to Y_0}^{1-\theta}\nrm{S}_{X_1\to Y_1}^{\theta}.\qedhere
\end{equation*} \end{proof}

In interpolation theory it is often useful to know that $X_0\cap X_1$ is dense in the intermediate spaces, which is the content of the next lemma.
\begin{proposition}\label{proposition:finiterankX0X1dense} The set of finite rank operators $T \in \mc{A}$ is dense in $\mc{A}$. In particular, $X_0\cap X_1$ is dense in $X_{\theta}$.
\end{proposition}

\begin{proof}  If $T\in \mc{A}$, we consider the operators $S_{\lambda,n}$ given by
\begin{equation*}
  S_{\lambda,n} \varphi(t) :=\sum_{\abs{k} \leq n } \has{\frac{1}{\lambda}\int_{k\lambda}^{(k+1)\lambda}\varphi(s)\ddn s} \ind_{[k\lambda,(k+1)\lambda)}(t), \qquad t \in \R
\end{equation*}
 for $\varphi \in L^2(\R)+L^2(\R,\ee^{-2t}\ddn t)$. As $T_jS_{\lambda,n}$ has finite rank, it suffices to show
 that for $j=0,1$
\begin{equation}\label{eq:densitySlambdan}
  \lim_{\lambda\to
0}\lim_{n\to\infty}\nrm{T_j-T_jS_{\lambda,n}}_{\alpha(\R,\ee^{-2jt}\ddn t;X_j)}=0.
\end{equation}
Note that for a finite rank operator $U\in \alpha(\R,\ee^{-2jt}\ddn t;X_j)$ and $j=0,1$
$$\lim_{\lambda\to 0}\lim_{n\to\infty}\nrm{U-US_{\lambda,n}}_{\alpha(\R,\ee^{-2jt}\ddn t;X_j)}=0$$
by the Lebesgue differentiation theorem. Moreover we have
\begin{align*}
  \nrm{S_{\lambda,n}}_{\mc{L}(L^2(\R))}&=1,\\
  \nrm{S_{\lambda,n}}_{\mc{L}(L^2(\R,\ee^{-2t}\ddn t))}&=\frac{\sinh \lambda }{\lambda},
\end{align*}
so by
density we obtain \eqref{eq:densitySlambdan} for $j=0,1$. To conclude note that if $T \in \mc{A}$ has finite rank, then necessarily $$T(L^2(\R)+L^2(\R,\ee^{-2t}\ddn t)) \subseteq X_0\cap X_1,$$
since $T\in   \alpha(\R,e^{-2jt}\ddn t;X_j)$ for $j=0,1$. Thus $X_0\cap X_1$ is dense in $X_{\theta}$.
\end{proof}

\subsection*{Duality}
If $X_0\cap X_1$ is dense in both $X_0$ and $X_1$, then
the pair $(X_0^*,X_1^*)$ is also compatible.  We can then define
the classes $\mc{A}^*, \mc{A}^*_{+}$  for the pair
$(X_0^*,X_1^*)$ with the global Euclidean structure $\alpha^*$ and define the interpolation spaces $(X_0^*,X_1^*)^{\alpha^*}_{\theta}$      and $(X_0^*,X_1^*)^{\alpha^*_+}_{\theta}$, which we write as $X_{\theta}^*$ and $X_{\theta,+}^*$ for brevity.

If $T \in \mc{A}_+^*$ we can view $T^*$ as the operator from $X_0\cap X_1$ to $L^2(\R)\cap L^2(\R,\ee^{-2t}\ddn t)$ so that for $x \in X_0 \cap X_1$
\begin{equation*}
  \ip{T^*x,\varphi} = \ip{x,T\varphi}, \qquad \varphi \in L^2(\R)+ L^2(\R,\ee^{-2t}\ddn t),
\end{equation*}
using the densely defined bilinear form $\ip{\cdot,\cdot}$ on $L^2(\R)+ L^2(\R,\ee^{-2t}\ddn t)$ given by
\begin{equation}\label{eq:bilinearform}
\ip{\varphi_1,\varphi_2} =\int_{\R}\varphi_1(t)\varphi_2(-t)\,\ddn t
\end{equation}
for all $\varphi_1$ and $\varphi_2$ such that $\varphi_1(\lcdot)\varphi_2(-\lcdot)\in L^1(\R)$, which holds in particular if
\begin{align*}
  \varphi_1 &\in L^2(\R)\cap L^2(\R,\ee^{-2t}\ddn t),\\
  \varphi_2 &\in L^2(\R)+ L^2(\R,\ee^{-2t}\ddn t).
\end{align*}
Then $T^*$ extends to the adjoints $T_j^*:X_j\to L^2(\R,\ee^{-2jt}\ddn t).$

\begin{lemma}\label{lemma:tracedualA}
  Suppose that $X_0 \cap X_1$ is dense in $X_0$ and $X_1$. If $S \in \mc{A}$ and $T \in \mc{A}^*_{+}$, then
  \begin{equation*}
    \tr(T_0^*S_0) = \tr(T_1^*S_1).
  \end{equation*}
\end{lemma}

\begin{proof}
  Let us fix $T \in \mc{A}^*_+$. The equality is trivial if $S$ has finite rank and thus range contained in $X_0\cap X_1$, since $T^*S$ then has finite rank and range contained in $L^2(\R)\cap L^2(\R,\ee^{-2t}\ddn t)$. Since the functionals $S \mapsto \tr(T_0^*S_0)$ and $S \mapsto \tr(T_1^*S_1)$ are continuous, the result follows from Proposition \ref{proposition:finiterankX0X1dense}
\end{proof}

By Lemma \ref{lemma:tracedualA} we can now define the pairing
\begin{equation*}
  \ip{S,T} := \tr(T_0^*S_0) = \tr(T_1^*S_1),\qquad S \in \mc{A}, \, T \in \mc{A}^*_+
\end{equation*}
and note that
\begin{equation}\label{eq:Apairinginequality}
  \abs{\ip{S,T}} \leq \min_{j=0,1} \nrm{S_j}_{\alpha(\R,\ee^{-2jt}\ddn t;X_j)}\nrm{T_j}_{\alpha^*_+(\R,\ee^{-2jt}\ddn t;X_j)} \leq \nrm{S}_{\mc{A}}\nrm{T}_{\mc{A}^*_+}
\end{equation}
for $S \in \mc{A}$ and $T \in \mc{A}^*_+$.

\begin{theorem}\label{theorem:interpolatealphadual}
   Suppose that $X_0\cap X_1$ is dense in $X_0$ and $X_1$. Then we have $(X_\theta)^* = X_{\theta,+}^*$ isomorphically.
\end{theorem}

\begin{proof}
Let $x^* \in X_{\theta,+}^*$ and take $T\in \mc{A}_+^*$ with $T(\ee_\theta) = x^*$. Fix  $x \in X_0 \cap X_1$ and take an $S \in \mc{A}$ with finite rank and $S(\ee_\theta)=x$. For $\tau \in \R$ let $U_\tau$ be the shift operator given by $U_\tau\varphi = \varphi(\cdot-\tau)$.  For $\varphi \in L^2(\R) \cap L^2(\R,\ee^{2jt}\ddn t)$ we note that
  \begin{equation*}
    \int_{\R} \ee^{\theta \tau} U_{\tau} \varphi \dd \tau = \int_{\R}\ee^{\theta (\tau+\cdot)} \varphi(-\tau) \dd \tau =   \ip{\ee_\theta, \varphi}\ee_\theta
  \end{equation*}
  as Bochner integral in $L^2(\R) + L^2(\R,\ee^{2jt}\ddn t)$. Thus, since the range of $T^*SU_{\tau}$ is contained in a fixed finite-dimensional subspace of $L^2(\R) \cap L^2(\R,\ee^{2jt}\ddn t)$ for all $\tau \in \R$, we have
  \begin{equation*}
    \int_{\R} \ee^{\theta \tau} \ipb{SU_{\tau},T}\ddn \tau = \ipb{S(e_\theta), T(\ee_\theta)}.
  \end{equation*}
  Now by \eqref{eq:Utau} and \eqref{eq:Apairinginequality} we have
  \begin{equation*}
    \ee^{\theta\tau }\ip{SU_{\tau},T} \leq \begin{cases}
      \ee^{(\theta -1)\tau}\nrm{S_1}_{\alpha(\R,\ee^{-2jt}\ddn t;X_j)}\nrm{T_1}_{\alpha_+^*(\R,\ee^{-2jt}\ddn t;X_j)},\quad &\tau \geq 0,\\
      \ee^{\theta\tau}\nrm{S_0}_{\alpha(\R,\ee^{-2jt}\ddn t;X_j)}\nrm{T_0}_{\alpha^*_+(\R,\ee^{-2jt}\ddn t;X_j)},\quad &\tau < 0,
    \end{cases}
  \end{equation*}
  from which it follows that
  \begin{equation*}
   \abs{\ip{x,x^*}} =  \abs{\ip{S(\ee_\theta),T(\ee_\theta)}} \leq (\theta(1-\theta))^{-1}\nrm{S}_{\mc{A}} \nrm{T}_{\mc{A}^*_+}.
  \end{equation*}
 Hence, taking the infimum over all such $S$ and $T$ and using Proposition \ref{proposition:finiterankX0X1dense}, we have
  \begin{equation*}
   \abs{\ip{x,x^*}} \leq   (\theta(1-\theta))^{-1}\nrm{x}_{X_\theta} \nrm{x^*}_{X_{\theta,+}^*}.
  \end{equation*}
  By the density of $X_0\cap X_1$ in $X_\theta$ this implies that $X_{\theta,+}^*$ embeds continuously into $(X_\theta)^*$.

  \bigskip

  We now turn to the other embedding. Given $x^* \in (X_{\theta})^*$ we must show  $x^* \in X_{\theta,+}^*$ with  $\nrm{x^*}_{X^*_{\theta,+}} \leq C\, \nrm{x^*}_{(X_{\theta})^*}$.
  First note that $x^*$ induces a linear functional $\psi$ on $\mc{A}$ by $\psi(S) = x^*(S(\ee_\theta))$ for $S \in \mc{A}$. Moreover there is a natural isometric embedding of $\mc{A}$ into
  \begin{equation*}
    \alpha(\R;X_0) \oplus_\infty \alpha(\R,\ee^{-2t}\ddn t;X_1)
  \end{equation*}
  via the map $S \mapsto (S_0,S_1)$. Hence by the Hahn-Banach theorem we can
  extend $x^*$ to a functional on this larger space, i.e. there is a
  \begin{equation*}
   T = (T_0,T_1) \in \alpha_{+}^*(\R,X_0^*) \oplus_1 \alpha_{+}^*(\R,\ee^{-2t}\ddn t,X_1^*)
  \end{equation*}
  such that $\nrm{T} = \nrm{x^*}_{(X_{\theta})^*}$ and
  \begin{equation*}
    \tr(T^*_0S_0) + \tr(T^*_1S_1) = x^*(S(\ee_\theta)),\qquad S \in \mc{A}.
  \end{equation*}

  Let us apply this to the rank one operator $S = \varphi \otimes x$ for some $ \varphi \in L^2(\R)\cap L^2(\R,\ee^{-2t}\ddn t)$ and $x \in X_0 \cap X_1$. Then
  \begin{equation*}
    \ip{x, T_0(\varphi)}+\ip{x,T_1(\varphi)} = x^*(x) \ip{\ee_\theta,\varphi},
  \end{equation*}
  so we have, by the density of $X_0 \cap X_1$, that
  \begin{equation}\label{eq:functionaly*}
    T_0(\varphi) +T_1 (\varphi )=  \ip{\ee_\theta,\varphi}x^*, \qquad \varphi \in L^2(\R)\cap L^2(\R,\ee^{-2t}\ddn t)
  \end{equation}
  as functionals on $X_\theta$. Let $U = \ee^{\theta} U_1-I$, where $U_1$ is the shift operator given by
  $U_1\varphi = \varphi(\cdot-1)$. Then we have
  \begin{equation}\label{eq:functionalzero}
    T_0(U\varphi) +T_1 (U\varphi )=\hab{\ee^{\theta}\ip{\ee_\theta,U_1\varphi}-\ip{\ee_\theta,\varphi}}x^* = 0.
  \end{equation}
  Note that
  \begin{align*}
    \nrm{T_0U}_{\alpha_+^*(L^2(\R),X_0^*)} &\leq (\ee^{\theta}+1) \nrm{x^*}_{(X_{\theta})^*},\\
    \nrm{T_1U}_{\alpha_+^*(L^2(\R,\ee^{-2t}\ddn t),X_1^*)}  &\leq (\ee^{\theta-1}+1) \nrm{x^*}_{(X_{\theta})^*}.
  \end{align*}
 So it follows from \eqref{eq:functionalzero} that $V\colon L^2(\R) + L^2(\R,\ee^{-2t}\ddn t) \to X_0 + X_1$ given by
 \begin{equation*}
   V\varphi = \begin{cases}
     T_0U\varphi, \qquad &\varphi \in L^2(\R)\\
     -T_1U\varphi, &\varphi \in L^2(\R,\ee^{-2t}\ddn t)
   \end{cases}
 \end{equation*}
 is a well-defined element of $\mc{A}^*_+$ and $\nrm{V}_{\mc{A}^*_+} \leq (\ee^\theta+1)\nrm{x^*}_{(X_{\theta})^*}$.
Let us compute $V(\ee_\theta)$. We have, using \eqref{eq:functionaly*}, that
\begin{align*}
  V(\ee_\theta) &= T_0U(\ee_\theta\ind_{(-\infty,0)})-T_1U(\ee_\theta \ind_{(0,\infty)})\\
  &=T_0(\ee_\theta \ind_{(0,1)}) + T_1(\ee_\theta \ind_{(0,1)})\\
  &= \ip{\ee_\theta,\ee_\theta\ind_{(0,1)}}x^* = x^*
\end{align*}
Thus we have $x^* \in X^*_{\theta,+}$ with  $\nrm{x^*}_{X^*_{\theta,+}} \leq \nrm{V}_{\mc{A}^*_+} \leq (\ee^\theta+1) \nrm{x^*}_{(X_\theta)^*}$ and the proof is complete.
\end{proof}
\section{A comparison with real and complex interpolation}
We will now compare the $\alpha$-interpolation method with the more well-known real and complex interpolation methods. We will only consider the $\alpha$-interpolation method in this section and leave the adaptations necessary to treat the $\alpha_+$-interpola\-tion method to the interested reader.
As in the previous section, throughout this section $\alpha$ is a global Euclidean structure, $(X_0,X_1)$ is a compatible pair of Banach spaces and $0<\theta<1$.

\subsection*{Real interpolation formulation}
We will start with a formulation of the $\alpha$-interpo\-lation method in the spirit of the real interpolation method. More precisely, we will give a formulation of the $\alpha$-interpolation method analogous to the Lions-Peetre mean method, which is equivalent to the real interpolation method in terms of the $K$-functional (see \cite{LP64}). Let $\mc{A}_{\bullet}$ be the set of all strongly measurable functions $f:\R_+ \to X_0 \cap X_1$ such that $t\mapsto t^jf(t) \in \alpha(\R_+, \frac{\ddn t}{t};X_j))$  for $j=0,1$. Define for $f \in \mc{A}_\bullet$
\begin{equation*}
  \nrm{f}_{\mc{A}_\bullet}:=\max_{j=0,1}\,\nrm{t \mapsto t^jf(t)}_{\alpha(\R_+,\frac{\ddn t}{t};X_j)}.
\end{equation*}

\begin{proposition}\label{proposition:alpharealmethod}
  For  $x \in X_0+X_1$ we have
  \begin{align*}
    \nrm{x}_{X_\theta} &= \inf \cbraceb{\nrm{f}_{\mc{A}_\bullet}: f \in \mc{A}_\bullet \text{ with } \int_{0}^\infty t^{\theta} f(t)\tfrac{\ddn t}{t} = x}
  \end{align*}
  where the integral converges in the Bochner sense in $X_0+X_1$.
\end{proposition}

\begin{proof}
Note that for $f \in \mc{A}_\bullet$ we have $t \mapsto f(e^t)\in \alpha(\R, \ee^{-2jt},X_j))$ for $j=0,1$. Therefore, using the transformation $t \mapsto \ee^t$, we may identify
 $\mc{A}_{\bullet}$ with a subset of $\mc{A}$. So the inequality ``$\leq$'' is immediate.

  To obtain the converse inequality take $x \in X_\theta$, let $\varepsilon>0$ and $T \in \mc{A}$ with $T(\ee_\theta)=x$ and $\nrm{T}_{\mc{A}} < (1+\varepsilon)\nrm{x}_{X_\theta}$. For $\lambda>0$ we consider the convolution operator
  \begin{equation*}
    K_\lambda \varphi = \frac{1}{2\lambda} \int_{-\lambda}^\lambda \varphi(\cdot-t) \ee^{\theta t}\dd t, \qquad \varphi \in L^2(\R) +L^2(\R,\ee^{-2t}\dd t).
  \end{equation*}
  Then $K_\lambda(\ee_\theta) = \ee_\theta$, hence $TK_\lambda(\ee_\theta) = x$. Note that for $j=0,1$
  \begin{equation*}
    \nrm{K_\lambda}_{\mc{L}(L^2(\R,\ee^{-2jt}\ddn t))} \leq \frac{1}{2\lambda}\int_{-\lambda}^\lambda \ee^{(\theta-j)t}\dd t \leq \begin{cases} \frac{\sinh(\theta \lambda)}{\theta \lambda} &j=0,\\
    \frac{\sinh((1-\theta) \lambda)}{(1-\theta) \lambda} &j=1.
    \end{cases}
  \end{equation*}
  Hence for small enough $\lambda>0$
  \begin{equation}\label{eq:TKlambdainequal}
    \nrm{TK_\lambda}_{\mc{A}} <(1+\varepsilon) \nrm{x}_{X_\theta}.
  \end{equation}
  Now we show that $TK_\lambda$ is representable by a function.
  Let
  \begin{equation*}
    F(t) = \begin{cases}
      T\hab{\ind_{(0,t)}\ee_{\theta}} &t>0,\\
      -T\hab{\ind_{(t,0]}\ee_{\theta}} &t \leq 0,
    \end{cases}
  \end{equation*}
  and note that we have
  \begin{equation*}
    K_\lambda\varphi = \frac{1}{2\lambda} \int_\R \varphi(t)\ind_{(t-\lambda,t+\lambda)}\ee^{\theta(\cdot-t)}\dd t, \qquad \varphi \in L^2(\R) +L^2(\R,\ee^{-2t}\dd t)
  \end{equation*}
  as a Bochner integral in $L^2(\R) +L^2(\R,\ee^{-2t}\dd t)$.
  Hence
  \begin{equation}\label{eq:Tklambdaequal}
    TK_\lambda\varphi = \frac{1}{2\lambda} \int_\R \varphi(t) \ee^{-\theta t}\hab{F(t+\lambda)-F(t-\lambda)}\dd t,
  \end{equation}
  so we can take
  \begin{equation*}
    g(t) = \frac{\ee^{-\theta t}}{2\lambda}\hab{F(t+\lambda)-F(t-\lambda)}, \qquad t \in \R.
  \end{equation*}
  Then, for $f(t) = g(\ln(t))$, we have by \eqref{eq:TKlambdainequal} and \eqref{eq:Tklambdaequal}
  \begin{equation*}
    \max_{j=0,1}\,\nrm{t \mapsto t^jf(t)}_{\alpha(\R_+,\frac{\ddn t}{t};X_j)} = \nrm{g}_{\mc{A}} = \nrm{TK_\lambda}_{\mc{A}} \leq ( 1+\varepsilon)\nrm{x}_{X_\theta},
  \end{equation*}
  which proves the inequality ``$\geq$''.
\end{proof}

The Lions-Peetre mean method also admits a discretized version. Using Proposition \ref{proposition:alpharealmethod} we can also give a discretized version of the $\alpha$-interpolation method in the same spirit. On a Banach space with finite cotype this will show that the $\gamma$-interpolation method is equivalent with the Rademacher interpolation method introduced in \cite[Section 7]{KKW06}. Moreover, it connects the $\alpha$-interpolation method to the abstract interpolation framework developed in \cite{LL21}.

Let $\mc{A}_{\#}$ be the set of all infinite sequences $(x_k)_{k \in \Z}$ in $X_0\cap X_1$ such that $(x_k)_{k \in \Z} \in \alpha(\Z;X_0)$ and $(2^{k}x_k)_{k\in \Z} \in \alpha(\Z;X_1)$, equipped with the norm
\begin{equation*}
  \nrm{(x_k)_{k \in \Z}}_{\mc{A}_{\#}}:= \max\cbraceb{\nrm{(x_k)_{k \in \Z}}_{\alpha(\Z;X_0)}, \nrm{(2^{k}x_k)_{k\in \Z}}_{\alpha(\Z;X_1)}}.
\end{equation*}

\begin{proposition}\label{proposition:alpharealmethoddiscrete}
  For  $x \in X_0+X_1$ we have
  \begin{align*}
    \nrm{x}_{X_\theta} &\simeq \inf \cbraceb{\nrm{\mb{y}}_{\mc{A}_{\#}}: \mb{y} \in \mc{A}_{\#}, \sum_{k \in \Z} 2^{k\theta} y_k = x},
  \end{align*}
  where the series converges in $X_0+X_1$.
\end{proposition}

\begin{proof}
  Fix $x \in X_0+X_1$. By  Proposition \ref{proposition:alpharealmethod} it suffices to prove
  \begin{equation}\label{eq:infequiv}
  \begin{aligned}
    \inf \cbraceb{\nrm{\mb{y}}_{\mc{A}_{\#}}: \mb{y} &\in \mc{A}_{\#}\text{ with } \sum_{k \in \Z} 2^{k\theta} y_k = x} \\&\simeq \inf \cbraceb{\nrm{f}_{\mc{A}_\bullet}: f \in \mc{A}_\bullet \text{ with } \int_{0}^\infty t^{\theta} f(t)\tfrac{\ddn t}{t} = x}.
\end{aligned}
  \end{equation}

 First suppose the right-hand side is finite and let $f \in \mc{A}_\bullet$ be such that $\int_{0}^\infty t^{\theta} f(t)\tfrac{\ddn t}{t} = x$. Define $g(t) = f(2^t)$, then we have
   $\ln(2) \cdot \int_{\R} 2^{t\theta} g(t) \dd t= x$ and for $j=0,1$   \begin{equation}\label{eq:gunderfinterpolation}
   \nrm{t\mapsto 2^{jt}g(t)}_{\alpha(\R;X_j)}  \lesssim \nrm{t \mapsto t^jf(t)}_{\alpha(\R_+,\frac{\ddn t}{t};X_j)}
   \end{equation}
   by the boundedness of the map $h \mapsto \hab{t\mapsto h(2^{t})}$ from $L^2(\R_+, \frac{\ddn t}{t})$ to $L^2(\R)$.
For $k \in \Z$ define
  \begin{equation*}
    y_k = \ln(2) \int_{k}^{k+1}2^{(t-k)\theta} g(t)\dd t \in X_0\cap X_1.
  \end{equation*}
  For $j=0,1$ we have, since the functions $$\varphi_k(t):=2^{(t -k)(\theta-j)} \ind_{[k,k+1)}, \qquad t \in \R$$ are orthogonal and uniformly bounded in $L^2(\R)$, that
  \begin{align*}
    \nrm{\ha{2^{jk}y_k}_{k \in \Z}}_{\alpha(\Z;X_j)}  &\leq \sup_{k \in \Z} \,\nrm{\varphi_k}_{L^2(\R)}  \nrm{t\mapsto 2^{jt}g(t)}_{\alpha(\R;X_j)}.
  \end{align*}
  Combined with \eqref{eq:gunderfinterpolation} this yields $\mb{y} \in \mc{A}_{\#}$ with
  \begin{align*}
    \nrm{\mb{y}}_{\mc{A}_{\#}}  \lesssim \nrm{f}_{\mc{A}_{\bullet}}.
  \end{align*}
  Since $\sum_{k \in \Z} 2^{k\theta} y_k = x$ this proves ``$\lesssim$'' of \eqref{eq:infequiv}.

   Conversely, assume the left-hand side is finite, take $\mb{y} \in \mc{A}_{\#}$ such that $\sum_{k \in \Z} 2^{k\theta} y_k = x$ and define
  \begin{equation*}
    f(t) := \sum_{k\in \Z} y_k 2^{(k-t)\theta} \ind_{[k,k+1)}(t), \qquad t \in \R.
  \end{equation*}
  Then $\int_{\R}2^{t\theta} f(t)\dd t = x$ and note that $f = \sum_{k\in \Z} \varphi_k \otimes y_k$ with
  \begin{align*}
    \varphi_k(t) &= 2^{(k-t)\theta} \ind_{[k,k+1)}(t), &&t \in \R.
  \end{align*}
  Since the $\varphi_k$'s are orthogonal and since we can compute the $\alpha(\R;X_0)$-norm of $f$ using a fixed orthonormal basis of $L^2(\R)$, this implies that
  \begin{equation*}
    \nrm{f}_{\alpha(\R;X_0)} \leq \sup_{k \in \N}\,  \nrm{\varphi_k}_{L^2(\R)} \nrm{\mb{y}}_{\alpha(\Z;X_0)} \lesssim \nrm{\mb{y}}_{\alpha(\Z;X_0)} .
  \end{equation*}
  Combined with a similar computation for the $\alpha(\R;X_1)$-norm of $t\mapsto 2^t f(t)$, this yields for
  $$g(t) =  \frac{f\hab{\ln(t)/\ln(2)}}{\ln(2)}$$
  that we have
$$     \nrm{g}_{\mc{A}_\bullet}  = \ln(2)^{-1/2} \max_{j=0,1} \nrm{t \mapsto 2^tf(t)}_{\alpha(\R;X_j)} \lesssim \nrm{\mb{y}}_{\mc{A}_{\#}}.
$$
Since $\int_0^\infty t^\theta g(t) \frac{\ddn t}{t}=x$, this proves ``$\gtrsim$'' of \eqref{eq:infequiv}.
\end{proof}

\subsection*{Complex interpolation formulation}
Next we will give a formulation of the $\alpha$-method in the spirit of the  complex interpolation method.  Denote by  the strip
$$\mathbb{S}=\cbrace{z \in \C:0 < \re(z) <1}.$$ Let $\mc{H}(\overline{\mathbb{S}})$ be the space of all bounded continuous functions $f: \overline{\mbb{S}}\to X_0+X_1$ such that
 \begin{itemize}
  \item $f$ is a holomorphic $(X_0+X_1)$-valued function on $\mathbb{S}$.
  \item $f_j(t):= f(j+it)$ is a bounded, continuous, $X_j$-valued function for $j=0,1$.
\end{itemize}
We let ${\mc{A}}_{\mathbb{S}}$ be the subspace of all $f \in \mc{H}(\overline{\mathbb{S}})$ such that $f_j \in \alpha(\R;X_j)$ and we define
\begin{equation*}
  \nrm{f}_{\mc{A}_{\mbb{S}}} := \max_{j=0,1} \,\nrm{f_j}_{\alpha(\R;X_j)}.
\end{equation*}

\begin{proposition}\label{proposition:alphacomplexmethod}
  For $x \in X_0+X_1$ we have
  \begin{equation*}
    \nrm{x}_{X_\theta} \leq (2\pi)^{-1/2} \inf\cbraceb{ \nrm{f}_{\mc{A}_{\mbb{S}}}: f \in \mc{A}_{\mbb{S}}, f(\theta) = x}.
  \end{equation*}
 Conversely, for $x \in X_0\cap X_1$ we have
  \begin{equation*}
    \nrm{x}_{X_\theta} \geq (2\pi)^{-1/2} \inf\cbraceb{ \nrm{f}_{\mc{A}_{\mbb{S}}}: f \in \mc{A}_{\mbb{S}}, f(\theta) = x}.
    \end{equation*}
\end{proposition}

\begin{proof}
  Let  $h_k \in C_c^\infty(\R)$  and $x_k \in X_0\cap  X_1$ for $k=1,\ldots,n$ and define
  \begin{equation}\label{eq:TfiniterankCc}
    T = \sum_{k=1}^n h_k \otimes x_k \in \mc{A}.
  \end{equation} Set $\ee_z(t) :=\ee^{tz}$ for $z \in \C$. Then we have for
  $f(z) := T(\ee_z)$ and $j=0,1$ that
   \begin{align*}
     f(j-2\pi it) &= \sum_{k=1}^n \int_{\R} h_k(\xi) \ee^{(j-2\pi it)\xi}\dd \xi \cdot x_k\\
     &=  \mc{F}\has{\sum_{k=1}^n \xi \mapsto h_k(\xi)\ee^{j\xi} \cdot x_k}.
   \end{align*}
   Therefore, by Example \ref{example:fourier}, we have
   \begin{align*}
     \nrm{f_j}_{\alpha(\R;X_j)} &= (2\pi)^{-1/2} \nrms{\xi \mapsto\sum_{k=1}^n  h_k(\xi)\ee^{j\xi}\cdot x_k}_{\alpha(\R;X_j)}\\&= (2\pi)^{-1/2} \nrm{T}_{\alpha(\R, \ee^{-2jt}\ddn t;X_j)},
   \end{align*}
   so we have $f \in \mc{A}_{\mbb{S}}$ with
   $(2\pi)^{-1/2} \nrm{f}_{\mc{A}_{\mbb{S}}}   = \nrm{T}_{\mc{A}}.$  Since $C_c^\infty(\R)$ is dense in $L^2(\R)\cap L^2(\R,\ee^{-2t}\ddn t)$, the collection of all $T$ as in \eqref{eq:TfiniterankCc} is dense in $\mc{A}$ by Proposition \ref{proposition:finiterankX0X1dense}. So  the inequality ``$\geq$'' follows.

  For the converse let $f \in \mc{A}_{\mbb{S}}$. Take $\varphi \in C^{\infty}_c(\R)$ and let $\tilde{\varphi}(z) := \int_{\R} \ee^{-zt}\varphi(t)\dd t$ be its Laplace transform. Then $\tilde{\varphi}$ is entire and for any $s_1<s_2$ we have an estimate
  \begin{equation*}
    \abs{\tilde{\varphi}(z)} \leq C(1+\abs{z})^{-2}, \qquad s_1 \leq \re{z} \leq s_2.
  \end{equation*}
 Therefore we can define
  \begin{equation*}
    T\varphi:=  \frac{1}{2\pi}\int_{\R}\tilde{\varphi}\hab{\tfrac{1}{2}+it}f\hab{\tfrac{1}{2}+it}\dd t
  \end{equation*}
  as a Bochner integral in $X_0+X_1$. An application of Cauchy's theorem shows that
    \begin{equation*}
    T\varphi=  \frac{1}{2\pi}\int_{\R}\tilde{\varphi}\ha{s+it}f\ha{s+it}\dd t
  \end{equation*}
  for $0<s<1$. By the dominated convergence theorem, using that $f$ is bounded and $t\mapsto \varphi(j+it) \in L^1(\R)$, we get for $j=0,1$
  \begin{equation*}
    T\varphi = \frac{1}{2\pi}\int_{\R}\tilde{\varphi}\ha{j+it}f_j\ha{t}\dd t
  \end{equation*}
  as Bochner integrals in $X_j$. Since we have
  \begin{equation*}
    \int_{\R}\abs{\tilde{\varphi}(j+it)}^2\dd t = 2\pi \int_{\R}\abs{\varphi(t)}^2\ee^{-2jt}\ddn t<\infty , \qquad j=0,1,
  \end{equation*}
  and $f_j \in \alpha(\R,X_j)$, it follows that $T$ extends to bounded operators $$T_j\colon L^2(\R,\ee^{-2jt}\ddn t) \to X_j, \qquad j=0,1.$$ Therefore $T$ can be extended to be in $\mc{A}$ and in particular we have
  \begin{equation*}
    \nrm{T}_{\mc{A}} = \max_{j=0,1} \,\nrm{T_j}_{\alpha(\R,\ee^{-2jt}\ddn t;X_j)} = \max_{j=0,1} \, (2\pi)^{-1/2}\nrm{f_j}_{\alpha(\R;X_j)} = (2\pi)^{-1/2}\nrm{f}_{\mc{A}_{\mathbb{S}}}.
  \end{equation*}
To conclude the proof of the inequality ``$\leq$'' we show that $T(\ee_\theta) = f(\theta)$. For this note that for any  $\varphi \in C^{\infty}_c(\R)$ we have
\begin{align*}
  T(\varphi \cdot \ee_\theta) &= \frac{1}{2\pi}\int_{\R} \int_{\R} \ee^{-(\theta+it)s}\varphi(s)\ee^{\theta s}  f(\theta+it) \dd s\dd t \\
  &=\frac{1}{2\pi} \int_{\R} \tilde{\varphi}(it)  f(\theta+it) \dd t.
\end{align*}
Now fix $\varphi$ such that $\varphi(0) = 1$ and for $n \in \N$ set
\begin{equation*}
  \varphi_n(t) = \varphi(nt),  \qquad t \in \R.
\end{equation*}
Then $\varphi_n\cdot\ee_\theta \to \ee_\theta$ in $L^2(\R) +L^2(\R,\ee^{-2t}\dd t)$ and therefore
\begin{equation*}
 T(\ee_\theta) = \lim_{n \to \infty}  T(\varphi_n\cdot\ee_\theta) = \lim_{n\to \infty} \frac{1}{2\pi} \int_{\R} \frac{1}{n}\tilde{\varphi}\hab{{it}/{n}}  f(\theta+it) \dd t = f(\theta),
\end{equation*}
where the last step follows from
\begin{equation*}
  \nrmb{t \mapsto \tilde{\varphi}(it)}_{L^1(\R)} = \nrmb{t \mapsto \widehat{\varphi}(t/2\pi)}_{L^1(\R)} = 2\pi \cdot \varphi(0) = 2\pi
\end{equation*}
and \cite[Theorem 2.3.8]{HNVW16}. This concludes the proof.
\end{proof}

\subsection*{A comparison of $\alpha$-interpolation with real and complex interpolation}
We conclude this section by comparing the $\alpha$-interpolation method with the actual real and complex interpolation methods. Recall that if $X_j$ has Fourier type $p_j\in [1,2]$ for $j=0,1$, i.e. if the Fourier transform is bounded from $L^{p_j}(\R;X_j)$ to $L^{p_j'}(\R;X_j)$, then by a result of Peetre \cite{Pe69} we know that we have continuous embeddings
\begin{equation}\label{eq:peetreembeddings}
  (X_0,X_1)_{\theta,p} \hookrightarrow  [X_0,X_1]_{\theta} \hookrightarrow (X_0,X_1)_{\theta,p'}
\end{equation}
where $\frac1p=\frac{1-\theta}{p_0} + \frac{\theta}{p_1}$. In particular the real method $(X_0,X_1)_{\theta,2} $ and the complex method $[X_0,X_1]_{\theta}$ are equivalent on Hilbert spaces. Using Proposition \ref{proposition:alpharealmethoddiscrete} we can prove a similar statement for the real and Gaussian interpolation method under type and cotype assumptions. Note that Fourier type $p$ implies type $p$ and cotype $p'$, but the converse only holds on Banach lattices (see \cite{GKT96}).

\begin{theorem}~\label{theorem:compareinterpolation}
  \begin{enumerate}[(i)]
  \item \label{it:interpolationcomparison1} If $X_0$ and $X_1$ have type $p_0,p_1 \in [1,2]$ and cotype $q_0,q_1 \in [2,\infty]$ respectively, then we have continuous embeddings
  \begin{equation*}
  (X_0,X_1)_{\theta,p} \hookrightarrow  (X_0,X_1)^\gamma_\theta \hookrightarrow (X_0,X_1)_{\theta,q}
\end{equation*}
where $\frac1p=\frac{1-\theta}{p_0} + \frac{\theta}{p_1}$ and $\frac1q=\frac{1-\theta}{q_0} + \frac{\theta}{q_1}$.
  \item \label{it:interpolationcomparison2} If $X_0$ and $X_1$ have type 2, then we have the continuous embedding $$[X_0,X_1]_\theta \hookrightarrow (X_0,X_1)^\gamma_\theta.$$
If $X_0$ and $X_1$ have cotype 2, then we have the continuous embedding  $$(X_0,X_1)^\gamma_\theta \hookrightarrow [X_0,X_1]_\theta.$$
  \item \label{it:interpolationcomparison4} If $X_0$ and $X_1$ are order-continuous Banach function spaces, then $$(X_0,X_1)^{\ell^2}_\theta= [X_0,X_1]_\theta$$
      isomorphically.
   \item \label{it:interpolationcomparison3} If $X_0$ and $X_1$ are Banach lattices with finite cotype, then $$(X_0,X_1)^{\gamma}_\theta = (X_0,X_1)^{\ell^2}_\theta$$
       isomorphically.
  \end{enumerate}
\end{theorem}

\begin{proof}
  For \ref{it:interpolationcomparison1} we note that we have, by the discrete version of the Lions--Peetre mean method (see \cite[Chapitre 2]{LP64}), that
    \begin{align*}
    &\nrm{x}_{(X_0,X_1)_{\theta,p}} \simeq \inf \cbraces{\max_{j=0,1}{\nrm{(2^{jk}y_k)_{k \in \Z}}_{\ell^{p_j}(\Z;X_j)}}:  \sum_{k \in \Z} 2^{k\theta} y_k = x},
  \end{align*}
  where the infimum is taken over all sequences $(y_k)_{k \in \Z}$ in $X_0\cap X_1$ such that the involved norms are finite. For a finitely non-zero sequence $(y_k)_{k \in \Z}$ in $X_0\cap X_1$ we have, using type $p_j$ of $X_j$ and Proposition \ref{proposition:gaussianradermacherl2comparison}
  \begin{equation*}
    \nrm{(2^{jk}y_k)_{k \in \Z}}_{\gamma(\Z;X_j)} \lesssim \nrm{(2^{jk}y_k)_{k \in \Z}}_{\ell^{p_j}(\Z;X_j)}, \qquad j=0,1,
  \end{equation*}
  By Proposition \ref{proposition:alphadef2} this inequality extends to any sequence in $X_0\cap X_1$ such that the right hand-side is finite.
  Therefore the first embedding in \ref{it:interpolationcomparison1} follows from Proposition \ref{proposition:alpharealmethoddiscrete}. The proof of the second embedding in \ref{it:interpolationcomparison1} is similar.

  For \ref{it:interpolationcomparison2} let $f \in \mc{H}(\overline{\mathbb{S}})$.
Then $g(z) := \ee^{z^2-\theta^2}f(z)$ has the property that $g(\theta)=f(\theta)$ and thus by Proposition \ref{proposition:alphacomplexmethod} and \eqref{eq:typeembedding}
\begin{align*}
  \nrm{f(\theta)}_{(X_0,X_1)^\gamma_\theta} &\leq \max_{j=0,1}\nrm{t\mapsto g(j+it)}_{\gamma(\R;X_j)} \\&\lesssim \max_{j=0,1}\nrm{t\mapsto g(j+it)}_{L^2(\R;X)} \\&\leq \max_{j=0,1}  \nrm{t \mapsto \ee^{(j+it)^2-\theta^2}}_{L^2(\R)}\sup_{t \in \R} \, \nrm{f(j+it)}_X,
  \\&\lesssim \sup_{t \in \R} \, \nrm{f(j+it)}_X,
\end{align*}
  from which the first embedding follows by the definition of the complex interpolation method. For the second embedding let $f \in \mc{A}_{\mbb{S}}$. Then we have by  \cite[Corollary C.2.11]{HNVW16} and \eqref{eq:cotypeembedding}
  \begin{align*}
    \nrm{f(\theta)}_{[X_0,X_1]_\theta} &\lesssim \max_{j=0,1} \nrm{t\mapsto f(j+it)}_{L^2(\R;X_j)}\\
    &\lesssim \max_{j=0,1} \nrm{t\mapsto f(j+it)}_{\gamma(\R;X_j)},
  \end{align*}
  from which the second embedding follows.

  For \ref{it:interpolationcomparison4} denote the measure space over which $X$ is defined by $(S,\mu)$. Note that $[X_0,X_1]_\theta$ is given by the Calder\'on-Lozanovskii space $X_0^{1-\theta}X_1^\theta$, which consists of all $x \in L^0(S)$ such that $\abs{x} = \abs{x_0}^{1-\theta}\abs{x_1}^{\theta}$ with $x_j \in X_j$ for $j=0,1$. The norm is given by
  \begin{equation*}
    \nrm{x}_{X_0^{1-\theta}X_1^\theta} = \inf\cbraceb{\max_{j=0,1}\nrm{x_j}_{X_j}:\abs{x} = \abs{x_0}^{1-\theta}\abs{x_1}^\theta, x_0 \in X_0,\,x_1 \in X_1},
  \end{equation*}
  see \cite{Ca64,Lo69}.

  First suppose that $0 \leq x \in X_0^{1-\theta}X_1^\theta$ factors in the form $x = \abs{x_0}^{1-\theta}\abs{x_1}^\theta$ with $x_j \in X_j$ for $j=0,1$ and $\max_{j=0,1}\,\nrm{x_j}_{X_j}\leq 2\,\nrm{x}_{[X_0,X_1]_\theta}$. We define
  \begin{equation*}
    f(z) := e^{z^2-\theta^2}\abs{x_0}^{1-z}\abs{x_1}^z, \qquad z \in \overline{\mbb{S}}.
  \end{equation*}
  Then since, for $j=0,1$, we have
  \begin{equation*}
    \has{\int_{\R}\abs{f(j+it)(s)}^2\dd t}^{\frac{1}{2}} = \has{\int_{\R}\ee^{2(j^2-t^2 - \theta^2)}\dd t}^{\frac{1}{2}} \abs{x_j(s)}, \qquad s \in S,
  \end{equation*}
  we have by Proposition \ref{proposition:l2representation} that $f_j \in \ell^2(\R;X_j)$ and therefore $f \in \mc{A}_{\mbb{S}}$. By Proposition \ref{proposition:alphacomplexmethod} this shows that
  \begin{align*}
    \nrm{x}_{(X_0,X_1)^{\ell^2}_\theta} &\leq \max_{j=0,1} \nrm{f_j}_{\ell^2(\R;X)} \lesssim \max_{j=0,1}\,\nrm{x_j}_{X_j}\leq 2\,\nrm{x}_{[X_0,X_1]_\theta}.
  \end{align*}
  For the converse direction take $f \in \mc{A}_{\mbb{S}}$. By \cite[Lemma C.2.10(2)]{HNVW16} with $X_0=X_1=\C$ and H\"older's inequality, we have for a.e. $s \in S$
  \begin{align*}
    \abs{f(\theta)(s)} &\lesssim \has{\has{\int_{\R}\abs{f(it)(s)}^2 \dd t}^{(1-\theta)/2}\cdot\has{\int_{\R}\abs{f(1+it)(s)}^2 \dd t}^{\theta/2}}.
  \end{align*}
Therefore, by Proposition \ref{proposition:l2representation}, we have
   \begin{align*}
    \nrm{f(\theta)}_{[X_0,X_1]_\theta} = \nrm{f(\theta)}_{X_0^{1-\theta}X_1^\theta}
    &\lesssim \, \max_{j=0,1} \, \nrms{\has{\int_{\R}\abs{f(j+it)}^2 \dd t}^{1/2}}_{X_j}
    \\&=   \max_{j=0,1}\, \nrm{f_j}_{\ell^2(\R;X_j)},
  \end{align*}
which implies the result by Proposition \ref{proposition:alphacomplexmethod}. Finally \ref{it:interpolationcomparison3} follows directly from Proposition \ref{proposition:compareEuclidean}.
\end{proof} 
\chapter{Sectorial operators and \texorpdfstring{$H^\infty$}{H}-calculus}\label{part:4}
On a Hilbert space  $H$, a sectorial operator $A$ has a bounded $H^\infty$-calculus if and only if it has $\BIP$. In this case $A$ has even a bounded $H^\infty$-calculus for operator-valued analytic functions which commute with the resolvent of $A$. If $A$ and $B$ are resolvent commuting sectorial operators with a bounded $H^\infty$-calculus, then $(A,B)$ has a joint $H^\infty$-calculus. Moreover if only one of the commuting operators has a bounded $H^\infty$-calculus, then still the ``sum of operators'' theorem holds, i.e.
\begin{equation*}
  \nrm{Ax}_H+\nrm{Bx}_H \leq \nrm{Ax+Bx}_H, \qquad x \in D(A)\cap D(B).
\end{equation*}
These theorems are very useful in regularity theory of partial differential operators and in particular in the theory of evolution equations. However, none of these important theorems hold in general Banach spaces without additional assumptions.

In this chapter we show that the missing ``ingredient'' in general Banach spaces is an $\alpha$-boundedness assumption, which allows one to reduce the problem via the representation in Theorem \ref{theorem:euclideanrepresentation} and its converse in Theorem \ref{theorem:Cboundedalphabounded} to the Hilbert space case.
Indeed, rather than designing an $\alpha$-bounded version of the Hilbertian proof for each of the aforementioned results, we will prove a fairly general ``transference principle'' (Theorem \ref{theorem:transference}) adapted to this task.
 Our analysis will in particular shed new light on the connection between the $\gamma$-structure and sectorial operators, which has been extensively studied (see \cite[Chapter 10]{HNVW17} and the references therein).

In the upcoming sections we will introduce the notions of (almost) $\alpha$-sectoriality, ($\alpha$)-bounded $H^\infty$-calculus and ($\alpha$)-$\BIP$ for a sectorial operator $A$. We will prove the following relations between these concepts:
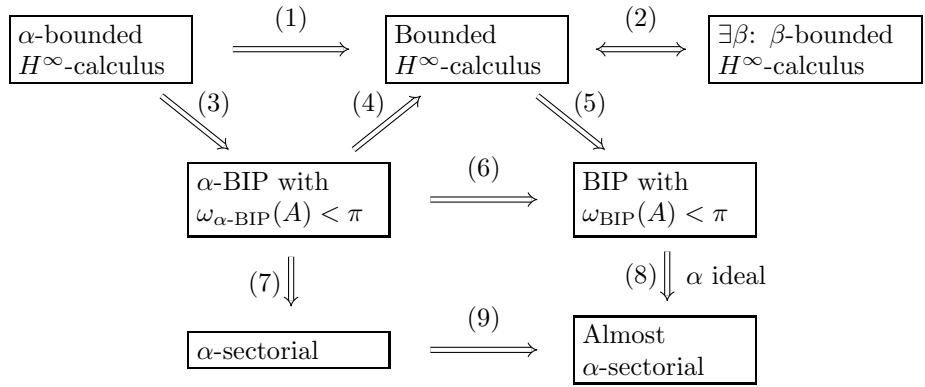
\begin{center}
  \begin{tikzpicture}
  \node (aH)[text width=2.5cm] at (0,4) {\framebox{\parbox{2.2cm}{$\alpha$-bounded \\ $H^\infty$-calculus}}};
  \node (bH)[text width=2.90cm] at (9.5,4) {\framebox{\parbox{2.6cm}{$\exists \beta$: $\beta$-bounded \\ $H^\infty$-calculus}}};
  \node (aB)[text width=2.75cm]  at (2.5,2) {\framebox{\parbox{2.45cm}{$\alpha$-$\BIP$ with\\ $\omega_{\alphaBIP}(A)<\pi$}}};
  \node (H)[text width=2.5cm]  at (5,4)
  {\framebox{\parbox{2.2cm}{Bounded \\ $H^\infty$-calculus}}};
  \node (B)[text width=2.5cm]  at (7.5,2)
  {\framebox{\parbox{2.2cm}{$\BIP$ with \\$\omega_{\BIP}(A)<\pi$}}};
  \node (a)[text width=2.75cm] at (2.5,0)
   {\framebox{\parbox{2.45cm}{$\alpha$-sectorial}}};
  \node (aa)[text width=2.5cm] at (7.5,0)
   {\framebox{\parbox{2.2cm}{Almost \\ $\alpha$-sectorial}}};
\node (ai)[text width=2.5cm] at (9,1)
   {$\alpha$ ideal};

  \draw[-implies,double equal sign distance, shorten <=2pt, shorten >=2pt] (aH) to node [shift={(0.25,0.25)}]{(3)} (aB);
  \draw[-implies,double equal sign distance, shorten <=2pt, shorten >=2pt] (aB) to node [shift={(-0.25,0.25)}] {(4)} (H);
  \draw[-implies,double equal sign distance, shorten <=2pt, shorten >=2pt] (H) to node [shift={(0.25,0.25)}] {(5)} (B);
  \draw[-implies,double equal sign distance, shorten <=4pt, shorten >=6pt] (aB) to node [left,midway] {(7)} (a);
  \draw[-implies,double equal sign distance, shorten <=2pt, shorten >=2pt] (B) to node [left,midway] {(8)} (aa);
  \draw[-implies,double equal sign distance, shorten <=10pt, shorten >=10pt] (a) to node [shift={(0,0.4)}] {(9)} (aa);
  \draw[-implies,double equal sign distance, shorten <=10pt, shorten >=10pt] (aB) to node [shift={(0,0.4)}] {(6)} (B);
  \draw[-implies,double equal sign distance, shorten <=10pt, shorten >=10pt] (aH) to node [shift={(0,0.4)}] {(1)} (H);
  \draw[implies-implies,double equal sign distance, shorten <=5pt, shorten >=5pt] (H) to node [shift={(0,0.4)}] {(2)} (bH);
\end{tikzpicture}
\end{center}
Implications (1), (3), (5), (6) and (9) are trivial. The `if and only if' statement in (2) is proven in Theorem \ref{theorem:Hinftyalphabounded}, implication (4) is one of our main results and is proven
Theorem \ref{theorem:BIPHinfty}, implication (7) follows from Theorem  \ref{theorem:alphaBIPalpha},
and implication (8) is contained in Proposition \ref{proposition:BIPalmostalpha} under the assumption that $\alpha$ is ideal.
 In the case that either $\alpha=\ell^2$ or $\alpha=\gamma$ and $X$ has Pisier's contraction property, implications (1), (3) and (4) are `if and only if' statements (see Theorem \ref{theorem:Hinftyell2bounded}). Moreover, if $X$ has the so-called triangular contraction property, then a bounded $H^\infty$-calculus implies $\gamma$-sectoriality (see \cite{KW01} or \cite[Theorem 10.3.4]{HNVW17}).

Besides these connections between the $\alpha$-versions of the boundedness of the $H^\infty$-calculus, $\BIP$ and sectoriality, we will study operator-valued and joint $H^\infty$-calculus using Euclidean structures in Section \ref{section:operatorjointcalculus}. In particular, we will use our  transference principle to deduce the boundedness of these calculi from $\alpha$-boundedness of the $H^\infty$-calculus. Moreover we will prove a sums of operators theorem.

Throughout this chapter we will keep the standing assumption that $\alpha$ is a Euclidean structure on $X$.

\section{The Dunford calculus}\label{section:dunford}
In this preparatory section we will recall the definition and some well-known properties of the so-called Dunford calculus. For a detailed treatment and proofs of the statements in this section we refer the reader to \cite[Chapter 10]{HNVW17} (see also \cite[Chapter 2]{Ha06} and \cite[Section 9]{KW04}).

If $0<\sigma<\pi$ we denote by $\Sigma_{\sigma}$ the sector in the complex plane given by
$$\Sigma_{\sigma}=\{z\in \C:\ z\neq 0,\ \abs{\arg z}<\sigma\}.$$
We let $\Gamma_{\sigma}$ be the boundary of $\Sigma_\sigma$, i.e. $\Gamma_\sigma = \cbrace{\abs{t}e^{i\sigma\,{\sgn (t)}}:t \in \R}$, which we orientate counterclockwise.
 A closed injective operator $A$ with dense domain $D(A)$  and dense range $R(A)$ is called {\it sectorial} if there exists a $0<\sigma<\pi$ so that
the spectrum of $A$, denoted by $\sigma(A)$, is contained in $\overline{\Sigma}_{\sigma}$ and the resolvent $R(\lambda,A):=(\lambda-A)^{-1}$ for $\lambda \in \C\setminus \sigma(A)=:\rho(A)$
satisfies
\begin{equation*}
\sup\cbraceb{\nrm{\lambda R(\lambda,A)}:  \lambda\in \C \setminus
\overline{\Sigma}_{\sigma}}\leq C_\sigma.
\end{equation*}
We denote by $\omega(A)$ the infimum of all $\sigma$ so that this inequality holds.
The definition of sectoriality varies in the literature. In particular, one could omit the dense domain, dense range and injectivity assumptions on $A$. However, these assumptions are not very restrictive, as one can always restrict to the part of $A$ in $\overline{D(A) \cap R(A)}$, which has dense domain and range and is injective. Moreover if $X$ is reflexive, then $A$ automatically has dense domain and we have a direct sum decomposition
\begin{equation*}
  X= N(A) \oplus \overline{R(A)}.
\end{equation*}

For $p \in [1,\infty]$ we define the Hardy space $H^p(\Sigma_\sigma)$ as the space of all holomorphic $f\colon \Sigma_\sigma \to \C$ such that
\begin{equation*}
  \nrm{f}_{H^p(\Sigma_\sigma)}:= \sup_{\abs{\theta}<\sigma}\nrm{t \mapsto f(\ee^{i\theta}t)}_{L^p(\R_+,\frac{\ddn t}{t})}
\end{equation*}
is finite. We will mostly work with the spaces $H^1(\Sigma_\sigma)$ and $H^\infty(\Sigma_\sigma)$. For $0<\sigma'<\sigma$ we have the continuous inclusion
\begin{equation}\label{eq:H1inHinfty}
  H^p(\Sigma_\sigma) \hookrightarrow H^\infty(\Sigma_{\sigma'}).
\end{equation}

\subsection*{The Dunford calculus}
Let $A$ be a sectorial operator on $X$, suppose $\omega(A)<\nu<\sigma<\pi$ and let $f\in H^1(\Sigma_{\sigma})$. Then we can define $f(A)\in \mc{L}(X)$ by the Bochner integral
\begin{equation}\label{eq:dunforddef}
f(A) =\frac{1}{2\pi i}\int_{\Gamma_{\nu}}f(z)R(z,A)\dd z.
\end{equation}
with norm estimate
\begin{equation*}
  \nrm{f(A)} \leq C_\nu\,  \nrm{f}_{H^1(\Sigma_\sigma)}<\infty.
\end{equation*}
 This is called the Dunford calculus of $A$. Let us note a few key properties of this calculus
 \begin{itemize}
   \item The definition is independent of $\nu$ by Cauchy's integral theorem.
   \item The calculus is multiplicative and thus commutative, i.e. if $f,g \in H_1(\Sigma_\sigma)$, then $f(A)g(A) = (fg)(A)$.
   \item For sectorial operators we have
   $$D(A)\cap R(A) = R(A(I+A)^{-2})$$
and using the Dunford calculus for $\varphi(z) = z(1+z)^{-2}$ we have  $\varphi(A) = A(I+A)^{-2}$. Thus for this $\varphi$ we have $\varphi(A)\colon X \to D(A)\cap R(A)$.
 \end{itemize}

Let $f \in H^1(\Sigma_\sigma)$, $x \in D(A)\cap R(A)$ and fix $y \in X$ such that $x=\varphi(A)y$ with $\varphi(z) = z(1+z)^{-2}$.
Using the multiplicativity of the calculus and Fubini's theorem we have
  \begin{equation*}
  \begin{split}
        \int_0^\infty \nrm{f(tA)x}_X \frac{\dd t}{t} &\leq C_\nu \int_0^\infty \int_{\Gamma_{\nu}}   \abs{f(t z)} \abs{\varphi(z)} \nrm{y}_X \frac{\abs{\ddn z}}{\abs{z}}\,\frac{\ddn t}{t}
    \\&\leq C_\nu \, \nrm{f}_{H^1(\Sigma_\sigma)}\nrm{\varphi}_{H^1(\Sigma_\sigma)} \nrm{y}_X <\infty,
  \end{split}
  \end{equation*}
so $t \mapsto f(tA)x$ is Bochner integrable. Since
\begin{equation*}
  \int_0^\infty f(tz)\varphi(z) \frac{\ddn t}{t} = c \varphi(z), \qquad z\in \Sigma_\sigma
\end{equation*}
with $c:= \int_0^\infty f(t)\frac{\ddn t}{t}$ by analytic continuation, we have the useful identity
\begin{equation}\label{eq:calderonrepnontrivial}
  \int_0^\infty f(tA)x\frac{\ddn t}{t} = c\,x, \qquad x \in D(A)\cap R(A).
\end{equation}

\subsection*{The extended Dunford calculus}
We extend will now extend the Dunford calculus to include functions like $e^{-w z}$ and $z^s$, for details we refer to \cite[Chapter 3]{Ha06} and \cite[Section 15]{KW04}.
 Define for $n \in \N$ the functions
\begin{align}\label{eq:phin}
\varphi_n(z) := \frac{n}{z+n}-\frac{1}{nz+1}, \qquad z \in \C \setminus (0,\infty).
\end{align}
These $\varphi_n$'s have the following properties:
\begin{enumerate}[(i)]
 \item $\varphi_n \in H^1(\Sigma_\sigma)$ for all $0<\sigma<\pi$ and $n\in \N$.
 \item By Cauchy's integral formula we have for $n \in \N$
\begin{align*}
  \varphi_n(A) &= -nR(-n,A) +n^{-1} R(-n^{-1},A)
\end{align*}
and thus by the sectoriality of $A$ we have $\sup_{n \in \N} \nrm{\varphi_n(A)} <\infty$
\item The range of each $\varphi_n(A)$ is $D(A) \cap R(A)$.
\item For all $x \in X$ we have $\varphi_n(A)x\to x$ as $n \to \infty$.
\end{enumerate}
Now suppose that  $f$ is a holomorphic function on $\Sigma_\sigma$ satisfying an estimate
\begin{equation*}
  \abs{f(z)} \leq C \abs{z}^{-\delta}(1+\abs{z})^{2\delta}, \qquad z \in \Sigma_\sigma
\end{equation*}
for some $\delta>0$ and $C>0$. For $x \in D(A^m)\cap R(A^m)$ with $m>\delta$ let $y\in X$ be such that $\varphi^m(A)y=x$ with $\varphi(z)= z(1+z)^{-2}$. Then we can define
\begin{equation*}
  f(A)x:= f\varphi^m(A)y,
\end{equation*}
which is independent of $m>\delta$. For $x \in D(A^m)\cap R(A^m)$ we have, by the multiplicativity of the Dunford calculus, that
\begin{equation*}
  f(A)x = \lim_{n \to \infty} (f\varphi_n^m)(A)x.
\end{equation*}
We extend this definition to the set  the set $D(f(A))$ of all $x \in X$ for which this limit exists. It can be shown that this defines $f(A)$ as a closed operator with dense domain for which $D(A^m)\cap R(A^m)$ is a core. Let us note a few examples of functions that are allows in the extended Dunford calculus.
\begin{itemize}
  \item If $\omega(A)<\pi/2$ we can take $f(z) = \ee^{-wz}$ for $w \in \Sigma_{\pi/2-\sigma}$. This leads to the bounded analytic semigroup  $(\ee^{-wA})_{w \in \Sigma_{\pi/2-\sigma}}$.
  \item Taking $f(z) = z^w$ we obtain the fractional powers $A^w$ for $w \in \C$. For $z,w \in \C$ we have
\begin{equation*}
  A^{z+w}x=A^zA^wx, \qquad x \in D(A^zA^w) = D(A^{z+w})\cap D(A^w)
\end{equation*}
and $A^{z+w}=A^zA^w$  if $\re z \cdot \re w>0$.
\end{itemize}
  The fractional powers $A^s$ for $s \in \R$ with $\abs{s} \, <\frac{\pi}{\omega(A)}$ are sectorial operators with $\omega(A^s) = \abs{s} \,\omega(A)$. For such $s \in \R$ we have $\varphi_s(A) = \varphi(A^s)$ with $$\varphi_s(z) = z^s(1+z^s)^{-2}, \qquad \varphi(z) = z(1+z)^{-2}$$ by the composition rule and therefore
      \begin{equation}\label{eq:Aepsilonrange}
        R(\varphi_s(A)) = R(A^s(I+A^{s})^{-2})=D(A^s)\cap R(A^s).
      \end{equation}
 Related to these fractional powers we have for $0<s<1$ and $f \in H^1(\Sigma_\sigma)$ the representation formula
\begin{equation}\label{eq:representationHinfty}
  f(A) = \frac{1}{2\pi i}\int_{\Gamma_\nu}f(z)z^{-s}A^{s}R(z,A)\dd z,
\end{equation}
This is sometimes a useful alternative to \eqref{eq:dunforddef}, since $A^sR(z,A) = \varphi_z(A)$ with $\varphi_z(w) = \frac{w^s}{z-w}$ and $\varphi_z$ is a  $H^1(\Sigma_\mu)$-function for $\omega(A)  <\mu <\abs{\arg(z)}$.

\section{(Almost) \texorpdfstring{$\alpha$}{a}-sectorial operators}
 After the preparations in the previous section, we start our investigation  by studying the boundedness of the resolvent of a sectorial operator $A$ on $X$. We say that $A$ is \emph{$\alpha$-sectorial} if there exists a $\omega(A)<\sigma< \pi$ such that
 $$\cbrace{\lambda R(\lambda,A): \lambda \in \C\setminus  \overline{\Sigma}_\sigma}$$ is $\alpha$-bounded and we let $\omega_\alpha(A)$ be the infimum of all such $\sigma$. $\alpha$-sectoriality has already been studied in the following special cases:
 \begin{itemize}
   \item  $\mc{R}$-sector\-iality, which is equivalent to maximal $L^p$-regularity (see \cite{CP01,We01}), has been studied thoroughly over the past decades (see e.g. \cite{DHP03, KKW06,KW01,KW04}). $\gamma$-sectoriality is equivalent to $\mc{R}$-sectoriality if $X$ has finite cotype by Proposition \ref{proposition:gaussianradermacherl2comparison}.
   \item $\ell^2$-sectoriality, or more generally $\ell^q$-sectoriality, has previously been studied in \cite{KU14}. We already used $\ell^2$-sectoriality in Subsection \ref{subsection:ell2vsUMD}.
 \end{itemize}

 We will also study a slightly weaker notion, analogous to the notion of almost $\mc{R}$-sectoriality and almost $\gamma$-sectoriality introduced in \cite{KKW06,KW16}.  We will say that $A$ is \emph{almost $\alpha$-sectorial} if there exists a $\omega(A)<\sigma< \pi$ such that the family $\cbrace{\lambda AR(\lambda,A)^2: \lambda \in \C\setminus  \overline{\Sigma}_\sigma}$ is $\alpha$-bounded and we let $\tilde\omega_\alpha(A)$ be the infimum of all such $\sigma$. This notion will play an important role in Section \ref{section:BIP} and Chapter \ref{part:5}.

\bigskip

 $\alpha$-sectoriality implies almost $\alpha$-sectoriality by Proposition \ref{proposition:alphaproperties}. The converse is not true, as we will show in Section \ref{section:almostalphanotalpha}. If an operator is $\alpha$-sectorial, then we do have equality of the angle of $\alpha$-sectoriality and almost $\alpha$-sectoriality.

\begin{proposition}\label{proposition:alphaalmoastalpha}
  Let $A$ be an almost $\alpha$-sectorial operator on $X$. If $$\cbrace{tR(-t,A):t>0}$$ is $\alpha$-bounded, then $A$ is $\alpha$-sectorial with $\omega_\alpha(A) = \tilde\omega_\alpha(A)$. In particular, if $A$ is $\alpha$-sectorial, then $\omega_\alpha(A) = \tilde\omega_\alpha(A)$.
\end{proposition}

\begin{proof}
  Take $\tilde\omega_\alpha(A)<\sigma<\pi$ and take $\lambda = t\ee^{i\theta}$ for some $t>0$ and $\sigma \leq \abs{\theta}<\pi$. Suppose that $\sigma\leq\theta<\pi$, then we have
  \begin{equation*}
    \lambda R(\lambda,A) + tR(-t,A) = i \int_{\theta}^\pi t\ee^{is}AR(t\ee^{is},A)^2\dd s.
  \end{equation*}
  A similar formula holds if $\sigma\leq -\theta<\pi$. Now since $\cbrace{tR(-t,A):t>0}$ is $\alpha$-bounded and $\sigma > \tilde\omega_\alpha(A)$ we know by Proposition \ref{proposition:alphaproperties} and Corollary \ref{corollary:L1mean} that
  \begin{equation*}
    \cbraces{\int_{\theta}^\pi t\ee^{is}AR(t\ee^{is},A)^2\dd s: \sigma\leq \abs{\theta} <\pi}
  \end{equation*}
  is $\alpha$-bounded. Therefore $\cbrace{\lambda R(\lambda,A): \abs{\arg (\lambda)} \geq \sigma}$ is $\alpha$-bounded, which means that $\omega_\alpha(A) \leq \sigma$. Combined with the trivial estimate $\tilde\omega_\alpha(A) \leq \omega_\alpha(A)$, the proposition follows.
\end{proof}

We can characterize almost $\alpha$-sectoriality nicely using the Dunford calculus of $A$, for which we will need the following consequence of the maximum modulus principle.
\begin{lemma}\label{lemma:maximummodulus}
  Let $0<\sigma<\pi$ and let $\Sigma$ be an open sector in $\C$ bounded by $\Gamma_\sigma$. Suppose that $f:\Sigma\cup\Gamma_\sigma \to \mc{L}(X)$ is bounded, continuous, and holomorphic on $\Sigma$. If $\cbrace{f(z):z \in \Gamma_\sigma}$ is $\alpha$-bounded, then $\cbrace{f(z):z \in \Sigma}$ is $\alpha$-bounded.
\end{lemma}

\begin{proof}
Suppose that $\mb{x} \in X^n$ and $z \in \Sigma$, then by the maximum modulus principle we have
\begin{equation*}
  \nrm{(x_1,\ldots,f(z)x_k,\ldots,x_n)}_{\alpha} \leq \sup_{w \in \Gamma_\sigma}\nrm{(x_1,\ldots,f(w)x_k,\ldots,x_n)}_\alpha .
\end{equation*}
By iteration we have for $z_1,\ldots,z_n \in \Sigma$ that
\begin{equation*}
  \nrm{(f(z_1)x_1,\ldots,f(z_n)x_n)}_{\alpha} \leq \sup_{w_1,\ldots,w_n \in \Gamma_\sigma}\nrm{(f(w_1)x_1,\ldots,f(w_n)x_n)}_\alpha,
\end{equation*}
which proves the lemma.
\end{proof}
\begin{proposition}\label{proposition:almostsectorialcharacterization}
  Let $A$ be a sectorial operator on $X$ and take $\omega(A)<\sigma<\pi$. The following conditions are equivalent:
  \begin{enumerate}[(i)]
    \item \label{it:almostsectorialcharacterization1} $A$ is almost $\alpha$-sectorial with $\tilde{\omega}_\alpha(A)<\sigma$.
    \item \label{it:almostsectorialcharacterization2a} There is a $0<\sigma'<\sigma$ such that for some (all) $0<s<1$ the set $$\cbraceb{\lambda^{s}A^{1-s}R(\lambda,A): \lambda \in \C \setminus \overline{\Sigma}_{\sigma'}}$$ is $\alpha$-bounded.
    \item \label{it:almostsectorialcharacterization3a} There is a $0<\sigma'<\sigma$ such that the set
\begin{equation*}
  \cbraceb{f(tA):t>0, \,f \in H^1(\Sigma_{\sigma'}), \,\nrm{f}_{H^1(\Sigma_{\sigma'})}\leq 1}
\end{equation*}
is $\alpha$-bounded.
    \item \label{it:almostsectorialcharacterization3b} There is a $0<\sigma'<\sigma$ such that for all $f \in H^1(\Sigma_{\sigma'})$ the set $$\cbrace{f(tA):t>0}$$ is $\alpha$-bounded.
  \end{enumerate}
\end{proposition}

\begin{proof}
  We start by proving the implication \ref{it:almostsectorialcharacterization1} $\Rightarrow$ \ref{it:almostsectorialcharacterization2a}. Fix $\tilde{\omega}_\alpha(A) <\mu<\sigma$ and $0<s<1$. For $\mu<\abs{\theta}<\sigma$ define
    \begin{align*}
    f(z) &:= (\ee^{-i\theta}z)^{1-s} (1-\ee^{-i\theta}z)^{-1}, &&z \in \Sigma_\mu
  \intertext{and set}
  F(z) &:= \int_{0}^{\abs{z}}\frac{f(t \ee^{i\arg z})}{t \ee^{i\arg z}}\dd t, &&z \in \Sigma_\mu.
   \intertext{Let $c := \int_0^\infty \frac{f(t)}{t} \dd t$ and define}
    G(z) &:= F(z) - c \frac{z}{1+z}, &&z \in \Sigma_\mu.
  \end{align*}
Since there is a $C>0$ such that
  \begin{equation*}
    \abs{f(z)} \leq C\,  \abs{z}^{1-s}\ha{1+\abs{z}}^{-1}, \qquad z \in \Sigma_\mu,
  \end{equation*}
one can show that $G \in H^1(\Sigma_\mu)$. Clearly $G'(z) = f(z)/z - c (1+z)^{-2}$, from which we can see that $zG'(z) \in H^1(\Sigma_\mu)$ as well. Since we have for $\tilde{\omega}_\alpha(A) <\nu<\mu$
  \begin{equation*}
    G(tA) = \frac{1}{2\pi i}\int_{\Gamma_\nu}G(z)R(z,tA) \dd z, \qquad t>0
  \end{equation*}
  as a Bochner integral,  we may differentiate under the integral sign by the dominated convergence theorem and obtain for $t>0$
  \begin{align*}
    tAG'(tA) &= \frac{1}{2\pi i}\int_{\Gamma_\nu}tz G'(tz)R(z,A) {\ddn z}\\
    &= t\frac{\ddn}{\ddn t} G(tA)\\
    &= \frac{1}{2\pi i}\int_{\Gamma_\nu}G(z)z tAR(z,tA)^2 \frac{\ddn z}{z}.
  \end{align*}
  Since $G \in H^1(\Sigma_\mu)$ and $tA$ is almost $\alpha$-sectorial, it follows from Corollary \ref{corollary:L1mean} that the set
  \begin{align*}
    \cbraceb{(t\ee^{i\theta})^sA^{1-s} R(t\ee^{i\theta},A):t>0} &= \cbrace{f(tA):t>0}\\ &= \cbraceb{tAG'(tA)+ctA(1+tA)^{-2}}
  \end{align*}
  is $\alpha$-bounded. Therefore by Lemma \ref{proposition:alphaproperties}\ref{it:propalha3} and Lemma \ref{lemma:maximummodulus} we deduce that $$\cbrace{\lambda^s A^{1-s}R(\lambda, A):\lambda \in \C \setminus \overline{\Sigma}_{\abs{\theta}}}$$ is $\alpha$-bounded.

  Next we show that \ref{it:almostsectorialcharacterization2a} $\Rightarrow$ \ref{it:almostsectorialcharacterization3a}. Fix $\sigma'<\nu<\sigma''<\sigma$. By \eqref{eq:representationHinfty} we have the following representation for  $f \in H^1(\Sigma_{\sigma''})$
  \begin{equation*}
    f(tA) = \frac{1}{2\pi i}\int_{\Gamma_\nu}f(tz)z^sA^{1-s}R(z,A)\frac{\dd z}{z}, \qquad t>0.
  \end{equation*}
  Since $f(t\cdot) \in H^1(\Sigma_{\sigma''})$ independent of $t>0$, it follows by Corollary \ref{corollary:L1mean} that \begin{equation*}
  \cbraceb{f(tA):t>0, f \in H^1(\Sigma_{\sigma''}), \nrm{f}_{H^1(\Sigma_{\sigma''})}\leq 1}.
\end{equation*}
 is $\alpha$-bounded. The implication \ref{it:almostsectorialcharacterization3a} $\Rightarrow$ \ref{it:almostsectorialcharacterization3b} is trivial.

For \ref{it:almostsectorialcharacterization3b} $\Rightarrow$ \ref{it:almostsectorialcharacterization1} take $f(z) = \ee^{-i\theta} z(1-\ee^{-i\theta} z)^{-2}$ with $\sigma'<\abs{\theta}<\sigma$. Then $f \in H^1(\Sigma_{\sigma'})$, so the set
  \begin{equation*}
    \cbrace{t\ee^{i\theta}AR(t\ee^{i\theta},A)^2:t>0}
  \end{equation*}
  is $\alpha$-bounded. Therefore by Lemma \ref{proposition:alphaproperties}\ref{it:propalha3} and Lemma \ref{lemma:maximummodulus} we deduce that $$\cbrace{\lambda A(1+\lambda A)^{-2}:\lambda \in \C \setminus \overline{\Sigma}_{\abs{\theta}}}$$ is $\alpha$-bounded and thus $\tilde{\omega}_\alpha(A)\leq\abs{\theta} <\sigma$.
\end{proof}

When $\omega(A)<\frac{\pi}{2}$, the sectorial operator $A$ generates an analytic semigroup. In the next proposition we connect the (almost) $\alpha$-sectoriality of $A$ to $\alpha$-boundedness of the associated semigroup.

\begin{proposition}
  Let $A$ be a sectorial operator on $X$ with $\omega(A) <\pi/2$ and take $\omega(A)<\sigma<\pi/2$. Then
  \begin{enumerate}[(i)]
    \item \label{it:semigroupalpha} $A$ is $\alpha$-sectorial with $\omega_{\alpha}(A)\leq \sigma$ if and only if $$\cbrace{\ee^{-zA}:z \in \Sigma_{\nu}}$$ is $\alpha$-bounded for all $0<\nu<\pi/2-\sigma$.
    \item \label{it:semigroupalmostalpha} $A$ is almost $\alpha$-sectorial with $\tilde{\omega}_{\alpha}(A)\leq\sigma$ if and only if $$\cbrace{zA\ee^{-zA}:z \in \Sigma_{\nu}}$$ is $\alpha$-bounded for all $0<\nu<\pi/2-\sigma$.
  \end{enumerate}
\end{proposition}

\begin{proof}
  For the `if' statement of \ref{it:semigroupalpha} take $\sigma<\nu'<\nu<\pi/2$. The $\alpha$-boundedness of $\cbrace{t\ee^{\pm i \nu}R(t\ee^{\pm i \nu},A):t>0}$ follows from the Laplace transform representation  of $R(t\ee^{\pm i \nu},A)$ in terms of the semigroups generated by $-\ee^{\pm i (\pi/2-\nu')}A$ (see \cite[Proposition G.4.1]{HNVW17}) and Corollary \ref{corollary:L1mean}. The $\alpha$-boundedness of $$\cbrace{\lambda R(\lambda ,A):\lambda \in \C\setminus \overline{\Sigma}_\nu}$$ then follows from Lemma \ref{lemma:maximummodulus}.

  For the only if take $0<\nu<\pi/2-\sigma$ and note that by \cite[Proposition 10.2.7]{HNVW17}
  \begin{equation*}
    \ee^{-zA} = z^{-1}R(z^{-1},A)+f_z(A), \qquad z \in \Sigma_{\nu},
  \end{equation*}
  where $f_z(w)= \ee^{-zw}-(1+zw)^{-1}$. Since $f_z \in H^1(\Sigma_\sigma)$, the $\alpha$-boundedness of $\ee^{-zA}$ on the boundary of $\Sigma_\nu$ follows from Proposition \ref{proposition:almostsectorialcharacterization} and the $\alpha$-boundedness in the interior of $\Sigma_\nu$ then follows from Lemma \ref{lemma:maximummodulus}.

  The proof of \ref{it:semigroupalmostalpha} is similar.  For the `if' statement one  uses an appropriate Laplace transform representation  of $R(t\ee^{\pm i \nu},A)^2$ and the `only if' statement is simpler as $zwe^{-zw}$ is an $H^1$-function.
\end{proof}

As noted in Section \ref{section:dunford}, the operator $A^s$ is sectorial as long as $\abs{s} \, <\frac{\pi}{\omega(A)}$ and in this case $\omega(A^s) = \abs{s} \,\omega(A)$.
We end this section with a similar result for (almost) $\alpha$-sectoriality.
\begin{proposition}
  Let $A$ be a sectorial operator on $X$
  \begin{enumerate}[(i)]
  \item \label{it:almostalphapower} Suppose that $A$ is almost $\alpha$-sectorial and $0<\abs{s}<\pi /\tilde{\omega}_{\alpha}(A)$. Then $A^s$ is almost $\alpha$-sectorial with $\tilde{\omega}_{\alpha}(A^s)=\abs{s}\,{\tilde{\omega}_{\alpha}}(A)$.
  \item \label{it:alphapower} Suppose that $A$ is $\alpha$-sectorial and $0<\abs{s}<\pi /{\omega_{\alpha}}(A)$. Then $A^s$ is  $\alpha$-sectorial with ${\omega}_{\alpha}(A^s)=\abs{s}\,{{\omega_{\alpha}}}(A)$.
  \end{enumerate}
\end{proposition}

\begin{proof}
Since $A$ is (almost) $\alpha$-sectorial if and only if $A^{-1}$ is (almost) $\alpha$-sectorial with equal angles by the resolvent identity, it suffices to consider the case $s>0$.
  \ref{it:almostalphapower} follows from Proposition \ref{proposition:almostsectorialcharacterization} and the fact that for $0<s<\pi /\tilde{\omega}_{\alpha}(A)$ we have $f \in H^1(\Sigma_\sigma)$ if and only if $g \in H^1(\Sigma_{\sigma s})$, where $g(z) = f(z^{s})$.

  For \ref{it:alphapower} suppose that $A$ is $\alpha$-sectorial and fix $0<s<\pi /{\omega}_{\alpha}(A)$. Define
  \begin{equation*}
    \psi(z) = \frac{z -z^s}{(1+z^s)(1-z)}, \qquad z \in \Sigma_\sigma
  \end{equation*}
  and note that $\psi \in H_1(\Sigma_\sigma)$ for $\sigma <\pi/s$.
  By \cite[Lemma 15.17]{KW04} we have
  \begin{equation*}
    -tR(-t,A^s) = -t^{1/s}R(-t^{1/s},A)+\psi(t^{-1/s}A), \qquad t>0.
  \end{equation*}
  Therefore $\cbrace{-tR(-t,A^s): t>0}$ is $\alpha$-bounded by the $\alpha$-sectoriality of $A$, Proposition \ref{proposition:almostsectorialcharacterization} and Proposition \ref{proposition:alphaproperties}. Therefore $A^s$ is $\alpha$-sectorial and by
\ref{it:almostalphapower} and Proposition \ref{proposition:alphaalmoastalpha} we have
$${\omega}_{\alpha}(A^s)=\tilde{\omega}_{\alpha}(A^s)={s}\, {{\tilde\omega_{\alpha}}}(A)={s}\,{{\omega_{\alpha}}}(A),$$
which finishes the proof.
\end{proof}

\section{\texorpdfstring{$\alpha$}{a}-bounded \texorpdfstring{$H^\infty$}{H}-calculus}\label{section:hinfty}
We now turn to the study of the $H^\infty$-calculus of a sectorial operator $A$ on $X$, which for Hilbert spaces dates back to the ground breaking paper of McIntosh \cite{Mc86}. For Banach spaces, in particular $L^p$-spaces, the central paper is by Cowling, Doust, McIntosh and Yagi \cite{CDMY96}. For examples of operators with or without a bounded $H^\infty$-calculus important in the theory of evolution equations, see e.g. \cite[Chapter 8]{Ha06}, \cite[Section 10.8]{HNVW17}, \cite[Section 14]{KW04} and the references therein.

 We will focus on situations where the $H^\infty$-calculus is $\alpha$-bounded. This has already been thoroughly studied for the $\gamma$-structure, through the notion of $\mc{R}$-boundedness, in \cite{KW01}. For a general Euclidean structure we will first use Theorem \ref{theorem:Cboundedalphabounded} to obtain an abstract result, which we afterwards make more specific under specific assumptions on $X$ and $\alpha$.

\bigskip

We will briefly recall the definition of the $H^\infty$-calculus and refer to \cite[Chapter 2]{Ha06}, \cite[Chapter 10]{HNVW17} or \cite[Section 9]{KW04} for a proper introduction. Note that some of these references take a slightly different, but equivalent approach to the $H^\infty$-calculus.

 The $H^\infty$-calculus for $A$ is an extension of the Dunford calculus to all functions in $H^\infty(\Sigma_\sigma)$ for some $\omega(A)<\sigma<\pi$. Recall that for $\varphi(z)=z(1+z)^{-2}$ we have $R(\varphi(A))=D(A)\cap R(A)$ and we can thus define for $f \in H^\infty(\Sigma_\sigma)$ and $x \in D(A)\cap R(A)$ the map
\begin{equation*}
  f(A)x:= (f\varphi)(A)y
\end{equation*}
where $y \in X$ is such that $x =\varphi(A)y$. This definition coincides with the
extended Dunford calculus and for $f \in H^1(\Sigma_\sigma)$ it coincides with the Dunford calculus. Moreover it
it is easy to check that $y=0$ implies $x=0$, so $f(A)x$ is well-defined.

By the properties of the $\varphi_n$'s as in \eqref{eq:phin} we have $\nrm{f(A)x}_X \leq C \nrm{x}_X$ for all $x \in D(A)\cap R(A)$ if and only if $\sup_{n \in \N} \nrm{(f\varphi_n)(A)} <\infty$. If one of these equivalent conditions hold we can extend $f(A)$  to a bounded operator on $X$ by density, for which we have
\begin{equation}\label{eq:calculusphin}
f(A)x = \lim_{n \to \infty} (f\varphi_n)(A)x, \qquad x \in X.
\end{equation}
 We say that $A$ has a \emph{bounded $H^\infty$-calculus} if there is a $\omega(A)<\sigma<\pi$  such that $f(A)$ extends to a bounded operator on $X$ for all $f \in H^\infty(\Sigma_\sigma)$ and we denote the infimum of all such $\sigma$ by $\omega_{H^\infty}(A)$. Just like the Dunford calculus, the $H^\infty$-calculus is multiplicative.
 We say that $A$ has an \emph{$\alpha$-bounded $H^\infty(\Sigma_\sigma)$-calculus} if the set
\begin{equation*}
  \cbraceb{f(A): f \in H^\infty(\Sigma_\sigma),\nrm{f}_{H^\infty(\Sigma_\sigma)}\leq 1}
\end{equation*}
is $\alpha$-bounded for some $\omega_{H^\infty}(A)<\sigma<\pi$. We denote the infimum of all such $\sigma$ by $\omega_{\alphaH}(A)$.

We note that the ($\alpha$-)bounded $H^\infty$-calculus of $A$ implies the
($\alpha$-)bounded $H^\infty$-calculus of $A^s$. This follows directly from the composition rule $f(A)=g(A^s)$ for $f \in H^\infty(\Sigma_\sigma)$ and $g \in H^\infty(\Sigma_{s\sigma})$ with $f(z)=g(z^s)$ (see e.g. \cite[Theorem 2.4.2]{Ha06}).
\begin{proposition}\label{proposition:hinftys}
  Let $A$ be a sectorial operator on $X$.
  \begin{enumerate}[(i)]
  \item \label{it:hinftypower} Suppose that $A$ has a bounded $H^\infty$-calculus and $0<\abs{s}<\pi /{\omega}_{H^\infty}(A)$. Then $A^s$ has a bounded $H^\infty$-calculus with ${\omega}_{H^\infty}(A^s)=\abs{s}\,{\omega}_{H^\infty}(A)$.
  \item \label{it:alphahinftypower} Suppose that $A$ has an $\alpha$-bounded $H^\infty$-calculus and $0<\abs{s}<\pi /{\omega}_{\alphaH}(A)$. Then $A^s$ has an $\alpha$-bounded $H^\infty$-calculus with ${\omega}_{\alphaH}(A^s)=\abs{s}\,{\omega}_{\alphaH}(A)$.
  \end{enumerate}
\end{proposition}

Our first major result with respect to an $\alpha$-bounded $H^\infty$-calculus follows almost immediately from the transference results in Chapter \ref{part:1}. Indeed, using Theorem \ref{theorem:Cboundedalphabounded} we can show that one can always upgrade a bounded $H^\infty$-calculus to an $\alpha$-bounded $H^\infty$-calculus.

\begin{theorem}\label{theorem:Hinftyalphabounded}
  Let $A$ be a sectorial operator with a bounded $H^\infty$-calculus. For every $\omega_{H^\infty}(A)<\sigma<\pi$ there exists a Euclidean structure $\alpha$ on $X$ such that $A$ has an $\alpha$-bounded $H^\infty$-calculus with $\omega_{\alphaH}(A) <\sigma$.
\end{theorem}

\begin{proof}
  Fix $\omega_{H^\infty}(A)<\nu<\sigma$. Note that $H^\infty(\Sigma_\nu)$ is a closed unital subalgebra of the $C^*$-algebra of bounded continuous functions on $\Sigma_\nu$ and that the algebra homomorphism $\rho\colon H^\infty(\Sigma_\nu) \to \mc{L}(X)$ given by $f \mapsto f(A)$ is bounded since $A$ has a bounded $H^\infty(\Sigma_\nu)$-calculus. Therefore the set
  \begin{equation*}
    \cbrace{f(A):f \in H^\infty(\Sigma_\nu), \nrm{f}_{H^\infty(\Sigma_\nu)} \leq 1}
  \end{equation*}
  is $C^*$-bounded. So by Theorem \ref{theorem:Cboundedalphabounded} we know that there is a Euclidean structure $\alpha$ such that $A$ has a bounded $H^\infty$-calculus with $\omega_{\alphaH}(A) \leq \nu$.
\end{proof}

\subsection*{Control over the Euclidean structure}
In general we have no control over the choice of the Euclidean structure $\alpha$ in Theorem \ref{theorem:Hinftyalphabounded}, as we will see in Example \ref{example:jointHinfty}. However, under certain geometric assumptions we can actually indicate a specific Euclidean structure such that $A$ has an $\alpha$-bounded $H^\infty$-calculus. The following proposition will play a key role in this.

\begin{proposition}\label{proposition:unconditionalalphabounded}
  Let $\alpha$ be a global ideal Euclidean structure and assume that
  \begin{equation*}
     \alpha(\N\times \N;X) =\alpha\hab{\N;\alpha(\N;X)}
   \end{equation*}
   isomorphically with constant $C_\alpha$.
      Let $(U_k)_{k \geq 1}$ and $(V_k)_{k\geq 1}$ be sequences of operators in $\mc{L}(X)$, which for all $n \in \N$ satisfy
  \begin{align*}
    \nrmb{(U_1x,\ldots,U_nx)}_\alpha &\leq M_U \nrm{x}_X, &&x \in X\\
    \nrmb{(V_1^*x^*,\ldots,V_n^*x^*)}_{\alpha^*} &\leq M_V \nrm{x^*}_{X^*}, &&x^* \in X^*
  \end{align*}
  for some constants $M_U,M_V>0$. If $\Gamma$ is an $\alpha$-bounded family of operators, then   the family
  \begin{equation*}
    \cbraces{\sum_{k=1}^n V_kT_kU_k: T_1,\ldots, T_n \in \Gamma, n \in \N}
  \end{equation*}
  is also $\alpha$-bounded with bound at most $C_\alpha^2 \,M_U M_V \, \nrm{\Gamma}_{\alpha}$.
\end{proposition}

\begin{proof}
Fix $n,m \in \N$ and define $U:X \to \alpha(\ell^2_n;X)$ by
\begin{equation*}
  Ux = (U_1x,\ldots,U_nx),\qquad  x \in X.
\end{equation*}
By assumption we have $\nrm{U} \leq M_U$. Take $\mb{x} \in X^m$. Using the global ideal property of $\alpha$ and the isomorphism between $\alpha(\ell^2_{mn};X)$ and $\alpha\hab{\ell^2_m;\alpha(\ell^2_n;X)}$, we have
\begin{align*}
  \nrmb{(U_kx_j)_{j,k=1}^{m,n}}_{\alpha(\ell^2_{mn};X)} &\leq C_\alpha\, \nrmb{(Ux_j)_{j=1}^m}_{\alpha(\ell^2_m;\alpha(\ell^2_n;X))} \leq C_\alpha\, M_U \, \nrm{\mb{x}}_{\alpha}.
\end{align*}
Analogously we have for any $\mb{x}^* \in (X^*)^m$ that
\begin{align*}
  \nrmb{(V^*_kx^*_j)_{j,k=1}^{m,n}}_{\alpha^*(\ell^2_{mn};X^*)} &\leq C_\alpha \,M_V \nrm{\mb{x}^*}_{\alpha^*}.
\end{align*}

Now let $S_j = \sum_{k=1}^n V_kT_{jk} U_k$ for $1 \leq j \leq m$ with $T_{jk} \in \Gamma\cup \cbrace{0}$. By the  duality $\alpha(\ell^2_m;X)^* = \alpha^*(\ell^2_m;X^*)$, we can pick $\mb{x}^* \in (X^*)^m$ such that $\nrm{\mb{x}^*}_{\alpha^*}=1$ and
\begin{equation*}
  \nrmb{(S_1x_1,\ldots,S_mx_m)}_{\alpha}= \sum_{j=1}^m \ip{S_jx_j,x_j^*}.
\end{equation*}
Using the $\alpha$-boundedness of $\Gamma$ we obtain
\begin{align*}
  \nrmb{(S_jx_j)_{j=1}^m}_{\alpha} &= \sum_{j=1}^m \sum_{k=1}^n \ip{T_{jk}U_kx_k,V^*_kx^*_k}\\
  &\leq \nrmb{(T_kU_kx_j)_{j,k=1}^{m,n}}_{\alpha(\ell^2_{mn};X)} \, \nrmb{(V^*_kx^*_j)_{j,k=1}^{m,n}}_{\alpha^*(\ell^2_{mn};X^*)} \\
  &\leq \nrm{\Gamma}_{\alpha}\, \nrmb{(U_kx_j)_{j,k=1}^{m,n}}_{\alpha(\ell^2_{mn};X)} \, \nrmb{(V^*_kx^*_j)_{j,k=1}^{m,n}}_{\alpha^*(\ell^2_{mn};X^*)} \\
  &\leq C^2_{\alpha} \,M_UM_V\,\nrm{\Gamma}_{\alpha}\, \nrm{\mb{x}}_{\alpha}.
\end{align*}
The theorem now follows by taking suitable $T_{jk}$.
\end{proof}

 Using the fact that the $\gamma$-structure is unconditionally stably, as shown in Proposition \ref{proposition:unconditionallystable}, we notice that
 Proposition \ref{proposition:unconditionalalphabounded} is a generalization of a similar statement for $\mc{R}$-boundedness in \cite[Theorem 3.3]{KW01}.

We will also need a special case of the following lemma, which is a generalization of \cite[Proposition H.2.3]{HNVW17}. We will use the full power of this generalization in Chapter \ref{part:6}.

\begin{lemma}\label{lemma:hadamardsum}
Fix $0<\nu<\sigma<\pi$ and let $(\lambda_k)_{k=1}^\infty$ be a sequence in $\Sigma_\nu$. Suppose that there is a $c>1$ such that $\abs{\lambda_{k+1}}\geq c \, \abs{\lambda_{k}}$ for all $k \in \N$. For $$g(z) :=\sum_{k=1}^\infty a_k f(\lambda_k z), \qquad f \in H^1(\Sigma_\sigma), \,\mbs{a} \in \ell^\infty$$
      we have $g \in H^\infty(\Sigma_{\sigma-\nu})$ with $\nrm{g}_{H^\infty(\Sigma_{\sigma-\nu})} \lesssim \,\nrm{\mbs{a}}_{\ell^\infty} \nrm{f}_{H^1(\Sigma_\sigma)}.$
\end{lemma}

\begin{proof}
We will prove the claim on the strip
  \begin{equation*}
      \mbb{S}_{{\sigma}}:= \cbrace{z \in \C:\abs{\im(z)}<{\sigma}}.
    \end{equation*}
    Define $\bar{f}\colon \mbb{S}_{{\sigma}}\to \C$ by $\bar{f}(z) = f(\ee^z)$, $\bar{g}\colon \mbb{S}_{{\sigma}-{\nu}}\to \C$ by $\bar{g}(z) = g(\ee^z)$, fix $z \in \Sigma_{\sigma-\nu}$ and set
    \begin{align*}
      \bar{\lambda}_k &:= \log(\lambda_k),\qquad
      \bar{z} := \log(z),\qquad
      \bar{c} := \log(c).
    \end{align*}
Then $\abs{\bar{\lambda}_j-\bar{\lambda}_k}>\bar{c}>0$, thus the disks
\begin{equation*}
  D_k:= \cbraces{w \in \C: \abs{w-(\bar{\lambda}_k+\bar{z})}<\frac{\bar{c}}{2}\wedge \sigma}, \qquad k \in \N
\end{equation*}
are pairwise disjoint and contained in $\mbb{S}_{{\sigma}}$. Therefore we have, by the mean value property, that
\begin{align*}
  \abs{g(z)}&\leq  \nrm{\mbs{a}}_{\ell^\infty} \sum_{k=1}^\infty \absb{\bar{f}(\bar{\lambda}_k+\bar{z})}\\
  &= \nrm{\mbs{a}}_{\ell^\infty} \sum_{k=1}^\infty  \abss{\frac{1}{\abs{D_k}}\int_{D_k} \bar{f}(x+iy) \dd x \dd y}\\
  &\leq \nrm{\mbs{a}}_{\ell^\infty}  \frac{1}{\pi (\frac{\bar{c}}{2}\wedge \sigma)^2} \int_{\mbb{S}_{\frac{\bar{c}}{2}\wedge \sigma}} \abs{\bar{f}(x+iy) }\dd x \dd y\\
  &\leq  \nrm{\mbs{a}}_{\ell^\infty} \frac{1}{\pi (\frac{\bar{c}}{2}\wedge \sigma)} \sup_{\abs{y} < \sigma} \int_{\R} \abs{\bar{f}(x+iy) }\dd x\\
  &\lesssim \nrm{\mbs{a}}_{\ell^\infty} \nrm{f}_{H^1(\Sigma_\sigma)}.
\end{align*}
This proves the norm estimate, from which $g \in H^\infty(\Sigma_{\sigma-\nu})$ follows directly.
\end{proof}

 We are now ready to prove some special cases in which we can indicate a Euclidean structure such that $A$ has an $\alpha$-bounded $H^\infty$-calculus.

\begin{theorem}\label{theorem:Hinftyell2bounded}
Let $A$ be a sectorial operator on $X$ with a bounded $H^\infty$-calculus.
\begin{enumerate}[(i)]
\item \label{it:Hinftygammabounded} If $X$ has Pisier's contraction property, then $A$ has a $\gamma$-bounded $H^\infty$-calculus with $\omega_{\gammaH}(A) = \omega_{H^\infty}(A)$.
\item \label{it:Hinftyl2bounded} If $X$ is a Banach lattice, then $A$ has a $\ell^2$-bounded $H^\infty$-calculus with $\omega_{\ellH}(A) = \omega_{H^\infty}(A)$.
\end{enumerate}
\end{theorem}

We refer to \cite{Pi78} and \cite[Section 7.5]{HNVW17} for the definition of Pisier's contraction property. Theorem \ref{theorem:Hinftyell2bounded}\ref{it:Hinftygammabounded} was already proven in \cite[Theorem 5.3]{KW01}. Here we will prove \ref{it:Hinftygammabounded} and \ref{it:Hinftyl2bounded} of Theorem \ref{theorem:Hinftyell2bounded} in a unified manner using Euclidean structures.

\begin{proof}[Proof of Theorem \ref{theorem:Hinftyell2bounded}]
Take $\alpha =\gamma$ in case \ref{it:Hinftygammabounded} and  $\alpha = \ell^g$ (which is equivalent to $\alpha = \ell^2$ by Proposition \ref{proposition:compareEuclidean}) in case \ref{it:Hinftyl2bounded}.
By Proposition \ref{proposition:alphaproperties} and \eqref{eq:calculusphin} it suffices to show that the family of operators
\begin{equation*}
  \cbraceb{f(A): f \in H^1(\Sigma_\sigma)\cap H^\infty(\Sigma_\sigma),\nrm{f}_{H^\infty(\Sigma_\sigma)}\leq 1}
\end{equation*}
is $\alpha$-bounded. For $f \in H^1(\Sigma_\sigma)\cap H^\infty(\Sigma_\sigma)$ we compute, using the representation formula \eqref{eq:representationHinfty},
\begin{equation}\label{eq:Hinftyform}
\begin{aligned}
    f(A) &= \frac{1}{2\pi i}\int_{\Gamma_\nu} z^{-1/2}f(z) A^{\frac{1}{2}}R(z,A)\dd z
  \\&= \sum_{k \in \Z} \sum_{\epsilon= \pm 1} \frac{-\epsilon}{2\pi i} \ee^{\epsilon i \nu/2} \int_1^2 f(\ee^{\epsilon i \nu} 2^kt)\varphi_{\ee^{\epsilon i \nu}}(t^{-1}2^{-k}A)\frac{\ddn t}{t}
\end{aligned}
\end{equation}
with $\varphi_z(w):=w^{1/2}/(z-w)$.

Now fix $\omega_{H^\infty}(A)<\mu<\nu$, $1\leq t \leq 2$ and $\epsilon = \pm 1$. Set $\psi := \varphi^{1/2}_{\ee^{\epsilon i \nu}}$ and note that $\psi  \in H^1(\Sigma_\mu)$ since $\ee^{\epsilon i \nu} \notin \Sigma_\mu$.
By Lemma \ref{lemma:hadamardsum} and the boundedness of the $H^\infty$-calculus of $A$, this means that there is a $C_0>0$ such that for any $n \in \N$
\begin{equation*}
  \sup_{\abs{\epsilon_k}=1}\nrms{\sum_{k=-n}^n \epsilon_k\psi(t^{-1}2^k A)} \leq C_0.
\end{equation*}
Note that $\alpha$ is unconditionally stable on $X$ by Proposition \ref{proposition:unconditionallystable}. Moreover
the family of multiplication operators $\cbrace{x \mapsto ax: \abs{a} \leq 1}$ on $X$
is $\alpha$-bounded by the right ideal property of $\alpha$. Furthermore we have
\begin{equation*}
     \alpha(\N\times \N;X) = \alpha\hab{\N;\alpha(\N;X)},
\end{equation*}
isomorphically, either by Pisier's contraction property if $\alpha = \gamma$ (see  \cite[Corollary 7.5.19]{HNVW17})                                       or since $\alpha$ is equivalent to the $\ell^2$-structure on Banach lattices if $\alpha = \ell^g$. Therefore by Proposition \ref{proposition:unconditionalalphabounded} the family of operators
\begin{equation*}
  \Gamma_{t,\epsilon} := \cbraces{\sum_{k=-n}^n a_k \psi(t^{-1}2^k A)^2: \abs{a_{-n}}, \ldots,\abs{a_n}\leq 1, n \in \N}
\end{equation*}
is $\alpha$-bounded and there is a constant $C_1>0$, independent of $t$ and $\epsilon$, such that $\nrm{\Gamma_{t,\epsilon}}_{\alpha} \leq C_1$.

Let $f_1,\ldots,f_m \in  f \in H^1(\Sigma_\sigma)\cap H^\infty(\Sigma_\sigma)$ with $\nrm{f_j}_{H^\infty(\Sigma_\sigma)} \leq 1$ and take $\mb{x} \in X^n$. Then we have, using \eqref{eq:Hinftyform} in the first step, that
\begin{align*}
  \nrmb{\hab{f_j(A)x_j}_{j=1}^m}_{\alpha} &\leq \frac{1}{\pi}\sup_{\epsilon = \pm 1} \nrms{\has{\sum_{k \in \Z}  \int_1^2 f_j(\ee^{\epsilon i \nu} 2^kt)\psi(t^{-1}2^{-k}A)^2x_j\frac{\ddn t}{t}}_{j=1}^m}_{\alpha}\\
  &\leq \sup_{\epsilon = \pm 1} \sup_{1 \leq t \leq 2} \sup_{n \in \N} \,
  \nrms{\hab{\sum_{k =-n}^n f_j(\ee^{\epsilon i \nu} 2^kt)\psi(t^{-1}2^{-k}A)^2x_j}_{j=1}^m}_{\alpha}\\
  &\leq \sup_{\epsilon = \pm 1} \sup_{1 \leq t \leq 2}\nrm{ \Gamma_{t,\epsilon}}_\alpha \, \nrm{\mb{x}}_{\alpha} \leq C_1 \nrm{\mb{x}}_{\alpha}.
\end{align*}
Hence we see that $A$ has an $\alpha$-bounded $H^\infty$-calculus with $\omega_{\alphaH}(A)\leq \sigma$.
\end{proof}

\section{Operator-valued and joint \texorpdfstring{$H^\infty$}{H}-calculus}\label{section:operatorjointcalculus}
In this section we will study of the operator-valued and joint functional calculus for sectorial operators by reducing the problem to the Hilbert space case via Euclidean structures and the general representation theorem (Theorem \ref{theorem:euclideanrepresentation}). We will also deduce a theorem on the closedness of the sum of two commuting sectorial operators.

The idea of an operator-valued $H^\infty$-calculus goes back to Albrecht, Franks and McIntosh \cite{AFM98} in Hilbert spaces. For the construction we take $0<\sigma<\pi$, $\Gamma \subseteq \mc{L}(X)$ and for $p \in [1,\infty]$ let $H^p(\Sigma_\sigma;\Gamma)$ be the set of all holomorphic functions $f\colon \Sigma_{\sigma} \to \Gamma$ such that
\begin{equation*}
  \nrm{f}_{H^p(\Sigma_\sigma;\Gamma)}:= \sup_{\abs{\theta}<\sigma}\nrm{t \mapsto f(\ee^{i\theta}t)}_{L^p(\R_+,\frac{\ddn t}{t};\mc{L}(X))}
\end{equation*}
is finite.

Take $\omega(A)<\nu<\sigma<\pi$ and let $\Gamma$ be a family of bounded operators on $X$ which commute with the resolvent of $A$. Then  we define $f(A)$ for $f\in H^1(\Sigma_{\sigma};\Gamma)$ by the contour integral
\begin{equation}\label{eq:operatorvaluedHinfty}
f(A) =\frac{1}{2\pi i}\int_{\Gamma_{\nu}}f(z)R(z,A)\dd z.
\end{equation}
We define for $f \in H^{\infty}(\Sigma_{\sigma};\Gamma)$ and $x \in D(A)\cap R(A)$
\begin{equation*}
  f(A)x := (f\varphi)(A)y
\end{equation*}
where $y \in X$ is such that $x = \varphi(A)y$ with $\varphi(z) = z(1+z)^{-2}$. As for the $H^\infty$-calculus, this coincides with \eqref{eq:operatorvaluedHinfty} for $f \in H^1(\Sigma_{\sigma};\Gamma)$.

If $\nrm{f(A)x}_X \leq C \, \nrm{x}_X$ for all $x \in D(A)\cap R(A)$ or equivalently if for $\varphi_n$ as in \eqref{eq:phin} we have $\sup_{n\in \N}\nrm{(f\varphi_n)(A)}<\infty$, we can extend $f(A)$ to a bounded operator on $X$ by density. We can then approximate $f(A)x$ as in \eqref{eq:calculusphin} for the $H^\infty$-calculus.

If there is a $\omega(A)<\sigma<\pi$ such that $f(A)$ extends to a bounded operator on $X$ for all $f \in H^\infty(\Sigma_\sigma;\Gamma)$  we say that $A$ has a \emph{bounded $H^\infty(\Gamma)$-calculus} and we denote the infimum of all such $\sigma$ by $\omega_{H^\infty(\Gamma)}(A)$. If the set
\begin{equation*}
  \cbraceb{f(A): f \in H^\infty(\Sigma_\sigma),\nrm{f}_{H^\infty(\Sigma_\sigma;\Gamma)}\leq 1}
\end{equation*}
is $\alpha$-bounded for some $\omega_{H^\infty(\Gamma)}(A)<\sigma<\pi$ we say that $A$ has a \emph{$\alpha$-bounded $H^\infty(\Gamma)$-calculus} and we denote the infimum of all such $\sigma$ by $\omega_{\alphaH(\Gamma)}(A)$.

\bigskip

We also want to study the joint functional calculus, first introduced in \cite{Al94} by Albrecht. For this let $(A,B)$ be a pair of sectorial operators which commute in the sense that $R(\lambda,A)$ and $R(\mu,B)$ commute for all $\lambda \in \rho(A)$ and $\mu \in \rho(B)$. Under these hypotheses
$$DR(A,B) := D(A)\cap D(B) \cap R(A) \cap R(B)$$ is dense in $X$. Indeed, $DR(A,B)$ is the range of $\varphi_n(A)\varphi_n(B)$ for each $n \in \N$ and $x = \lim_{n\to \infty} \varphi_n(A) \varphi_n(B)x$.

 Suppose that $\omega(A)<\nu_A<\sigma_A<\pi$ and $\omega(B)<\nu_B<\sigma_B<\pi$ and $p \in [1,\infty]$. W let $H^p(\Sigma_{\sigma_A}\times \Sigma_{\sigma_B})$ be the set of all holomorphic $f\colon \Sigma_{\sigma_A}\times \Sigma_{\sigma_B} \to \C$ such that
\begin{equation*}
  \nrm{f}_{H^p(\Sigma_{\sigma_A}\times \Sigma_{\sigma_B})}:= \sup_{\abs{\theta_A}<\sigma_A, \abs{\theta_B}<\sigma_B}\nrm{(s,t) \mapsto f(\ee^{i\theta_A}s,\ee^{i\theta_B}t)}_{L^p(\R_+\times \R_+,\frac{\ddn s}{s}\frac{\ddn t}{t})}
\end{equation*}
is finite. Then for $f \in H^1(\Sigma_{\sigma_A}\times \Sigma_{\sigma_B})$ we can define the operator $f(A,B)$ by
\begin{equation}\label{eq:jointcalculus}
  f(A,B) = -\frac{1}{4\pi^2} \int_{\Gamma_{\nu_A}}\int_{\Gamma_{\nu_B}} f(z,w)R(z,A)R(w,B) \dd w \dd z.
\end{equation}
If $f \in H^\infty(\Sigma_{\sigma_A}\times \Sigma_{\sigma_B})$ and $x \in DR(A,B)$ we define $f(A,B)x$ by
\begin{equation*}
  f(A,B)x := (f{\varphi})(A,B)y,
\end{equation*}
where $y \in X$ is such that $x = \varphi(A,B)y$ with $$\varphi(z,w) =z(1+z)^{-2}w(1+w)^{-2}, \qquad (z,w) \in \Sigma_{\sigma_A}\times \Sigma_{\sigma_B}.$$ Again this calculus is well-defined and it coincides with \eqref{eq:jointcalculus} if $f \in H^1(\Sigma_{\sigma_A}\times \Sigma_{\sigma_B})$.

As before for $f \in H^\infty(\Sigma_{\sigma_A}\times \Sigma_{\sigma_B})$ we have that $f(A,B)$ extends to a bounded operator on $X$ if $\nrm{f(A,B)x}_X\leq C\nrm{x}_X$ for all $x \in DR(A,B)$ or equivalently if \begin{equation*}
  \sup_{n \in \N} \nrm{(f\psi_n)(A,B)} < \infty,
\end{equation*}
where $\psi_n(z,w) = \varphi_n(z)\varphi_n(w)$.
We say that $(A,B)$ has a bounded joint $H^\infty$-calculus if there are $\omega(A)<\nu_A<\sigma_A<\pi$ and $\omega(B)<\nu_B<\sigma_B<\pi$ such that $f(A,B)$ extends to a bounded operator for all $f \in H^\infty(\Sigma_{\sigma_A}\times \Sigma_{\sigma_B})$ and denote the infimum over all such $(\sigma_A,\sigma_B)$ by $\omega_{H^\infty}(A,B)$.

\subsection*{A general transference principle}
Our main results in this and the next section will be based on the following transference principle, which basically tells us that the $\alpha$-bounded versions of the introduced properties of sectorial operators may be studied in the Hilbert space setting.
\begin{theorem}\label{theorem:transference}
Suppose that $A$ and $B$ are resolvent commuting $\alpha$-sectorial operators on $X$. Take $\omega_\alpha(A)<\sigma_A<\pi$ and $\omega_\alpha(B)<\sigma_B<\pi$.
\begin{itemize}
  \item Let $\Xi_A$ and $\Xi_B$ be subsets of $H^\infty(\Sigma_{\sigma_A})$ and $H^\infty(\Sigma_{\sigma_B})$ such that $\cbrace{f(A):f \in \Xi_A}$ and $\cbrace{f(B):f \in \Xi_B}$ are $\alpha$-bounded.
  \item Let $\Gamma_A$ be an $\alpha$-bounded subset of $\mc{L}(X)$, which commutes with the resolvent of $A$.
\end{itemize}
Then there is a Hilbert space $H$ and resolvent commuting operators $\widetilde{A}$ and $\widetilde{B}$ on $H$ with $\omega(\widetilde{A})< \sigma_A$ and $\omega(\widetilde{B})< \sigma_B$, so that:
\begin{enumerate}[(i)]
\item \label{it:transference1} There is a $C>0$ such that
\begin{align*}
  \sup\cbraceb{\nrmb{f(\widetilde{A})}_{\mc{L}(H)}: f \in \Xi_A}&\leq C,\\
  \sup\cbraceb{\nrmb{f(\widetilde{B})}_{\mc{L}(H)}: f \in \Xi_B}&\leq C.
\end{align*}
\item \label{it:transference2} There is a $C>0$ such that for all $f \in H^1(\Sigma_{\sigma_A}\times \Sigma_{\sigma_B})$ we have
\begin{equation*}
  \nrm{f(A,B)}_{\mc{L}(X)} \leq C \,\nrmb{f(\widetilde{A},\widetilde{B})}_{\mc{L}(H)} .
\end{equation*}
\item \label{it:transference3} There is a $C>0$ and a bounded subset $\widetilde{\Gamma}_A$ of $\mc{L}(H)$ commuting with the resolvent of $\widetilde{A}$ such that all $f \in H^1(\Sigma_{\sigma_A};\Gamma_A)$ there is a $\widetilde{f} \in H^1(\Sigma_{\sigma_A};\widetilde{\Gamma}_A)$ with
    \begin{equation*}
      \nrm{f(A)}_{\mc{L}(X)} \leq C\, \nrmb{\widetilde{f}(\widetilde{A})}_{\mc{L}(H)}
    \end{equation*}
\end{enumerate}
\end{theorem}

\begin{proof}
Fix $\omega_\alpha(A)<\nu_A<\sigma_A$ and $\omega_\alpha(B)<\nu_B<\sigma_B$. Define
\begin{align*}
  \Gamma_0 = \cbrace{\lambda R(\lambda,A):\lambda &\in \C\setminus \overline{\Sigma}_{\nu_A}} \cup \cbrace{\lambda R(\lambda,B):\lambda \in \C \setminus \overline{\Sigma}_{\nu_B}} \\
  &\cup \cbrace{f(A):f \in \Xi_A} \cup \cbrace{f(B):f \in \Xi_B} \cup \Gamma_A
\end{align*}
Let $\Gamma$ be the closure in the strong operator topology of the absolutely convex hull of
\begin{equation*}
   \cbraceb{T_1T_2T_3:T_1,T_2,T_3 \in \Gamma_0\cup \cbrace{I}},
\end{equation*}
where $I$ denotes the identity operator on $X$.
Then $\Gamma$ is $\alpha$-bounded by Proposition \ref{proposition:alphaproperties}. Denote by $\mc{L}_\Gamma(X)$ the linear span of $\Gamma$ normed by the Minkowski functional
\begin{equation*}
  \nrm{T}_{\Gamma}  = \inf\cbrace{\lambda>0:\lambda^{-1}T \in \Gamma}.
\end{equation*}
Then the map $z \mapsto R(z,A)$ is continuous from $\C \setminus \overline{\Sigma}_{\nu_A}$ to $\mc{L}_\Gamma(X)$. This follows directly from the fact that for $z,w \in \C \setminus \overline{\Sigma}_{\nu_A}$ we have
\begin{equation*}
  R(z,A)-R(w,A) = (z^{-1}-w^{-1})zwR(z,A)R(w,A) \in (z^{-1}-w^{-1})\Gamma.
\end{equation*}
The same holds for the map $z \mapsto R(z,B)$ from $\C \setminus \overline{\Sigma}_{\nu_B}$ to $\mc{L}_\Gamma(X)$. Analogously, the map $(z,w) \mapsto R(z,A)R(w,B)$ is continuous from $(\C \setminus \overline{\Sigma}_{\nu_A})\times (\C \setminus \overline{\Sigma}_{\nu_B})$ to $\mc{L}_\Gamma(X)$.

By Theorem \ref{theorem:euclideanrepresentation} there is a closed subalgebra $\mc{B}$ of $\mc{L}(H_0)$ for some Hilbert space $H_0$, a bounded algebra homomorphism $\rho:\mc{B} \to \mc{L}(X)$ and a bounded linear operator $\tau:\mc{L}_\Gamma(X) \to \mc{B}$ so that $\rho\tau(T)=T$ for all $T \in \mc{L}_{\Gamma}(X)$. Furthermore, $\tau$ extends to an algebra homomorphism on the algebra $\mc{A}$ generated by $\Gamma$.

Set $R_A(z) = \tau(R(z,A))$ for $z \in \C \setminus\overline{\Sigma}_{\nu_A}$ and $R_B(z) = \tau(R(z,B))$ for $z \in \C \setminus\overline{\Sigma}_{\nu_B}$. Then, since $\tau$ is an algebra homomorphism on $\mc{A}$, we know that $R_A$ and $R_B$ are commuting functions which obey the resolvent equations
\begin{align*}
  R_A(z) - R_A(w) &= (w-z)R_A(z)R_A(w), \qquad z,w \in \C \setminus\overline{\Sigma}_{\nu_A},\\
  R_B(z) - R_B(w) &= (w-z)R_B(z)R_B(w), \qquad z,w \in \C \setminus\overline{\Sigma}_{\nu_B}.
\end{align*}
Furthermore we have
\begin{align}\label{eq:resolventforRA}
  \sup_{\lambda \in \C \setminus\overline{\Sigma}_{\nu_A}}\nrm{\lambda R_A(\lambda)}&\leq \nrm{\tau}, \qquad
  \sup_{\lambda \in \C \setminus\overline{\Sigma}_{\nu_B}}\nrm{ \lambda R_B(\lambda)}\leq\nrm{\tau}.
\end{align}
 Finally we note that, since $z \to R(z,A)$ is continuous from $\C \setminus\overline{\Sigma}_{\nu_A}$ into $\mc{L}_\Gamma(X)$, the map $R_A$ is also continuous. A similar statement holds for $R_B$. Therefore  it follows from the resolvent equation that both $R_A$ and $R_B$ are holomorphic.

Now let $H$ be the subspace of $H_0$ of all $\xi \in H_0$ such that
\begin{equation}
\begin{aligned}
  \lim_{t \to \infty} \xi + tR_A(-t)\xi = \lim_{t \to 0} tR_A(-t)\xi&=0,\\
\lim_{t \to \infty} \xi + tR_B(-t)\xi = \lim_{t \to 0} tR_B(-t)\xi&=0.
\end{aligned}
\label{eq:densityandinjectivity}
\end{equation}
As the operators $tR_A(-t)$ and $tR_B(-t)$ are uniformly bounded for $t>0$, $H$ is closed. Moreover, since $R_A$ and $R_B$ commute, $R_A(z)(H) \subseteq H$ for $z \in \C \setminus \overline{\Sigma}_{\nu_A}$ and $R_B(z)(H) \subseteq H$ for $z \in \C \setminus \overline{\Sigma}_{\nu_B}$.

For $\varphi_n$ as in \eqref{eq:phin} we have $\varphi_n(A),\varphi_n(B) \in \mc{L}_\Gamma(X)$ with
\begin{align}\label{eq:varphiAB}
  \sup_{n\in \N}\, \nrm{\varphi_n(A)}_{\Gamma} \leq C,\qquad
  \sup_{n\in \N}\, \nrm{\varphi_n(B)}_{\Gamma} \leq C.
\end{align}
Moreover we claim that for all $n \in \N$ we have
\begin{equation}\label{eq:tauvarphiH}
  \tau(\varphi_n(A)\varphi_n(B))(H_0) \subseteq H.
\end{equation} To prove this claim it suffices to show that if $\xi = \tau(\varphi_n(A))\eta$ for some $\eta \in H_0$, then
\begin{equation}\label{eq:Hdef}
  \lim_{t \to \infty} \xi + tR_A(-t)\xi = \lim_{t \to 0} tR_A(-t)\xi=0,
\end{equation}
 and an identical statement for $B$. We have
 $$\tau(\varphi_n(A)) = n^{-1}R_A(-n^{-1}) - nR_A(-n)$$
 and therefore if $t \neq n, n^{-1}$ we have
 \begin{align*}
   tR_A(-t)\tau(\varphi_n(A)) &= tn^{-1}(t-n^{-1})^{-1}\hab{R_A(-t)-R_A(-n^{-1})} \\&\hspace{2cm} - tn(t-n)^{-1}\hab{R_A(-t)-R_A(-n)}.
 \end{align*}
 Combined with the uniform boundedness of $tR_A(-t)$ one can deduce \eqref{eq:Hdef} by taking the limits $t \to 0$ and $t \to \infty$ on each of the terms in this expression.

 We can now define the sectorial operator $\widetilde{A}$ on $H$ using $R_A$. For $\xi \in H$ we have by the resolvent equation that if $R_A(z)\xi =0$ for some $z \in \C \setminus \overline{\Sigma}_{\nu_A}$ we have $tR_A(-t)\xi =0$ for all $t>0$. Hence  $R_A(z)|_H$ is injective by \eqref{eq:densityandinjectivity}. As domain we take the range of $R_A(-1)$ and define
$$\widetilde{A}\hab{R_A(-1)\xi} := -\xi - R_A(-1)\xi,\qquad \xi \in H.$$
Then $\widetilde{A}$ is injective and has dense domain and range by \eqref{eq:densityandinjectivity} (See \cite[Section II.4.a]{EN00} and \cite[Proposition 10.1.7(3)]{HNVW17} for the details). Moreover by the resolvent equation we have $R(\cdot,\widetilde{A}) = R_A|_H$ and thus $\widetilde{A}$ is sectorial on $H$ with $\omega(\widetilde{A}) \leq \nu_A<\sigma_A$ by \eqref{eq:resolventforRA}. We make a similar definition for $\widetilde{B}$.

\bigskip

Finally, we turn to the inequalities in \ref{it:transference1}-\ref{it:transference3}.
For \ref{it:transference1} take $f \in \Xi_A$ and let $\omega(\tilde{A})<\mu_A<\sigma_A$. For any $n \in \N$ we have
\begin{equation*}
  (f\varphi_n)(A) = \frac{1}{2\pi i}\int_{\Gamma_{\mu_A}} f(z)\varphi_n(z) R(z,A)\dd z
\end{equation*}
and this integral converges as a Bochner integral in $\mc{L}_\Gamma(X)$. Therefore, using the boundedness of $\tau$, we have
\begin{equation*}
(f\varphi_n)(\widetilde{A})\xi = \tau\hab{(f \varphi_n)(A)}\xi,\qquad \xi \in H.
\end{equation*}
By the multiplicativity of the $H^\infty$-calculus, the boundedness of $\tau$ and \eqref{eq:varphiAB} we obtain that there is a $C>0$ such that for any $n \in \N$
\begin{equation*}
  \nrm{(f\varphi_n)(\widetilde{A})}_{\mc{L}(H)} \leq \nrm{\tau} \nrm{(f \varphi_n)(A)}_\Gamma \leq C.
\end{equation*}
We can prove an analogous estimate for $f(\widetilde{B})$ for any $f \in \Xi_B$ and thus \ref{it:transference1} follows.

For \ref{it:transference2} take $f \in H^1(\Sigma_{\sigma_A} \times \Sigma_{\sigma_B})$. We can express $f(A,B)$ as a Bochner integral in $\mc{L}_{\Gamma}(X)$ using \eqref{eq:jointcalculus}. By the boundedness of $\tau$ we conclude that
\begin{equation*}
  \tau\hab{f(A,B)}\xi = f(\widetilde{A},\widetilde{B})\xi, \qquad \xi \in H
\end{equation*}
Fix $n \in \N$. Using the fact that $\tau$ extends to an algebra homomorphism on the algebra generated by $\Gamma$ and \eqref{eq:tauvarphiH}, we have  for $n \in \N$
\begin{equation*}
  \tau\hab{f(A,B)\varphi_n(A)\varphi_n(B)}\eta = f(\widetilde{A},\widetilde{B})\tau\hab{\varphi_n(A)\varphi_n(B)}\eta, \qquad \eta \in H_0.
\end{equation*}
This means, by the boundedness of $\tau$ and \eqref{eq:varphiAB}, that
\begin{equation*}
  \nrm{\tau\hab{f(A,B)\varphi_n(A)\varphi_n(B)}}_{\mc{L}(H_0)} \leq C_0 \, \nrmb{f(\widetilde{A},\widetilde{B})}_{\mc{L}(H)}.
\end{equation*}
with $C_0>0$ independent of $f$ and $n$. Since $\rho$ is also bounded this implies by a limiting argument that
\begin{equation*}
  \nrmb{f(A,B)}_{\mc{L}(X)} \leq C_1 \, \nrmb{f(\widetilde{A},\widetilde{B})}_{\mc{L}(H)}
\end{equation*}
with $C_1>0$ again independent of $f$, proving \ref{it:transference2}.

Finally for \ref{it:transference3} take $f \in H^1(\Sigma_{\sigma_A};\Gamma_A)$. We can express $f(A)$ as a Bochner integral in $\mc{L}_\Gamma(X)$ using \eqref{eq:operatorvaluedHinfty}. Define $\widetilde{\Gamma}_A := \cbrace{\tau(T):T \in \Gamma_A}$ and $\widetilde{f}(z) := \tau(f(z))$. By the boundedness of $\tau$ we have
\begin{equation*}
  \tau(f(A))\xi = \widetilde{f}(\widetilde{A})\xi,\qquad \xi \in H.
\end{equation*}
Arguing analogously to the proof of \ref{it:transference2} we can now deduce that
\begin{equation*}
  \nrm{f(A)}_{\mc{L}(X)} \leq C\, \nrmb{\widetilde{f}(\widetilde{A})}_{\mc{L}(H)},
\end{equation*}
proving \ref{it:transference3}.
\end{proof}

\subsection*{The operator-valued $H^\infty$-calculus}
On a Hilbert space, any sectorial operator with a bounded $H^\infty$-calculus has a bounded operator-valued $H^\infty$-calculus, a result that is implicit in \cite{Le96} (see also \cite[Remark 6.5]{LLL98} and \cite{AFM98}). As a first application of the transference principle of Theorem \ref{theorem:transference} we obtain an analog of this statement in Banach spaces under additional $\alpha$-boundedness assumptions. Similar results using $\mc{R}$-boundedness techniques are contained in \cite{KW01,LLL98}.

\begin{theorem}\label{theorem:operatorHinfty}
Suppose that $A$ is a sectorial operator on $X$ with an $\alpha$-bounded $H^\infty$-calculus. Let $\Gamma$ be an $\alpha$-bounded subset of $\mc{L}(X)$ which commutes with the resolvent of $A$. Then $A$ has a bounded $H^\infty(\Gamma)$-calculus with $\omega_{H^\infty(\Gamma)}(A) \leq \omega_{\alphaH}(A)$.
\end{theorem}

\begin{proof}
  Fix $\omega_{\alphaH}(A)<\sigma<\pi$. We apply the transference principle of Theorem \ref{theorem:transference} to the sectorial operator $A$ with $\Xi_A = H^\infty(\Sigma_\sigma)$ and $\Gamma_A = \Gamma$. Then there is a sectorial operator $\widetilde{A}$ on a Hilbert space $H$ and a uniformly bounded family  of operators $\widetilde{\Gamma}$ on $H$ such that for all $f \in H^1(\Sigma_\sigma;\Gamma)$ there is a $\widetilde{f} \in H^1(\Sigma_\sigma;\widetilde{\Gamma})$ with
  \begin{equation*}
    \nrm{f(A)}_{\mc{L}(X)} \leq C\, \nrmb{\widetilde{f}(\widetilde{A})}_{\mc{L}(H)}.
  \end{equation*}
  As stated before the theorem,  any sectorial operator on a Hilbert space with a bounded $H^\infty$-calculus has a bounded operator-valued $H^\infty$-calculus. So for any $f \in H^\infty(\Sigma_\sigma;\Gamma)$ we have
  \begin{equation*}
 \sup_{n\in \N}\nrm{(f\varphi_n)(A)} \leq C\, \sup\cbraceb{\nrmb{\widetilde{f}(\widetilde{A})}:\widetilde{f} \in H^1(\Sigma_\sigma;\widetilde{\Gamma})} \leq C,
\end{equation*}
which shows that $A$ has a bounded $H^\infty(\Gamma)$-calculus with $\omega_{H^\infty(\Gamma)}(A) \leq \sigma$.
\end{proof}

In Theorem \ref{theorem:operatorHinfty} we cannot avoid the $\alpha$-boundedness assumptions. In \cite{LLL98} it is shown that if the conclusion of Theorem \ref{theorem:operatorHinfty} holds for all sectorial operators with a bounded $H^\infty$-calculus and for all bounded and resolvent commuting families $\Gamma \subseteq \mc{L}(X)$, then $X$ is isomorphic to a Hilbert space.

We can combine Theorem \ref{theorem:Hinftyell2bounded} and Theorem \ref{theorem:operatorHinfty} to improve
Theorem \ref{theorem:operatorHinfty}
 in case the Euclidean structure $\alpha$ is either the $\gamma$- or the $\ell^2$-structure. A similar result using $\mc{R}$-boundedness can be found in \cite[Theorem 4.4]{KW01}.

\begin{corollary} \label{corollary:operatorHinftygammaell2}
Let $A$ be a sectorial operator on $X$ with a bounded $H^\infty$-calculus and let $\Gamma$ be a subset of $\mc{L}(X)$ which commutes with the resolvent of $A$.
\begin{enumerate}[(i)]
\item \label{it:operatorHinftygammabounded} If $X$ has Pisier's contraction property and $\Gamma$ is $\gamma$-bounded, then $A$ has a $\gamma$-bounded $H^\infty(\Gamma)$-calculus with $\omega_{\gammaH(\Gamma)}(A) \leq \omega_{H^\infty}(A)$.
\item \label{it:operatorHinftyl2bounded} If $X$ is a Banach lattice and $\Gamma$ is $\ell^2$-bounded, then $A$ has an $\ell^2$-bounded $H^\infty(\Gamma)$-calculus with $\omega_{\ellH(\Gamma)}(A) \leq \omega_{H^\infty}(A)$.
\end{enumerate}
\end{corollary}

\begin{proof}
  Either take $\alpha = \gamma$ or $\alpha = \ell^2$. By Theorem \ref{theorem:Hinftyell2bounded} we know that $A$ has an $\alpha$-bounded $H^\infty$-calculus with $\omega_{\alphaH}(A) = \omega_{H^\infty}(A)$. Then by Theorem \ref{theorem:operatorHinfty} we know that $A$ has a bounded $H^\infty(\Gamma)$-calculus with $\omega_{H^\infty(\Gamma)}(A) \leq \omega_{H^\infty}(A)$. Finally, by a repetition of the proof of Theorem \ref{theorem:Hinftyell2bounded} using the operator family
  \begin{equation*}
  \Gamma_{t,\epsilon} := \cbraces{\sum_{k=-n}^n T_k \psi(t^{-1}2^k A)^2: T_{-n},\ldots,T_n \in \Gamma, n \in \N}
\end{equation*}
we can prove that $A$ has a $\alpha$-bounded $H^\infty(\Gamma)$-calculus with $\omega_{\alphaH(\Gamma)}(A) \leq \omega_{H^\infty}(A)$.
\end{proof}

\subsection*{The joint $H^\infty$-calculus}
On a Hilbert space any pair of resolvent commuting sectorial operators with bounded $H^\infty$-calculi has a bounded joint $H^\infty$-calculus (see \cite[Corollary 4.2]{AFM98}). Moreover the converse of this statement is trivial. Again using the transference principle of Theorem \ref{theorem:transference} we obtain a characterization of the boundedness of the joint $H^\infty$-calculus of a pair of commuting sectorial operators $(A,B)$ on a Banach space $X$ in terms of the $\alpha$-boundedness of the $H^\infty$-calculi of $A$ and $B$.

\begin{theorem}\label{theorem:jointHinfty}
  Suppose that $A$ and $B$ are resolvent commuting sectorial operators on $X$.
  \begin{enumerate}[(i)]
  \item If $A$ and $B$ have an $\alpha$-bounded $H^\infty$-calculus, then $(A,B)$ has a bounded joint $H^\infty$-calculus with $$\omega_{H^\infty}(A,B) \leq \hab{\omega_{\alphaH}(A),\omega_{\alphaH}(B)}.$$
  \item If $(A,B)$ has a bounded joint $H^\infty$-calculus, then for any $$\omega_{H^\infty}(A,B)<(\sigma_A,\sigma_B)<(\pi,\pi)$$ there is a Euclidean structure $\alpha$ such that $A$ and $B$ have $\alpha$-bounded $H^\infty$-calculi with $\omega_{\alphaH}(A) \leq \sigma_A$ and $\omega_{\alphaH}(B) \leq \sigma_B$.
  \end{enumerate}
\end{theorem}

\begin{proof}
  The first part is a typical application of Theorem \ref{theorem:transference}. Let $\omega_{\alphaH}(A)<\sigma_A<\pi$ and $\omega_{\alphaH}(B)<\sigma_B<\pi$. Using Theorem \ref{theorem:transference} we can find a pair of resolvent commuting sectorial operators $\widetilde{A}$ and $\widetilde{B}$ on a Hilbert space $H$ such that $\omega_{H^\infty}(\widetilde{A})<\sigma_A$, $\omega_{H^\infty}(\widetilde{B})<\sigma_B$ and such that $$\nrm{f(A,B)}_{\mc{L}(X)} \leq C\,\nrm{f(\widetilde{A},\widetilde{B})}_{\mc{L}(H)}$$ for all $f \in H^1(\Sigma_{\sigma_A} \times \Sigma_{\sigma_B})$. On a Hilbert space any pair of sectorial operators with a bounded $H^\infty$-calculus has a bounded joint $H^\infty$-calculus (see \cite[Corollary 4.2]{AFM98}), so by approximation this proves the first part.

  For the second part note that  $H^\infty(\Sigma_{\sigma_A}\times \Sigma_{\sigma_B})$ is a closed unital subalgebra of the $C^*$-algebra of bounded continuous functions on $\Sigma_{\sigma_A}\times \Sigma_{\sigma_B}$ and that the algebra homomorphism $\rho\colon H^\infty(\Sigma_{\sigma_A}\times \Sigma_{\sigma_B}) \to \mc{L}(X)$ given by $f \mapsto f(A,B)$ is bounded since $(A,B)$ has a bounded $H^\infty$-calculus. Therefore the set
  \begin{equation*}
    \cbrace{f(A,B):f \in H^\infty(\Sigma_{\sigma_A}\times \Sigma_{\sigma_B}), \nrm{f}_{H^\infty(\Sigma_{\sigma_A}\times \Sigma_{\sigma_B})} \leq 1}
  \end{equation*}
  is $C^*$-bounded, from which the claim follows by Theorem \ref{theorem:Cboundedalphabounded} and restricting to functions $f:\Sigma_{\sigma_A}\times \Sigma_{\sigma_B} \to \C$ that are constant in one of the variables.
\end{proof}

As in the operator-valued $H^\infty$-calculus case, in Theorem \ref{theorem:jointHinfty} we cannot omit the assumption of an $\alpha$-bounded $H^\infty$-calculus. We illustrate this with an example, see also \cite{KW01,LLL98}.
\begin{example}\label{example:jointHinfty}
  Consider the Schatten class $\mc{S}^p$ for $p \in (1,\infty)$. We represent a member $x \in \mc{S}^P$ by an infinite matrix, i.e. $x = (x_{jk})_{j,k=1}^\infty$, and define $Ax = (2^jx_{jk})_{j,k=1}^\infty$ with as domain the set of all $x \in \mc{S}^p$ such that $Ax \in \mc{S}^p$. Analogously we define $Bx = (2^kx_{jk})_{j,k=1}^\infty$. Then $A$ and $B$ are both sectorial operators with bounded $H^\infty$-calculus and $\omega_{H^\infty}(A) = \omega_{H^\infty}(B) =0$. However $(A,B)$ do not have a bounded joint $H^\infty$-calculus for any choice of angles, unless $p=2$ (see \cite[Theorem 3.9]{LLL98}).
\end{example}

\begin{remark}
  In particular Example \ref{example:jointHinfty} shows that the Euclidean structure $\alpha$ given by Theorem \ref{theorem:Hinftyalphabounded} for $A$ must fail the ideal property. Indeed, let $\alpha$ be an ideal Euclidean structure such that $A$ has an $\alpha$-bounded $H^\infty$-calculus. Then for any $f \in H^\infty(\Sigma_\sigma)$ we have $f(B)=Tf(A)T$, where $T$ is the transpose operator on $\mc{S}^p$. So $B$ has an $\alpha$-bounded $H^\infty$-calculus as well by the ideal property of $\alpha$ and therefore Theorem \ref{theorem:jointHinfty} would imply that $(A,B)$ has a bounded joint $H^\infty$-calculus, a contradiction with Example \ref{example:jointHinfty}.
\end{remark}

We can combine Theorem \ref{theorem:Hinftyell2bounded} and Theorem \ref{theorem:jointHinfty} to recover the following result of Lancien, Lancien and Le Merdy \cite{LLL98} (see also \cite{AFM98, FM98, KW01}).

\begin{corollary}
  Suppose that $X$ has Pisier's contraction property or is a Banach lattice. Let $A$ and $B$ be resolvent commuting sectorial operators on $X$ with a bounded $H^\infty$-calculus. Then $(A,B)$ has a bounded joint $H^\infty$-calculus with
  $\omega_{H^\infty}(A,B) = \hab{\omega_{H^\infty}(A),\omega_{H^\infty}(B)}.$
\end{corollary}

\subsection*{The sum of closed operators}
We end this section with a sum of closed operators theorem. It is well known that an operator-valued $H^\infty$-calculus implies theorems on the closedness of the sum of commuting operators, see e.g. \cite{AFM98,KW01, LLL98} and \cite[Theorem 12.13]{KW04}. However, here we prefer to employ the transference principle of Theorem \ref{theorem:transference} once more.

\begin{theorem}\label{theorem:sumsofoperators}
Let $A$ and $B$ be resolvent commuting sectorial operators on $X$. Suppose that $A$ has an $\alpha$-bounded $H^\infty$-calculus and $B$ is $\alpha$-sectorial with $\omega_{\alphaH}(A) +\omega_{\alpha}(B)<\pi$. Then $A+B$ is closed on $D(A)\cap D(B)$ and
\begin{equation*}
  \nrm{Ax}_X+\nrm{Bx}_X \lesssim \nrm{Ax+Bx}_X,\qquad x \in D(A)\cap D(B).
\end{equation*}
Moreover $A+B$ is sectorial with $\omega(A+B) \leq \max\cbrace{\omega_{\alphaH}(A), \omega_\alpha(B)}$.
\end{theorem}

\begin{proof}
Take $\sigma_A>\omega_{H^\infty}(A)$ and $\sigma_B>\omega_{\alpha}(B)$ with $\sigma_A+\sigma_B<\pi$. Choose $\Xi_A = H^\infty(\Sigma_{\sigma_A})$ and apply Theorem \ref{theorem:transference} to find a Hilbert space $H$ and resolvent commuting sectorial operators $\widetilde{A}$, $\widetilde{B}$ on $H$ with $\omega_{H^\infty}(\widetilde{A})<\sigma_A$ and $\omega(\widetilde{B})<\sigma_B$. By the sum of operators theorem on Hilbert spaces due to Dore and Venni \cite[Remark 2.11]{DV87} (see also \cite{AFM98}) we deduce that $\widetilde{A}+\widetilde{B}$ is a sectorial operator on $D(\widetilde{A})\cap D(\widetilde{B})$ with
\begin{equation}\label{eq:A+BleqA+B}
  \nrmb{\widetilde{A}\xi}_H+\nrmb{\widetilde{B}\xi}_H \lesssim \nrmb{\widetilde{A}\xi+\widetilde{B}\xi}_H,\qquad \xi \in D(\widetilde{A})\cap D(\widetilde{B}).
\end{equation}
Using the joint functional calculus we wish to transfer this inequality to $A$ and $B$. For this note that the function $f(z,w) = z(z+w)^{-1}$ belongs to $H^\infty(\Sigma_{\sigma_A}\times \Sigma_{\sigma_B})$ since $\sigma_A+\sigma_B<\pi$.  Set
$g_n(z)=(z+w)\varphi_n(z)^2\varphi_n(w)^2$ with $\varphi_n$ as in \eqref{eq:phin}. Then $g \in H^1(\Sigma_{\sigma_A}\times \Sigma_{\sigma_B})$ and by the resolvent identity we have
\begin{equation*}
  g_n(\widetilde{A},\widetilde{B}) = (\widetilde{A}+\widetilde{B})\varphi_n(\widetilde{A})^2\varphi_n(\widetilde{B})^2.
\end{equation*}
Therefore by the multiplicativity of the joint $H^\infty$-calculus and \eqref{eq:A+BleqA+B} we have for $\eta \in R(\widetilde{A}+\widetilde{B})$ and $\xi \in D(\widetilde{A})\cap D(\widetilde{B})$ with $\eta = \widetilde{A}\xi+\widetilde{B}\xi$
\begin{equation*}
\begin{aligned}
  \nrm{f(\widetilde{A},\widetilde{B})\varphi_n(\widetilde{A})^2\varphi_n(\widetilde{B})^2\eta}_H &= \nrm{f(\widetilde{A},\widetilde{B}) (\widetilde{A}\xi+\widetilde{B}\xi)\varphi_n(\widetilde{A})^2\varphi_n(\widetilde{B})^2\xi}_H\\
  &= \nrm{f(\widetilde{A},\widetilde{B})g_n(\widetilde{A},\widetilde{B})\xi}_H \\
  &= \nrm{\widetilde{A} \varphi_n(\widetilde{A})^2\varphi_n(\widetilde{B})^2\xi}_H \\
  &\lesssim \nrm{\varphi_n(\widetilde{A})^2\varphi_n(\widetilde{B})^2\eta}_H
  \end{aligned}
\end{equation*}
Taking the limit $n \to \infty $ and using the density of $R(\widetilde{A}+\widetilde{B})$ in $H$, we see that $f(\widetilde{A},\widetilde{B})$ is bounded on $H$. By part \ref{it:transference2} of Theorem \ref{theorem:transference} we therefore obtain
\begin{equation*}
   \sup_{n \in \N} \nrm{(f\psi_n)(A,B)x}  \leq  \sup_{n \in \N}  \nrm{(f\psi_n)(\widetilde{A},\widetilde{B})} \lesssim \nrm{f(\widetilde{A},\widetilde{B})}
\end{equation*}
with $\psi_n(z,w) =\varphi_n(z)\varphi_n(w)$. It follows that $f(A,B)$ extends to a bounded operator on $X$. Therefore we have for all $x \in D(A)\cap D(B)$
\begin{align*}
  \nrm{A \varphi_n(A)^2\varphi_n(B)^2x}_X &= \nrm{f(A,B)g_n(A,B)x}_X\\
  &\lesssim\nrm{(A+B)\varphi_n(A)^2\varphi_n(B)^2x}_X
\end{align*}
and taking the limit $n \to \infty$ this implies
\begin{equation*}
    \nrm{Ax}_X + \nrm{Bx}_X  \leq 2\nrm{Ax}_X + \nrm{Ax+Bx}_X \lesssim  \nrm{Ax+Bx}_X.
  \end{equation*}
The closedness of $A+B$ now follows from the closedness of $A$ and $B$. The sectoriality of $A+B$ is proven for example in \cite[Theorem 3.1]{AFM98}.
\end{proof}

As we have seen before in this section, Theorem \ref{theorem:sumsofoperators} can be strengthened if the Euclidean structure $\alpha$ is either the $\gamma$- or the $\ell^2$-structure. For a similar statement using $\mc{R}$-sectoriality we refer to \cite{KW01}.

\begin{corollary}
  Let $A$ and $B$ be resolvent commuting sectorial operators on $X$. Suppose that $A$ has a bounded $H^\infty$-calculus and $B$ is $\alpha$-sectorial with $\omega_{H^\infty}(A) +\omega_{\alpha}(B)<\pi$. Assume one of the following conditions:
  \begin{enumerate}[(i)]
\item $X$ has Pisier's contraction property and $\alpha = \gamma$.
\item $X$ is a Banach lattice and $\alpha = \ell^2$.
\end{enumerate}
 Then $A+B$ is closed on the domain $D(A)\cap D(B)$ and
\begin{equation*}
  \nrm{Ax}_X+\nrm{Bx}_X \leq C\,\nrm{Ax+Bx}_X,\qquad x \in D(A)\cap D(B).
\end{equation*}
Moreover $A+B$ is $\alpha$-sectorial with $\omega_\alpha(A+B) \leq \max\cbrace{\omega_{H^\infty}(A), \omega_{\alpha}(B)}$.
\end{corollary}

\begin{proof}
The first part of the statement follows directly from Theorem \ref{theorem:Hinftyell2bounded} and Theorem  \ref{theorem:sumsofoperators}. It remains to prove the $\alpha$-sectoriality of $A+B$. Fix
 $\omega_{\alphaH}(A)<\sigma_A<\pi$ and $\omega_{\alpha}(B)<\sigma_B<\pi$ such that $\sigma_A+\sigma_B<\pi$ and take $\max\cbrace{\sigma_A,\sigma_B}<\nu<\pi$. Let $\lambda \in \C \setminus \overline{\Sigma}_\nu$ and define
  \begin{equation*}
    g_\lambda(z) := \frac{\lambda}{\lambda-z}\hab{(\lambda-z)R(\lambda -z,B)}, \qquad z \in \Sigma_{\sigma_A}
  \end{equation*}
  Then $g_\lambda \in H^\infty(\Sigma_{\sigma_A};\Gamma)$ with
$$\Gamma := \cbrace{\lambda R(\lambda,B):\lambda \in \C \setminus \overline{\Sigma}_{\sigma_B}}.$$
Note that $\frac{\lambda}{\lambda-z}$ is uniformly bounded for $\lambda \in \C \setminus \overline{\Sigma}_\nu$ and $z \in \Sigma_{\sigma_A}$. Therefore since $\Gamma$ is $\alpha$-bounded it follows from Corollary \ref{corollary:operatorHinftygammaell2} that the family $$\cbrace{g_{\lambda}(A):\lambda \in \C \setminus \overline{\Sigma}_\mu}$$ is $\alpha$-bounded. By an approximation argument similar to the one presented in Theorem \ref{theorem:sumsofoperators} we have $g_{\lambda}(A) = R(\lambda,A+B)$. Therefore it follows that $A+B$ is $\alpha$-sectorial of angle $\nu$.
\end{proof}

\section{\texorpdfstring{$\alpha$}{a}-bounded imaginary powers}\label{section:BIP}
Before the development of the $H^\infty$-calculus for a sectorial operator $A$, the notion of bounded imaginary powers, i.e. $A^{is}$ for $s \in \R$, played an important role in the study of sectorial operators. We refer to  \cite{BdL92,DV87,Mo97,PS90} for a few breakthrough results in using bounded imaginary powers.

Defined by the extended Dunford calculus, $A^{is}$ for $s \in \R$ is a possibly unbounded operator whose domain includes $D(A) \cap R(A)$. $A$ is said to have \emph{bounded imaginary powers}, denoted by $\BIP$, if $A^{is}$ is bounded for all $s \in \R$. In this case $(A^{is})_{s \in \R}$ is a $C_0$-group and by semigroup theory we then know that there are $C,\theta>0$ such that $\nrm{A^{is}} \leq C\ee^{\theta\abs{s}}$ for $s \in \R$. Thus we can define
\begin{equation*}
  \omega_{\BIP}(A) := \inf\cbrace{\theta:\nrm{A^{is}} \leq C\ee^{\theta\abs{s}}, s \in \R}.
\end{equation*}
It is a celebrated result of Pr\"uss and Sohr \cite{PS90} that $\omega_{\BIP}(A) \geq \omega(A)$ and it is possible to have $\omega_{\BIP}(A) \geq \pi$, see \cite[Corollary 5.3]{Ha03}. If $A$ has a bounded $H^\infty$-calculus, then $A$ has $\BIP$ and since
\begin{equation}\label{eq:HinftyBIP}
\sup_{z\in \Sigma_\sigma} z^{it} \leq \ee^{\sigma t}, \qquad t \in \R
\end{equation}
we have $\omega_{\BIP}(A) \leq \omega_{H^\infty}(A)<\pi$. Furthermore Cowling, Doust, McIntosh and Yagi \cite{CDMY96} showed that in this case  $\omega_{\BIP}(A) = \omega_{H^\infty}(A)$.
Conversely if $X$ is a Hilbert space and $A$ has $\BIP$ with $\omega_{\BIP}(A)<\pi$, then $A$ has a bounded $H^\infty$-calculus. However, the example given in \cite{CDMY96} shows that even for $X=L^p$ with $p \neq 2$ this result fails, i.e. it is possible for a sectorial operator $A$ on $X$ without a bounded $H^\infty$-calculus to have $\omega_{\BIP}(A)<\pi$.

We will try to understand this from the point of view of Euclidean structures. For this we say that a sectorial operator $A$  has $\alpha$-$\BIP$ if the family $\cbrace{\ee^{-\theta\abs{s}}A^{is}:s \in \R}$ is $\alpha$-bounded for some $\theta \geq 0$. In this case we set
\begin{equation*}
  \omega_{\alphaBIP}(A) = \inf\cbraceb{\theta:(\ee^{-\theta\abs{s}}A^{is})_{s \in \R} \text{ is $\alpha$-bounded}}.
\end{equation*}

Since $(A^s)^{it} = A^{ist}$ for $\abs{s} \leq \pi/\omega(A)$ and $t \in \R$ (see \cite[Theorem 15.16]{KW04}), we know that $A^s$ has ($\alpha$-)$\BIP$ if $A$ has ($\alpha$-)$\BIP$ with
\begin{align*}
  \omega_{\BIP}(A^s) &= \abs{s}\,\omega_{\BIP}(A)\\ \omega_{\alphaBIP}(A^s) &= \abs{s}\,\omega_{\alphaBIP}(A).
\end{align*}
Moreover $\alpha$-$\BIP$ implies $\BIP$ with $\omega_{\BIP}(A) \leq \omega_{\alphaBIP}(A)$. If $\alpha$ is ideal, we have equality of angles.

\begin{proposition}\label{proposition:BIPideal}
  Let $\alpha$ be an ideal Euclidean structure and let $A$ be a sectorial operator on $X$. Suppose that $A$ has $\alpha$-$\BIP$, then $\omega_{\BIP}(A) = \omega_{\alphaBIP}(A)$.
\end{proposition}

\begin{proof}
  Since $\alpha$ is ideal, we have the estimate $\nrm{A^{in}}_{\alpha} \leq C \, \nrm{A^{in}}$ for $n \in \Z$. Take $\theta > \omega_{\BIP}(A)$, then by Proposition \ref{proposition:alphaproperties}\ref{it:propalha3} we know that $$\cbrace{\ee^{\theta\abs{n}}A^{in}:n \in \Z}$$ is $\alpha$-bounded. Combined with the fact that $\cbrace{A^{is}:s \in [-1,1]}$ is $\alpha$-bounded we obtain by Proposition \ref{proposition:alphaproperties}\ref{it:propalha1} that $\omega_{\alphaBIP}(A)<\theta$
\end{proof}

\subsection*{The connection between ($\alpha$)-$\BIP$ and (almost) $\alpha$-sectoriality}
We have an integral representation of $\lambda^sA^s(1+\lambda A)^{-1}$ in terms of the imaginary powers of $A$, which will allow us to connect $\BIP$ to almost $\alpha$-sectoriality. The representation is based on the Mellin transform.

\begin{lemma}\label{lemma:mellin}
  Let $A$ be a sectorial operator with $\BIP$ with $\omega_{\BIP}(A)<\pi$. Then we have for $\lambda \in \C$ with $\abs{\arg(\lambda)} + \omega_{\BIP}(A)<\pi$ that
  \begin{align*}
\lambda^sA^s(1+\lambda A)^{-1} &= \frac{1}{2}\int_{\R}\frac{1}{\sin\hab{\pi(s-it)}}\lambda^{it}A^{it} \dd t ,\qquad 0<s<1.
  \end{align*}
\end{lemma}

\begin{proof}
Recall the following Mellin transform (see e.g. \cite{Ti86})
\begin{equation}\label{eq:mellin1}
  \int_0^\infty \frac{z^{s-1}}{1+z}\dd z = \frac{\pi}{\sin(\pi s)}, \qquad 0<\re(s)<1.
\end{equation}
Using the substitution $z=e^{2\pi \xi}$ this becomes a Fourier transform:
\begin{equation*}
  2 \int_{\R}\frac{e^{2\pi s\xi}}{1+e^{2\pi \xi}}\ee^{-2\pi it\xi}\dd \xi = \frac{1}{\sin(\pi (s-it))}, \qquad 0<s<1,\, t \in \R.
\end{equation*}
Thus by the Fourier inversion theorem we have
\begin{equation*}
   \int_{\R}  \frac{ \ee^{2\pi i t \xi}}{\sin(\pi (s-it))}\dd t= \frac{2 e^{2\pi s\xi}}{1+e^{2 \pi \xi}}, \qquad 0<s<1,\, \xi \in \R.
\end{equation*}
Therefore using the substitution $z = \ee^{2\pi  \xi}$ we have
  \begin{equation}\label{eq:mellin}
    \int_{\R}\frac{z^{it}}{\sin\hab{\pi(s-it)}}\dd t = \frac{2z^s}{1+ z}, \qquad 0<s<1,\, z \in \R_+.
  \end{equation}
  for $z \in \R_+$, which extends by analytic continuation to all $z \in \C$ with $-\pi < \arg(z)<\pi$.

   Take $\omega(A)<\nu<\pi - \abs{\arg(\lambda)}$  and let $x \in D(A) \cap R(A)$. Then  $A^{it}x$ is given by the Bochner integral
  \begin{equation*}
    A^{it}x = \frac{1}{2\pi i}\int_{\Gamma_\nu}z^{it}\varphi(z)R(z,A)y \dd z,
  \end{equation*}
  where $\varphi(z) = z(1+z)^{-2}$ and $y \in X$ is such that $x =\varphi(A)y$.
  Thus, by Fubini's theorem, \eqref{eq:mellin} and  $\abs{\arg(\lambda)}+\nu<\pi$, we have for $0<s<1$
  \begin{align*}
    \frac{1}{2}\int_{\R}\frac{1}{\sin\hab{\pi(s-it)}}\lambda^{it}A^{it}x \dd t
    &= \frac{1}{4\pi i}\int_{\Gamma_\nu}\int_{\R} \frac{\lambda^{it}z^{it} }{\sin\hab{\pi(s-it)}} \dd t\, \varphi(z) R(z,A)y \dd z \\
    &= \lambda^sA^s(1+\lambda A)^{-1}x.
  \end{align*}
  As $\lambda A$ has $\BIP$ with $\omega_{\BIP}(\lambda A)<\pi$, the lemma now follows by a density argument.
\end{proof}

As announced this lemma allows us to connect $\BIP$ to almost $\alpha$-sectoriality.

\begin{proposition}\label{proposition:BIPalmostalpha}
  Let $A$ be a sectorial operator on $X$.
  \begin{enumerate}[(i)]
    \item \label{it:almostalpha1} If $A$ has $\alpha$-$\BIP$ with $\omega_{\alphaBIP}<\pi$, then $A$ is almost $\alpha$-sectorial with $\tilde{\omega}_\alpha(A) \leq \omega_{\alphaBIP}(A)$.
    \item \label{it:almostalpha2} If $A$ has $\BIP$ with $\omega_{\BIP}<\pi$ and $\alpha$ is ideal, then $A$ is almost $\alpha$-sectorial with $\tilde{\omega}_\alpha(A) \leq \omega_{\BIP}(A)$.
  \end{enumerate}
\end{proposition}

\begin{proof}
Either fix $\omega_{\alphaBIP}(A)<\theta<\pi$ for \ref{it:almostalpha1} or fix $\omega_{\BIP}(A)<\theta<\pi$ for \ref{it:almostalpha2}. Suppose that $\lambda_1,\ldots,\lambda_n\in \C$ satisfy $\abs{\arg(\lambda_k)} \leq \pi - \theta $. Then for $0<s<1$ and $\mb{x} \in X^n$ we have by Lemma \ref{lemma:mellin}
\begin{align*}
  \nrmb{(\lambda_k^{1-s}A^sR(-\lambda_k,A)x_k)_{k=1}^n}_{\alpha}&= \nrmb{(\lambda_k^{-s}A^s(1+\lambda_k^{-1}A)^{-1}x_k)_{k=1}^n}_{\alpha}\\&\leq
  \frac{1}{2}\int_{\R}\frac{1}{\absb{\sin(\pi(s-it))}}\nrmb{(\lambda_k^{-it}A^{it}x_k)_{k=1}^n}_\alpha \dd t\\
  &\leq \frac{1}{2}\int_{\R}\frac{1}{{\absb{\sin(\pi(s-it))}}} \ee^{(\pi-\theta)\abs{t}}\nrm{A^{it}}_\alpha \dd t \, \nrm{\mb{x}}_\alpha \\&\leq C \, \nrm{\mb{x}}_{\alpha},
\end{align*}
where we used that there is a $\theta_0<\theta$ such that
\begin{equation*}
  \ee^{-\theta_0\abs{t}} \nrm{A^{it}}_\alpha \leq C
\end{equation*}
with $C>0$ independent of $t\in \R$ in the last step.
\end{proof}

With some additional effort we can self-improve Proposition \ref{proposition:BIPalmostalpha}\ref{it:almostalpha1} to conclude that $A$ is actually $\alpha$-sectorial rather than almost $\alpha$-sectorial. They key ingredient will be the $\alpha$-multiplier theorem (Theorems \ref{theorem:pointwisemultipliers1} and \ref{theorem:pointwisemultipliers2}).

\begin{theorem}\label{theorem:alphaBIPalpha}
  Let  $A$ be a sectorial operator on $X$. If $A$ has $\alpha$-$\BIP$ with $\omega_{\alphaBIP}<\pi$, then $A$ is $\alpha$-sectorial with $\omega_\alpha(A) \leq \omega_{\alphaBIP}(A)$.
\end{theorem}

\begin{proof}
We will show that for $0<s<\frac{1}{2}$ the families of operators
$$\Gamma_s :=\cbrace{t^sA^s(1+tA)^{-1}:t>0}$$ are $\alpha$-bounded uniformly in $s$. Since
we have for $x \in D(A)\cap R(A)$
  $$\lim_{s \to 0} t^{1-s}A^s(t+A)^{-1}x = -tR(-t,A)x$$  by the dominated convergence theorem, we obtain for $x_1,\ldots,x_n \in D(A)\cap R(A)$ and $t_1,\ldots,t_n>0$ that
  \begin{align*}
    \nrmb{\hab{-t_kR(-t_k,A)x_k}_{k=1}^n}_\alpha \leq \liminf_{s \to 0} \nrmb{\hab{t_k^{1-s}A^s(t_k+ A)^{-1}x_k}_{k=1}^n}_{\alpha}.
  \end{align*}
  This implies that $A$ is $\alpha$-sectorial with
  \begin{equation*}
    \omega_\alpha(A) = \tilde{\omega}_\alpha(A) \leq \omega_{\alphaBIP}(A).
  \end{equation*}
  by Proposition \ref{proposition:alphaalmoastalpha} and Proposition \ref{proposition:BIPalmostalpha}.

 We claim that it suffices to prove for $f$ in the Schwartz class $\mc{S}(\R;X)$ that
  \begin{equation}\label{eq:alphaBIPclaim}
    \nrms{t \mapsto \int_{\R} k_s(t-u)A^{i(t-u)}f(u)\dd u}_{\alpha(\R;X)} \leq C\,\nrm{f}_{\alpha(\R;X)},
  \end{equation}
  where $C>0$ is independent of $0<s<\frac{1}{2}$ and
\begin{equation*}
    k_s(t) := \frac{1}{2 \sin\hab{\pi(s-it)}}, \qquad t \in \R.
  \end{equation*}
  Indeed, assuming this claim for the moment, we know by Fubini's theorem and Lemma \ref{lemma:mellin}
\begin{align*}
\int_{\R}\int_{\R} k_s(t-u)A^{i(t-u)}f(u)\dd u \,\ee^{-2\pi i t \xi}\dd t &= \int_{\R} k_s(t)A^{it} \ee^{-2\pi i t \xi}\dd t \int_{\R} f(u) \ee^{-2\pi i u \xi}\dd u\\
&= \ee^{-2\pi\xi s}A^s (1+ \ee^{-2\pi\xi }A)^{-1} \hat{f}(\xi)
  \end{align*}
  for any $\xi \in \R$.
  Thus since the Fourier transform is an isometry on $\alpha(\R;X)$ by Example \ref{example:fourier}, we deduce that for any $g \in \mc{S}(\R;X)$
  \begin{equation*}
    \nrm{\xi \mapsto \ee^{-2\pi s\xi}A^s(1+\ee^{-2\pi \xi}A)^{-1}g(\xi)}_{\alpha(\R;X)}\leq C \, \nrm{g}_{\alpha(\R;X)},
  \end{equation*}
  which extends to all strongly measurable $g\colon S \to X$ in $\alpha(S;X)$ by density, see Proposition \ref{proposition:densealphaspace}. Then the converse of the $\alpha$-multiplier theorem (Theorem \ref{theorem:pointwisemultipliers2}) implies that $\Gamma_s$ is $\alpha$-bounded, which completes the proof.

  To prove the claim fix $0<s<\frac12$ and set $I_m = [2m-1,2m+1)$ for $m \in \Z$. For $n \in \Z$ we define the kernel
  \begin{equation*}
    K_n(t,u) : = \sum_{j \in \Z}k_s(t-u) \ind_{I_j}(t)\ind_{I_{j+n}}(u),\qquad t,u \in \R,
  \end{equation*}
  where the sum consists of only one element for any point $(t,u)$.
  Since $k_s \in L^1(\R)$, the operator $T_n:L^2(\R)\to L^2(\R)$ given by
  \begin{equation*}
    T_n\varphi(t) := \int_\R  K_n(t,u)\varphi(u) \dd u, \qquad t \in \R,
  \end{equation*}
  is well-defined. By the Mellin transform as in \eqref{eq:mellin} we know
  \begin{equation*}
    \hat{k}_s(\xi) = \frac{\ee^{-s\xi}}{1+\ee^{-\xi}} \leq 1, \qquad \xi \in \R,
  \end{equation*}
  so by Plancherel's theorem we have for $\varphi \in L^2(\R)$
  \begin{align*}
    \nrm{T_n\varphi}_{L^2(\R)}^2 &= \sum_{j \in \Z} \int_{\R} \ind_{I_j}(t)\abss{\int_{\R} k_s(t-u)\varphi(u)\ind_{I_{j+n}}(u)\dd u}^2\dd t \\
    &\leq \sum_{j \in \Z} \int_{I_{j+n}}\abs{\varphi(t)}^2 \dd t= \nrm{\varphi}_{L^2(\R)}^2.
  \end{align*}
  Moreover since $\abs{K_n(t,u)} \leq \abs{k_s(t-u)}\ind_{\abs{t-u} \geq 2(\abs{n}-1)}$ for $t,u \in \R$ and
  \begin{equation*}
    \abs{k_s(t)} \leq \frac{1}{2\abs{\sinh(\pi t)}} \leq  \ee^{-\pi \abs{t}}, \qquad \abs{t} \geq 1,
  \end{equation*} we have by Young's inequality
  \begin{equation*}
    \nrm{T_n}_{L^2(\R) \to L^2(\R)} \leq C_0 \, \ee^{-2\pi\abs{n}}, \qquad \abs{n} \geq 2
  \end{equation*}
  for some constant $C_0>0$.  We conclude that $T_n$ extends to a bounded operator on $\alpha(\R;X)$ for all $n \in \Z$ with
  \begin{equation*}
    \nrm{T_n}_{\alpha(\R:X)\to \alpha(\R;X)} \leq C_0 \,\ee^{-2\pi\abs{n}}.
  \end{equation*}

  For $t \in \R$ define $p(t) = 2j$ with $j \in \Z$ such that $t \in I_j$. Then $\abs{p(t)-t} \leq 1$ for all $t \in \R$. Take $\omega_{\alphaBIP}(A) <\theta<\pi$ and let $C_1,C_2>0$ be such that
  \begin{align*}
    \nrmb{\cbrace{A^{is}:s \in [-1,1]}}_\alpha &\leq C_1,\\
    \nrmb{A^{is}}_{\alpha} &\leq C_2\,\ee^{\theta\abs{s}}, \qquad s \in \R.
  \end{align*}
  Now take a Schwartz function $f \in \mc{S}(\R;X)$ and fix $n \in \Z$. Noting that $p(t)= p(u) -2n$ on the support of $K_n$, we estimate
  \begin{align*}
    \nrms{t \mapsto &\int_{\R}K_n(t,u)A^{i(t-u)}f(u)\dd u}_{\alpha(\R;X)} \\&= \nrms{t \mapsto \int_{\R}K_n(t,u)A^{i(t-p(t)+ p(u)-u-2n)}f(u)\dd u}_{\alpha(\R;X)}\\
    &\leq C_1 \nrms{t \mapsto \int_{\R}K_n(t,u)A^{i(p(u)-u-2n)}f(u)\dd u}_{\alpha(\R;X)}\\
    &\leq C_0 C_1 \,\ee^{-2\pi\abs{n}}\nrmb{t \mapsto A^{i(p(t)-t-2n)}f(t)}_{\alpha(\R;X)}\\
    &\leq C_0 C_1^2C_2 \,\ee^{-2(\pi-\theta)\abs{n}}\nrm{f}_{\alpha(\R;X)}
  \end{align*}
  using Theorem \ref{theorem:pointwisemultipliers1} in the second and last step.  Since
  \begin{equation*}
    k_s(t-u) = \sum_{n\in \Z} K_n(t,u), \qquad t,u \in \R,
  \end{equation*}
  the claim in \eqref{eq:alphaBIPclaim} now follows from the triangle inequality.
\end{proof}

\begin{remark}
  If the $X$ has the $\UMD$ property and $A$ is a sectorial operator with $\BIP$, then it was shown in \cite[Theorem 4]{CP01} that $A$ is $\gamma$-sectorial. The proof of that result can be generalized to a Euclidean structure $\alpha$ under the assumption that $\alpha(\R;X)$ has the $\UMD$ property, which in case of the $\gamma$-structure is equivalent to the assumption that $X$ has the $\UMD$ property. Note that the proofs of Theorem \ref{theorem:alphaBIPalpha} and \cite[Theorem 4]{CP01} are of a similar flavour. The key difference being the point at which one gets rid of the singular integral operators, employing their boundedness on $\alpha(\R;X)$ and $L^p(\R;X)$ respectively.
\end{remark}

\subsection*{The characterization of $H^\infty$-calculus in terms of $\alpha$-$\BIP$}
 With Theorem \ref{theorem:alphaBIPalpha} at our disposal we turn to the main result of this section, which characterizes when $A$ has a bounded $H^\infty$-calculus in terms of $\alpha$-$\BIP$. For this we will combine the Mellin transform arguments from \ref{theorem:alphaBIPalpha} with the
 self-improvement of a bounded $H^\infty$-calculus in Theorem \ref{theorem:Hinftyalphabounded} and the transference principle in Theorem \ref{theorem:transference}.

\begin{theorem}\label{theorem:BIPHinfty}
  Let $A$ be a sectorial operator on $X$. The following conditions are equivalent:
  \begin{enumerate}[(i)]
    \item $A$ has $\BIP$ with $\omega_{\BIP}(A) < \pi$ and $\alpha$-$\BIP$ for some Euclidean structure $\alpha$ on $X$.
    \item $A$ has a bounded $H^\infty$-calculus.
  \end{enumerate}
  If one of these equivalent conditions holds, we have
  \begin{equation*}
    \omega_{H^\infty}(A) = \omega_{\BIP}(A) = \inf\cbrace{\omega_{\alphaBIP}(A):\alpha \text{ is a Euclidean structure on $X$}}
  \end{equation*}
\end{theorem}

\begin{proof}
  Suppose that $A$ has a bounded $H^\infty$-calculus and let $\omega_{H^\infty}(A)<\sigma<\pi$. Then, by Theorem \ref{theorem:Hinftyalphabounded}, there is a Euclidean structure $\alpha$ on $X$ so that $A$ has a $\alpha$-bounded $H^\infty(\Sigma_\sigma)$-calculus. By \eqref{eq:HinftyBIP} this implies that $A$ has $\alpha$-$\BIP$ with $\omega_{\alphaBIP}\leq\sigma$ and therefore
  \begin{equation}\label{eq:angle1}
    \inf\cbrace{\omega_{\alphaBIP}(A):\alpha \text{ is a Euclidean structure on $X$}} \leq \omega_{H^\infty}(A)
  \end{equation}

  For the converse direction pick $s>0$ so that $\omega_{\alphaBIP}(A^s)<\pi$. Then $A^s$ is $\alpha$-sectorial by Theorem \ref{theorem:alphaBIPalpha} with $\omega_{\alpha}(A^s) \leq \omega_{\alphaBIP}(A^s)$. Take $\omega_{\alphaBIP}(A^s)<\sigma<\pi$, then by Theorem \ref{theorem:transference} we can find a sectorial operator $\widetilde{A}$ on a Hilbert space $H$ with $\omega(\widetilde{A})=\omega_{\BIP}(\widetilde{A})<\sigma$ and such that  $$\nrm{f(A)}_{\mc{L}(X)} \lesssim \nrm{f(\widetilde{A})}_{\mc{L}(H)}, \qquad f \in H^\infty(\Sigma_\sigma).$$
  Since $\BIP$ implies a bounded $H^\infty$-calculus on a Hilbert space by \cite{Mc86}, $\widetilde{A}$ has a bounded $H^\infty(\Sigma_\sigma)$-calculus. Therefore $A^s$ has a bounded $H^\infty$-calculus with $\omega_{H^\infty}(A^s) <\pi$. So since the $\BIP$ and $H^\infty$-calculus angles are equal for sectorial operators with a bounded $H^\infty$-calculus, it follows that
  \begin{equation}\label{eq:angle2}
    \omega_{H^\infty}(A^s)= \omega_{\BIP}(A^s) = s\,\omega_{\BIP}(A) < s \,\pi.
  \end{equation}
  Thus $A$ has a bounded $H^\infty$-calculus with $\omega_{H^\infty}(A) = s^{-1}\omega_{H^\infty}(A^s) = \omega_{\BIP}(A)$ by Proposition \ref{proposition:hinftys}. The claimed angle equalities follow by combining \eqref{eq:angle1} and \eqref{eq:angle2}.
\end{proof}

Combining Theorem \ref{theorem:BIPHinfty} with Theorem \ref{theorem:Hinftyell2bounded} and Proposition \ref{proposition:BIPideal} we obtain the following corollary, of which the first part recovers \cite[Corollary 7.5]{KW16}

\begin{corollary}
  Let $A$ be a sectorial operator on $X$.
  \begin{enumerate}[(i)]
  \item If $X$ has Pisier's contraction property, then $A$ has a bounded $H^\infty$-calculus if and only if $A$ has $\gamma$-$\BIP$ with $\omega_{\gammaBIP}(A) < \pi$. In this case
      \begin{equation*}
        \omega_{H^\infty}(A) =\omega_{\gammaH}(A)  = \omega_{\BIP}(A) = \omega_{\gammaBIP}(A)
      \end{equation*}
  \item If $X$ is a Banach lattice, then $A$ has a bounded $H^\infty$-calculus if and only if $A$ has $\ellBIP$ with $\omega_{\ellBIP}(A) < \pi$. In this case
      \begin{equation*}
        \omega_{H^\infty}(A) =\omega_{\ellH}(A)  = \omega_{\BIP}(A) = \omega_{\ellBIP}(A)
      \end{equation*}
  \end{enumerate}
\end{corollary}

\chapter{Sectorial operators and generalized square functions}\label{part:5}
Continuing our analysis of the connection between the $H^\infty$-calculus of sectorial operators and Euclidean structures, we will  characterize whether a sectorial operator $A$ has a bounded $H^\infty$-calculus in terms of generalized square function estimates and in terms of the existence of a dilation to a group of isometries in this chapter. Furthermore, for a given Euclidean structure $\alpha$  we will introduce certain spaces close to $X$ on which $A$ always admits a bounded $H^\infty$-calculus. In order to do so we will need the full power of the vector-valued function spaces introduced in Chapter \ref{part:3}, in particular the $\alpha$-multiplier theorem.

Our inspiration stems from \cite{CDMY96}, where Cowling, Doust, McIntosh and Yagi describe a general construction of some spaces associated to a given sectorial operator $A$ on $L^p$ for $p \in (1,\infty)$. They consider norms of the form
\begin{equation*}
   \nrms{\has{\int_0^\infty \abs{\psi(tA)x}^2\frac{\ddn t}{t}}^{1/2}}_{L^p},\qquad x \in D(A) \cap R(A),
\end{equation*}
where $\psi \in H^1(\Sigma_\sigma)$ for some $\omega(A)<\sigma<\pi$. They
 characterize the boundedness of the $H^\infty$-calculus of $A$ on $X$ in terms of the equivalence of such expressions with $\nrm{x}_{L^p}$.
 Further developments in this direction can for example be found in \cite{AMN97, FM98,KU14,KW16b, LL05,Le04,Le12}.

In the language of this memoir the norms from \cite{CDMY96} can be interpreted as
\begin{equation*}
  \nrms{\has{\int_0^\infty \abs{\psi(tA)x}^2\frac{\ddn t}{t}}^{1/2}}_{L^p} = \nrm{t\mapsto\psi(tA)x}_{\ell^2(\R_+,\frac{\ddn t}{t};X)},
\end{equation*}
which suggests to extend these results to the framework of Euclidean structures by replacing the $\ell^2$-structure with a general Euclidean structure $\alpha$. Therefore, for a sectorial operator $A$ on a general Banach space $X$ equipped with a Euclidean structure $\alpha$ we will introduce the generalized square function norms
$
 \nrm{t\mapsto \psi(tA)x}_{\alpha(\R_+,\frac{\ddn t}{t};X)}
$
along with a discrete variant and study their connection with the $H^\infty$-calculus of $A$ in Section \ref{section:squarefunctions}. In particular, we will characterize the boundedness of the $H^\infty$-calculus of $A$ in terms of a norm equivalence between these generalized square function norms and the usual norm on $X$. For the $\gamma$-structure, which is equivalent to the $\ell^2$-structure on $L^p$, this was already done in \cite{KW16} (see also \cite[Section 10.4]{HNVW17}).
In Section \ref{section:dilation} we will use these generalized square function norms to construct dilations of sectorial operators on the spaces $\alpha(\R;X)$, which characterize the boundedness of the $H^\infty$-calculus of $A$.

Afterwards we introduce a scale of spaces $H^\alpha_{\theta,A}$ for $\theta \in \R$ in Section \ref{section:scalespaces}, which are endowed with such a generalized square function norm. These spaces
are very close to the homogeneous fractional domain spaces, but  behave better in many respects.  In particular we will show that $A$ induces a sectorial operator on these spaces which always has a bounded $H^\infty$-calculus. Moreover  we will show that these generalized square function spaces form an interpolation scale for the complex method and that when one applies the $\alpha$-interpolation method as introduced in Section \ref{section:interpolation} to the fractional domain spaces of $A$, one obtains these generalized square function spaces. We will end this chapter with an investigation of the generalized square function spaces for sectorial operators that are not necessarily almost $\alpha$-bounded in Section \ref{section:Hspacenotalmostalpha}. This will allow us to construct some interesting counterexamples on the angle of the $H^\infty$-calculus in Section \ref{section:angleexample}.

As in the previous two chapters, we keep the standing assumption that $\alpha$ is a Euclidean structure on $X$ throughout this chapter.

\section{Generalized square function estimates}\label{section:squarefunctions}
Let $A$ be a sectorial operator on $X$. As announced in the introduction of this chapter, we start by studying the generalized square function norm
\begin{equation*}
 \nrm{t\mapsto \psi(t A)x}_{\alpha(\R_+,\frac{\ddn t}{t};X)},
\end{equation*}
and its discrete analog
 \begin{equation*}
   \sup_{t \in [1,2]}\nrmb{(\psi(2^ntA)x)_{n\in \Z}}_{\alpha(\Z;X)}
 \end{equation*}
for appropriate $x \in X$. For $\alpha=\gamma$ these norms were already studied in \cite{KW16} (see also \cite[Section 10.4]{HNVW17}).

We would like to work with $x$ such that $t \mapsto \psi(t A)x$ defines an element of $\alpha(\R_+,\frac{\ddn t}{t};X)$, rather than just being an element of the larger space $\alpha_+(\R_+,\frac{\ddn t}{t};X)$. Our main tool, the $\alpha$-multiplier theorem (Theorem \ref{theorem:pointwisemultipliers1}), asserts that $\alpha$-bounded pointwise multipliers act boundedly from $\alpha(\R_+,\frac{\ddn t}{t};X)$ to $\alpha_+(\R_+,\frac{\ddn t}{t};X)$. We will frequently use the following lemma to ensure that such a multiplier actually maps to $\alpha(\R_+,\frac{\ddn t}{t};X)$ for certain $x \in X$.

\begin{lemma}\label{lemma:squarefunctionalpha}
  Let $A$ be a sectorial operator on $X$ and take $\omega(A)<\sigma<\pi$. Let $x \in R(\varphi(A))$ for some $\varphi \in H^1(\Sigma_\sigma)$, e.g. take $x \in D(A)\cap R(A)$. Then for $f \in H^\infty(\Sigma_\sigma)$ and $\psi \in H^1(\Sigma_\sigma)$ we have
  \begin{align*}
    \hab{t\mapsto f(A)\psi(tA)x} &\in \alpha\hab{\R_+,\tfrac{\ddn t}{t};X},\\
    \hab{n \mapsto f(A)\psi(2^ntA)x} &\in \alpha(\Z;X), \qquad t \in [1,2].
  \end{align*}
\end{lemma}

\begin{proof}
We will only show the first statement, the second being proven analogously. Take $\omega(A) <\nu<\sigma$ and let $y \in X$ be such that $x = \varphi(A)y$.  By the multiplicativity of the $H^\infty$-calculus we have
\begin{equation}\label{eq:dunfordpsiphi}
  f(A)\psi(tA)x = \frac{1}{2\pi i}\int_{\Gamma_\nu}f(z)\psi(tz)\varphi(z)R(z,A)y\dd z.
\end{equation}
  For all $z \in \Gamma_\nu$ the function $\psi(\cdot z)\otimes f(z)\varphi(z)R(z,A)y$ belongs to $\alpha(\R_+,\frac{\ddn t}{t};X)$, with norm
  \begin{equation*}
   \nrm{\psi(\cdot z)}_{L^2(\R_+,\frac{\ddn t}{t})} \nrm{f(z)\varphi(z)R(z,A)y}_X.
  \end{equation*}
  By \eqref{eq:H1inHinfty} we know that for any $\xi \in H^1(\Sigma_\sigma)$ we have $\xi \in H^2(\Sigma_{\sigma'})$ for $\nu<\sigma'<\sigma$, so
  $$\sup_{z\in \Gamma_\nu} \nrm{\psi(\cdot z)}_{L^2(\R_+,\frac{\ddn t}{t})}< \infty.$$ We can therefore interpret the integral \eqref{eq:dunfordpsiphi} as a Bochner integral in $\alpha(\R_+,\frac{\ddn t}{t};X)$, which yields that $f(A)\psi(\cdot A)x$ defines an element of $\alpha(\R_+,\frac{\ddn t}{t};X)$.
\end{proof}

\subsection*{The equivalence of discrete and continuous generalized square function norms}
Next we will show that it does not matter whether one studies the discrete or the continuous generalized square functions, as these norms are equivalent. Because of this equivalence we will only state results for the continuous generalized square function norms in the remainder of this section.  The statements for discrete generalized square function norms are left to the interested reader, see also \cite[Section 10.4.a]{HNVW17}. Situations in which one can take $\delta=0$ in the following proposition will be discussed in Corollary \ref{corollary:equivalencediscretecontinuous} and Proposition \ref{proposition:equivdiscretecontinuousstable}.

 \begin{proposition}\label{proposition:equivalencediscretecontinuous}
 Let $A$ be a sectorial operator on $X$, take $\omega(A)<\sigma<\pi$ and let $\psi \in H^1(\Sigma_\sigma)$. For all $0<\delta<\sigma-\omega(A)$ there is a $C>0$ such that for $x \in D(A)\cap R(A)$
 \begin{align*}
  \sup_{t \in [1,2]}\nrmb{(\psi(2^ntA)x)_{n\in \Z}}_{\alpha(\Z;X)} \leq C \, \max_{\epsilon = \pm \delta}\, \nrm{\psi_{\epsilon}(\cdot A)x}_{\alpha(\R_+,\frac{\ddn t}{t};X)},\intertext{and}
   \nrm{\psi(\cdot A)x}_{\alpha(\R_+,\frac{\ddn t}{t};X)}\leq C \,\sup_{\abs{\epsilon}< \delta}\sup_{t \in [1,2]}\nrmb{(\psi_{\epsilon}(2^ntA)x)_{n\in \Z}}_{\alpha(\Z;X)}
 \end{align*}
 where $\psi_{\epsilon}(z) = \psi(\ee^{i\epsilon}z)$.
 \end{proposition}

\begin{proof}
For $x \in D(A)\cap R(A)$ and $\abs{\epsilon}\leq \delta$ we know that  \begin{align*}
    \hab{t\mapsto \psi_{\epsilon}(tA)x} &\in \alpha\hab{\R_+,\tfrac{\ddn t}{t};X},\\
    \hab{n \mapsto \psi_{\epsilon}(2^ntA)x} &\in \alpha(\Z;X), \qquad t \in [1,2].
\end{align*} by Lemma \ref{lemma:squarefunctionalpha}. Therefore if $\alpha=\gamma$ the first inequality follows from
    \cite[Proposition 9.7.10]{HNVW17} with $a=0$, $b = \delta/\log(2)$, $\alpha = \sigma/\log(2)$ and the observation that  $z \mapsto 2^z$ maps the strip
    \begin{equation*}
      \mbb{S}_{\sigma/\log(2)}:= \cbrace{z \in \C:\abs{\im(z)}<\sigma/\log(2)}
    \end{equation*}
    onto the sector $\Sigma_\sigma$. Analogously the second inequality for $\alpha=\gamma$ follows from  \cite[Proposition 9.7.20]{HNVW17} with $\alpha=\delta/\log(2)$. The proofs carry over to an arbitrary Euclidean structure, as the only properties of the $\gamma$-structure used in the proof of \cite[Proposition 9.7.10 and Proposition 9.7.20]{HNVW17} are \eqref{eq:E1x} and the right ideal property in \eqref{eq:E2x} in the form of Proposition \ref{proposition:operatoronfunctions}.
\end{proof}

\subsection*{The equivalence of the generalized square function norms for different $\psi$}
The first major result of this section will be the equivalence of the continuous generalized square function norms for different choices of $\psi \in H^1(\Sigma_\sigma)$. As a preparation let us note the following easy corollary of Jensen's inequality.

\begin{lemma}\label{lemma:productconvolution}
  Let $h \in L^1(\R_+,\frac{\ddn t}{t})$. The operator $S_h$ on $L^2(\R_+,\frac{\ddn t}{t})$ given by
  \begin{equation*}
    S_hu(s) := \int_0^\infty h(st)u(t)\frac{\ddn t}{t}, \qquad s \in \R_+
  \end{equation*}
  is bounded with $\nrm{S_h} \leq \nrm{h}_{L^1(\R_+,\frac{\ddn t}{t})}$.
\end{lemma}

\begin{proof}
Let $c := \nrm{h}_{L^1(\R_+,\frac{\ddn t}{t})}$. By Jensen's inequality and Fubini's theorem we have
\begin{equation*}
  \int_0^\infty \abss{\int_0^\infty h(st)u(t)\frac{\ddn t}{ct}}^2\frac{\ddn s}{s}\leq \int_0^\infty \int_0^\infty \abs{h(st)}\abs{u(t)}^2\frac{\ddn t}{ct}\frac{\ddn s}{s} = \nrm{u}^2_{L^2(\R_+,\frac{\ddn t}{t})},
\end{equation*}
which yields $\nrm{S_hu}_{L^2(\R_+,\frac{\ddn t}{t})} \leq \nrm{h}_{L^1(\R_+,\frac{\ddn t}{t})} \nrm{u}_{L^2(\R_+,\frac{\ddn t}{t})}$.
\end{proof}

We are now ready to prove the announced equivalence, which relies upon the $\alpha$-multiplier theorem. For the $\ell^2$- and the $\gamma$-structure this recovers the corresponding results from \cite{Le04} and \cite{KW16} respectively.

\begin{proposition}\label{proposition:equivsquarefunctions} Let $A$ be an almost $\alpha$-sectorial operator on $X$. Take $\tilde{\omega}_{\alpha}(A)<\sigma<\pi$ and fix arbitrary non-zero $\psi,\varphi \in H^1(\Sigma_\sigma)$. For all $f \in H^\infty(\Sigma_\sigma)$ and $x \in D(A) \cap R(A)$ we have
\begin{equation*}
  \nrmb{f(A) \psi(\cdot A)x}_{\alpha(\R_+,\frac{\ddn t}{t};X)} \lesssim \nrm{f}_{H^\infty(\Sigma_\sigma)}\nrmb{\varphi(\cdot A)x}_{\alpha(\R_+,\frac{\ddn t}{t};X)}.
\end{equation*}
In particular, for $f \equiv 1$, we have
\begin{equation*}
  \nrmb{\psi(\cdot A)x}_{\alpha(\R_+,\frac{\ddn t}{t};X)} \simeq  \nrmb{\varphi(\cdot A)x}_{\alpha(\R_+,\frac{\ddn t}{t};X)}
\end{equation*}
\end{proposition}

\begin{proof}
 First fix $f \in H^1(\Sigma_\sigma)\cap H^\infty(\Sigma_\sigma)$, let $x \in D(A) \cap R(A)$ and note that both $f(A) \psi(\cdot A)x$ and $\varphi(\cdot A)x$ are in $\alpha(\R_+,\frac{\ddn t}{t};X)$ by Lemma \ref{lemma:squarefunctionalpha}. Let $\xi,\eta \in H^1(\Sigma_\sigma)$ be non-zero such that
  \begin{equation*}
    \int_0^\infty \xi(t)\eta(t)\varphi(t)\frac{\ddn t}{t} = 1.
  \end{equation*}
  Then
    \begin{equation*}
    \int_0^\infty \xi(tz)\eta(tz)\varphi(tz)\frac{\ddn t}{t} = 1,\qquad  z \in \Sigma_\sigma,
  \end{equation*}
  which is clear for $z \in \R_+$ and then in general by analytic continuation. Take $\tilde{\omega}_{\alpha}(A)<\nu<\sigma$. We use the properties of the Dunford calculus of $A$ and Fubini's theorem to calculate
  \begin{align*}
    f(A) &= \frac{1}{2\pi i} \int_{\Gamma_\nu} \has{\int_0^\infty\xi(tz)\eta(tz)\varphi(tz)\frac{\ddn t}{t}} f(z)R(z,A)\dd z\\
    &= \int_0^\infty \has{\frac{1}{2\pi i} \int_{\Gamma_\nu} \xi(tz)\eta(tz)\varphi(tz) f(z)R(z,A)\dd z}\frac{\ddn t}{t}\\
    &=\int_0^\infty \xi(tA)\eta(tA) f(A)\varphi(tA)\frac{\ddn t}{t}
  \end{align*}
  With this identity, Fubini's theorem and \eqref{eq:representationHinfty} we obtain for $x \in X$ and $s \in \R_+$
  \begin{align*}
    f(A)\psi(sA)x &= \int_0^\infty \psi(sA)\xi(tA)\eta(tA)f(A)\varphi(tA) x\frac{\ddn t}{t}\\
    &= \frac{1}{2\pi i} \int_{\Gamma_\nu} \psi(sz)z^{\frac{1}{2}}A^{\frac{1}{2}}R(z,A)\int_0^\infty \xi(tz)\eta(tA)f(A)\varphi(tA) x\frac{\ddn t}{t}\frac{\ddn z}{z}\\
    &= \frac{1}{2\pi i} \sum_{\epsilon =\pm 1}
    \int_{0}^\infty \psi(s\lambda \ee^{i \epsilon \nu})M(\lambda \ee^{i \epsilon \nu})\int_0^\infty \xi(t\lambda \ee^{i \epsilon \nu})N(t)\varphi(tA) x\frac{\ddn t}{t}\frac{\ddn \lambda}{\lambda}\\
    &= \frac{1}{2\pi i} \sum_{\epsilon =\pm 1} S_\psi\has{\lambda \mapsto M(\lambda \ee^{i \epsilon \nu} )\cdot S_{\xi}\hab{t \mapsto N(t)\varphi(tA)x}(\lambda \ee^{i \epsilon \nu})}(s),
  \end{align*}
  where
  \begin{align*}
    M(z) &:= z^{\frac{1}{2}}A^{\frac{1}{2}}R(z,A), && z \in \Sigma_\sigma\\
    N(t) &:= \eta(tA)f(A),&&t \in \R_+
  \end{align*}
and $S_h$ for $h \in L^1(\R_+,\frac{\ddn t}{t})$ is as in Lemma \ref{lemma:productconvolution}, which extends to a bounded operator on functions in $\alpha(\R_+,\frac{\ddn t}{t};X)$ by Proposition \ref{proposition:operatoronfunctions}.

By Proposition \ref{proposition:almostsectorialcharacterization} we have that $\cbrace{M(z):z \in \Sigma_\sigma}$ is $\alpha$-bounded. For $N$ we recall the representation of \eqref{eq:representationHinfty}
\begin{align*}
  N(t)
  &=\frac{1}{2\pi i} \int_{\Gamma_\nu} \eta(tz)f(z)z^{\frac{1}{2}}A^{\frac{1}{2}}R(z,A)\frac{\ddn z}{z}.
\end{align*}
Thus since
\begin{equation*}
  \int_{\Gamma_\nu} \abs{\eta(tz)}\abs{f(z)}\absb{\frac{\ddn z}{z}} \leq \nrm{f}_{H^\infty(\Sigma_\sigma)} \nrm{\eta}_{H^1(\Sigma_\sigma)}
\end{equation*}
we have by the almost $\alpha$-sectoriality of $A$ and Corollary \ref{corollary:L1mean} that  $$\nrm{\cbrace{N(t):t \in \R_+}}_\alpha \lesssim \nrm{f}_{H^\infty(\Sigma_\sigma)} \nrm{\eta}_{H^1(\Sigma_\sigma)}.$$ Moreover by Lemma \ref{lemma:squarefunctionalpha} we have
\begin{equation*}
   N(\cdot)\varphi(\cdot A)x = \eta(\cdot A) f(A)\varphi(\cdot A)  x \in \alpha\hab{\R_+,\frac{\ddn t}{t};X}
\end{equation*}
where implicit constant depends on $A$.

Applying the boundedness of $S_\xi$ and $S_\psi$ and the $\alpha$-multiplier theorem (Theorem \ref{theorem:pointwisemultipliers1}) on $M$ and $N$  we obtain
\begin{align*}
  \nrm{f(A)\psi(\cdot A)x}_{\alpha(\R_+,\frac{\ddn t}{t};X)} &\leq\nrm{S_\psi} \nrm{S_\xi} \nrmb{\cbrace{M(z):z \in \Sigma_\sigma}}_\alpha \\
   & \hspace{1cm} \cdot\nrm{\cbrace{N(t):t \in \R_+}}_\alpha \nrmb{\varphi(\cdot A)x}_{\alpha(\R_+,\frac{\ddn t}{t};X)}\\
   &\lesssim \nrm{f}_{H^\infty(\Sigma_\sigma)} \nrm{\varphi(\cdot A)x}_{\alpha(\R_+,\frac{\ddn t}{t};X)},
\end{align*}
where the implicit constant depends on $\psi$, $\xi$, $\eta$ and $A$.

The same estimate  for a general $f \in H^\infty(\Sigma_\sigma)$ follows by approximating $f$ by $f_n := f \varphi_n \in H^1(\Sigma_\sigma)\cap H^\infty(\Sigma_\sigma)$ with $\varphi_n$ as in  \eqref{eq:phin}, noting that $f_n(A)\psi(tA)x \to f(A)\psi(tA)x$ for all $t \in \R_+$ and appealing to Proposition \ref{proposition:alphaspaceconvergence}\ref{it:convergenceproperty1}.
\end{proof}

From Proposition \ref{proposition:equivsquarefunctions} we can see that the square function norms corresponding to $\psi$ and $\psi_\epsilon$ as in Proposition \ref{proposition:equivalencediscretecontinuous} are equivalent when $A$ is almost $\alpha$-sectorial. Thus we can take $\delta=0$ in Proposition \ref{proposition:equivsquarefunctions}.

 \begin{corollary}\label{corollary:equivalencediscretecontinuous}
 Let $A$ be an almost $\alpha$-sectorial operator on $X$, take $\tilde{\omega}_{\alpha}(A)<\sigma<\pi$ and let $\psi \in H^1(\Sigma_\sigma)$. For all $x \in D(A)\cap R(A)$ we have
 \begin{align*}
  \nrm{\psi(\cdot A)x}_{\alpha(\R_+,\frac{\ddn t}{t};X)} \simeq \sup_{t \in [1,2]}\nrmb{(\psi(2^ntA)x)_{n\in \Z}}_{\alpha(\Z;X)}.
  \end{align*}
 \end{corollary}

If the generalized square function norms are equivalent with $\nrm{\cdot}_X$, it follows immediately from Proposition \ref{proposition:equivsquarefunctions} that $A$ has a bounded $H^\infty$-calculus, which is the content of the next theorem.

\begin{theorem}\label{theorem:squaretocalculus}
  Let $A$ be an almost $\alpha$-sectorial operator on $X$ and take $\tilde{\omega}_{\alpha}(A)<\sigma<\pi$. If there are non-zero $\psi,\varphi \in H^1(\Sigma_\sigma)$ such that for all $x \in D(A)\cap R(A)$
  \begin{equation*}
    \nrmb{\psi(\cdot A)x}_{\alpha(\R_+,\frac{\ddn t}{t};X)} \lesssim \nrm{x}_X \lesssim \nrmb{\varphi(\cdot A)x}_{\alpha(\R_+,\frac{\ddn t}{t};X)},
  \end{equation*}
  then $A$ has a bounded $H^\infty$-calculus with $\omega_{H^\infty}(A) \leq \tilde{\omega}_{\alpha}(A)$.
\end{theorem}

\begin{proof}
  Our claim follows directly from Proposition \ref{proposition:equivsquarefunctions}. Indeed, for $f \in H^\infty(\Sigma_\nu)$ for any $\tilde{\omega}_{\alpha}(A) <\nu \leq \sigma$  and $x \in D(A)\cap R(A)$ we have
  \begin{align*}
    \nrm{f(A)x} &\lesssim \nrmb{f(A)\varphi(tA)x}_{\alpha(\R_+,\frac{\ddn t}{t};X)} \\&\lesssim \nrmb{\psi(tA)x}_{\alpha(\R_+,\frac{\ddn t}{t};X)} \lesssim \nrm{x}_X
  \end{align*}
which extends by density to all $x \in X$.
\end{proof}

\subsection*{The equivalence of the generalized square function norms with the norm of $X$}
In the second halve of this section we will turn to the converse of Theorem \ref{theorem:squaretocalculus}, i.e. we will study when the boundedness of the $H^\infty$-calculus of $A$ implies the equivalence of the generalized square function norms with $\nrm{\cdot}_X$.
In order to prove our results, we will need to use the adjoint of $A$. Recall that the adjoint of a sectorial operator is a closed operator, which may not be a sectorial operator as it may only have dense domain and dense range in the weak$^*$-topology. To remedy this we introduce the so-called \emph{moon dual}, see e.g. \cite{FW06,KW04}. Define $X^{\sharp}$ as $\overline{D(A^*)}\cap \overline{R(A^*)}$, where the closures are taken in the norm topology of $X^*$. The moon-dual operator $A^{\sharp}$ of $A$ is the part of $A^*$ in $X^{\sharp}$, i.e.
\begin{equation*}
  A^{\sharp}x = A^*x^*, \qquad x \in D(A^{\sharp})= \cbraceb{x^* \in D(A^*)\cap \overline{R(A^*)}: A^*x^*  \in \overline{D(A^*)}}.
\end{equation*}
Then the following properties hold:
\begin{itemize}
  \item $A^{\sharp}$ is a sectorial operator on $X^{\sharp}$ with spectrum $\rho(A^{\sharp}) = \rho(A^*) = \rho(A)$.
  \item $X^{\sharp}\subseteq X^*$ is norming for $X$.
  \item $R(z,A^{\sharp})$ is the restriction of $R(z,A)^*$ to $X^{\sharp}$.
  \item $\varphi(A)^*x = \varphi(A^{\sharp})x$ for $\varphi \in H^1(\Sigma_\sigma)$ and $x^*\in X^{\sharp}$.
  \item If $X$ is reflexive, then $X^{\sharp} = X^*$ and $A^\sharp = A^*$.
  \item If $A$ has a bounded $H^\infty$-calculus, then $A^\sharp$ has a bounded $H^\infty$-calculus with $\omega_{H^\infty}(A^\sharp) = \omega_{H^\infty}(A)$.
\end{itemize}

We will start by showing that, up to a smoothing factor $\varphi(A)$ for $\varphi \in H^1(\Sigma_\sigma)$, we always have the equivalence of the generalized square function norms with $\nrm{\cdot}_X$. Note that similar estimates hold for the adjoint $A^*$ on $X^*$ equipped with the Euclidean structure $\alpha^*$, by applying the following proposition to $A^\sharp$ on $X^\sharp$ equipped with the Euclidean structure induced by $\alpha^*$.

\begin{proposition}\label{proposition:functionspaceclose}
  Let $A$ be a sectorial operator on $X$, let $\omega(A)<\sigma<\pi$ and take non-zero $\psi,\varphi \in H^1(\Sigma_\sigma)$. Then we have
  \begin{align*}
     &\nrm{\psi(\cdot A)\varphi(A)x}_{\alpha(\R_+,\frac{\ddn t}{t};X)} \lesssim \nrm{x}_X, && x \in X\\
      \intertext{and}
     &\nrm{\varphi(A)x}_X \lesssim  \nrm{\psi(\cdot A)x}_{\alpha(\R_+,\frac{\ddn t}{t};X)}, && x \in D(A)\cap R(A).
\end{align*}
\end{proposition}

\begin{proof}
  For the first inequality fix $x \in X$. Then we have $\psi(\cdot A)\varphi(A)x \in \alpha(\R_+,\frac{\ddn t}{t};X)$ by Lemma \ref{lemma:squarefunctionalpha}.
  Furthermore by \eqref{eq:representationHinfty} we have for $\omega(A)<\nu<\sigma$ and $t>0$
  \begin{align*}
    \psi(tA)
    &=\frac{1}{2 \pi i}  \int_{\Gamma_\nu} \psi(tz)  (z^{-1}A)^{1/2}R(1,z^{-1}A)  \frac{\ddn z}{z}\\
    &=\sum_{\epsilon=\pm 1}\frac{-\epsilon}{2 \pi i}  \int_{0}^\infty \psi(ste^{\epsilon i\nu})  (\ee^{-\epsilon i\nu }s^{-1}A)^{1/2}(1-\ee^{-\epsilon i\nu }s^{-1}A) \frac{\ddn s}{s}\\
    &= \sum_{\epsilon=\pm 1}\frac{-\epsilon}{2 \pi i}  \int_{0}^\infty \psi(s^{-1}te^{\epsilon i\nu})f_{\epsilon} (sA)\frac{\ddn s}{s}
  \end{align*}
  with $f_\epsilon(z) := (\ee^{-\epsilon i\nu}z)^{1/2}(1-\ee^{-\epsilon i\nu}z)^{-1}$. As $f_\epsilon \in H^1(\Sigma_{\sigma'})$ for $\omega(A)<\sigma'<\nu$,  we have by Fubini's theorem and the multiplicativity of the Dunford calculus
  \begin{align*}
    \int_0^\infty \nrm{f_{\epsilon}(sA)\varphi(A)} \frac{\dd s}{s}\lesssim \nrm{f}_{H^1(\Sigma_{\sigma'})} \nrm{\varphi}_{H^1(\Sigma_\sigma)}.
  \end{align*}
  Therefore, by property \eqref{eq:E1x} of a Euclidean structure and \eqref{eq:H1inHinfty}, we have
  \begin{align*}
    &\nrm{\psi(\cdot A)\varphi(A)x}_{\alpha(\R_+,\frac{\ddn t}{t};X)}\\ &\hspace{1cm}\leq \sum_{\epsilon=\pm 1} \nrms{t \mapsto \int_0^\infty\psi(s^{-1}t\ee^{\epsilon i\nu})f_\epsilon(sA)\varphi(A)x\frac{\dd s}{s}}_{\alpha(\R_+,\frac{\ddn t}{t};X)}\\
    &\hspace{1cm}\leq \sum_{\epsilon=\pm 1} \int_0^\infty \nrmb{t\mapsto \psi(s^{-1}t\ee^{\epsilon i \nu})}_{L^2(\R_+,\frac{\ddn t}{t})}  \nrm{f_\epsilon(sA)\varphi(A)x}_X\frac{\ddn s}{s}\\
    &\hspace{1cm}= \sum_{\epsilon=\pm 1}  \nrmb{t\mapsto\psi(t\ee^{\epsilon i \nu})}_{L^2(\R_+,\frac{\ddn t}{t})} \int_0^\infty \nrm{f_\epsilon(sA)\varphi(A)x}_X\frac{\ddn s}{s} \lesssim \nrm{x}_X,
  \end{align*}
   which proves the first inequality. Applying this result to $A^\sharp$ on $X^\sharp$ equipped with the Euclidean structure induced by $\alpha^*$ yields
    \begin{equation}\label{eq:dualsquarefunlemma}
      \nrm{\psi(\cdot A)^*\varphi(A)^*x^*}_{\alpha^*(\R_+,\frac{\ddn t}{t};X^*)} \lesssim \nrm{x^*}_{X^*}, \qquad x^*\in X^\sharp.
    \end{equation}

   For the second inequality take $x \in D(A)\cap R(A)$. Then by Lemma \ref{lemma:squarefunctionalpha} we have $\psi(\cdot A)x \in \alpha(\R_+,\frac{\ddn t}{t};X)$.
Thus, since $\psi \in H^2(\Sigma_{\sigma'})$ for $\omega(A)<\sigma'<\sigma$ by \eqref{eq:H1inHinfty}, we have by \eqref{eq:calderonrepnontrivial} and the multiplicativity of the Dunford calculus
   \begin{equation*}
     c\,x = \int_0^\infty \psi(tA)\psi^*(tA)x\frac{\ddn t}{t}
   \end{equation*}
   where $\psi^*(z):= \overline{\psi(\overline{z})}$ and $c = \int_0^\infty \abs{\psi(t)}^2\frac{\ddn t}{t}>0$. Applying Proposition \ref{proposition:alphaholder} and \eqref{eq:dualsquarefunlemma} we deduce for any $x^*\in X^\sharp$
   \begin{align*}
     \abs{\ip{\varphi(A)x,x^*}} &\leq c^{-1} \, \int_0^\infty \abs{\ip{ \psi(tA)x, \psi^*(tA)^*\varphi(A)^*x^*}} \frac{\ddn t}{t}\\
     &\leq c^{-1} \, \nrm{\psi(\cdot A)x}_{\alpha(\R_+,\frac{\ddn t}{t};X)}  \nrm{ \psi^*(\cdot A)^* \varphi(A)^*x}_{\alpha^*(\R_+,\frac{\ddn t}{t};X)}\\
     &\lesssim \nrm{\psi(\cdot A)x}_{\alpha(\R_+,\frac{\ddn t}{t};X)} \nrm{x^*}_{X^*},
   \end{align*}
   so taking the supremum over all $x^*\in X^\sharp$ yields the second inequality.
\end{proof}

If we assume the Euclidean structure $\alpha$ to be unconditionally stable and $A$ to have a bounded $H^\infty$-calculus, we can get rid of the $\varphi(A)$-terms in Proposition \ref{proposition:functionspaceclose}. For the Euclidean structures $\ell^2$ and $\gamma$, this recovers results from \cite{CDMY96, KW16}

\begin{theorem}\label{theorem:calculustosuqare}
Let  $A$ be a sectorial operator on $X$ with a bounded $H^\infty$-calculus and assume that $\alpha$ is unconditionally stable.
Take $\omega_{H^\infty}(A)<\sigma<\pi$ and let $\psi \in H^1(\Sigma_\sigma)$ be non-zero. Then for all $x  \in D(A)\cap R(A)$ we have
  \begin{equation*}
\nrm{x}_X \simeq \nrmb{t\mapsto \psi(tA)x}_{\alpha(\R_+,\frac{\ddn t}{t};X)}\simeq \sup_{t \in [1,2]} \nrmb{(\psi(2^ntA)x)_{n \in \Z}}_{\alpha(\Z;X)}.
  \end{equation*}
\end{theorem}

\begin{proof}
 Let $\omega_{H^\infty}(A)<\sigma' < \sigma$ and $\varphi \in H^1(\Sigma_{\sigma'})$. Note that $(\varphi(2^ntA)x)_{n\in \Z}$ is an element of $\alpha(\Z;X)$ by Lemma \ref{lemma:squarefunctionalpha} and
the functions
  \begin{equation*}
    f(z) = \sum_{k =-n}^n\epsilon_k \varphi(2^ktz)
  \end{equation*}
are uniformly bounded in $H^\infty(\Sigma_\sigma)$ for $t \in [1,2]$, $\abs{\epsilon_k}=1$ and $n \in \N$ by Lemma \ref{lemma:hadamardsum}. Therefore, since $\alpha$ is unconditionally stable and $A$ admits a bounded $H^\infty$-calculus, we have for all $x  \in D(A)\cap R(A)$
  \begin{align*}
    \sup_{t\in [1,2]}\nrmb{(\varphi(2^ntA)x)_{n\in \Z}}_{\alpha(\Z;X)} &\lesssim \sup_{t \in [1,2]} \sup_{n\in \N} \sup_{\abs{\epsilon_k} = 1} \nrmb{\sum_{k=-n}^n \epsilon_k \varphi(2^ktA)x}_X\\
    &\lesssim \nrm{x}_X.
  \end{align*}
Taking $\varphi=\psi$ in this inequality yields the first halve of the equivalence between the generalized discrete square function norms and $\nrm{\cdot}_X$. Furthermore, using this inequality with  $\varphi =\psi_\epsilon = \psi(\ee^{i\epsilon }\cdot)$ with $\epsilon<\sigma-\sigma'$, we have by Proposition \ref{proposition:equivalencediscretecontinuous}
 \begin{equation*}
   \nrm{\psi(\cdot A)x}_{\alpha(\R_+,\frac{\ddn t}{t};X)} \lesssim  \sup_{\abs{\epsilon}<\sigma-\sigma'} \sup_{t \in [1,2]} \nrmb{(\psi_{\epsilon}(2^ntA)x)_{n \in \Z}}_{\alpha(\Z;X)} \lesssim \nrm{x}.
 \end{equation*}

For the converse inequality we apply this result to the moon dual $A^\sharp$ on $X^\sharp$ equipped with the Euclidean structure induced by $\alpha^*$ to obtain
  \begin{equation*}
    \nrm{\psi(\cdot A)^*x^*}_{\alpha^*(\R_+,\frac{\ddn t}{t};X^*)}  \lesssim  \nrm{x^*}_{X^*}, \qquad x^* \in D(A^*)\cap R(A^*).
  \end{equation*}
Since $\psi \in H^2(\Sigma_{\sigma'})$ for $\omega(A)<\sigma'<\sigma$ by \eqref{eq:H1inHinfty}, we have by the Calder\'on reproducing formula \eqref{eq:calderonrepnontrivial}
   \begin{equation*}
     c\,x = \int_0^\infty \psi(tA)\psi^*(tA)x \frac{\ddn t}{t}
   \end{equation*}
   where $\psi^*(z):= \overline{\psi(\overline{z})}$ and $c = \int_0^\infty \abs{\psi(t)}^2\frac{\ddn t}{t}>0$. Applying Proposition \ref{proposition:alphaholder}, we deduce for any $x^*\in D(A^*)\cap R(A^*)$
   \begin{align*}
     \abs{\ip{x,x^*}} &\leq c^{-1} \, \int_0^\infty \abs{\ip{ \psi(tA)x, \psi^*(tA)^*x^*}} \frac{\ddn t}{t}\\
     &\leq c^{-1} \, \nrm{\psi(\cdot A)x}_{\alpha(\R_+,\frac{\ddn t}{t};X)}  \nrm{\psi^*(\cdot A)^* x}_{\alpha^*(\R_+,\frac{\ddn t}{t};X)}\\
     &\lesssim  \nrm{\psi(\cdot A)x}_{\alpha(\R_+,\frac{\ddn t}{t};X)} \nrm{x^*}_{X^*}.
   \end{align*}
   So since $D(A^*)\cap R(A^*)$ is norming for $X$, this yields
   \begin{equation*}
     \nrm{x}_X \lesssim  \nrm{\psi(\cdot A)x}_{\alpha(\R_+,\frac{\ddn t}{t};X)}.
   \end{equation*}
   Another application of Proposition \ref{proposition:equivalencediscretecontinuous} yields the same inequality for the discrete generalized square function norm, finishing the proof.
\end{proof}

\subsection*{The equality of the angles of almost $\alpha$-sectoriality and $H^\infty$-calculus}
To conclude this section, we note that, by combining Theorem \ref{theorem:squaretocalculus} and Theorem \ref{theorem:calculustosuqare}, we are now able to show the equality of the almost $\alpha$-sectoriality angle and the $H^\infty$-calculus angle of a sectorial operator $A$. Using the global, ideal, unconditionally stable Euclidean structure $\ell^g$ this in particular reproves the equality of the $\BIP$ and bounded $H^\infty$-calculus angles, originally shown in \cite[Theorem 5.4]{CDMY96}. Furthermore if $A$ is $\alpha$-sectorial this implies $\omega_{H^\infty}(A) = {\omega}_\alpha(A)$, which for the $\gamma$-structure was shown in \cite{KW01}.

\begin{corollary}\label{corollary:equalangles}
  Let $A$ be an $\alpha$-sectorial operator on $X$ with a bounded $H^\infty$-calculus and assume that $\alpha$ is unconditionally stable.
  \begin{enumerate}[(i)]
    \item \label{it:angle1} If $A$ almost $\alpha$-sectorial, then $\omega_{H^\infty}(A) \leq \tilde{\omega}_\alpha(A)$.
    \item \label{it:angle2}  If $\alpha$ is ideal, then $A$ is almost $\alpha$-sectorial with
  $$\omega_{H^\infty}(A) = \omega_{\BIP}(A) = \tilde{\omega}_\alpha(A).$$
  \end{enumerate}
\end{corollary}

\begin{proof}
For \ref{it:angle1}  we know by Theorem \ref{theorem:calculustosuqare} that for $\omega_{H^\infty}<\sigma<\pi$ and a non-zero $\psi \in H^1(\Sigma_\sigma)$
  \begin{equation*}
    \nrm{x} \simeq \nrmb{t\mapsto \psi(tA)x}_{\alpha(\R_+,\frac{\ddn t}{t};X)}, \qquad x \in D(A)\cap R(A).
  \end{equation*}
  Thus, by Theorem \ref{theorem:squaretocalculus}, we know that $\omega_{H^\infty}(A) \leq \tilde{\omega}_{\alpha}(A)$. \ref{it:angle2}  follows from  \eqref{eq:HinftyBIP} and Proposition \ref{proposition:BIPalmostalpha}.
\end{proof}

\section{Dilations of sectorial operators}\label{section:dilation}
Extending a dilation result of Sz-Nagy \cite{Na47}, Le Merdy showed in \cite{Le96, Le98} that a sectorial operator $A$ on a Hilbert space $H$ with $\omega(A)<\frac\pi2$ has a bounded $H^\infty$-calculus if and only if the associated semigroup $(\ee^{-tA})_{t\geq 0}$ has a dilation to a unitary group $(U(t))_{t \in \R}$ on a larger Hilbert space $\widetilde{H}$, i.e. $A$ has a dilation to a normal operator $\widetilde{A}$ on $\widetilde{H}$. By the spectral theorem for normal operators (see e.g. \cite[Theorem X.4.19]{Co90}) we can think of $\widetilde{A}$ as a multiplication operator.

In this section we will use the generalized square functions to characterize the boundedness of the $H^\infty$-calculus of a sectorial operator $A$ on a general Banach space $X$ in terms of dilations. We say that a semigroup $(U(t))_{t \geq 0}$ on a Banach space $\widetilde{X}$ is a dilation of $(\ee^{-tA})_{t\geq 0}$  if there is an isomorphic embedding $J\colon X \to \widetilde{X}$ and a bounded operator $Q \colon \widetilde{X}\to X$ such that
\begin{equation*}
  \ee^{-tA} = QU(t)J, \qquad t \geq 0.
\end{equation*}
A sectorial operator $\widetilde{A}$ on $\widetilde{X}$ is called a dilation of $A$ if there are such $J$ and $Q$ with
\begin{equation*}
         R(\lambda,A) = Q R(\lambda,\widetilde{A})J,\qquad \lambda \in \C \setminus \overline{\Sigma}_{\max\ha{\omega(A), \omega(\widetilde{A})}}.
\end{equation*}
This can be expressed in terms of the commutation of the following diagrams
\begin{center}
\begin{tikzcd}
 \widetilde{X} \arrow{r}{U(t)} & \widetilde{X}      \arrow{d}{Q}        \\
 X \arrow{u}{J} \arrow{r}{\ee^{-tA}} &X
\end{tikzcd}\hspace{1cm}\begin{tikzcd}
 \widetilde{X} \arrow{r}{R(\lambda,\widetilde{A})} & \widetilde{X}      \arrow{d}{Q}        \\
 X \arrow{u}{J} \arrow{r}{R(\lambda,A)} &X
\end{tikzcd}
\end{center}
Taking $t=0$ in the semigroup case we see that $QJ=I$ and $JQ$ is a bounded projection of $\widetilde{X}$ onto $R(J)$. The same conclusion can be drawn in the sectorial operator case by
\begin{equation*}
  x = \lim_{\lambda \to \infty} \lambda (\lambda+A)^{-1}x = \lim_{\lambda \to \infty} \lambda Q (\lambda+\widetilde{A})^{-1}Jx = QJx, \qquad x \in X,
\end{equation*}

We will choose $\widetilde{X} = \alpha(\R;X)$ for an unconditionally stable Euclidean structure $\alpha$ on $X$  and for $s>0$ consider the multiplication operator $\mc{M}_s$ given by
\begin{equation*}
  \mc{M}_sg(t):=(it)^{\frac{2}{\pi}s} g(t), \qquad t \in \R
\end{equation*}
for strongly measurable $g\colon \R \to X$ such that $g,\mc{M}_sg \in \alpha(\R;X)$. Note that the spectrum of $\mc{M}_s$ is given by
\begin{equation*}
  \sigma(\mc{M}_s) = \partial \Sigma_{ s}
\end{equation*}
 and that for a bounded measurable function $f\colon \Sigma_{ \sigma} \to \C$ with $s<\sigma<\pi$ the operator $f(\mc{M}_s)$ defined by
 \begin{equation*}
   f(\mc{M}_s)g(t) = f\hab{(it)^{\frac{2}{\pi}s}}g(s), \qquad t \in \R
 \end{equation*}
 extends to a bounded operator on $\alpha(\R;X)$ by Example \ref{example:pointwisemulti}. Hence $\mc{M}_s$ has a bounded Borel functional calculus and is therefore a worthy replacement for normal operators on a Hilbert space.

 If $\mc{M}_s$ on $\alpha(\R;X)$ is a dilation of a sectorial operator $A$ on $X$ for $\omega(A)<s<\pi$, we have for $f\in H^1(\Sigma_\sigma)\cap H^\infty(\Sigma_\sigma)$ with $s<\nu<\sigma<\pi$ that
  \begin{equation*}
    f(A) =  {\frac{1}{2\pi i}\int_{\Gamma_{\nu}}f(z)QR(z,\mc{M}_s)J\dd z},
    = Qf(\mc{M}_s)J,
  \end{equation*}
  where $f(\mc{M}_s)$ can either be interpreted in the Borel functional calculus sense or the Dunford calculus sense. Therefore, the fact that $\mc{M}_s$ is a dilation of $A$ implies that $A$ has a bounded $H^\infty$-calculus and we have for $f \in H^\infty(\Sigma_\sigma)$
  \begin{equation}\label{eq:dilationcalculusHinfty}
    f(A) = Qf(\mc{M}_s)J.
  \end{equation}
 The converse of this statement is the main result in this section, which characterizes the boundedness of the $H^\infty$-calculus of $A$ in terms of dilations.

 \begin{theorem}\label{theorem:dilationresolvent}
   Let $A$ be a sectorial operator on $X$ and $\omega(A)<s<\pi$. Consider the following statements:
   \begin{enumerate}[(i)]
   \item \label{it:dilation1} $A$ has a bounded $H^\infty(\Sigma_\sigma)$-calculus for some $\omega(A)<\sigma<s$.
   \item \label{it:dilation2} The operator $\mc{M}_s$ on $\alpha(\R;X)$ is a dilation of $A$ for all unconditionally stable Euclidean structures $\alpha$ on $X$.
   \item \label{it:dilation3} The operator $\mc{M}_s$ on $\alpha(\R;X)$ is a dilation of $A$ for some Euclidean structure $\alpha$ on $X$.
   \item \label{it:dilation4} $A$ has a bounded $H^\infty(\Sigma_\sigma)$-calculus for all $s<\sigma<\pi$.
   \end{enumerate}
   Then \ref{it:dilation1}$\implies$\ref{it:dilation2}$\implies$\ref{it:dilation3}$\implies$\ref{it:dilation4}. Moreover, if $A$ is almost $\alpha$-sectorial with $\tilde{\omega}_{\alpha}(A)<s$ for some unconditionally stable Euclidean structure $\alpha$, then \ref{it:dilation4}$\implies$\ref{it:dilation1}.
 \end{theorem}

Since $\gamma(\R;H)=L^2(\R;H)$ and $\omega(A) = \tilde{\omega}_\gamma(A)$ for a sectorial operator $A$ on a Hilbert space $H$, Theorem \ref{theorem:dilationresolvent} extends the classical theorem on Hilbert spaces by Le Merdy \cite{Le96}. If $X$ has finite cotype, the $\gamma$-structure is unconditionally stable by Proposition \ref{proposition:unconditionallystable}, so we also recover the main result from Fr\"ohlich and the third author \cite[Theorem 5.1]{FW06}. For further results on dilations in $\UMD$ Banach spaces and $L^p$-spaces we refer to \cite{FW06} and \cite {AFL17}.

\begin{proof}[Proof of Theorem \ref{theorem:dilationresolvent}]
  For \ref{it:dilation1}$\implies$\ref{it:dilation2} we may assume without loss of generality that $s=\frac{\pi}{2}$, as we can always rescale by defining a sectorial operator $B := A^{\frac{\pi}{2s}}$ and using Proposition \ref{proposition:hinftys} and the observation that $(\mc{M}_{s})^{\frac{\pi}{2s}} = \mc{M}_{\frac{\pi}{2}}$.

   Define
  for $x \in D(A)\cap R(A)$
  \begin{equation*}
    Jx(t):= A^{1/2}R(it,A)x \qquad t \in \R.
  \end{equation*}
  Setting $\psi_{\pm}(z) = \frac{z^{1/2}}{\pm i-z}$, we have
  \begin{align*}
    Jx(t) &= t^{-1/2}\psi_+(t^{-1}A)x, &&t \in \R_+,\\
    Jx(-t) &= t^{-1/2}\psi_-(t^{-1}A)x, &&t \in \R_+.
  \end{align*}
Therefore $Jx \in \alpha(\R;X)$ by Lemma \ref{lemma:squarefunctionalpha} and using Proposition \ref{proposition:operatoronfunctions} we obtain
  \begin{align*}
    \nrm{Jx}_{\alpha(\R;X)} \simeq \nrm{\psi_+(tA)x}_{\alpha(\R_+,\frac{\ddn t}{t};X)} + \nrm{\psi_-(tA)x}_{\alpha(\R_+,\frac{\ddn t}{t};X)}.
  \end{align*}
  Now by Theorem \ref{theorem:calculustosuqare} the bounded $H^\infty(\Sigma_\sigma)$-calculus of $A$ implies that
    \begin{align*}
    \nrm{Jx}_{\alpha(\R;X)} \simeq \nrm{x},
  \end{align*}
  so by density $J$ extends to an isomorphic embedding $J\colon X \to \alpha(\R;X)$. Next take $g \in \alpha(\R;X)$ such that $\nrm{g(t)}_X \lesssim (1+\abs{t})^{-1}$ and define the operator
  \begin{align*}
    Qg&:= \frac{1}{\pi}\int_\R A^{1/2}R(-it,A)g(t) \dd t \in X,
  \end{align*}
  where the integral converges in the Bochner sense in $X$, since
  \begin{equation*}
    \nrm{A^{1/2}R(-it,A)} \lesssim \frac{1}{\abs{t}^{1/2}}, \qquad t \in \R.
  \end{equation*}
  By the $\alpha$-H\"older inequality (Proposition \ref{proposition:alphaholder}) and Theorem \ref{theorem:calculustosuqare} applied to the moon dual $A^\sharp$ on $X^\sharp$ equipped with the Euclidean structure induced by $\alpha^*$, we have for $x^* \in D(A^*)\cap R(A^*)$ that
  \begin{align*}
    \abs{\ip{Qg, x^*}} &\leq \frac1\pi \int_0^\infty  \absb{\ipb{g(t),t^{-1/2}\hab{\psi_{+}(t^{-1}A)^* +\psi_{-}(t^{-1}A)^*} x^*}}\dd t\\
    &\lesssim  \, \nrm{g}_{\alpha(\R;X)}\hab{ \nrm{\psi_+(tA)^*x^*}_{\alpha^*(\R_+,\frac{\ddn t}{t};X^*)} + \nrm{\psi_-(tA)^*x^*}_{\alpha^*(\R_+,\frac{\ddn t}{t};X^*)}}\\
    &\lesssim \nrm{g}_{\alpha(\R;X)} \nrm{x^*}.
  \end{align*}
  Since $D(A^*)\cap R(A^*)$ is norming for $X$ and using Proposition \ref{proposition:densealphaspace}, it follows that $Q$ extends to a bounded operator $Q\colon\alpha(\R;X) \to X$.
  To show that $\mc{M}_{\frac{\pi}{2}}$ on $\alpha(\R;X)$ is a dilation of $A$ we will show that
  \begin{equation}\label{eq:dilationproof}
    R(\lambda,A) = QR(\lambda,\mc{M}_{\frac{\pi}{2}})J, \qquad  \lambda \in \C \setminus \overline{\Sigma}_{\frac{\pi}{2}}.
  \end{equation}
  First note that for $t\in \R$ we have by the resolvent identity
  \begin{align*}
    AR(it,A)R(-it,A) &= -\frac{1}{2it}(AR(it,A)-AR(-it,A))\\
    &=  -\frac{1}{2}(R(it,A)+R(-it,A)).
  \end{align*}
 So since
 \begin{equation*}
   \sup_{t \in \R}\, \nrm{A^{1/2}R(it,A)x}_X<\infty, \qquad x \in R(A^{1/2})
 \end{equation*}
 by the resolvent equation, we have for $x \in D(A)\cap R(A)$  and $\lambda \in \C \setminus \overline{\Sigma}_{\frac{\pi}{2}}$ that
 \begin{equation*}
   \nrmb{\frac{1}{\lambda - it}A^{1/2}R(it,A)x}_X \lesssim (1+\abs{t})^{-1}, \qquad t\in \R
 \end{equation*}
 and therefore
  \begin{align*}
    QR(\lambda,\mc{M}_{\frac{\pi}{2}} )Jx &= \frac{1}{\pi}\int_\R A^{1/2}R(-it,A)\frac{1}{\lambda - it}A^{1/2}R(it,A)x \dd t\\
    &=-\frac{1}{2\pi}\int_{\R} \frac{1}{\lambda - it} R(it,A)x \dd t- \frac{1}{2\pi}\int_{\R} \frac{1}{\lambda - it}  R(-it,A)x \dd t\\
    &=\frac{1}{2\pi i}\int_{\Gamma_{\frac{\pi}{2}}} \frac{1}{\lambda - z} R(z,A)x \dd z+\frac{1}{2\pi i}\int_{\Gamma_{\frac{\pi}{2}}} \frac{1}{\lambda + z} R(z,A)x \dd z\\
    &= R(\lambda,A)x,
  \end{align*}
  where the last step follows from \cite[Example 10.2.9]{HNVW17} and Cauchy's theorem. This proves \eqref{eq:dilationproof} by density.

  The implication \ref{it:dilation2}$\implies$\ref{it:dilation3} follows directly from the fact that the global lattice structure $\ell^g$ is unconditionally stable on any Banach space $X$ by Proposition \ref{proposition:unconditionallystable}. Implication  \ref{it:dilation3}$\implies$\ref{it:dilation4} is a direct consequence of \eqref{eq:dilationcalculusHinfty}. Finally, if $A$ is almost $\alpha$-sectorial with $\tilde{\omega}_\alpha(A)<s$ for some unconditionally stable Euclidean structure $\alpha$, \ref{it:dilation4}$\implies$\ref{it:dilation1} is a consequence of Corollary \ref{corollary:equalangles}.
\end{proof}

As a direct corollary we obtain a dilation result for the semigroup $(\ee^{-tA})_{t\geq 0}$. Note that we could use any $\mc{M}_s$ with $\tilde{\omega}_\alpha(A)<s<\pi$, but only $\mc{M}_{{\frac{\pi}{2}}}$ yields a group of isometries.

\begin{corollary}
   Let $A$ be an almost $\alpha$-sectorial operator on $X$ with $\tilde{\omega}_\alpha(A)<\frac\pi2$ and assume that $\alpha$ is unconditionally stable. Then the following are equivalent
   \begin{enumerate}[(i)]
   \item \label{it:dilationsemi1} $A$ has a bounded $H^\infty$-calculus.
   \item \label{it:dilationsemi2} The group of isometries $(U(t))_{t \in \R}$
    on $\alpha(\R;X)$ given by $U(t) = \ee^{-t\mc{M}_{\frac{\pi}{2}}}$ is a dilation of the semigroup $(\ee^{-tA})_{t\geq 0}$.
   \end{enumerate}
\end{corollary}

\begin{proof}
  The implication \ref{it:dilationsemi1}$\Rightarrow$\ref{it:dilationsemi2} follows directly from Theorem \ref{theorem:dilationresolvent} and \eqref{eq:dilationcalculusHinfty} for $f_t(z) = \ee^{-tz}$ with $t \geq 0$. For the implication \ref{it:dilationsemi2}$\Rightarrow$\ref{it:dilationsemi1} we note that from the Laplace transform (see \cite[Proposition G.4.1]{HNVW17})
  \begin{equation*}
    R(\lambda,A)x = -\int_0^\infty \ee^{\lambda t}\ee^{-tA}x\dd t, \qquad \re \lambda<0, \, x\in X,
  \end{equation*}
 and a similar equation for $\mc{M}_{\frac{\pi}{2}}$ we obtain that $\mc{M}_{\frac{\pi}{2}}$ on $\alpha(\R;X)$ is a dilation of $A$, which implies the statement by  Theorem \ref{theorem:dilationresolvent}.
\end{proof}

To conclude this section, we note that for Banach lattices we can actually construct a dilation of $(\ee^{-tA})_{t\geq 0}$ consisting of \emph{positive} isometries. This provides a partial converse to the result of the third author in  \cite[Remark 4.c]{We01b}  that the negative generator of any bounded analytic semigroup of positive contractions on $L^p$ has a bounded $H^\infty$-calculus with  $\omega_{H^\infty}(A)<\frac{\pi}{2}$.
 For more elaborate results in this direction  and a full $L^p$-counterpart to the Hilbert space result from \cite{Le98} we refer to \cite{AFL17,Fa14b}
\begin{corollary}
   Let $A$ be a sectorial operator on an order-continuous Banach function space $X$ and suppose that $A$ has a bounded $H^\infty$-calculus with $\omega_{H^\infty}(A)<\frac{\pi}{2}$. Then the semigroup $(\ee^{-tA})_{t\geq 0}$ has a dilation to a positive $C_0$-group of isometries $(U(t))_{t \in \R}$ on $\ell^2(\R;X)$.
\end{corollary}

\begin{proof}
  Let $J$ and $Q$ be the embedding and projection operator of the dilation in Theorem \ref{theorem:dilationresolvent}\ref{it:dilation2} with $\alpha = \ell^2$. Let $\mc{F}$ denote the Fourier transform on $\ell^2(\R;X)$ and define
  \begin{equation*}
         R(\lambda,A) = Q R(\lambda,\mc{M}_{\frac{\pi}{2}})J = Q_{\mc{N}} R(\lambda,\mc{N})J_{\mc{N}},\qquad \lambda \in \C \setminus \overline{\Sigma}_{\frac{\pi}{2}},
\end{equation*}
where $$\mc{N} := \mc{F}^{-1}\mc{M}_{\frac{\pi}{2}}\mc{F} = \frac{1}{2\pi}\frac{\dd}{\dd t}$$ on $\ell^2(\R;X)$, $J_{\mc{N}}:=\mc{F}^{-1}J$ and $Q_{\mc{N}}:=Q\mc{F}$. Since the Fourier transform is bounded on $\ell^2(\R;X)$ by Example \ref{example:fourier}, we obtain that $(\ee^{-t\mc{N}})_{t \in \R}$ is a dilation of $(\ee^{-tA})_{t\geq 0}$ by  \eqref{eq:dilationcalculusHinfty} for $f_t(z) = \ee^{-tz}$ with $t \geq 0$.
Now the corollary follows from the fact that $(\ee^{-t\mc{N}})_{t \in \R}$ is the translation group on $\ell^2(\R;X)$, which is a positive $C_0$-group of isometries by the order-continuity of $X$, the dominated convergence theorem and Proposition \ref{proposition:l2representation}.
\end{proof}

\section{A scale of generalized square function spaces}\label{section:scalespaces}
For a sectorial operator $A$ on the Banach space $X$ the scale of homogeneous fractional domain spaces $\dot{X}_{\theta,A}$ reflects many properties of $X$ and is very useful in spectral theory.
However, the operators on $\dot{X}_{\theta,A}$ induced by $A$
may not have a bounded $H^\infty$-calculus or $\BIP$, the scale $\dot{X}_{\theta,A}$ may not be an interpolation scale and even for a differential operator $A$ they may not be easy to identify as function spaces.
Therefore one also considers e.g. the real interpolation spaces $(X,D(A))_{\theta,q}$ for $q \in [1,\infty]$, on which the restriction of an invertible sectorial operator $A$ always has a bounded $H^\infty$-calculus (see \cite{Do99}), and which, in the case of $A=-\Delta$ on $L^p(\R^d)$, equal the Besov spaces $B_{p,q}^{2\theta}(\R^d)$. However, these spaces almost never equal the fractional domain scale $\dot{X}_{\theta,A}$ (see \cite{KW05}).

In this section we will introduce a scale of intermediate spaces $H_{\theta,A}^\alpha$ which are defined in terms of the generalized square functions of Section \ref{section:squarefunctions}. These spaces have, under reasonable assumptions on $A$ and the Euclidean structure $\alpha$, the following advantages:
\begin{enumerate}[(i)]
\item They are ``close'' to the homogeneous fractional domain spaces, i.e. for $\eta_1<\theta<\eta_2$ we have continuous embeddings
\begin{equation*}
  \dot{X}_{\eta_1,A} \cap \dot{X}_{\eta_2,A} \hookrightarrow  H^{\alpha}_{\theta, A} \hookrightarrow \dot{X}_{\eta_1,A}+\dot{X}_{\eta_2,A},
\end{equation*}
see Theorem \ref{theorem:functionspaceclose}.
\item The sectorial operator $A|_{H_{\theta,A}^\alpha}$ induced by $A$ on $H_{\theta,A}^\alpha$ has a bounded $H^\infty$-calculus, see Theorem \ref{theorem:Hproperties}.
\item The spaces $H_{\theta,A}^\alpha$ and $\dot{X}_{\theta,A}$ are isomorphic essentially if and only if $A$ has a bounded $H^\infty$-calculus (see Theorem \ref{theorem:Hproperties}). In this case the spaces $H_{\theta,A}^\alpha$ provide a generalized form of the Littlewood--Paley decomposition for $\dot{X}_{\theta,A}$, which enables certain harmonic analysis methods in the spectral theory of $A$. In particular, if $A = -\Delta$ on $L^p(\R^d)$ with $1<p<\infty$, then $H_{\theta,A}^{\gamma} = \dot{H}^{2\theta,p}(\R^d)$ is a Riesz potential space.
\item They form an interpolation scale for the complex interpolation method and are realized as $\alpha$-interpolation spaces of the homogeneous fractional domain spaces. (see Theorems \ref{theorem:interpolateHcomplex} and \ref{theorem:interpolateH}).
\end{enumerate}

Let us fix a framework to deal with the fractional domain spaces of a sectorial operator $A$ on $X$.
Let $\theta \in \R$ and $m \in \N$ with $\abs{\theta}<m$.
We define the homogeneous fractional domain space  $\dot{X}_{\theta,A}$ as the completion of $D(A^\theta)$ with respect to the norm $x \mapsto \nrm{A^\theta x}_X$. We summarize a few properties of $\dot{X}_{\theta,A}$ in the following proposition. We refer to \cite[Section 15.E]{KW04} or \cite[Chapter 6]{Ha06} for the proof.
\begin{proposition}\label{proposition:homspaceprop}
Let $A$ be a sectorial operator on $X$ and take $\theta \in \R$.
  \begin{enumerate}[(i)]
\item \label{it:homspaceprop1} $D(A^m)\cap R(A^m)$ is dense in $\dot{X}_{\theta,A}$ for $m \in \N$ with $\abs{\theta}<m$.
\item \label{it:homspaceprop2} For $\eta_1,\eta_2\geq 0$ we have $\dot{X}_{\eta_1,A}\cap \dot{X}_{-\eta_2,A}=D(A^{\eta_1})\cap R(A^{\eta_2})$.
\item \label{it:homspaceprop3} For $\eta_1<\theta<\eta_2$ we have the continuous embeddings \begin{align*}
   \dot{X}_{\eta_1,A}\cap \dot{X}_{\eta_2,A} &\hookrightarrow \dot{X}_{\theta,A}\hookrightarrow \dot{X}_{\eta_1,A}+ \dot{X}_{\eta_2,A}
 \end{align*}
\end{enumerate}
\end{proposition}

\subsection*{The spaces $H_{\theta,A}^\alpha$ and their properties.}
Now let us turn to the spaces $H^{\alpha}_{\theta,A}$,
 for which we first introduce a version depending on a choice of $\psi \in H^1(\Sigma_\sigma)$.
  Let $A$ be a sectorial operator on $X$. Assume either of the following conditions:
  \begin{itemize}
    \item $\alpha$ is ideal and set $\omega_A:=\omega(A)$.
    \item $A$ is almost $\alpha$-sectorial and set $\omega_A:=\tilde{\omega}_\alpha(A)$.
  \end{itemize}
Let $\omega_A<\sigma<\pi$, $\psi \in H^1(\Sigma_\sigma)$  and take $\theta \in \R$ and $m \in \N$ with $\abs{\theta}+1<m$.
  We define $H_{\theta,A}^\alpha(\psi)$ as the completion of $D(A^m)\cap R(A^m)$  with respect to the norm
  \begin{equation*}
    x \mapsto \nrmb{ \psi(\cdot A)A^\theta x}_{\alpha(\R_+,\frac{\ddn t}{t};X)}.
  \end{equation*}
 We write
 \begin{equation*}
     H^\alpha_{\theta,A}:= H^\alpha_{\theta,A}(\varphi),\qquad \varphi(z) := z^{1/2}(1+z)^{-1}.
   \end{equation*}

 By Lemma \ref{lemma:squarefunctionalpha} we know that $\psi(\cdot A)A^\theta x \in \alpha(\R_+,\frac{\ddn t}{t};X)$ for any $x \in D(A^m)\cap R(A^m)$ and $\psi(\cdot A)A^\theta x=0$ if and only if $x=0$ by \eqref{eq:calderonrepnontrivial}, so the norm on $H_{\theta,A}^\alpha(\psi)$ is well-defined.
 \begin{remark}~
 \begin{itemize}
      \item On Hilbert spaces these spaces were already studied in \cite{AMN97}. For the $\gamma$-structure on a Banach space these spaces are implicitly used in \cite[Section 7]{KKW06}  and they are studied in \cite{KW16b} for $0$-sectorial operators with a so-called Mihlin functional calculus.
 In \cite{Ha06b} (see also (\cite[Chapter 6]{Ha06}), these spaces using $L^p(\R_+,\frac{\ddn t}{t};X)$-norms instead of $\alpha(\R_+,\frac{\ddn t}{t};X)$-norms were studied and identified as real interpolation spaces. Furthermore, for Banach function spaces  using $X(\ell^q)$-norms instead of $\alpha(\R_+,\frac{\ddn t}{t};X)$-norms, these spaces were developed in \cite{KU14,Ku15}.
 \item For $\psi \in H^1(\Sigma_\sigma)$
such that $\widetilde{\psi} \in H^1(\Sigma_\sigma)$ for $\widetilde{\psi}(z) := z^{\theta} \psi(z)$, we have the norm equality
  \begin{equation*}
    \nrm{x}_{H_{\theta,A}^\alpha(\psi)} =\nrmb{t \mapsto t^{-\theta}\widetilde{\psi}(tA) x}_{\alpha(\R_+,\frac{\ddn t}{t};X)}
  \end{equation*}
  for $x \in D(A^m)\cap R(A^m)$.
Viewing $\widetilde{\psi}(tA)$ as a generalized continuous Littlewood-Paley decomposition, this connects our scale of spaces to the more classical fractional smoothness scales.
 \end{itemize}
 \end{remark}

 Before turning to more interesting results, we will first prove that the $H^\alpha_{\theta,A}(\psi)$-spaces are independent of the parameter $m>\abs{\theta}+1$. This is the reason why we do not include it in our notation.

 \begin{lemma}
   The definition of $H^\alpha_{\theta,A}(\psi)$ is independent of $m>\abs{\theta}+1$.
 \end{lemma}
 \begin{proof}
   It suffices to show that $D(A^{m+1})\cap R(A^{m+1})$ is dense in $H^\alpha_{\theta,A}(\psi)$, which is defined as the completion of $D(A^{m})\cap R(A^{m})$. Fix $x
    \in D(A^m)\cap R(A^m)$ and let $\varphi_n$ as in \eqref{eq:phin}. Then $\varphi_n(A)$ maps $D(A^{m})\cap R(A^{m})$ into $D(A^{m+1})\cap R(A^{m+1})$. We consider two cases:
 \begin{itemize}
   \item If $\alpha$ is ideal, then since $\varphi_n(A)x \to x$ in $X$ we have  $$\psi(\cdot A)A^\theta \varphi_n(A)x \to \psi(\cdot A)A^\theta x$$ in $\alpha(\R_+,\frac{\ddn t}{t};X)$ by Proposition \ref{proposition:alphaspaceconvergence}\ref{it:convergenceproperty5}.
   \item If $A$ is almost $\alpha$-sectorial, let $y \in D(A^{m-1}\cap R(A^{m-1})$ be such that $x = \varphi(A)y$ with $\varphi(z) = {z}{(1+z)^{-2}}$. Since we have for any $n\in \N$ and
   $z \in \Sigma_\sigma$ that
       \begin{align*}
         \abs{\varphi(z)(\varphi_n(z)-1)} &= \abss{ \frac{z}{n+z} \frac{z}{(1+z)^2}- \frac{\frac{1}{z}}{n+\frac{1}{z}} \frac{z}{(1+z)^2}}\\
         &\lesssim  \frac{1}{n}\has{\frac{z^2}{(1+z)^2}+\frac{1}{(1+z)^2}}\leq  \frac{2}{n},
       \end{align*}
we deduce by Proposition \ref{proposition:equivsquarefunctions} that
       \begin{align*}
       \lim_{n \to \infty} \nrm{\psi(\cdot A)&A^\theta \varphi_n(A)x - \psi(\cdot A)A^\theta x}_{\alpha(\R_+,\frac{\ddn t}{t};X)}\\ &\leq \lim_{n \to \infty} \nrm{\varphi(\varphi_n -1)}_{H^\infty(\Sigma_\sigma)} \nrm{\psi(\cdot A)A^\theta y}_{\alpha(\R_+,\frac{\ddn t}{t};X)} = 0
       \end{align*}
 \end{itemize}
Thus we obtain in both cases  that $D(A^{m+1})\cap R(A^{m+1})$ is dense in $H_{\theta,A}^\alpha(\psi)$.
 \end{proof}

If $A$ is almost $\alpha$-sectorial, then the spaces $H^\alpha_{\theta,A}(\psi)$ are independent of the choice of $\psi \in H^1(\Sigma_\sigma)$ as well by Proposition \ref{proposition:equivsquarefunctions} and thus all isomorphic to $H^\alpha_{\theta,A}$.

\bigskip

  We start our actual analysis of the $H_{\theta,A}^\alpha(\psi)$-spaces by proving embeddings that show that they are ``close'' to the fractional domain spaces $\dot{X}_{\theta,A}$.

\begin{theorem}\label{theorem:functionspaceclose}
  Let $A$ be a sectorial operator on $X$. Assume either of the following conditions:
  \begin{itemize}
    \item $\alpha$ is ideal and set $\omega_A:=\omega(A)$.
    \item $A$ is almost $\alpha$-sectorial and set $\omega_A:=\tilde{\omega}_\alpha(A)$.
  \end{itemize}
Let $\omega_A<\sigma<\pi$ and take a non-zero $\psi \in H^1(\Sigma_\sigma)$, then for $\eta_1<\theta<\eta_2$ we have continuous embeddings
    \begin{equation*}
      \dot{X}_{\eta_1,A} \cap \dot{X}_{\eta_2,A} \hookrightarrow  H^{\alpha}_{\theta, A}(\psi) \hookrightarrow \dot{X}_{\eta_1,A}+\dot{X}_{\eta_2,A}
    \end{equation*}
\end{theorem}

\begin{proof}
By density it suffices to show the embeddings for $x \in D(A^m)\cap R(A^m)$ for some $m \in \N$ with $\eta_1,-\eta_2<m-1$. Set
$$\epsilon = \min\cbrace{\theta-\eta_1,\eta_2-\theta}$$ and define $\varphi(z)=z^\epsilon(1+z^\epsilon)^{-2}$. Then by \eqref{eq:Aepsilonrange} we have that $$\varphi(A)^{-1}\colon D(A^\epsilon)\cap R(A^\epsilon) \mapsto X$$ is given by $\varphi(A)^{-1} = A^{\epsilon}+A^{-\epsilon}+2I$.
For the first embedding we have by Proposition \ref{proposition:functionspaceclose} and Proposition \ref{proposition:homspaceprop}\ref{it:homspaceprop3}
\begin{align*}
  \nrm{x}_{H^\alpha_{\theta,A}(\psi)} &= \nrm{\psi(\cdot A)\varphi(A)\varphi(A)^{-1}A^{\theta}x}_{\alpha(\R_+,\frac{\ddn t}{t};X)}\\
  &\lesssim \nrm{A^{\theta+\epsilon}x+A^{\theta-\epsilon}x+2A^{\theta}x}_{X} \\
  &\lesssim \nrm{x}_{\dot{X}_{\eta_1,A} \cap \dot{X}_{\eta_2,A} }
\end{align*}
For the second embedding we have by $A^\theta x \in D(A) \cap R(A)$, Proposition \ref{proposition:homspaceprop}\ref{it:homspaceprop3} and Proposition \ref{proposition:functionspaceclose}
\begin{align*}
  \nrm{x}_{\dot{X}_{\eta_1,A}+\dot{X}_{\eta_2,A}} &\lesssim \nrmb{\varphi(A)(A^\epsilon+A^{-\epsilon}+2I)x}_{\dot{X}_{\theta-\epsilon,A}+\dot{X}_{\theta+\epsilon,A}}\\
  &\leq  \nrm{\varphi(A)A^\epsilon x}_{\dot{X}_{\theta-\epsilon,A}} + \nrm{\varphi(A)A^{-\epsilon} x}_{\dot{X}_{\theta+\epsilon,A}}+2\nrm{\varphi(A) x}_{\dot{X}_{\theta,A}}
   \\&\lesssim \nrm{\psi(\cdot A)A^\theta x}_{\alpha(\R_+,\frac{\ddn t}{t};X)},
\end{align*}
which finishes the proof of the theorem.
\end{proof}

\subsection*{The sectorial operators $A|_{H^\alpha_{\theta,A}(\psi)}$ and their properties}
In the scale of fractional domain spaces $\dot{X}_{\theta,A}$ one can define a sectorial operator $A|_{\dot{X}_{\theta,A}}$ on $\dot{X}_{\theta,A}$ for $\theta \in \R$, which coincides with $A$ on $$D(A^{-\theta} A A^\theta) = \dot{X}_{\min\cbrace{\theta,0},A}\cap \dot{X}_{1+\max\cbrace{\theta,0},A},$$ see \cite[Proposition 15.24]{KW04}. We would like to have a similar situation for the spaces
$H^\alpha_{\theta,A}(\psi)$, which is the content of the following proposition.

\begin{proposition}\label{proposition:AdefHalpha}
  Let $A$ be a sectorial operator on $X$ and take $\eta_1<\theta <\eta_2$. Assume either of the following conditions:
  \begin{itemize}
    \item $\alpha$ is ideal and set $\omega_A:=\omega(A)$.
    \item $A$ is almost $\alpha$-sectorial and set $\omega_A:=\tilde{\omega}_\alpha(A)$.
  \end{itemize}
  Let $\omega_A<\sigma<\pi$ and $\psi \in H^1(\Sigma_\sigma)$. Then there is a sectorial operator ${A}|_{H^\alpha_{\theta,A}(\psi)}$ on $H_{\theta,A}^{\alpha}(\psi)$ with $\omega({A}|_{H^\alpha_{\theta,A}(\psi)})\leq  \omega_A$
satisfying
\begin{align*}
  {A}|_{H^\alpha_{\theta,A}(\psi)}x &= A x, &&x \in \dot{X}_{\min\cbrace{\eta_1,0},A}\cap \dot{X}_{1+\max\cbrace{\eta_2,0},A},
  \intertext{and for $\lambda \in \C\setminus \overline{\Sigma}_{\omega_A}$}
    R(\lambda,{A}|_{H^\alpha_{\theta,A}(\psi)})x &= R(\lambda,A) x, &&x \in \dot{X}_{\min\cbrace{\eta_1,0},A}\cap \dot{X}_{\max\cbrace{\eta_2,0},A}.
\end{align*}

\end{proposition}

\begin{proof}
Let $m \in \N$ such that $\abs{\theta}+1<m$. Either by the ideal property of $\alpha$ or by Proposition \ref{proposition:equivsquarefunctions} we have for $x \in D(A^m)\cap R(A^m)$ and $\omega_A<\nu<\pi$ that
\begin{equation}\label{eq:Bsect}
  \nrm{\lambda R(\lambda,A)x}_{H^\alpha_{\theta,A}(\psi)} \leq C_{\nu}\, \nrm{x}_{H^\alpha_{\theta,A}(\psi)}, \qquad \lambda \in \C\setminus \overline{\Sigma}_\nu.
\end{equation}
 Thus, since $D(A^m)\cap R(A^m)$ is dense in $H^\alpha_{\theta,A}(\psi)$, $R(\lambda,A)$ extends to a bounded operator $R_{A}(\lambda)$ on $H^\alpha_{\theta,A}(\psi)$ for $\lambda$ in the open sector
 $$ \Sigma:= \C\setminus \overline{\Sigma}_{\omega_A}.$$

We will construct ${A}|_{H^\alpha_{\theta,A}(\psi)}$ from $R_{A}$, as we also did in the proof of Theorem \ref{theorem:transference}. Note that for $x \in \dot{X}_{\eta_1,A}\cap \dot{X}_{\eta_2,A}$ we have
\begin{align*}
  \lim_{t \to \infty} \, \nrm{tR_{A}(-t)x+x}_{\dot{X}_{\eta_1,A}\cap \dot{X}_{\eta_2,A}} &=0\\
  \lim_{t \to 0}\, \nrm{tR_{A}(-t)x}_{\dot{X}_{\eta_1,A}\cap \dot{X}_{\eta_2,A}} &=0
\end{align*}
and thus, by density and one of the continuous embeddings in Theorem \ref{theorem:functionspaceclose}, we have for all $x \in H^\alpha_{\theta,A}(\varphi)$
\begin{align}
 \label{eq:sectB1} \lim_{t \to \infty} \, \nrm{tR_{A}(-t)x+x}_{H^\alpha_{\theta,A}(\psi)} &=0\\
  \label{eq:sectB2}  \lim_{t \to 0}\, \nrm{tR_{A}(-t)x}_{H^\alpha_{\theta,A}(\psi)} &=0
\end{align}
Using the density of $D(A^m)\cap R(A^m)$ in $H^\alpha_{\theta,A}(\psi)$ we also have the resolvent equation
\begin{align*}
  R_{A}(z) - R_{A}(w) &= (w-z)R_{A}(z)R_{A}(w), \qquad z,w \in \Sigma,
\end{align*}
which in particular implies that if $R_{A}(z)x=0$ for some $z \in \Sigma$ and $x \in H^\alpha_{\theta,A}(\psi)$, then $R_{A}(-t)x=0$ for all $t>0$, so $R_{A}(z)$ is injective by \eqref{eq:sectB1}.

We are now ready to define ${A}|_{H^\alpha_{\theta,A}(\psi)}$. As domain we take the range of $R_{A}(-1)$ and we define
\begin{equation*}
  {A}|_{H^\alpha_{\theta,A}(\psi)}\ha{R_{A}(-1)x}:=-x-R_{A}(-1)x, \qquad x \in H^\alpha_{\theta,A}(\psi).
\end{equation*}
Then by the resolvent equation we have $R(\lambda, {A}|_{H^\alpha_{\theta,A}(\psi)}) = R_{A}(\lambda)$ for $\lambda \in \Sigma$.
Furthermore $A|_{H^\alpha_{\theta,A}(\psi)}$ is injective, has dense domain by \eqref{eq:sectB1} and dense range by \eqref{eq:sectB2} (see \cite[Section III.4.a]{EN00}  and \cite[Proposition 10.1.7(3)]{HNVW17} for the details). So by \eqref{eq:Bsect} we can conclude that $A|_{H^\alpha_{\theta,A}(\psi)}$ is a sectorial operator with $\omega(A|_{H^\alpha_{\theta,A}(\psi)}) \leq \omega_A$.

To conclude take $x \in \dot{X}_{\min\cbrace{\eta_1,0},A}\cap \dot{X}_{1+\max\cbrace{\eta_2,0},A}$ and let $y :=(I+A)x$. Then $y \in H^\alpha_{\theta,A}(\psi) \cap X$ by the embeddings in Proposition \ref{proposition:homspaceprop} and Theorem \ref{theorem:functionspaceclose}
 and thus $R_A(-1)y=R(-1,A)y=-x$. Therefore we have
\begin{equation*}
  A|_{H^\alpha_{\theta,A}(\psi)}{x} = -x+y = Ax.
\end{equation*}
Similarly for
$x \in \dot{X}_{\min\cbrace{\eta_1,0},A}\cap \dot{X}_{\max\cbrace{\eta_2,0},A}$ we have $x \in H^\alpha_{\theta,A}(\psi) \cap X$ and therefore
$$R(\lambda, {A}|_{H^\alpha_{\theta,A}(\psi)})x= R_A(\lambda)x = R(\lambda,A)x, \qquad \lambda \in \Sigma,$$
which concludes the proof.
\end{proof}

If $A$ is almost $\alpha$-sectorial, then the spaces $H^\alpha_{\theta,A}(\psi)$ are all isomorphic to  $H^\alpha_{\theta,A}$ as noted before. In this case the spaces $H^\alpha_{\theta,A}$ and the operators ${A}|_{H^\alpha_{\theta,A}}$ have the following nice properties:

\begin{theorem}\label{theorem:Hproperties}
  Let $A$ be an almost $\alpha$-sectorial operator on $X$.
  \begin{enumerate}[(i)]
  \item \label{it:Hproperties1} $A|_{H^\alpha_{\theta,A}}$ has a bounded $H^\infty$-calculus on $H^\alpha_{\theta,A}$ with $\omega_{H^\infty}(A|_{H^\alpha_{\theta,A}}) \leq \tilde{\omega}_\alpha(A)$ for all $\theta \in \R$.
  \item \label{it:Hproperties2} If $H^\alpha_{\theta,A}=\dot{X}_{\theta,A}$ isomorphically for some $\theta \in \R$, then $A$ has a bounded $H^\infty$-calculus on $X$.
  \item \label{it:Hproperties3} If $A$ has a bounded $H^\infty$-calculus on $X$ and $\alpha$ is unconditionally stable, then $H^\alpha_{\theta,A}=\dot{X}_{\theta,A}$ for all $\theta \in \R$.
  \end{enumerate}
\end{theorem}

\begin{proof}
  Fix $\theta \in \R$ and $m \in \N$ with $\abs{\theta}+1<m$. Let $x \in D(A^m)\cap R(A^m)$, then by Proposition \ref{proposition:AdefHalpha} we know $f(A)x=f(A|_{H^\alpha_{\theta,A}})x$ for $f \in H^\infty(\Sigma_\sigma)$ with $\tilde{\omega}(A)<\sigma<\pi$, so by Proposition \ref{proposition:equivsquarefunctions} we have
  \begin{align*}
    \nrm{f(A|_{H^\alpha_{\theta,A}})x}_{H^\alpha_{\theta,A}} &=\nrm{f(A)\varphi(tA) A^\theta x}_{\alpha(\R_+,\frac{\ddn t}{t};X)} \\&\lesssim \nrm{\varphi(tA) A^\theta x}_{\alpha(\R_+,\frac{\ddn t}{t};X)} =  \nrm{x}_{H^\alpha_{\theta,A}}
  \end{align*}
  with  $\varphi(z) := z^{1/2}(1+z)^{-1}$. Now \ref{it:Hproperties1} follows by the density of $D(A^m)\cap R(A^m)$ in $H^\alpha_{\theta,A}$. For  \ref{it:Hproperties2} we set $y = A^{-\theta} x$ and estimate
  \begin{align*}
    \nrm{f(A)x}_X = \nrm{f(A)y}_{\dot{X}_{\theta,A}} \simeq \nrm{f(A)y}_{H^\alpha_{\theta,A}} \lesssim \nrm{y}_{H^\alpha_{\theta,A}} \simeq \nrm{y}_{\dot{X}_{\theta,A}} = \nrm{x}_X,
  \end{align*}
   from which the claim follows by density.
   Finally \ref{it:Hproperties3} is an immediate consequence of Theorem \ref{theorem:calculustosuqare} and another density argument.
\end{proof}

\subsection*{Interpolation of  square function spaces}
We will now show that there is a rich interpolation theory of the $H^\alpha_{\theta,A}$-spaces.
First of all we note that $H^\alpha_{\theta,A}$ is the fractional domain space of order $\theta$ of the operator $A|_{H^\alpha_{0,A}}$ on $H^\alpha_{0,A}$ and $A|_{H^\alpha_{0,A}}$ has a bounded $H^\infty$-calculus, and thus in particular $\BIP$, by Theorem \ref{theorem:Hproperties}. Therefore it follows from \cite[Theorem 15.28]{KW04} that $H^\alpha_{\theta,A}$ is an interpolation scale for the complex method. We record this observation in the following theorem.

\begin{theorem}\label{theorem:interpolateHcomplex}
  Let $A$ be an almost $\alpha$-sectorial operator on $X$. Let
  $\theta_0,\theta_1 \in \R$, $0<\eta<1$ and $\theta = (1-\eta)\theta_0+\eta\theta_1$.
   Then
   \begin{align*}
     H^\alpha_{\theta,A}&= [H^\alpha_{\theta_0,A},H^\alpha_{\theta_1,A}]_{\eta}
   \end{align*}
   isomorphically.
\end{theorem}

Our main interpolation result will be the interpolation of the fractional domain spaces using the
 the $\alpha$-interpolation method developed in Section \ref{section:interpolation}. We will show that this yields exactly the spaces $H^\alpha_{\theta,A}$. In \cite[Section 7]{KKW06} this result was already implicitly shown for the Rademacher interpolation method, which is connected to the $\gamma$-interpolation method by Proposition \ref{proposition:alpharealmethoddiscrete}.

We know that $A|_{H^\alpha_{\theta,A}}$ has a bounded $H^\infty$-calculus on $H^\alpha_{\theta,A}$ by Theorem \ref{theorem:Hproperties}. Therefore one can view the following theorem as an $\alpha$-interpolation version of the theorem of Dore, which states that $A$ always has a bounded $H^\infty$-calculus on the real interpolation spaces $(\dot{X}_{\theta_0,A}, \dot{X}_{\theta_1,A})_{\eta,q}$ for $q \in [1,\infty]$ (see \cite{Do99} and its generalizations in \cite{Do01,Ha06b, KK10}).

\begin{theorem}\label{theorem:interpolateH}
  Let $A$ be an almost $\alpha$-sectorial operator on $X$. Let
  $\theta_0,\theta_1 \in \R$, $0<\eta<1$ and $\theta = (1-\eta)\theta_0+\eta\theta_1$.
   Then
   \begin{align*}
     H^\alpha_{\theta,A}&= (\dot{X}_{\theta_0,A}, \dot{X}_{\theta_1,A})_{\eta}^\alpha
   \end{align*}
   isomorphically.
\end{theorem}

\begin{proof}
 Assume without loss of generality that $\theta_1>\theta_0$, take $x \in D(A^m)\cap R(A^m)$ for $m \in \N$ with $\abs{\theta}+1<m$ and fix $\tilde{\omega}_\alpha(A)<\sigma<\pi$. Let $\psi \in H^1(\Sigma_\sigma)$ be such that $\int_0^\infty \psi(t)\frac{\dd t}{t}=1$ and
 $$\psi_j:=\hab{z \mapsto z^{\theta_j-\theta}\psi(z)} \in H^1(\Sigma_\sigma),\qquad j=0,1.$$
 First consider the strongly measurable function $f:\R_+ \to D(A^m)\cap R(A^m)$ given by
 \begin{equation*}
   f(t) = \frac{ t^{-\eta}}{\theta_1-\theta_0} \psi\hab{t^{\frac{1}{\theta_1-\theta_0}}A}x, \qquad t \in \R_+
 \end{equation*}
 Then, by \eqref{eq:calderonrepnontrivial} and a change of variables, we have $\int_0^\infty t^\eta f(t)\frac{\ddn t}{t} = x$ and thus, by  Proposition \ref{proposition:alpharealmethod} and Proposition \ref{proposition:equivsquarefunctions}, we have
 \begin{align*}
   \nrm{x}_{(\dot{X}_{\theta_0,A}, \dot{X}_{\theta_1,A})_{\eta}^\alpha} &\leq  \max_{j=0,1}\,\nrm{t \mapsto t^jf(t)}_{\alpha(\R_+,\frac{\ddn t}{t};\dot{X}_{\theta_j,A})}\\
   &=  \max_{j=0,1}\,\nrmb{t \mapsto \frac{ 1}{\theta_1-\theta_0} \psi_j(t^{\frac{1}{\theta_1-\theta_0}}A)A^\theta x}_{\alpha(\R_+,\frac{\ddn t}{t};X)}\\
   &\simeq \nrm{x}_{H^\alpha_{\theta,A}}.
 \end{align*}

 Conversely, take  a strongly measurable function  $f\colon\R_+\to D(A^m)\cap R(A^m)$ such that $t\mapsto t^jf(t) \in \alpha(\R_+,\frac{\ddn t}{t};\dot{X}_{\theta_j,A})$ for $j=0,1$ and
$
   \int_{0}^\infty t^\eta f(t) \frac{\ddn t}{t} = x.
$
 Let $\varphi \in H^1(\Sigma_\sigma)$ be such that
 $$\varphi_j:=\hab{z \mapsto z^{\theta-\theta_j}\varphi(z)} \in H^1(\Sigma_\sigma),\qquad j=0,1.$$
Then, since $A$ is almost $\alpha$-sectorial, we have by Proposition \ref{proposition:equivsquarefunctions}, Proposition \ref{proposition:almostsectorialcharacterization} and Theorem \ref{theorem:pointwisemultipliers1} that
\begingroup
\allowdisplaybreaks
 \begin{align*}
   \nrm{x}_{H^\alpha_{\theta,A}} &\simeq \nrmb{\varphi(tA)A^\theta  \int_{0}^\infty (st^{\theta_1-\theta_0})^{\eta } f(st^{\theta_1-\theta_0}) \frac{\ddn s}{s}}_{\alpha(\R_+,\frac{\ddn t}{t};X)}
   \\
   &\lesssim \int_{0}^1 s^{\eta} \nrmb{{\varphi}_0(tA) A^{\theta_0}f(st^{\theta_1-\theta_0}) }_{\alpha(\R_+,\frac{\ddn t}{t};X)} \frac{\ddn s}{s}\\
   &\hspace{1cm} + \int^{\infty}_1 s^{\eta} \nrmb{{\varphi}_1(tA) t^{(\theta_1-\theta_0)} A^{\theta_1}f(st^{\theta_1-\theta_0}) }_{\alpha(\R_+,\frac{\ddn t}{t};X)} \frac{\ddn s}{s}\\
   &\lesssim \int_{0}^1 s^{\eta} \nrmb{f(st^{\theta_1-\theta_0}) }_{\alpha(\R_+,\frac{\ddn t}{t};\dot{X}_{\theta_0,A})} \frac{\ddn s}{s}\\
   &\hspace{1cm} + \int^{\infty}_1 s^{(\eta-1)} \nrmb{s t^{{\theta_1-\theta_0}} f(st^{\theta_1-\theta_0}) }_{\alpha(\R_+,\frac{\ddn t}{t};\dot{X}_{\theta_1,A})} \frac{\ddn s}{s}\\
   &\lesssim \max_{j=0,1}\,\nrm{t \mapsto t^jf(t)}_{ \alpha(\R_+,\frac{\ddn t}{t};\dot{X}_{\theta_j,A})}.
 \end{align*}
 \endgroup
 Taking the infimum over all such $f$ we obtain by Proposition \ref{proposition:alpharealmethod}
\begin{equation*}
  \nrm{x}_{H^\alpha_{\theta,A}}\lesssim\nrm{x}_{(\dot{X}_{\theta_0,A}, \dot{X}_{\theta_1,A})_{\eta}^\alpha},
\end{equation*}
so the norms of $H^\alpha_{\theta,A}$ and $(\dot{X}_{\theta_0,A}, \dot{X}_{\theta_1,A})_{\eta}^\alpha$ are equivalent on  $D(A^m)\cap R(A^m)$. As $D(A^m)\cap R(A^m)$ is dense in both spaces, this proves the theorem.
\end{proof}

 In \cite[Theorem 5.3]{AMN97} Auscher, McIntosh and Nahmod proved that a sectorial operator $A$ on a Hilbert space $H$ has a bounded $H^\infty$-calculus if and only if the fractional domain spaces of $A$ form a interpolation scale for the complex method. As a direct corollary of Theorem \ref{theorem:Hproperties} and Theorem \ref{theorem:interpolateH} we can now deduce a similar  characterization of the boundedness of the $H^\infty$-calculus of a sectorial operator on a Banach space in terms of the $\alpha$-interpolation method.

\begin{corollary}\label{corollary:interptoHinfty}
  Let $A$ an almost $\alpha$-sectorial operator on $X$ and suppose that $\alpha$ is unconditionally stable. Then $A$ has a bounded $H^\infty$-calculus if and only if
  $$\dot{X}_{\theta,A} = (\dot{X}_{\theta_0,A}, \dot{X}_{\theta_1,A})_{\eta}^\alpha$$
   for some $\theta_0,\theta_1 \in \R$, $0<\eta<1$ and $\theta = (1-\eta)\theta_0+\eta\theta_1$.
\end{corollary}

In \cite{KKW06} perturbation theory for $H^\infty$-calculus is developed using the Rademacher interpolation method, which is equivalent to the $\gamma$-interpolation method on spaces with finite cotype by Proposition \ref{proposition:alpharealmethoddiscrete} and Proposition \ref{proposition:gaussianradermacherl2comparison}. Naturally, these results can also be generalized to the Euclidean structures framework. In particular, let us prove a version of \cite[Theorem 5.1]{KKW06} in our framework. We leave the extension of the other perturbation results from \cite{KKW06} (see also \cite{Ka07, KW13,KW17}) to the interested reader.

\begin{corollary}
Let $A$ be an almost $\alpha$-sectorial operator on $X$ and suppose that $\alpha$ is unconditionally stable. Suppose that $A$ has a bounded $H^\infty$-calculus and $B$ is almost $\alpha$-sectorial. Assume that for two different, non-zero $\theta_0,\theta_1 \in \R$ we have
\begin{equation*}
  \dot{X}_{\theta_j,A} = \dot{X}_{\theta_j,B}, \qquad j=0,1.
\end{equation*}
Then $B$ has a bounded $H^\infty$-calculus.
\end{corollary}

\begin{proof}
Let $\tilde{\theta}_0,\tilde{\theta}_1,\tilde{\theta}  \in \cbrace{0,\theta_0,\theta_1}$ be such that $\tilde{\theta}_0<\tilde{\theta}<\tilde{\theta}_1$ and let $\eta \in (0,1)$ be such that $\tilde{\theta} = (1-\eta)\tilde{\theta}_0+\eta\tilde{\theta}_1$. Then by Theorem \ref{theorem:Hproperties} and \ref{theorem:interpolateH} we have
\begin{align*}
  \dot{X}_{\tilde{\theta},B} = \dot{X}_{\tilde{\theta},A} = (\dot{X}_{\tilde{\theta}_0,A}, \dot{X}_{\tilde{\theta}_1,A})_{\eta}^\alpha=(\dot{X}_{\tilde{\theta}_0,B}, \dot{X}_{\tilde{\theta}_1,B})_{\eta}^\alpha,
\end{align*}
so the corollary follows from Corollary \ref{corollary:interptoHinfty}.
\end{proof}

\section{Generalized square function spaces without almost \texorpdfstring{$\alpha$}{a}-sectoriality}\label{section:Hspacenotalmostalpha}
In Section \ref{section:scalespaces} we have seen that the spaces $H^\alpha_{\theta,A}(\psi)$ behave very nicely when $A$ is almost $\alpha$-sectorial. In this section we will take a closer look at the $H^\alpha_{\theta,A}(\psi)$-spaces for sectorial operators $A$ which are not necessarily almost $\alpha$-sectorial.
 In this case the spaces $H^\alpha_{\theta,A}(\psi)$  may be different for different $\psi$ and whether $A|_{H^\alpha_{\theta,A}(\psi)}$  has a  bounded $H^\infty$-calculus may depend on the choice of $\psi$. This unruly behaviour will allow us to construct some interesting counterexamples in Section \ref{section:angleexample}.

\subsection*{The spaces $H^\alpha_{\theta,A}(\varphi_{s,r})$, and their properties}
Conforming with the definition of $H^\alpha_{\theta,A}(\psi)$ we need to assume that $\alpha$ is ideal throughout this section. Let $0<s<\frac{\pi}{\omega(A)}$, $0<r<1$, $\omega(A)<\sigma<\frac{\pi}{s}$ and set
\begin{equation*}
  \varphi_{s,r}(z) := \frac{z^{{sr}}}{1+z^s}, \qquad z \in \Sigma_{\sigma}.
\end{equation*}
 We will focus our attention on the spaces $H^\alpha_{\theta,A}(\varphi_{s,r})$, for which we have $H^\alpha_{\theta,A} = H^\alpha_{\theta,A}(\varphi_{1,\frac12})$. We will start our analysis by computing an equivalent norm on these
spaces, which will be more suited for our analysis.
\begin{proposition}\label{proposition:Hequivnorm}
  Let $A$ be a sectorial operator on $X$ and assume that $\alpha$ is ideal. Let $0<s<\frac{\pi}{\omega(A)}$ and $0<r<1$. Then for $x \in D(A)\cap R(A)$ and $t \in [1,2]$ we have
  \begin{align*}
    \nrm{\varphi_{s,r}(\cdot A)x}_{\alpha(\R_+,\frac{\ddn t}{t};X)} &\simeq \nrm{ \ee^{-\frac{\pi}{s}\abs{\cdot}}A^{i\cdot}x}_{\alpha(\R;X)},\\
    \nrmb{(\varphi_{s,r}(2^ntA) x)_{n\in \Z}}_{\alpha(\Z;X)} &\simeq \nrms{\sum_{m \in \Z} \ee^{-\frac{\pi}{s}\abs{\cdot+2mb}} A^{i(\cdot+2mb)}x }_{\alpha([-b,b];X)}
  \end{align*}
  with $b=\pi/\log(2)$ and the implicit constants only depend on $s$ and $r$.
\end{proposition}

\begin{proof}
  Let $\omega(A)<\sigma<\min\cbrace{\frac{\pi}{s},\pi}$, then for $\xi \in \R$ and $z \in \R_+$ we have, using
  the change of coordinates $u^{1/s} = e^{2\pi t} z$ and the Mellin transform as in \eqref{eq:mellin1}, that
  \begin{align*}
   \int_{\R} \ee^{-2\pi i t\xi}\varphi_{s,r}(e^{2\pi t} z)\dd t
    &= \frac{z^{i\xi} }{2\pi s}\int_0^\infty u^{-i \xi/s}  \frac{u^{r}}{1+u}\frac{\ddn u}{u}\\  &=\frac{z^{i\xi}}{2s \sin(\pi(r-i\xi/s))} =:z^{i\xi} g(\xi),
  \end{align*}
  which extends to all $z \in \Sigma_\sigma$ by analytic continuation. Note that
    \begin{equation}\label{eq:sinhequive}
   \abs{g(\xi)} =  \frac{1}{2s\,  \abs{\sin(\pi(r-i\xi/s))}} \simeq \ee^{-\frac{\pi}{s}\abs{\xi}}, \qquad \xi \in \R.
  \end{equation}
By Fourier inversion we have for all $z \in \Sigma_\sigma$ and $t \in \R$
  \begin{equation*}
       \varphi_{s,r} (\ee^{2\pi t}z)  = \int_{\R}\ee^{2 \pi i t \xi} g(\xi) z^{i\xi} \dd \xi.
  \end{equation*}
 Thus, by the definition of the $H^\infty$-calculus and Fubini's theorem, we have for $x \in D(A)\cap R(A)$
  \begin{equation}\label{eq:repmellin}
   \varphi_{s,r} (\ee^{2\pi t}A)x  = \int_{\R}\ee^{2 \pi i t \xi} g(\xi) A^{i\xi}x \dd \xi.
  \end{equation}
as a Bochner integral, since, for $\xi \in \R$, $\omega(A)<\nu<\frac{\pi}{s}$, $\varphi(z) = z(1+z)^{-2}$ and $y \in X$ such that $\varphi(A)y=x$, we have
\begin{equation}\label{eq:gabsolutelyconv}
  \begin{aligned}
  \nrm{g(\xi) A^{i\xi}x}_X &= \frac{\abs{g(\xi)}}{2\pi} \nrms{\int_{\Gamma_\nu} z^{i\xi}\varphi(z)R(z,A)y\dd z}_X\\
  &\lesssim \ee^{-\frac{\pi}{s}\abs{\xi}} \cdot  \int_{\Gamma_\nu} \ee^{\nu \abs{\xi}} \frac{\abs{z}}{\abs{1+z}^2}\nrm{y}_X\frac{\abs{\ddn z}}{\abs{z}} \\
  &\lesssim \ee^{-\ha{\frac{\pi}{s}-\nu}\abs{\xi}} \nrm{y}_X.
\end{aligned}
\end{equation}

  Now to prove the equivalence for the continuous square function norm, define
 $T:L^2(\R_+,\frac{\ddn t}{t})\mapsto L^2(\R)$ by
$$Tf(t):= \sqrt{2 \pi}\, f(e^{2\pi t}), \qquad t \in \R.$$
Then $T$ is an isometry, so by Proposition \ref{proposition:operatoronfunctions} we have for any $f \in L^2(\R_+,\frac{\ddn t}{t})$
\begin{equation}\label{eq:Ttransform}
  \nrm{f}_{\alpha(\R_+,\frac{\ddn t}{t};X)}=\sqrt{ 2 \pi}\, \nrm{t\mapsto f(e^{2\pi t} )}_{\alpha(\R;X)}.
\end{equation}
Let $x \in D(A)\cap R(A)$ and note that by the definition of the Dunford calculus and Fubini's theorem
\begin{align*}
  (t \mapsto \varphi_{s,r}(t A)x )&\in L^2(\R_+,\tfrac{\ddn t}{t};X),\\
 ( t \mapsto \varphi_{s,r}(e^{2\pi t} A)x) &\in L^1(\R;X).
\end{align*} So by \eqref{eq:repmellin}, \eqref{eq:Ttransform} and the invariance of the $\alpha$-norms under the Fourier transform (see Example \ref{example:fourier}) we have
  \begin{align*}
    \nrm{t \mapsto \varphi_{s,r}(t A)x}_{\alpha(\R_+,\frac{\ddn t}{t};X)} &= \sqrt{2 \pi} \, \nrm{t\mapsto \varphi_{s,r}(e^{2\pi t} A)x}_{\alpha(\R;X)}
\\ &= {\sqrt{2 \pi}}  \,\nrmb{t \mapsto \int_{\R}\ee^{2 \pi i t \xi} g(\xi) A^{i\xi}x\dd \xi}_{\alpha(\R;X)}
\\ &= {\sqrt{2 \pi}} \,\nrmb{\xi  \mapsto g(\xi) A^{i\xi}x}_{\alpha(\R;X)},
  \end{align*}
  which proves the equivalence for the continuous square function by \eqref{eq:sinhequive}.

  For the discrete square function norm note that by \eqref{eq:repmellin} we have for $x \in D(A)\cap R(A)$ and $t \in \R$
   \begin{align*}
     \varphi_{s,r} (\ee^{2\pi t}A)x  = \sum_{m \in \Z}\int_{-b}^b\ee^{2 \pi i t (\xi+2mb)} g(\xi+2mb) A^{i(\xi+2mb)} x\dd \xi.
\end{align*}
The sum converges absolutely by \eqref{eq:gabsolutelyconv}.
Thus, using $2^{i n\cdot 2mb}=1$ and setting $2^nu = \ee^{2\pi t}$, we have
\begin{align}\label{eq:varphis2n}
  \varphi_{s,r}(2^nuA)x = \int_{-b}^b  2 ^{in\xi} \sum_{m \in \Z} u^{i(\xi+2mb)} g(\xi+2mb) A^{i(\xi+2mb)}x \dd \xi.
\end{align}
By Parseval's theorem and Proposition \ref{proposition:operatoronfunctions} for any $h \in L^1([-b,b];X)$ with $h \in \alpha([-b,b];X)$, we have
\begin{equation*}
  \nrm{h}_{\alpha([-b,b];X)}= \sqrt{2b} \, \nrmb{(\widehat{h}(n))_{n \in \Z}}_{\alpha(\Z;X)},
\end{equation*}
where $b=\pi/\log(2)$ and
\begin{equation*}
  \widehat{h}(n):= \frac{1}{2b}\int_{-b}^b h(\xi)2^{- i  n \xi}\dd \xi.
\end{equation*}
And thus, using \eqref{eq:varphis2n} and the fact that $\abs{u^{i(\xi+2mb)}}=1$ for all $\xi \in [-b,b]$, $m \in \Z$ and  $u \in [1,2]$, we obtain
 \begin{align*}
 \nrmb{(\varphi_{s,r}(2^nuA) x)_{n\in \Z}}_{\alpha(\Z;X)} \simeq
 \nrms{\sum_{m \in \Z} g(\cdot+2mb) A^{i(\cdot+2mb)}x }_{\alpha([-b,b];X)},
 \end{align*}
 which combined with \eqref{eq:sinhequive} proves
 the equivalence for the discrete square function norm.
\end{proof}

From Proposition \ref{proposition:Hequivnorm} we can immediately deduce embeddings between the $H^\alpha_{\theta,A}(\varphi_{s,r})$-spaces.
\begin{corollary}\label{corollary:hspacesisomorphic}
   Let $A$ be a sectorial operator on $X$ and assume that $\alpha$ is ideal. Fix $0<s_1\leq s_2<\frac{\pi}{\omega(A)}$, $0<r_1,r_2<1$ and $\theta \in \R$. For $\nu = \pi(\frac{1}{s_1}-\frac{1}{s_2})$ and $\omega(A)<\sigma<\frac{\pi}{s_2}$
  set
\begin{equation*}
  \varphi_{s_1,r_1}^{\pm\nu}(z):=\varphi_{s_1,r_1}(\ee^{\pm i\nu}z), \qquad z \in \Sigma_\sigma .
\end{equation*}
   Then we have the continuous embedding
   \begin{equation*}
     H^\alpha_{\theta,A}(\varphi_{s_2,r_2})\hookrightarrow H^\alpha_{\theta,A}(\varphi_{s_1,r_1})
   \end{equation*}
   and
        \begin{equation*}
      H^\alpha_{\theta,A}(\varphi_{s_1,r_1}^{+\nu}) \cap H^\alpha_{\theta,A}(\varphi_{s_1,r_1}^{-\nu})=H^\alpha_{\theta,A}(\varphi_{s_2,r_2}).
    \end{equation*}
    isomorphically.
\end{corollary}

\begin{proof}
  Without loss of generality we may assume $\theta = 0$. The claimed embedding is a direct consequence of Proposition \ref{proposition:Hequivnorm} and the density of $D(A)\cap R(A)$ in $ H^\alpha_{\theta,A}(\varphi_{s_2,r_2})$. For the isomorphism fix $x \in D(A)\cap R(A)$. Then by Proposition \ref{proposition:Hequivnorm} we have
  \begin{align*}
  \nrm{\varphi_{s_1,r_1}^{+\nu}(\cdot A)x}_{\alpha(\R_+,\frac{\ddn t}{t};X)} &\simeq \nrmb{ \ee^{-(\nu+\frac{\pi}{s_1})\abs{\cdot}}A^{i\cdot}x \ind_{(0,\infty)}}_{\alpha(\R;X)} \\&\hspace{2cm} +\nrmb{ \ee^{-\frac{\pi}{s_2}\abs{\cdot}}A^{i\cdot}x \ind_{(-\infty,0)}}_{\alpha(\R;X)},\\
  \nrm{\varphi_{s_1,r_1}^{-\nu}(\cdot A)x}_{\alpha(\R_+,\frac{\ddn t}{t};X)} &\simeq \nrmb{ \ee^{-\frac{\pi}{s_2}\abs{\cdot}}A^{i\cdot}x \ind_{(0,\infty)}}_{\alpha(\R;X)} \\&\hspace{2cm} +\nrmb{ \ee^{-(\nu+\frac{\pi}{s_1})\abs{\cdot}}A^{i\cdot}x \ind_{(-\infty,0)}}_{\alpha(\R;X)},\\
  \nrm{\varphi_{s_2,r_2}(\cdot A)x}_{\alpha(\R_+,\frac{\ddn t}{t};X)} &\simeq  \nrmb{ \ee^{-\frac{\pi}{s_2}\abs{\cdot}}A^{i\cdot}x \ind_{(0,\infty)}}_{\alpha(\R;X)} \\&\hspace{2cm} +\nrmb{ \ee^{-\frac{\pi}{s_2}\abs{\cdot}}A^{i\cdot}x \ind_{(-\infty,0)}}_{\alpha(\R;X)}.
  \end{align*}
  Since $\nu+\frac{\pi}{s_1} \geq \frac{\pi}{s_2}$, the corollary now follows by density and Example \ref{example:pointwisemulti}.
\end{proof}

From Corollary \ref{corollary:hspacesisomorphic} we can see that the spaces $H^\alpha_{\theta,A}(\varphi_{s,r})$ are independent of $r$, which is why we will focus on the spaces
$H^\alpha_{\theta,A}(\varphi_{s})$ for
\begin{equation*}
  \varphi_s(z):= \varphi_{s,\frac12}(z) = \frac{z^{{s/2}}}{1+z^s} , \qquad z \in \Sigma_\sigma
\end{equation*}
with $\omega(A)<\sigma<\frac{\pi}{s}$ for the remainder of this section. Moreover, Corollary \ref{corollary:hspacesisomorphic} states that the $H^\alpha_{\theta,A}(\varphi_{s})$-spaces shrink as  $s$ increases.

\subsection*{The operators $A|_{H^\alpha_{\theta,A}(\varphi_s)}$ and their properties}
We will now analyse the properties of the operators $A|_{H^\alpha_{\theta,A}(\varphi_s)}$ on $H^\alpha_{\theta,A}(\varphi_s)$. As a first observation, we note that from Proposition \ref{proposition:Hequivnorm} we immediately deduce for  $0<s<\frac{\pi}{\omega(A)}$ and $\theta \in \R$ that $A|_{H^\alpha_{\theta,A}(\varphi_s)}$ has $\BIP$ with
  \begin{equation}\label{eq:BIPangle}
    \omega_{\BIP}(A|_{H^\alpha_{\theta,A}(\varphi_s)}) \leq \frac{\pi}{{s}}.
  \end{equation}
Using the characterization of $\alpha$-$\BIP$ in Theorem \ref{theorem:BIPHinfty} and the transference result of Theorem \ref{theorem:transference}, we can say more if $s>1$.

\begin{theorem}\label{theorem:XalphaHinfty}
  Let  $A$ be a sectorial operator on $X$ and assume that $\alpha$ is ideal. Fix $1<s<\frac{\pi}{\omega(A)}$ and $\theta \in \R$. Then $A|_{H^\alpha_{\theta,A}(\varphi_s)}$ has a bounded $H^\infty$-calculus on $H^\alpha_{\theta,A}(\varphi_s)$ with
  \begin{equation*}
    \omega_{H^\infty}(A|_{H^\alpha_{\theta,A}(\varphi_s)}) \leq \frac{\pi}{s}.
  \end{equation*}
\end{theorem}

We give two proofs. The first is far more elegant, relying on the transference result in Chapter \ref{part:1}. In particular, we will use the characterization of $\alpha$-$\BIP$ in Theorem \ref{theorem:BIPHinfty}. We include a sketch of a second, more direct and elementary, but highly technical proof. This leads to a proof for the angle of the $H^\infty$-calculus counterexample in Section \ref{section:angleexample} which does not rely on the theory in Chapter \ref{part:1}

\begin{proof}[Proof of Theorem \ref{theorem:XalphaHinfty}]
  Define a Euclidean structure $\beta$ on $\alpha(\R;X)$ by defining for $T_1,\ldots,T_n \in \alpha(\R;X)$
  \begin{equation*}
    \nrm{(T_1,\ldots,T_n)}_{\beta} = \nrm{T_1\oplus\cdots\oplus T_n}_{\alpha(L^2(\R)^n;X)},
  \end{equation*}
  where we view $T_1\oplus\cdots\oplus T_n$ as an operator from $L^2(\R)^n$ to $X$ given by
  \begin{equation*}
    \hab{T_1\oplus\cdots\oplus T_n} (h_1,\ldots,h_n) := \sum_{k=1}^n T_kh_k, \qquad (h_1,\ldots,h_n) \in L^2(\R)^n.
  \end{equation*}
  By Proposition \ref{proposition:Hequivnorm}, the space $H^\alpha_{\theta,A}(\varphi_s)$ is continuously embedded in $\alpha(\R;X)$ via the map
  \begin{equation*}
    x \mapsto \hab{t \mapsto \ee^{-\frac{\pi}{s} \abs{t}}A^{it+\theta} x}, \qquad x \in D(A^m)\cap R(A^m)
  \end{equation*}
  with $m \in \N$ such that $\abs{\theta}+1<m$. Therefore $\beta$ can be endowed upon $H^\alpha_{\theta,A}(\varphi_s)$. We will show that
  \begin{equation*}
    \Gamma := \cbraces{\ee^{-\frac{\pi}{s}\abs{t}} (A|_{H^\alpha_{\theta,A}(\varphi_s)})^{it}: \,t\in \R}
  \end{equation*}
  is $\beta$-bounded, which combined with Theorem \ref{theorem:BIPHinfty} yields the theorem. Suppose that $t_1,\ldots,t_n \in \R$ and $x_1,\ldots,x_n \in D(A^m)\cap R(A^m)$. Then
  \begin{align*}
    &\nrmb{\hab{\ee^{-\frac{\pi}{s} \abs{t_k}}(A|_{H^\alpha_{\theta,A}(\varphi_s)})^{it_k}x_k}_{k=1}^n}_\beta \\&\hspace{2cm}= \nrmb{\oplus_{k=1}^n \hab{t\mapsto \ee^{-\frac{\pi}{s} (\abs{t}+\abs{t_k})}A^{i(t+t_k)+\theta}x_k}}_{\alpha(L^2(\R)^n;X)}\\
    &\hspace{2cm} = \nrmb{\oplus_{k=1}^n \hab{t\mapsto \ee^{-\frac{\pi}{s} (\abs{t-t_k}+\abs{t_k})}A^{it+\theta}x_k}}_{\alpha(L^2(\R)^n;X)}\\
        &\hspace{2cm} \leq \nrmb{\oplus_{k=1}^n \hab{t\mapsto \ee^{-\frac{\pi}{s} (\abs{t})}A^{it+\theta}x_k}}_{\alpha(L^2(\R)^n;X)} \\&\hspace{2cm}= \nrm{(x_k)_{k=1}^n}_{\beta}
  \end{align*}
  Now the $\beta$-boundedness of $\Gamma$ follows by the density of $D(A^m)\cap R(A^m)$ in $H^\alpha_{\theta,A}(\varphi_s)$, which proves the theorem.
\end{proof}

\begin{proof}[Sketch of an alternative proof of Theorem \ref{theorem:XalphaHinfty}]
Without loss of generality we may assume $\theta=0$. Fix $\frac{\pi}{s}<\nu<\sigma<\pi$, take $f \in H^\infty(\Sigma_\sigma)$ with $\nrm{f}_{H^\infty(\Sigma_\sigma)} \leq 1$
and fix $0<a,b,c<1$ such that $a+b = 1+c$.
Then, using a similar calculation as in the proof of Proposition \ref{proposition:functionspaceclose}, we can write for $\nu' = \pi-\nu$ and $x \in D(A^2)\cap R(A^2)$
\begin{equation*}
  f(A)x = \sum_{\epsilon=\pm 1}\frac{-\epsilon\ee^{-ia\pi}}{2 \pi i}  \int_{0}^\infty f(s^{-1}e^{\epsilon i\nu})\varphi_{1,a}(s \ee^{\epsilon i\nu'}A)x\frac{\ddn s}{s}.
\end{equation*}
To estimate $\nrm{f(A)}_{H^\alpha_{0,A}(\varphi_s)}$ we will first consider the integral for $\epsilon = 1$. Note that we have the identity
\begin{equation*}
  \varphi_{1,a}(\lambda A)\varphi{1_b}(\mu A) = \frac{\lambda^{1-b}\mu^b}{\mu-\lambda}\varphi_{1,c}(\lambda A) + \frac{\lambda^a\mu^{1-a}}{\lambda-\mu}\varphi_{1,c}(\mu A)
\end{equation*}
for $\abs{\arg \lambda}, \abs{\arg \mu}<\pi-\omega(A)$. Thus for $s,t>0$ and $\nu'' = \pm\pi (1- \frac{1}{s})$
\begin{equation*}
  \varphi_{1,a}(st\ee^{i\nu'} A)\varphi_{1,b}(t\ee^{i\nu''} A) = \kappa_1(t) \frac{s^{1-b}}{1+s}\varphi_{1,c}(st\ee^{i\nu'} A) + \kappa_2(t) \frac{s^a}{1+s}\varphi_{1,c}(t\ee^{i\nu''} A),
\end{equation*}
where $\kappa_1,\kappa_2:\R_+\to \C$ are bounded and continuous functions. Therefore
\begin{align*}
  \nrms{\int_{0}^\infty f(s^{-1}e^{ i\nu})&\varphi_{1,a}(s \ee^{ i\nu'}A)x\frac{\ddn s}{s}}_{H^\alpha_{0,A}(\varphi_{1,b}(\ee^{i\nu''}\cdot))} \\&= \nrms{\int_{0}^\infty f(s^{-1}t^{-1}e^{ i\nu})\varphi_{1,a}(st \ee^{ i\nu'}A)\varphi_{1,b}(t \ee^{ i\nu''}A)x \frac{\ddn t}{t}}_{\alpha(\R,\frac{\ddn t}{t};X)}\\
  &\lesssim \int_0^\infty \frac{s^{a}}{1+s} \nrm{\varphi_{1,c}(st\ee^{i\nu'} A)}_{\alpha(\R,\frac{\ddn t}{t};X)} x\frac{\ddn s}{s} \\&\hspace{1cm}+ \int_0^\infty \frac{s^{1-b}}{1+s} \nrm{\varphi_{1,c}(t\ee^{i\nu''} A)}_{\alpha(\R,\frac{\ddn t}{t};X)}x\frac{\ddn s}{s}
  \\ &\lesssim {\nrm{x}_{H^\alpha_{0,A}(\varphi_{1,c}(\ee^{i\nu'}\cdot))}+ \nrm{x}_{H^\alpha_{0,A}(\varphi_{1,c}(\ee^{i\nu''}\cdot))}}.
\end{align*}
Combining the estimates for $\nu''=\pm\pi (1- \frac{1}{s})$ with the isomorphism from Corollary \ref{corollary:hspacesisomorphic} with parameters $s_1=1$, $s_2=s$, $r_1=b,c$ and $r_2 = \frac{1}{2}$, we obtain
\begin{align*}
  \nrms{\int_{0}^\infty f(s^{-1}e^{ i\nu})&\varphi_{1,a}(s \ee^{ i\nu'}A)x\frac{\ddn s}{s}}_{H^\alpha_{0,A}(\varphi_{s})} \\&\lesssim {\nrm{x}_{H^\alpha_{0,A}(\varphi_{1,c}(\ee^{i\nu'}\cdot))} + \nrm{x}_{H^\alpha_{0,A}(\varphi_{s})}}
\end{align*}
and since $\nu' \leq (1- \frac{1}{s})$, applying Corollary \ref{corollary:hspacesisomorphic} once more yields
\begin{equation*}
  \nrm{x}_{H^\alpha_{0,A}(\varphi_{1,c}(\ee^{i\nu'}\cdot))} \lesssim \nrm{x}_{H^\alpha_{0,A}(\varphi_{s})}.
\end{equation*}
Doing a similar computation for $\epsilon = -1$ yields the theorem.
\end{proof}

If we have a strict inequality
\begin{equation*}
  \omega_{\BIP}(A|_{H^\alpha_{\theta,A}(\varphi_s)}) < \pi,
\end{equation*}
we can extend Theorem \ref{theorem:XalphaHinfty} to $s=1$. So in this case $A|_{H^\alpha_{\theta,A}}$ on $H^\alpha_{\theta,A}$ ``behaves'' like a Hilbert space operator, as it has $\BIP$ if and only if it has a bounded $H^\infty$-calculus.

\begin{theorem}\label{theorem:XABIPHinftyequiv}
 Let  $A$ be a sectorial operator on $X$, suppose that $\alpha$ is ideal and fix $\theta \in \R$. The following are equivalent:
  \begin{enumerate}[(i)]
    \item \label{it:XABIPHinftyequiv1} $A|_{H^\alpha_{\theta,A}}$ has $\BIP$ with $\omega_{\BIP}(A|_{H^\alpha_{\theta,A}})<\pi$
    \item \label{it:XABIPHinftyequiv2} There is a $1<\sigma<\frac{\pi}{\omega(A)}$ such that the spaces $H^\alpha_{\theta,A}(\varphi_s)$ are isomorphic for all $0<s<\sigma$.
    \item \label{it:XABIPHinftyequiv3} $A|_{H^\alpha_{\theta,A}}$ has a bounded $H^\infty$-calculus.
  \end{enumerate}
\end{theorem}

\begin{proof}
The implication \ref{it:XABIPHinftyequiv2} $\Rightarrow$ \ref{it:XABIPHinftyequiv3} follows directly from Theorem \ref{theorem:XalphaHinfty} and \ref{it:XABIPHinftyequiv3} $\Rightarrow$ \ref{it:XABIPHinftyequiv1} is immediate from \eqref{eq:HinftyBIP}. For \ref{it:XABIPHinftyequiv1} $\Rightarrow$ \ref{it:XABIPHinftyequiv2} let $\sigma>1$ be such that
\begin{equation}\label{eq:BIPforH}
  \nrmb{(A|_{H^\alpha_{\theta,A}})^{it}} \leq C\ee^{\frac{\pi}{\sigma}\abs{t}}, \qquad t \in \R.
\end{equation}
Fix $x \in D(A^m)\cap R(A^m)$ with $m \in \N$ such that $\abs{\theta}+1<m$ and take $0<s'<s<\sigma$. Then by Proposition \ref{proposition:Hequivnorm}, \eqref{eq:BIPforH} and the ideal property of $\alpha$ we have
\begin{align*}
  \nrm{x}_{H^\alpha_{\theta,A}(\varphi_{s})} &\lesssim \nrmb{t\mapsto \ee^{-\frac{\pi}{s}\abs{t}}A^{it+\theta}(t)x}_{\alpha(\R;X)}\\
  &\leq \sum_{n \in \Z}\nrmb{t\mapsto \ee^{-\frac{\pi}{s}\abs{t}}A^{it+\theta}\ind_{[n,n+1)}(t)x}_{\alpha(\R;X)}\\
  &\leq \ee^{\frac{\pi}{s}} \sum_{n \in \Z}\ee^{-\frac{\pi}{s} \abs{n}} \nrmb{(A|_{H^\alpha_{\theta,A}})^{in}} \nrmb{t\mapsto A^{it+\theta}\ind_{[0,1)}(t) x}_{\alpha(\R;X)}\\
  &\leq \ee^{\pi(s^{-1}+s'^{-1})} \sum_{n \in \Z}\ee^{-\pi (s^{-1}-\sigma^{-1})\abs{n}}  \nrmb{t\mapsto \ee^{-\frac{\pi}{s'} \abs{t}}A^{it+\theta}\ind_{[0,1)}(t) x}_{\alpha(\R;X)}\\
  &\lesssim \nrm{x}_{H^\alpha_{\theta,A}(\varphi_{s'})}.
\end{align*}
Moreover by Corollary \ref{corollary:hspacesisomorphic} we have the converse estimate
\begin{equation*}
  \nrm{x}_{H^\alpha_{\theta,A}(\varphi_{s'})} \lesssim \nrm{x}_{H^\alpha_{\theta,A}(\varphi_{s})},
\end{equation*}
so by the density of $D(A^m)\cap R(A^m)$ the spaces $H^\alpha_{\theta,A}(\varphi_{s'})$ and $H^\alpha_{\theta,A}(\varphi_{s})$ are isomorphic.
\end{proof}

Using Theorem \ref{theorem:XABIPHinftyequiv}, we end this section with another theorem on the equivalence of discrete and continuous square functions, as treated in Proposition \ref{proposition:equivalencediscretecontinuous} and Corollary \ref{corollary:equivalencediscretecontinuous}. This time for a very specific choice of $\psi$ and under the assumption that one of the  equivalent statements of  Theorem \ref{theorem:XABIPHinftyequiv} holds. Note that in this special case we can also omit the supremum over $t \in [1,2]$ for the discrete square functions.

\begin{proposition}\label{proposition:equivdiscretecontinuousstable}
   Let $A$ be a sectorial operator on $X$ and suppose that $\alpha$ is ideal. Assume that $A|_{H^\alpha_{0,A}}$ has a bounded $H^\infty$-calculus on $H^\alpha_{0,A}$. Then there is a $1<\sigma<\frac{\pi}{\omega(A)}$ such that for all $0<s<\sigma$, and  $x \in D(A)\cap R(A)$ we have
 \begin{align*}
  \nrm{\varphi_s(\cdot A) x}_{\alpha(\R_+,\frac{\ddn t}{t};X)} \simeq \nrmb{(\varphi_s(2^nA) x)_{n\in \Z}}_{\alpha(\Z;X)}.
  \end{align*}
\end{proposition}

\begin{proof}
Take $1<\sigma<\pi/\omega(A)$ as in Theorem \ref{theorem:XABIPHinftyequiv}\ref{it:XABIPHinftyequiv2}, let
$0<s<\sigma$ and $0<\delta< \frac{\pi}{s}-\omega(A)$. Then by Proposition \ref{proposition:equivalencediscretecontinuous} and Proposition \ref{proposition:Hequivnorm} we have
\begin{align*}
  \nrmb{(\varphi_s(2^nA) x)_{n\in \Z}}_{\alpha(\Z;X)} &\lesssim \max_{\epsilon = \pm \delta} \,\nrm{t \mapsto \varphi_s(t \ee^{i\epsilon} A) x}_{\alpha(\R_+,\frac{\ddn t}{t};X)} \\
  &\lesssim \max_{\epsilon = \pm \delta} \,\nrm{t \mapsto \ee^{-\frac{\pi}{s}\abs{t}}\ee^{-\epsilon t} A^{it}x}_{\alpha(\R;X)}\\
  &\leq  \nrm{t \mapsto \ee^{-\frac{\pi}{s'}\abs{t}} A^{it}x}_{\alpha(\R;X)}\\
  &\lesssim \nrm{\varphi_{s'}(\cdot  A) x}_{\alpha(\R_+,\frac{\ddn t}{t};X)}
\end{align*}
with $s' = \frac{\pi s}{\pi-\delta s}$. So taking $\delta$ small enough such that $0<s'<\sigma$ it follows from Theorem \ref{theorem:XABIPHinftyequiv}\ref{it:XABIPHinftyequiv2} that
\begin{equation*}
  \nrmb{(\varphi_s(2^nA) x)_{n\in \Z}}_{\alpha(\Z;X)} \lesssim \nrm{\varphi_s(\cdot A) x}_{\alpha(\R_+,\frac{\ddn t}{t};X)}.
\end{equation*}
For the converse inequality let $s' = \frac{\pi s}{\pi+\delta s}$ for some $0<\delta< \frac{\pi}{s}-\omega(A)$. Then we have by Proposition \ref{proposition:equivalencediscretecontinuous} and Proposition \ref{proposition:Hequivnorm}
\begin{align*}
  \nrm{\varphi_{s}(\cdot  A) x}_{\alpha(\R_+,\frac{\ddn t}{t};X)}  &\lesssim \nrm{\varphi_{s'}(\cdot  A) x}_{\alpha(\R_+,\frac{\ddn t}{t};X)} \\
  &\lesssim \sup_{\abs{\epsilon}< \delta}\sup_{t \in [1,2]}\nrmb{(\varphi_{s'}(2^nt\ee^{i\epsilon} A)x)_{n\in \Z}}_{\alpha(\Z;X)}\\
  &\lesssim\sup_{\abs{\epsilon}< \delta} \nrms{\sum_{m \in \Z} \ee^{-\pi\abs{\cdot+2mb}/s'} \ee^{-\epsilon(\cdot+2mb)} A^{i(\cdot+2mb)}}_{\alpha([-b,b];X)}\\
   &\leq \nrms{\sum_{m \in \Z}\ee^{-\pi\abs{\cdot+2mb}/s} A^{i(\cdot+2mb)}}_{\alpha([-b,b];X)}\\
  &\lesssim \nrmb{(\varphi_{s}(2^n A)x)_{n\in \Z}}_{\alpha(\Z;X)},
\end{align*}
finishing the proof.
\end{proof} 
\chapter{Some counterexamples}\label{part:6}
In Chapter \ref{part:4} we introduced various properties of a sectorial operator $A$ on $X$ and proved the following relations between these properties:
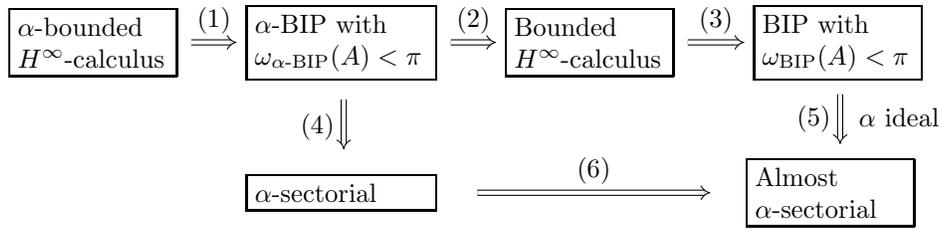
\vspace{-10pt}
\begin{center}
  \begin{tikzpicture}
  \node (aH)[text width=2.3cm] at (0,2) {\framebox{\parbox{2cm}{$\alpha$-bounded \\ $H^\infty$-calculus}}};
  \node (aB)[text width=2.6cm]  at (3.3,2) {\framebox{\parbox{2.3cm}{$\alpha$-$\BIP$ with\\ $\omega_{\alphaBIP}(A)<\pi$}}};
  \node (H)[text width=2.3cm]  at (6.6,2)
  {\framebox{\parbox{2cm}{Bounded \\ $H^\infty$-calculus}}};
  \node (B)[text width=2.3cm]  at (9.9,2)
  {\framebox{\parbox{2cm}{$\BIP$ with \\$\omega_{\BIP}(A)<\pi$}}};
  \node (a)[text width=2.6cm] at (3.3,0)
   {\framebox{\parbox{2.3cm}{$\alpha$-sectorial}}};
  \node (aa)[text width=2.5cm] at (9.9,0)
   {\framebox{\parbox{2cm}{Almost \\ $\alpha$-sectorial}}};
   \node (ai)[text width=2.3cm] at (11.3,1)
   {$\alpha$ ideal};

  \draw[-implies,double equal sign distance] (aH) to node [above,midway] {(1)} (aB);
  \draw[-implies,double equal sign distance] (aB) to node [above,midway] {(2)} (H);
  \draw[-implies,double equal sign distance] (H) to node [above,midway] {(3)} (B);
  \draw[-implies,double equal sign distance, shorten <=4pt, shorten >=6pt] (aB) to node [left,midway] {(4)} (a);
  \draw[-implies,double equal sign distance, shorten <=2pt, shorten >=2pt] (B) to node [left,midway] {(5)} (aa);
  \draw[-implies,double equal sign distance, shorten <=10pt, shorten >=10pt] (a) to node [above,midway] {(6)} (aa);
\end{tikzpicture}
\end{center}
Moreover we noted that (1) and (2) are `if and only if' statements if $\alpha=\ell^2$ or $\alpha=\gamma$ and $X$ has Pisier's contraction property. Statements (3), (4) and (5) cannot be turned into `if and only if' statements for $\alpha=\ell^2$ or $\alpha=\gamma$. Indeed, there are counterexamples on spaces admitting an unconditional Schauder basis disproving the converse of (3), (4) and (5) for $\mc{R}$-boundedness, which is equivalent to $\ell^2$- and $\gamma$-boundedness if $X$ has finite cotype by Proposition \ref{proposition:gaussianradermacherl2comparison}. We refer to the survey of Fackler \cite{Fa15} and the references therein for an overview of these counterexamples

In this chapter we will show that (6) can also not be turned into an `if and only if' statement for any unconditionally stable Euclidean structure $\alpha$ on a Banach space admitting an unconditional Schauder basis. Moreover in the same setting we will show that even the weakest property, the almost $\alpha$-sectoriality of $A$, does not follow from the sectoriality of $A$.

In Chapter \ref{part:4} we have also seen that under reasonable assumptions on $\alpha$ the angles of (almost) $\alpha$-sectoriality, ($\alpha$-)$\BIP$ and of the ($\alpha$-)-bounded $H^\infty$-calculus are equal whenever $A$ has these properties. Strikingly absent in this list is the angle of sectoriality of $A$. In Section \ref{section:BIP} we already remarked that it is possible to have $\omega_{\BIP}(A) \geq \pi$ and thus $\omega_{\BIP}(A) > \omega(A)$, see \cite[Corollary 5.3]{Ha03}. Moreover in \cite{Ka03} it was shown that it is also possible to have $\omega_{H^\infty}(A)>\omega(A)$. However, the Banach space used in  \cite{Ka03} is quite unnatural. We will end this chapter with an example of a sectorial operator with $\omega_{H^\infty}(A)>\omega(A)$ on a closed subspace of $L^p$, using the $H^\alpha_{\theta,A}$-spaces introduced in Chapter \ref{part:5}.

\section{Schauder multiplier operators}\label{section:schaudersectorial}
We start by introducing the class of operators that we will use in our examples. This will be the class of so-called Schauder multiplier operators.
The idea of using Schauder multiplier operators to construct counterexamples in the context of sectorial operators goes back to Clement and Baillon \cite{BC91} and Venni \cite{Ve93}, where Schauder multipliers were used to construct examples of sectorial operators without $\BIP$. It has since proven to be a fruitful method to construct counterexamples in this context, see for example \cite{AL19, CDMY96,Fa13, Fa14, Fa15, Fa16, KL00, KL02, La98, Le04}. For $L^1(S)$- and $C(K)$-spaces different counterexamples, connected to the breakdown of the theory of singular integral operators, are available, see e.g. \cite{HKK04,KK08,KW05}.

\subsection*{Schauder decompositions}
Let $(X_k)_{k=1}^\infty$ be a sequence of closed subspaces of $X$. Then $(X_k)_{k=1}^\infty$ is called a \emph{Schauder decomposition} of $X$ if every $x \in X$ has a unique representation of the form $x = \sum_{k=1}^\infty x_k$ with $x_k \in X_k$ for every $k \in \N$. A Schauder decomposition induces a sequence of coordinate projections $(P_k)_{k=1}^\infty$ on $X$ by putting
$$P_k \has{\sum_{j=1}^\infty x_j} := x_k, \qquad k \in \N.$$
We denote the partial sum projection by $S_n := \sum_{k=1}^n P_k$. Both the set of coordinate and the set of partial sum projections are uniformly bounded. A Schauder decomposition is called \emph{unconditional} if for every $x \in X$, the expansion $x = \sum_{k=1}^\infty x_k$ with $x_k \in X_k$ converges unconditionally. In this case the set of operators $U_\epsilon := \sum_{k=1}^\infty \epsilon_kP_k$, where $\epsilon = (\epsilon_k)_{k=1}^\infty$ is a sequence of signs, is also uniformly bounded.

 A Schauder decomposition $(X_k)_{k=1}^\infty$ of $X$ with $\dim(X_k)=1$ for all $k \in \N$ is called a \emph{Schauder basis}. In this case we represent $(X_k)_{k=1}^\infty$ by $\mb{x} = (x_k)_{k=1}^\infty$ with $x_k \in X_k$ for all $k \in \N$. Then there is a unique sequence of scalars $(a_k)_{k=1}^\infty$ such that  $x = \sum_{k=1}^\infty a_k x_k$ for any $x \in X$.
 The sequence of linear functionals $\mb{x}^* = (x_k^*)_{k=1}^\infty$ defined by
\begin{equation*}
  x_k^*\has{\sum_{j=1}^\infty a_j x_j} := a_k, \qquad k \in \N,
\end{equation*}
is called the biorthogonal sequence of $\mb{x}$, which is a Schauder basis of $\overline{\spn}\cbrace{x_k^*:k \in \N}$. If $\mb{x}$ is unconditional, then $\mb{x}^*$ is as well.
If $\mb{x}$ is a Schauder basis for $X$ and $\mb{y}$ is a Schauder basis for $Y$, then we say that $\mb{x}$ and $\mb{y}$ are \emph{equivalent} if $\sum_{k=1}^\infty a_kx_k$ converges in $X$ if and only if  $\sum_{k=1}^\infty a_ky_k$ in $Y$ for any sequence of scalars $(a_k)_{k=1}^\infty$. In this case $X$ and $Y$ are isomorphic.
For a further introduction to Schauder decompositions and bases, we refer to \cite{LT77}.

\subsection*{Schauder mutliplier operators}
Fix $0<\sigma<\pi$ and let $(\lambda_k)_{k=1}^\infty$ be a sequence in $\Sigma_\sigma$. We call $(\lambda_k)_{k=1}^\infty$ \emph{Hadamard} if $\abs{\lambda_1}>0$ and there is a $c>1$ such that $\abs{\lambda_{k+1}}\geq c \, \abs{\lambda_{k}}$ for all $k \in \N$.
Let $(X_k)_{k=1}^\infty$ be a Schauder decomposition of $X$ and let $(\lambda_k)_{k=1}^\infty$ be either a Hadamard sequence or an increasing sequence in $\R_+$. Consider the unbounded diagonal operator defined by
  \begin{align*}
    Ax&:= \sum_{k=1}^\infty \lambda_k P_k x,\\
    D(A)&:= \cbraces{x \in X: \sum_{k=1}^\infty \lambda_k P_k x \text{ converges in } X}.
  \end{align*}
We call $A$ the \emph{Schauder multiplier operator} associated to $(X_k)_{k=1}^\infty$ and $(\lambda_k)_{k=1}^\infty$. We will first establish that this is a sectorial operator, for which we will need the following lemma.

\begin{lemma}\label{lemma:hadamardincreasingestimate}
 Let $(\lambda_k)_{k=1}^\infty$ be either a Hadamard sequence or an increasing sequence in $\R_+$. There is a $C>0$ such that for all $\lambda \in \C \setminus \cbrace{\lambda_k:k \in \N}$
  \begin{equation*}
    \sum_{k=1}^\infty\abss{\frac{\lambda}{\lambda - \lambda_{k+1}}-\frac{\lambda}{\lambda-\lambda_{k}}} \leq C \,\sup_{k \in \N} \has{\frac{\max\cbrace{\abs{\lambda},\abs{\lambda_k}}}{\abs{\lambda-\lambda_k}}}^2.
  \end{equation*}
 \end{lemma}

\begin{proof}
For $n \in \N$ define $\mu_n := \abs{\lambda_1} + \sum_{k=1}^{n-1}\abs{\lambda_{k+1}-\lambda_{k}}$. In both cases there exists a $C_\mu>0$ such that $\abs{\lambda_k} \leq \mu_k \leq C_\mu\, \abs{\lambda_k}$ for all $k \in \N$. Fix $\lambda \in \C \setminus \cbrace{\lambda_k:k \in \N}$ and define
  \begin{equation*}
    C_\lambda := \sup_{k \in \N}  \frac{\max\cbrace{\abs{\lambda},\abs{\lambda_k}}}{\abs{\lambda-\lambda_k}}<\infty.
  \end{equation*}
  We have for all $k \in \N$
  \begin{equation*}
    \abss{\frac{\lambda}{\lambda - \lambda_{k+1}}-\frac{\lambda}{\lambda-\lambda_{k}}} \leq C_\lambda^2 \, \frac{\abs{\lambda}\abs{\lambda_{k+1}-\lambda_k}}{\max\cbrace{\abs{\lambda},\abs{\lambda_{k+1}}}\cdot \max\cbrace{\abs{\lambda},\abs{\lambda_{k}}}}.
  \end{equation*}
  Fix $n \in \N$ such that $\abs{\lambda_{n}}\leq\abs{\lambda}<\abs{\lambda_{n+1}}$ (or take $n=0$ if this is not possible). Then
  \begin{align*}
    \sum_{k=1}^{n-1}\abss{\frac{\lambda}{\lambda - \lambda_{k+1}}-\frac{\lambda}{\lambda-\lambda_{k}}}
    &\leq C_\lambda^2\, \abs{\lambda} \sum_{k=1}^{n-1} \frac{\abs{\lambda_{k+1}-\lambda_{k}}}{\abs{\lambda}^2}
    \\&\leq  C_\lambda^2 \,\abs{\lambda}^{-1} \abs{\mu_n}  \leq C_\mu C_\lambda^2,
  \end{align*}
  and
  \begin{align*}
    \sum_{k=n+1}^{\infty}\abss{\frac{\lambda}{\lambda - \lambda_{k+1}}-\frac{\lambda}{\lambda-\lambda_{k}}} &\leq C_\lambda^2 \abs{\lambda} \sum_{k=n+1}^\infty \frac{\abs{\lambda_{k+1}-\lambda_{k}}}{\abs{\lambda_{k+1}}\abs{\lambda_k}}\\
    &\leq C_\mu^2 C_\lambda^2\abs{\lambda} \sum_{k=n+1}^\infty \frac{\mu_{k+1}-\mu_{k}}{\mu_{k+1}\,\mu_k}\\
    &\leq C_\mu^2C_\lambda^2 \has{\frac{\abs{\lambda}}{\mu_{n+1}} + \lim_{k \to \infty} \frac{\abs{\lambda}}{\mu_{k}}} \leq C_\mu^2 C_\lambda^2,
  \end{align*}
  and finally
  \begin{equation*}
    \abss{\frac{\lambda}{\lambda - \lambda_{n+1}}-\frac{\lambda}{\lambda-\lambda_{n}}} \leq C_\lambda^2 \abs{\lambda} \frac{\abs{\lambda_{n+1}-\lambda_{n}}}{\abs{\lambda_{n+1}}\abs{\lambda}} \leq 2C_\lambda^2.
  \end{equation*}
  Combined this proves the lemma.
\end{proof}

To show that an operator associated to a Schauder decomposition and a Hada\-mard or increasing sequence is sectorial is now straightforward.

\begin{proposition}\label{proposition:schaudersectorial}
  Let $(X_n)_{n=1}^\infty$ be a Schauder decomposition of $X$. Let $(\lambda_k)_{k=1}^\infty$ be either a Hadamard sequence or an increasing sequence in $\R_+$. Let $A$ be the operator associated to $(X_k)_{k=1}^\infty$ and $(\lambda_k)_{k=1}^\infty$. Then $A$ is sectorial with
  $$\omega(A) = \inf\cbraceb{0<\sigma<\pi: \lambda_k \in \Sigma_\sigma \text{ for all } k \in \N}.$$
\end{proposition}

\begin{proof}
Fix $\lambda \in \C \setminus \overline{\cbrace{\lambda_k:k \in \N}}$ and define
\begin{align*}
    C_\lambda := \sup_{k \in \N}  \frac{\max\cbrace{\abs{\lambda},\abs{\lambda_k}}}{\abs{\lambda-\lambda_k}}<\infty,\qquad \qquad
    C_S := \sup_{k \in \N}\, \nrm{S_{k}}
\end{align*}
Note that for any $n \in \N$
\begin{equation}\label{eq:resolventcomputation}
  \sum_{k=1}^n \frac{1}{\lambda - \lambda_k} P_{k} = \frac{1}{\lambda-\lambda_{n+1}} S_{n}- \sum_{k={1}}^n \has{\frac{1}{\lambda - \lambda_{k+1}} - \frac{1}{\lambda - \lambda_k}} S_{k}.
\end{equation}
So by Lemma \ref{lemma:hadamardincreasingestimate} we have for all $n \in \N$
\begin{equation}\label{eq:sectorialcomputation}
\begin{aligned}
    \nrms{\sum_{k=1}^n \frac{1}{\lambda - \lambda_k} P_{k}} &= \nrms{\frac{1}{\lambda-\lambda_{n+1}} S_{n}- \sum_{k={1}}^n \has{\frac{1}{\lambda - \lambda_{k+1}} - \frac{1}{\lambda - \lambda_k}} S_{k}}\\
  &\leq \frac{C_\lambda}{\abs{\lambda}} \nrm{S_{n}} + \frac{CC_\lambda^2}{\abs{\lambda}} \, \sup_{1\leq k \leq n} \nrm{S_{k}}\\
 &\leq C C_\lambda^2 C_S \abs{\lambda}^{-1}.
 \end{aligned}
\end{equation}
By a similar computation we see that $\hab{\sum_{k=1}^n \frac{1}{\lambda - \lambda_k} P_{k}}_{n=1}^\infty$ is a Cauchy sequence and therefore convergent. Thus
\begin{equation*}
  R(\lambda):= \sum_{k=1}^\infty \frac{1}{\lambda - \lambda_k} P_{k},
\end{equation*}
is a well-defined, bounded operator on $X$. Moreover we have
\begin{equation*}
  (\lambda-A)R(\lambda)x = \sum_{k=1}^\infty \frac{\lambda}{\lambda-\lambda_k}P_{k}x-\sum_{j=1}^\infty \lambda_j P_{j} \sum_{k=1}^\infty \frac{1}{\lambda-\lambda_k}P_{k}x = x
\end{equation*}
for all $x \in X$ and similarly $R(\lambda)(\lambda-A)x = x$ for $x \in D(A)$. Therefore $\lambda \in \rho(A)$ and $R(\lambda,A) = R(\lambda)$.

Since $(X_n)_{n=1}^\infty$ is a Schauder decomposition, A is injective and $x_n \in D(A)\cap R(A)$ for $x_n \in X_n$, so $A$ has dense domain and dense range. Moreover, if we fix
$$\inf\cbraceb{0<\sigma<\pi: \lambda_k \in \Sigma_\sigma \text{ for all } k \in \N}<\sigma'<\pi,$$ then there is a $C_{\sigma'}>0$ such that
\begin{equation*}
    C_\lambda = \sup_{k\in \N} \frac{\max\cbrace{\abs{\lambda},\abs{\lambda_k}}}{\abs{\lambda - \lambda_k}} \leq C_{\sigma'} , \qquad  \lambda \in \C \setminus \overline{\Sigma}_{\sigma'}.
\end{equation*}
  So by \eqref{eq:sectorialcomputation} $A$ is sectorial with $\omega(A) \leq \sigma'$. Equality follows since $\lambda_n \in \sigma(A)$ for all $n\in \N$.
\end{proof}

From the proof of Proposition \ref{proposition:schaudersectorial} we can also see that
\begin{equation*}
  \rho(A) = \C \setminus \overline{\cbrace{\lambda_k: k \in \N}}
\end{equation*}
and for $\lambda \in \rho(A)$ we have
  \begin{equation}\label{eq:resolventmultiplier}
  R(\lambda,A)= \sum_{k=1}^\infty \frac{1}{\lambda - \lambda_k} P_{k} = \sum_{k={1}}^\infty \has{\frac{1}{\lambda - \lambda_{k}} - \frac{1}{\lambda - \lambda_{k+1}}} S_{k}.
\end{equation}
Indeed, this follows by taking limits in \eqref{eq:resolventcomputation}. Let $\omega(A)<\nu<\sigma<\pi$. Using \eqref{eq:sectorialcomputation} and the dominated convergence theorem we have for $f \in H^1(\Sigma_\sigma)$
\begin{equation}\label{eq:calculusmultiplier}
  f(A) = \int_{\Gamma_\nu}f(z)R(z,A)\dd z = \sum_{k=1}^\infty \int_{\Gamma_\nu}\frac{f(z)}{z-\lambda_k}P_{k}\dd z = \sum_{k=1}^\infty f(\lambda_k)P_{k}.
\end{equation}
To extend this to the extended Dunford calculus let $f\colon \Sigma_\sigma\to \C$ be holomorphic satisfying
\begin{equation*}
  \abs{f(z)} \leq C \abs{z}^{-\delta}\ha{1+\abs{z}}^{2\delta}
\end{equation*}
for some $C,\delta>0$ and fix $x \in X$ with $P_kx=0$ for all $k \geq N$ for some $N \in \N$. Then we have by \eqref{eq:calculusmultiplier} that
\begin{equation}\label{eq:calculusmultipliercutoff}
  f(A)x = \lim_{n \to \infty} \sum_{k=1}^N f(\lambda_k)\varphi_n^m(\lambda_k) P_k x = \sum_{k=1}^N f(\lambda_k) P_k x
\end{equation}
with $m>\delta$.

\subsection*{(Almost) $\alpha$-bounded Schauder decompositions}
Let $\alpha$ be a Euclidean structure on $X$. For the operator $A$ associated to a Schauder decomposition $(X_k)_{k=1}^\infty$ and a Hadamard sequence $(\lambda_k)_{k=1}^\infty$ we can reformulate (almost) $\alpha$-sectoriality in terms of the projections associated to $(X_k)_{k=1}^\infty$. Motivated by the following result we call $(X_k)_{k=1}^\infty$ \emph{almost $\alpha$-bounded} if the family of coordinate projections $\cbrace{P_{k}:k \in \N}$ is $\alpha$-bounded and we call $(X_k)_{k=1}^\infty$ \emph{$\alpha$-bounded} if the family of partial sum projections $\cbrace{S_{k}:k \in \N}$ is $\alpha$-bounded.

\begin{proposition}\label{proposition:alphaschauder}
  Let $(X_k)_{k=1}^\infty$ be a Schauder decomposition of $X$, $(\lambda_k)_{k=1}^\infty$ a Hadamard sequence and $A$ the sectorial operator associated to $(X_k)_{k=1}^\infty$ and $(\lambda_k)_{k=1}^\infty$. Let $\alpha$ be a Euclidean structure on $X$. Then
 \begin{enumerate}[(i)]
 \item \label{it:alphaschauder1} $A$ is almost $\alpha$-sectorial if and only if $(X_k)_{k=1}^\infty$ is almost $\alpha$-bounded. In this case $\tilde{\omega}_\alpha(A) = \omega(A)$.
 \item \label{it:alphaschauder2} $A$ is $\alpha$-sectorial if and only if $(X_k)_{k=1}^\infty$ is $\alpha$-bounded. In this case $\omega_\alpha(A) = \omega(A)$
 \end{enumerate}
\end{proposition}

In the proof of Proposition \ref{proposition:alphaschauder} we will need the following interpolating property of $H^\infty(\Sigma_\sigma)$-functions evaluated in the points of a Hadamard sequence.

\begin{lemma}\label{lemma:hadamardprop}
Fix $0<\sigma<\nu<\pi$ and let $(\lambda_k)_{k=1}^\infty$ be a Hadamard sequence in $\Sigma_\sigma$.
 For all $\mbs{a} \in \ell^\infty$ there exists an $f \in H^\infty(\Sigma_\sigma)$ such that
  \begin{equation*}
    f(\lambda_k) = a_k, \qquad k \in \N
  \end{equation*}
  and  $\nrm{f}_{H^\infty(\Sigma_\sigma)} \lesssim \nrm{\mbs{a}}_{\ell^\infty}$.
\end{lemma}

\begin{proof}
  The lemma states that $(\lambda_k)_{k=1}^\infty$ is an \emph{interpolating sequence} for $H^\infty(\Sigma_\sigma)$. On the upper half-plane a theorem due to Carleson (see for example \cite{Ga07}) states that $(\zeta_k)_{k=1}^\infty$ is an interpolating sequence if and only if
  \begin{equation*}
    \prod_{j \in \N \setminus \cbrace{k}}\abss{\frac{\zeta_k-\zeta_j}{\zeta_k -\overline{\zeta}_j}}>0, \qquad k \in \N.
  \end{equation*}
   Since the function $z \mapsto i z^{\frac{\pi}{2\nu}}$ conformally maps $\Sigma_\sigma$ onto the upper half-plane, it suffices to show
  \begin{equation*}
    \prod_{j \in \N \setminus \cbrace{k}}\abss{\frac{\mu_k-\mu_j} {\mu_k +\overline{\mu}_j}}>0
  \end{equation*}
  for $\mu_k = \lambda^{\frac{\pi}{2\sigma}}_k$. Fix $k \in \N$, then we have
  \begin{equation*}
    \prod_{j =1}^{k-1}\abss{\frac{\mu_k-\mu_j} {\mu_k +\overline{\mu}_j}} \geq \prod_{j =1}^{k-1}\frac{\abs{\mu_k}-\abs{\mu_j}}{\abs{\mu_k} + \abs{{\mu}_j}} = \prod_{j =1}^{k-1}\has{1- \frac{2\abs{\mu_j}}{\abs{\mu_k}+\abs{\mu_j}}} \geq \prod_{j =1}^{k-1}\has{1- \frac{2}{c^{k-j}+1}},
  \end{equation*}
  where $c>1$ is such that $\abs{\mu_{k+1}}\geq c \, \abs{\mu_{k}}$ for all $k \in \N$. A similar inequality holds  for the product with $j \geq k+1$. Therefore, since $\sum_{j=1}^\infty \frac{2}{c^{j}+1}<\infty$, it follows that
  \begin{equation*}
    \prod_{j \in \N \setminus \cbrace{k}}\abss{\frac{\mu_k-\mu_j} {\mu_k +\overline{\mu}_j}} \geq \has{\prod_{j =1}^{\infty}\has{1- \frac{2}{c^{j}+1}}} ^2>0,
  \end{equation*}
which finishes the proof.
\end{proof}

\begin{proof}[Proof of Proposition \ref{proposition:alphaschauder}]
Fix $\omega(A) < \nu<\sigma<\pi$. For statement \ref{it:alphaschauder1} first assume that $\cbrace{P_{k}:k \in \N}$ is $\alpha$-bounded. Take $f \in H^1(\Sigma_\nu)$, then by \eqref{eq:calculusmultiplier} we have for $t>0$
\begin{equation*}
  f(tA) = \sum_{k=1}^\infty f(t\lambda_k)P_{k}
\end{equation*}
  and by Lemma  \ref{lemma:hadamardsum} we have $\sum_{k=1}^\infty \abs{f(t\lambda_k)} \leq C$ for $C>0$ independent of $t$. Therefore it follows by Proposition \ref{proposition:alphaproperties} that $\cbrace{f(tA):t>0}$ is $\alpha$-bounded. Thus  $A$ is almost $\alpha$-sectorial with $\tilde{\omega}(A) \leq \sigma$ by Proposition \ref{proposition:almostsectorialcharacterization}.

 Conversely assume that $A$ is almost $\alpha$-sectorial and set $t_k = \abs{\lambda_k}$ for $k \in \N$. By Lemma \ref{lemma:hadamardprop} there is a sequence of functions  $(f_j)^\infty_{j=1}$ in $H^\infty(\Sigma_\sigma)$ with $\nrm{f_j}_{H^\infty(\Sigma_\sigma)} \leq C$ such that $f_j(\lambda_k) = \delta_{jk}$ for all $j,k \in \N$. Take
 \begin{equation*}
   g_j(z) = \frac{z}{(1+z)^{2}}\frac{(t_j+\lambda_j)^2}{t_j\lambda_j}f_j(t_jz),\qquad z \in \Sigma_\sigma,
 \end{equation*}
 then $(g_j)^\infty_{j=1}$ is uniformly in $H^1(\Sigma_\sigma)$.
 Therefore $\cbrace{g_j(t_j^{-1}A): j \in \N}$ is $\alpha$-bounded by Proposition \ref{proposition:almostsectorialcharacterization}. By \eqref{eq:calculusmultiplier} we have for $j \in \N$
 \begin{align*}
   g_j(t_j^{-1}A)  &= \frac{t_j^{-1}\lambda_j}{(1+t_j^{-1}\lambda_j)^2} \frac{(t_j+\lambda_j)^2}{t_j\lambda_j}P_{j} = P_{j}.
 \end{align*}
 So the family of coordinate projections $\cbrace{P_{k}:k \in \N}$ is $\alpha$-bounded, i.e. $(X_k)_{k=1}^\infty$ is almost $\alpha$-bounded.

 \bigskip

 For \ref{it:alphaschauder2} assume that $\cbrace{S_{k}:k \in \N}$ is $\alpha$-bounded and take $\lambda \in \C \setminus \overline{\Sigma}_\sigma$. Using the expression for the resolvent of $A$ from \eqref{eq:resolventmultiplier}, we have
\begin{equation*}
  \lambda R(\lambda,A) = \sum_{k={1}}^\infty \has{\frac{\lambda}{\lambda - \lambda_{k+1}} - \frac{\lambda}{\lambda - \lambda_k}} S_{k}.
\end{equation*}
Therefore the set $\cbrace{\lambda R(\lambda,A):\lambda \in \C \setminus \overline{\Sigma}_\sigma}$ is $\alpha$-bounded by Lemma \ref{lemma:hadamardincreasingestimate} and Proposition \ref{proposition:alphaproperties}, so $A$ is $\alpha$-sectorial with $\omega_\alpha(A) \leq \sigma$.

Conversely assume that $A$ is $\alpha$-sectorial and set $t_k = \abs{\lambda_k}$ for $k \in \N$. As $(\lambda_k)_{k=1}^\infty$ is an interpolating sequence for $H^\infty(\Sigma_\sigma)$ by Lemma \ref{lemma:hadamardprop}, we can find a sequence of functions $(f_j)^\infty_{j=1}$ in $H^\infty(\Sigma_\sigma)$ such that
  \begin{equation*}
    f_j(\lambda_k) = \begin{cases}
      1+\lambda_kt_j^{-1} &1 \leq k \leq j\\
      -1-t_j\lambda_k^{-1} &j<k
    \end{cases}
  \end{equation*}
  for all $j,k \in \N$ with $\nrm{f_j}_{H^\infty(\Sigma_\sigma)} \leq C$. Now let $g_j(z) = z(1+z)^{-2}f_j(t_jz)$. Then $(g_j)^\infty_{n=1}$ is uniformly in $H^1(\Sigma_\sigma)$
 and therefore $\cbrace{g_j(t_j^{-1}A): j \in \N}$ is $\alpha$-bounded by Proposition \ref{proposition:almostsectorialcharacterization}. Again using \eqref{eq:calculusmultiplier}, we have for all $n \in \N$
 \begin{align*}
   g_n(t_n^{-1}A)
   &= \sum_{k=1}^n \frac{t_n^{-1}\lambda_k (1+t_n^{-1}\lambda_k )}{(1+t_n^{-1}\lambda_k)^{2}}P_{k} - \sum_{k=n+1}^\infty \frac{t_n^{-1}\lambda_k (1+t_n\lambda_k^{-1} )}{(1+t_n^{-1}\lambda_k)^{2}}P_{k}\\
   &= \sum_{k=1}^n \frac{t_n^{-1}\lambda_k }{1+t_n^{-1}\lambda_k}P_{k} - \sum_{k=n+1}^\infty \frac{1}{1+t_n^{-1}\lambda_k}P_{k}\\
   &= \sum_{k=1}^n P_{k} - \sum_{k=1}^\infty \frac{t_n}{t_n+\lambda_k}P_{k}\\
   &= S_{n} + t_nR(-t_n,A).
 \end{align*}
 Since $A$ is $\alpha$-sectorial, the set $\cbrace{t_kR(-t_k,A),k \in \N}$ is $\alpha$-bounded. Therefore the family of partial sum projections $\cbrace{S_{k}:k \in \N}$ is $\alpha$-bounded, i.e. $(X_k)_{k=1}^\infty$ is $\alpha$-bounded.
\end{proof}

\section{Sectorial operators which are not almost \texorpdfstring{$\alpha$}{a}-sectorial}\label{section:sectorialnotalmostalpha}
In this section we will start our series of examples based on the sectorial operators defined in Section \ref{section:schaudersectorial}. In this section we will construct a Schauder basis which is not almost $\alpha$-bounded for any unconditionally stable Euclidean structure $\alpha$, e.g. the $\gamma$-structure on a space with finite cotype or the $\ell^2$-structure on a Banach lattice. In view of Proposition \ref{proposition:alphaschauder} this yields sectorial operators which are not almost $\alpha$-sectorial. Our proof will basically be a reconstruction of the idea of Lancien and the first author \cite{KL00} to construct sectorial operators that are not $\mc{R}$-sectorial. This idea has been further developed in a sequence of papers by Fackler \cite{Fa13, Fa14, Fa15, Fa16} and was recently revisited by Arnold and Le Merdy \cite{AL19}.

\subsection*{The spaces $\ell^1$ and $c_0$}
As a warm up we consider the sequence spaces $\ell^1$ and $c_0$.
\begin{proposition}\label{proposition:l1c0schauder}
Both $\ell^1$ and $ c_0$ have a Schauder basis $\mb{x}$ which is not almost $\alpha$-bounded for any unconditionally stable Euclidean structure $\alpha$ on $X$.
\end{proposition}

\begin{proof}
  For $c_0$ we consider the so-called summing basis $\mb{x}$, given by $x_n = \sum_{k=1}^n e_k$, where $(e_k)_{k=1}^\infty$ is the canonical basis of $c_0$. The biorthogonal sequence $\mb{x}^*$ in $\ell^1$ is given by $x_n^* = e^*_n-e^*_{n+1}$ for $n \in \N$, where $(e_k^*)_{k=1}^\infty$ is the canonical basis of $\ell^1$. Let $\alpha$ be an unconditionally stable Euclidean structure on $c_0$ and suppose that $\mb{x}$ is almost $\alpha$-bounded. Let $(P_k^{\mb{x}})_{k=1}^\infty$ be the coordinate projections associated to $\mb{x}$. Then  we have for any $n \in \N$
  \begin{align*}
    \nrmb{(x_k)_{k=1}^n}_{\alpha} = \nrmb{(P_k^{\mb{x}}e_k)_{k=1}^n}_{\alpha}\lesssim \nrmb{(e_k)_{k=1}^n}_{\alpha} \lesssim \sup_{\abs{\epsilon_k}=1}\nrms{\sum_{k=1}^n \epsilon_k e_k}_{\ell^\infty} =1.
  \end{align*}
  Since $e_1^*(x_k)=1$ for all $k \in \N$, we also have by Proposition \ref{proposition:finitedimensionalalpha}
\begin{equation*}
  n^{1/2} = n^{-1/2}\sum_{k=1}^n \abs{e_1^*(x_k)} \leq n^{-1/2}\nrmb{(x_k)_{k=1}^n}_{\alpha} \nrmb{(e_1^*)_{k=1}^n}_{\alpha^*} = \nrmb{(x_k)_{k=1}^n}_{\alpha},
\end{equation*}
a contradiction. So $\mb{x}$ is not almost $\alpha$-bounded.

The argument for $\ell^1$ is dual. We consider the basis $\mb{y}^*:=(e_1^*,x_1^*,x_2^*,\ldots)$ with biorthogonal sequence $(e,x_1-e,x_2-e,\ldots)$, where $e \in \ell^\infty$ is the sequence $(1,1,\ldots)$. Now let $\beta$ be an unconditionally stable Euclidean structure on $\ell^1$ and suppose that $\mb{y}^*$ is almost $\beta$-bounded.  Then we have for any $n \in \N$
\begin{align*}
  n = \sum_{k=1}^n  x_k^*(e_k) \leq \nrmb{( x_k^*)_{k=1}^n}_\beta \nrmb{(e_k)_{k=1}^n}_{\beta^*} \lesssim \nrmb{(x_k^*)_{k=1}^n}_{\beta}
\end{align*}
and, denoting the coordinate projections associated to $\mb{y}^*$ by $(P_k^{\mb{y}^*})_{k=1}^\infty$, we have
\begin{equation*}
\nrmb{(x_k^*)_{k=1}^n}_{\beta}  = \nrmb{(P_{k+1}^{\mb{y}^*}e_n^*)_{k=1}^n}_{\beta} \lesssim \nrmb{(e_n^*)_{k=1}^n}_{\beta} \leq  n^{1/2},
\end{equation*}
a contradiction. So $\mb{y}^*$ is not almost $\beta$-bounded.
\end{proof}

\subsection*{The general case}
The general case will follow from the following lemma, which is a consequence of a result by Lindenstrauss and Zippin \cite{LZ69}.

\begin{lemma}\label{lemma:lindenstrausszippin}
  Suppose that $X$ has  an unconditional Schauder basis and is not isomorphic to $\ell^1$, $\ell^2$ or $c_0$. Then there is an unconditional Schauder basis $\mb{x}$ of $X$, a permutation $\pi\colon\N\to \N$ and a sequence of scalars $(a_k)_{k=1}^\infty$ such that $\sum_{k=1}^\infty a_kx_{2k-1}$ converges but $\sum_{k=1}^\infty a_kx_{\pi(2k)}$ does not converge.
\end{lemma}

\begin{proof}
  By \cite[Note (1) at the end]{LZ69} we know that $X$ has an unconditional, \emph{non-symmetric} basis $\mb{x}$, i.e. there is a permutation $\pi\colon\N\to\N$ such that $\mb{x}$ and $(x_{\pi(k)})_{k=1}^\infty$ are not equivalent. This implies the claim by the first part of the proof of \cite[Chapter 2, Proposition 23.2]{Si70}.
\end{proof}

We are now ready to prove the main result of this section.
\begin{theorem}\label{theorem:notalmostalphasectorial}
  Suppose that $X$ has  an unconditional Schauder basis and is not isomorphic to $\ell^2$. Then $X$ has a Schauder basis which is not almost $\alpha$-bounded for any unconditionally stable Euclidean structure $\alpha$.
\end{theorem}

\begin{proof}
  If $X$ is isomorphic to $\ell^1$ or $c_0$, the theorem follows from Proposition \ref{proposition:l1c0schauder}. Otherwise we can use Lemma \ref{lemma:lindenstrausszippin} to find an unconditional Schauder basis $\mb{x}$ of $X$, a permutation $\pi\colon\N\to \N$ and a sequence of scalars $(a_k)_{k=1}^\infty$ such that $\sum_{k=1}^\infty a_kx_{2k-1}$ converges but $\sum_{k=1}^\infty a_kx_{\pi(2k)}$ does not converge.

  Define for $k \in \N$
  \begin{equation*}
    y_k = \begin{cases}
      x_k + x_{\pi(k+1)}\quad &\text{if $k$ is odd},\\
      x_{\pi(k)}\quad &\text{if $k$ is even}.
    \end{cases}
  \end{equation*}
  Then $\mb{y}$ is an unconditional Schauder basis of $X$ and its biorthogonal sequence $\mb{y}^*$ is an unconditional Schauder basis for $\overline{\spn}\cbrace{y_k^*:k \in \N}$.
  Let $\alpha$ be an unconditionally stable Euclidean structure and assume that $\mb{y}$ is almost $\alpha$-bounded. Fix $m,n \in \N$ and let $(b_k)_{k=m+1}^n$ be such that $\sum_{k=m+1}^nb_ky_{2k}^*$ has norm $1$ and
  \begin{equation*}
     \nrms{\sum_{k=m+1}^n a_kx_{\pi(2k)}}_X = \sum_{k=m+1}^n \absb{b_ky_{2k}^*(a_ky_{2k})}.
  \end{equation*}
  Let $(P_k^{\mb{y}})_{k=1}^\infty$ be the coordinate projections associated to $\mb{y}$. Then since $x_{2k-1} = y_{2k-1}-y_{2k}$, we obtain by Proposition \ref{proposition:alphaschauder} and the unconditionality of $\mb{y}$, $\mb{y}^*$ and $\alpha$
  \begin{align*}
    \nrmb{\sum_{k=m+1}^n a_kx_{\pi(2k)}}_X &=  \sum_{k=m+1}^n \absb{b_ky_{2k}^*(a_ky_{2k})}\\
    &\leq \nrmb{(-a_kP^{\mb{y}}_{2k}x_{2k-1})_{k=m+1}^n}_\alpha \nrmb{(b_ky^*_{2k})_{k=m+1}^n}_{\alpha^*}\\
    &\lesssim \sup_{\abs{\epsilon_k}=1} \nrmb{\sum_{k=m+1}^n \epsilon_k a_k x_{2k-1}}_X \sup_{\abs{\epsilon_k}=1} \nrmb{\sum_{k=m+1}^n \epsilon_k b_k y^*_{2k}}_{X^*}\\
    &\lesssim  \nrmb{\sum_{k=m+1}^n a_kx_{2k-1}}_X.
  \end{align*}
But since $\sum_{k=1}^\infty a_kx_{2k-1}$ converges this implies that  $\sum_{k=1}^\infty a_kx_{\pi(2k)}$ converges, a contradiction. So $\mb{y}$ is not almost $\alpha$-bounded.
\end{proof}

Theorem \ref{theorem:notalmostalphasectorial} combined with Proposition \ref{proposition:alphaschauder} yields the result we were after in this section:
\begin{corollary}\label{corollary:notalmostalpha}
  Suppose that $X$ has  an unconditional Schauder basis and is not isomorphic to $\ell^2$. Then $X$ has a Schauder basis $\mb{x}$ such that for any Hadamard sequence $(\lambda_k)_{k=1}^\infty$ the operator associated to $\mb{x}$ and $(\lambda_k)_{k=1}^\infty$ is not almost $\alpha$-sectorial for any unconditionally stable Euclidean structure $\alpha$.
\end{corollary}

In particular this implies the following corollary, since the Haar basis is unconditional in $L^p(\R)$ for $p \in (1,\infty)$.

\begin{corollary}
  Let $p \in (1,\infty)\setminus \cbrace{2}$. Then there is a sectorial operator on $L^p(\R)$ which is not almost $\gamma$-sectorial.
\end{corollary}

\section{Almost \texorpdfstring{$\alpha$}{a}-sectorial operators which are not \texorpdfstring{$\alpha$}{a}-sectorial}\label{section:almostalphanotalpha}
Building upon the results of the previous section, we will now construct a Schauder basis that is almost $\alpha$-bounded, but not $\alpha$-bounded for any unconditionally stable Euclidean structures $\alpha$. In view of Proposition \ref{proposition:alphaschauder} this yields examples of sectorial operators that are almost $\alpha$-sectorial, but not $\alpha$-sectorial.

We start with a useful criterion for the almost $\alpha$-boundedness of a Schauder basis.

\begin{lemma}\label{lemma:schauderduality}
  Let $\mb{x}$ be a Schauder basis of $X$ with biorthogonal sequence $\mb{x}^*$, let $\alpha$ be a Euclidean structure on $X$ and take $p \in [1,\infty]$. If there is a $C>0$ such that for all sequences of scalars $(a_k)_{k=1}^n$ and $(b_k)_{k=1}^n$ we have
  \begin{align*}
    \nrmb{(a_1x_1,\ldots,a_nx_n)}_{\alpha} &\leq C \, \nrmb{(a_k)_{k=1}^n}_{\ell^p_n} ,\\
    \nrmb{(b_1x^*_1,\ldots,b_nx^*_n)}_{\alpha^*} &\leq C \, \nrmb{(b_k)_{k=1}^n}_{\ell^{p'}_n},
  \end{align*}
  then $\mb{x}$ is almost $\alpha$-bounded.
\end{lemma}

\begin{proof}
Fix $n \in \N$ and take $\mb{y} \in X^n$. Define $a_k = x^*_k(y_k)$ for $k =1,\ldots,n$. Let $(b_k)_{k=1}^n$ be such that $\nrmb{(b_k)_{k=1}^n}_{\ell^{p'}_n} =1$ and
$  \sum_{k=1}^n a_kb_k = \nrmb{(a_k)_{k=1}^n}_{\ell^p_n} .$
Let $(P_k^{\mb{x}})_{k=1}^\infty$ be the coordinate projections associated to $\mb{x}$. Then
  \begin{align*}
    \nrmb{(P_k^{\mb{x}}y_k)_{k=1}^n}_{\alpha} &= \nrmb{(a_kx_k)_{k=1}^n}_{\alpha} \leq C \, \nrmb{(a_k)_{k=1}^n}_{\ell^p_n} = C\, \sum_{k=1}^n b_k x^*_k(y_k)
    \leq C^2\, \nrm{\mb{y}}_{\alpha}.
  \end{align*}
 In the same way we obtain for $m_1,\ldots,m_n \in \N$ distinct that
    \begin{align*}
    \nrmb{(P_{m_k}^{\mb{x}}y_k)_{k=1}^n}_{\alpha} \leq C^2\, \nrm{\mb{y}}_{\alpha}.
  \end{align*}
  To allow repetitions we consider index sets $I_j = \cbrace{k \in \N: m_k = j}$ and let $N \in \N$ be such that $I_j = \emptyset$ for $j >N$. By the right ideal property of a Euclidean structure and choosing appropriate $c_{jk}$'s with $\sum_{k \in I_j} \abs{c_{jk}}^2 =1$ for $j=1,\ldots,N$ we have
  \begin{align*}
    \nrmb{(P_{m_k}^{\mb{x}}y_k)_{k=1}^n}_{\alpha} &=  \nrms{\has{\hab{\sum_{k \in I_j}\abs{x_j^*(y_k)}^2}^{1/2}x_j}_{j=1}^N}_\alpha\\
    &= \nrms{\hab{\sum_{k \in I_j}c_{jk}x_j^*(y_k)x_j}_{j=1}^N}_\alpha\\
    &= \nrms{\has{P_j^{\mb{x}}\hab{\sum_{k \in I_j}c_{jk}y_k}}_{j=1}^N}_\alpha \\
    &\leq C^2 \, \nrms{\hab{\sum_{k \in I_j}c_{jk}y_k}_{j=1}^N}_\alpha
    = C^2\, \nrm{\mb{y}}_{\alpha}.
  \end{align*}
  Therefore $\cbrace{P_k^{\mb{x}}:k\in \N}$ is $\alpha$-bounded, i.e. $\mb{x}$ is almost $\alpha$-bounded.
\end{proof}

With this lemma at our disposal we can now turn to the main result of this section. We take the example in Theorem \ref{theorem:notalmostalphasectorial} as a starting point to construct an example of a Schauder basis that is almost $\alpha$-sectorial, but not $\alpha$-sectorial.

\begin{theorem}\label{theorem:almostalphasectorialnotalphasectorial}
Let $\alpha$ be an ideal unconditionally stable Euclidean structure on $X$. Suppose that
\begin{itemize}
  \item $X$ has a Schauder basis $\mb{x}$ which is not $\alpha$-bounded.
  \item $X$ has a complemented subspace isomorphic to $\ell^p$ for some $p \in [1,\infty)$ or isomorphic to $c_0$.
\end{itemize}
 Then $X$ has a Schauder basis $\mb{y}$ which is almost $\alpha$-bounded, but not $\alpha$-bounded.
\end{theorem}

\begin{proof}
We will consider the $\ell^p$-case, the calculations for $c_0$ are similar and left to the reader. By assumption there is a $p \in [1,\infty)$ and a subspace $W$ of $X$ for which we have the following chain of isomorphisms
  \begin{equation*}
    X = W \oplus \ell^p = W \oplus \ell^p \oplus \ell^p = X \oplus \ell^p.
  \end{equation*}
  Thus we can write $X = Y \oplus Z$ where $Y$ is isomorphic to $X$ and $Z$ is isomorphic to $\ell^p$. We denote the projection from $X$ onto $Y$ by $P_Y$ and let $V\colon Y  \to X$ be an isomorphism.

  Let $(e_k)_{k=1}^\infty$ be a Schauder basis of $Z$ equivalent to the canonical basis of $\ell^p$. We consider the Schauder basis $\mb{u}$ of $X$ given by
  \begin{equation*}
    u_k =\begin{cases}
    e_1 &\text{if } k=1\\
    e_{k-j+1} &\text{if } 2^j+1 \leq k \leq  2^{j+1}-1 \text{ for } j \in \N\\
    \frac{1}{\nrm{V^{-1}x_j}_X}V^{-1}x_j &\text{if } k = 2^j \text{ for } j \in \N.
    \end{cases}
  \end{equation*}
  For $j\in \N$ we define
  \begin{align*}
    v_j &= \has{(2^j-1)^{-1/p}\sum_{k=2^j+1}^{2^{j+1}-1}u_k}-u_{2^j}\\
    v_j^* &= \has{(2^j-1)^{-1/p'}\sum_{k=2^j+1}^{2^{j+1}-1}u_k^*}+u^*_{2^j}.
  \end{align*}
  If we now define  the operators $T_jx = v_j^*(x)v_j$ for $x \in X$  and $j \in \N$, then
  \begin{itemize}
    \item $\nrm{T_j}\leq 4$, since $\nrm{v_j}_X, \nrm{v_j^*}_{X^*} \leq 2$ .
    \item $T_j^2 = 0$, since $v_j^*(v_j)=0$.
    \item $T_j$ leaves the subspace ${\spn}\cbrace{u_k:2^j \leq k \leq 2^{j+1}-1}$ invariant.
  \end{itemize}
  Therefore $I+T_j$ is an automorphism of ${\spn}\cbrace{u_k:2^{j-1}+1 \leq k \leq 2^{j}}$, so we can make a new basis $\mb{y}$ of $X$, given by
  \begin{equation*}
    y_k = \begin{cases}
      u_1 &\text{if } k=1,\\
   \frac{1}{\nrm{(I+T_j)u_k}_{X}} \,
   (I+T_j)u_k &\text{if } 2^j \leq k \leq  2^{j+1}-1 \text{ for } j \in \N.
    \end{cases}
  \end{equation*}
  Let $(S_k^{\mb{x}})_{k=1}^\infty$, $(S_k^{\mb{y}})_{k=1}^\infty$ and $(S_k^{\mb{u}})_{k=1}^\infty$ be the partial sum projections associated to $\mb{x}$, $\mb{y}$ and $\mb{u}$ respectively. Then $S^{\mb{u}}_{2^{k+1}-1}= S^{\mb{y}}_{2^{k+1}-1}$ and thus
  \begin{equation*}
    S_k^{\mb{x}} = VP_YS^{\mb{u}}_{2^{k+1}-1}V^{-1} =  VP_YS^{\mb{y}}_{2^{k+1}-1}V^{-1}
  \end{equation*}
  for all $k \in \N$. Since $\alpha$ is ideal and $(S_k^{\mb{x}})_{k=1}^\infty$ is not $\alpha$-bounded, we have by Proposition \ref{proposition:idealsingleoperator} that $(S_k^{\mb{y}})_{k=1}^\infty$ is not $\alpha$-bounded. So $\mb{y}$ is not $\alpha$-bounded.

  Next we show that $\mb{y}$ is almost $\alpha$-bounded. We will prove that there is a $C>0$ such that for all scalar sequences $(a_k)^n_{k=1}$ and $(b_k)^n_{k=1}$  we have
  \begin{align}
   \label{eq:yschauder} \nrmb{(a_1y_1,\ldots,a_ny_n)}_{\alpha} &\leq C \, \has{\sum_{k=1}^n\abs{a_k}^p}^{1/p},\\
    \label{eq:ystarschauder}\nrmb{(b_1y^*_1,\ldots,b_ny^*_n)}_{\alpha^*} &\leq C \, \has{\sum_{k=1}^n\abs{b_k}^{p'}}^{1/p'},
  \end{align}
  for all $n \in \N$. By Lemma \ref{lemma:schauderduality} this implies  that $\mb{y}$ is almost $\alpha$-bounded.
   The calculations for \eqref{eq:yschauder} and \eqref{eq:ystarschauder} are similar, so we will only treat \eqref{eq:yschauder}. Fix $m \in \N$, let $n = 2^{m+1}-1$ and define $c_j = 2^j-1$ for $j \in \N$. First suppose that $a_{2^j}=0$ for $1\leq j \leq m$. Then, using the triangle inequality, the unconditonal stability of $\alpha$ and the fact that $(e_k)_{k=1}^\infty$ is equivalent to the canonical basis of $\ell^p$,
   we have
   \begin{align*}
     \nrmb{(a_ky_k)_{k=1}^n}_{\alpha} &\leq \nrmb{(a_ku_k)_{k=1}^n}_{\alpha} + \nrms{\has{c_j^{-1/p'}\has{\sum_{k=2^j+1}^{2^{j+1}-1}\abs{a_k}^2}^{1/2}v_j}_{j=1}^m}_{\alpha}\\
     &\leq C \,\has{\sum_{k=1}^n\abs{a_k}^p}^{1/p} + \sum_{j=1}^m c_j^{-1/p'}\has{\sum_{k=2^j+1}^{2^{j+1}-1}\abs{a_k}^2}^{1/2} \nrm{v_j}_X,
   \end{align*}
  If $p \in [1,2]$, we estimate the second term by H\"older's inequality
   \begin{align*}
     \sum_{j=1}^m c_j^{-1/p'}\has{\sum_{k=2^j+1}^{2^{j+1}-1}\abs{a_k}^2}^{1/2}\nrm{v_j}_X  &\leq 2 \sum_{j=1}^m c_j^{-1/p'}\has{\sum_{k=2^j+1}^{2^{j+1}-1}\abs{a_k}^p}^{1/p} \\ &\leq 2
      \has{\sum_{j=1}^m c_j^{-1}}^{1/p'} \has{\sum_{k=1}^{n}\abs{a_k}^p}^{1/p}
   \end{align*}
and, if $p \in (2,\infty)$, we estimate the second term by applying H\"older's inequality twice
   \begin{align*}
     \sum_{j=1}^m c_j^{-1/p'}\has{\sum_{k=2^j+1}^{2^{j+1}-1}\abs{a_k}^2}^{1/2}\nrm{v_j}_X  &\leq  2\sum_{j=1}^m c_j^{-1/p}\has{\sum_{k=2^j+1}^{2^{j+1}-1}\abs{a_k}^p}^{1/p}\\ &\leq 2 \has{\sum_{j=1}^m c_j^{-p'/p}}^{1/p'} \has{\sum_{k=1}^{n}\abs{a_k}^p}^{1/p}.
   \end{align*}
Combined this yields \eqref{eq:yschauder} if $a_{2^j}=0$ for $1\leq j \leq m$. Now assume that $a_{k}=0$ unless $k = 2^j$ for $1\leq j \leq m$. Since we have
   \begin{equation*}
     y_{2^j} =\frac{u_{2^j} + v_j}{\nrm{u_{2^j} + v_j}_{X}}  = \frac{c_j^{-1/p}}{\nrm{u_{2^j} + v_j}_{X}}
     \sum_{k=2^j+1}^{2^{j+1}-1}u_k,
   \end{equation*}
   we immediately obtain
   \begin{align*}
     \nrmb{(a_ky_k)_{k=1}^n}_{\alpha} &= C \has{\sum_{k=1}^n\abs{a_k}^p}^{1/p}
   \end{align*}
   again using that $\alpha$ is unconditionally stable and $(e_k)_{k=1}^\infty$ is equivalent to the canonical basis of $\ell^p$. The estimate for general $(a_k)^n_{k=1}$ now follows by the triangle inequality.
\end{proof}

Theorem \ref{theorem:almostalphasectorialnotalphasectorial} combined with Theorem \ref{theorem:notalmostalphasectorial} and Proposition \ref{proposition:alphaschauder} yields examples of sectorial operators that are almost $\alpha$-sectorial, but not $\alpha$-sectorial:

\begin{corollary}\label{corollary:almostalphanotalphasectorial}
Let $\alpha$ be an unconditionally stable Euclidean structure on $X$. Suppose that
\begin{itemize}
  \item $X$ has an unconditional Schauder basis.
  \item $X$ is not isomorphic to $\ell^2$.
  \item $X$ has a complemented subspace isomorphic to $\ell^p$ for some $p \in [1,\infty)$ or isomorphic to $c_0$.
\end{itemize}
 Then $X$ has a Schauder basis $\mb{x}$ such that for any Hadamard sequence $(\lambda_k)_{k=1}^\infty$ the operator associated to $\mb{x}$ and $(\lambda_k)_{k=1}^\infty$ is almost $\alpha$-sectorial, but not $\alpha$-sectorial.
\end{corollary}

Specifically for the $\gamma$-structure we have:
\begin{corollary}
  Let $p \in [1,\infty)\setminus \cbrace{2}$. There is a sectorial operator on $L^p(\R)$ which is almost $\gamma$-sectorial, but not $\gamma$-sectorial.
\end{corollary}

\begin{proof}
  If $p \in (1,\infty)\setminus \cbrace{2}$ this follows from Corollary \ref{corollary:almostalphanotalphasectorial}, since the Haar basis is unconditional. Any Schauder basis of $L^1(\R)$ is not $\mc{R}$-bounded and thus not $\gamma$-bounded by \cite[Theorem 3.4]{HKK04}, so for $p=1$ we can directly apply Theorem \ref{theorem:almostalphasectorialnotalphasectorial}.
\end{proof}

\section{Sectorial operators with \texorpdfstring{$\omega_{H^\infty}(A)>\omega(A)$}{wH(A)>w(A)}} \label{section:angleexample}
Let $A$ be a sectorial operator on $X$ and $\alpha$ a Euclidean structure on $X$. We have seen in Proposition \ref{proposition:alphaalmoastalpha}, Proposition \ref{proposition:BIPideal}, Theorem \ref{theorem:BIPHinfty} and Corollary \ref{corollary:equalangles} that under reasonable assumptions on $\alpha$ the angles of (almost) $\alpha$-sectoriality, ($\alpha$-)$\BIP$ and of the ($\alpha$-)-bounded $H^\infty$-calculus are equal whenever $A$ has these properties. Strikingly absent in this list is the angle of sectoriality.

In general the angle of sectoriality is not equal to the other introduced angles in Chapter \ref{part:4}. As we already noted in Section \ref{section:BIP}, Haase showed in \cite[Corollary 5.3]{Ha03} that there exists a sectorial operator $A$ with $\omega_{\BIP}(A)>\omega(A)$. The first counterexample to the equality $\omega_{H^\infty}(A)=\omega(A)$ was given by Cowling, Doust, McIntosh and Yagi \cite[Example 5.5]{CDMY96}, who constructed an operator (without dense range) with a bounded $H^\infty$-calculus, such that $\omega(A) <\omega_{H^\infty}(A)$. Subsequently, the first author constructed a sectorial operator with $\omega(A) <\omega_{H^\infty}(A)$ in \cite{Ka03}. Both these examples are on very specific (non-reflexive) Banach spaces  and it is an open problem whether every infinite-dimensional Banach space admits such an example. In particular in \cite[Problem P.13]{HNVW17} it was asked whether there exists examples on $L^p$. In this section we will provide an example on a subspace of $L^p$ for any $p \in (1,\infty)$.

Let us note that all known examples that make their appearance in applications actually satisfy $\omega_{H^\infty}(A) = \omega(A)$. This holds in particular for classical operators like the Laplacian on $L^p(\R^d)$, but also for far more general elliptic operators as shown \cite{Au07}, which is based on earlier results in \cite{BK03, DM99, DR96}. More recent developments in this direction can for example be found in \cite{CD16b, CD19, Eg18, Eg18b, EHRT19}. Also for example for the Ornstein-Uhlenbeck operator we have $\omega_{H^\infty}(A) = \omega(A)$ (see e.g. \cite{Ca09,CD16,GMMST01, Ha19}). Even in more abstract situations, like the H\"ormander-type holomorphic functional calculus for symmetric contraction semigroups on $L^p$, the angle of the functional calculus, is equal to the angle of sectoriality (see \cite{CD17}). This means that our example will have to be quite pathological.

\subsection*{The general idea}
We will proceed as follows: We will construct a Banach space $X$ and a Schauder multiplier operator such that $\omega(A) = 0$ and such that, on the generalized square function spaces introduced in Section \ref{section:scalespaces}, the induced operator $A^s|_{H^\gamma_{\theta,A^s}}$ does not have a bounded $H^\infty$-calculus for $s>1$. Then $\omega_{\BIP}({A^s|_{H^\gamma_{\theta,A^s}}}) = \pi$ by Theorem  \ref{theorem:XABIPHinftyequiv} and \eqref{eq:BIPangle}, so using Theorem \ref{theorem:XalphaHinfty} and Proposition \ref{proposition:Hequivnorm} we know that  $A|_{H^\gamma_{\theta,A}(\varphi_s)}$ with $\varphi_s(z)=\frac{z^{s/2}}{1+z^s}$  has a bounded $H^\infty$-calculus. Therefore
\begin{equation*}
  \omega_{H^\infty}\hab{A|_{H^\gamma_{\theta,A}(\varphi_s)}} = \omega_{\BIP}\hab{A|_{H^\gamma_{\theta,A}(\varphi_s)}} = \frac{1}{s} \,\omega_{\BIP}\hab{A^s|_{H^\gamma_{\theta,A^s}}} = \frac{\pi}{s}
\end{equation*}
and by Proposition \ref{proposition:AdefHalpha}
\begin{equation*}
  \omega\hab{A|_{H^\gamma_{\theta,A}(\varphi_s)}} = \omega(A)=0.
\end{equation*}
Therefore $A|_{H^\gamma_{\theta,A}(\varphi_s)}$ on ${H^\gamma_{\theta,A}(\varphi_s)}$, which will be a closed subspace of $L^p$, is an example of an operator that we are looking for.

The remainder of this section will be devoted to the construction of this $A$. As a first guess, we could try the operators we used in Section \ref{section:sectorialnotalmostalpha} and Section \ref{section:almostalphanotalpha} for our examples. The following theorem shows that this will not work.
\begin{theorem}
  Let $\mb{x}$ be a Schauder basis for $X$, $(\lambda_k)_{k=1}^\infty$ a Hadamard sequence and $A$ the sectorial operator associated to $\mb{x}$ and $(\lambda_k)_{k=1}^\infty$. Let $\alpha$ be an ideal Euclidean structure on $X$ and fix $\theta \in \R$. Then $A|_{H^\alpha_{\theta,A}}$ has a bounded $H^\infty$-calculus on $H^\alpha_{\theta,A}$ with $\omega_{H^\infty}(A|_{H^\alpha_{\theta,A}})=\omega(A)$
\end{theorem}

\begin{proof}
For simplicity we assume $(\lambda_k)_{k=1}^\infty \subseteq \R_+$, i.e. $\omega(A)=0$, and leave the case $(\lambda_k)_{k=1}^\infty \subseteq \Sigma_\sigma$ for $0<\sigma<\pi$ to the interested reader.
Let $c>1$ be the constant in the definition of a Hadamard sequence and take $u>\max\cbrace{{1,1/\log(c)}}$. Define $\mu_k:= u \log{\ha{\lambda_k}}$, then we have
\begin{equation*}
  \inf_{j \neq k}\, \abs{\mu_j -\mu_k} = \inf_{j \neq k} \, u \,\abs{\log\hab{{\lambda_j}/{\lambda_k}}} >1 ,
\end{equation*}
so $(\mu_k)_{k=1}^\infty$ is \emph{uniformly discrete}. Moreover, denoting by $n^+(r)$ the largest number of points of $(\mu_k)_{k=1}^\infty$ in any interval $I \subset \R$ of length $r>0$, we have for the \emph{upper Beurling density} of $(\mu_k)_{k=1}^\infty$
\begin{equation*}
  D^+((\mu_k)_{k=1}^\infty):=\lim_{r \to \infty}\frac{n^+(r)}{r} <1.
\end{equation*}
Therefore, by \cite[Theorem 2.2]{Se95}, we know that $(\ee^{i\mu_k t})_{k=1}^\infty$ is a \emph{Riesz sequence} in $L^2(-\pi,\pi)$, i.e. there exists a $C>0$ such that for any sequence $\mbs{a} \in \ell^2_n$ we have
\begin{equation*}
  C^{-1}\,\nrm{\mbs{a}}_{\ell^2_n}  \leq \nrms{t\mapsto \sum_{k=1}^n a_k \ee^{i\mu_k t}}_{L^2(-\pi,\pi)} \leq C \,\nrm{\mbs{a}}_{\ell^2_n},
\end{equation*}
and thus we have
\begin{align*}
 \nrm{\mbs{a}}_{\ell^2_n} \leq C\, \nrms{\sum_{k=1}^n a_k \ee^{i\mu_k \cdot}}_{L^2(-\pi,\pi)}
 &\leq \frac{C}{\sqrt{u}}\, \ee^{\frac{\pi^2}{s} u} \nrms{ \ee^{-\frac{\pi}{s} \abs{\cdot}}\sum_{k=1}^n a_k \lambda_k^{i\cdot}}_{L^2(-\pi u, \pi u)}.
\end{align*}
Conversely we have that
\begin{align*}
   \sup_{\nrm{\mbs{a}}_{\ell^2_n}\leq 1 }\nrms{\ee^{-\frac{\pi}{s} \abs{\cdot}}\sum_{k=1}^n a_k \lambda_k^{i\cdot}}_{L^2(\R)} &\leq \sup_{\nrm{\mbs{a}}_{\ell^2_n}\leq 1 } \sqrt{u} \cdot \nrms{ \sum_{k=1}^n a_k \ee^{i\mu_k \cdot}}_{L^2(-\pi ,\pi )} \\&\hspace{1.5cm}+  \nrms{ \ee^{-\frac{\pi}{s} \abs{\cdot}} \sum_{k=1}^n a_k \lambda_k^{i\cdot} }_{L^2(\R\setminus(-\pi u,\pi u))}\\
   &\leq C\sqrt{u}  +2\sup_{\nrm{\mbs{a}}_{\ell^2_n} \leq 1} \ee^{-\frac{\pi^2}{s}u} \nrms{\ee^{-\frac{\pi }{s}\abs{\cdot}}\sum_{k=1}^n a_k \lambda_k^{i\cdot}}_{L^2(\R)},
\end{align*}
using the change of variables $t'=t \pm \pi u$ in the second step.
So for any $0<s<\pi u$ we have
\begin{equation*}
   \nrms{t \mapsto \ee^{-\frac{\pi}{s} \abs{t}}\sum_{k=1}^n a_k \lambda_k^{it}}_{L^2(\R)} \simeq \sqrt{u} \cdot \nrm{\mbs{a}}_{\ell^2_n}.
\end{equation*}
Therefore $T_s\colon \ell^2 \to L^2(\R)$ given by
\begin{equation*}
  (T_s{\mbs{a}})(t) := \sum_{k=1}^\infty a_k\ee^{-\pi\abs{t}/u}\lambda_k^{it}, \qquad t \in \R
\end{equation*}
is an isomorphism onto the closed subspace of $L^2(\R)$ generated by the functions $(t\mapsto \ee^{-\frac{\pi}{s} \abs{t}} \lambda_k^{it})_{k=1}^\infty$. Its adjoint $T_s^*\colon  L^2(\R)\to \ell^2$ is given by
\begin{equation*}
  (T_s\varphi)_k  = \int_\R \varphi(t) \ee^{-\pi\abs{t}/u}\lambda_k^{it}\dd t, \qquad k \in \N.
\end{equation*}
 Now fix $\mbs{a} \in \ell^2_n$ and define $x = \sum_{k=1}^n  a_kx_k$. Then for $\varphi_s(z)=z^{s/2}{(1+z^s)^{-1}}$ we have by Proposition \ref{proposition:Hequivnorm} and \eqref{eq:calculusmultipliercutoff}
\begin{equation*}
  \nrm{x}_{H^{\alpha}_{\theta,A}(\varphi_s)} = \nrmb{t\mapsto  \sum_{k=1}^n a_k  \ee^{-\frac{\pi}{s} \abs{t}} \lambda_k^{it+\theta} x_k}_{\alpha(\R;X)}.
\end{equation*}
Now if $S:L^2(\R) \to X$ is the operator represented by  $t\mapsto  \sum_{k=1}^n a_k  \ee^{-\frac{\pi}{s} \abs{t}} \lambda_k^{it+\theta} x_k$, then  we have $S=S' \circ T_s^*$,
where $S':\ell^2 \to X$ is the finite rank operator given by $S' = \sum_{k=1}^n e_k \otimes a_k\lambda_k^\theta x_k$ and $(e_k)_{k=1}^\infty$ is the canonical basis of $\ell^2$. Since $T_s$ is an isomorphism, we see
\begin{equation*}
  \nrm{x}_{H^{\alpha}_{\theta,A}(\varphi_s)} = \nrm{S}_{\alpha(\R;X)} \simeq \sqrt{u} \cdot \nrm{S'}_{\alpha}=\sqrt{u}\cdot \nrm{(a_1\lambda_1^\theta x_1,\ldots,a_n\lambda_n^\theta x_n)}_{\alpha}.
\end{equation*}
Thus by density we deduce that the spaces $H^{\alpha}_{\theta,A}(\varphi_s)$ are isomorphic for all $0<s<\pi u$ and since $u$ could be taken arbitrarily large they are isomorphic for all $s>0$. In particular $H^{\alpha}_{\theta,A}$ is isomorphic to $H^{\alpha}_{\theta,A}(\varphi_s)$ for any $s>1$ and thus by Theorem \ref{theorem:XalphaHinfty} we deduce that $A|_{H^\alpha_{\theta,A}}$ has a bounded $H^\infty$-calculus on $H^\alpha_{\theta,A}$ with $\omega_{H^\infty}(A|_{H^\alpha_{\theta,A}})=0$
\end{proof}

\subsection*{The construction of the Banach space $X$}
Since Hadamard sequences will not work for the example we are looking for, we will construct an operator based on a Schauder basis and an increasing sequence in $\R_+$, which is also sectorial by Proposition \ref{proposition:schaudersectorial}.
Let us first define the Banach space $X$ that we will work with. Fix $1<q<p<\infty$ and denote the space of all sequences which are eventually zero by $c_{00}$. Let $(x_k^*)_{k=1}^\infty$ be a sequence in
$c_{00} \cap \cbrace{x^* \in \ell^{p'}: \nrm{x^*}_{\ell^{p'}}\leq 1}$ such that $\cbrace{x_k^*:k\in \N}$ is dense in $\cbrace{x^* \in \ell^{p'}: \nrm{x^*}_{\ell^{p'}}\leq 1}$ and each element of $\cbrace{x_k^*:k\in \N}$ is repeated infinitely often. For each $k \in \N$ let $F_k \subseteq \N$ be the support of $x_k^*$ and write $F_k = \cbrace{s_{k,1},\ldots,s_{k,\abs{F_k}}}$, where $s_{k,1}<\cdots<s_{k,\abs{F_k}}$. Define $N_0=0$ and $N_k = \abs{F_1}+\cdots+ \abs{F_k}$ for $k \in \N$.

Let $(e_j)_{j=1}^\infty$ be the canonical basis of $\ell^p$. For $k \in \N$ we define the bounded linear operator $U_k:\ell^p \to \ell^p$ by
\begin{equation*}
  U_k(e_j) := \begin{cases}
    e_{s_{k,(j-N_{k-1})}} &\text{if } N_{k-1}<j\leq N_k\\
    0 &\text{otherwise}
  \end{cases}
\end{equation*}
and the partial inverse $V_k:\ell^p \to \ell^p$ by $V_k x = U_k^{-1}(x \ind_{F_k})$. Now we define $X=X_{p,q}$ as the completion of $c_{00}$ under the norm
\begin{equation*}
  \nrm{x}_{X_{p,q}} := \nrm{x}_{\ell^p} +\nrmb{\hab{\ip{U_kx,x_k^*}}_{k=1}^\infty}_{\ell^q}.
\end{equation*}
Then $X$ is isomorphic to a closed subspace of $\ell^p \oplus \ell^q$, which can be seen using the embedding $X\hookrightarrow \ell^p \oplus \ell^q$ given by
 \begin{equation*}
  x\mapsto x \oplus \hab{\ip{U_kx,x_k^*}}_{k=1}^\infty
\end{equation*}
We consider $X$, and therefore all parameters introduced above, to be fixed for the remainder of this section.

\begin{lemma}\label{lemma:examplecanonicalbasis}
The canonical basis of $\ell^p$ is a Schauder basis of $X$.
\end{lemma}
\begin{proof}
  It suffices to show that the partial sum projections $(S_j)_{j=1}^\infty$ associated to the canonical basis of $\ell^p$ are uniformly bounded. Fix $m,n \in \N$ such that $N_{n-1}<m\leq N_{n}$ and take $x \in X$. Then  since $\nrm{S_m}_{\mc{L}(\ell^p)}\leq 1$ and $\nrm{x_n^*}_{\ell^{p'}}\leq 1$  we have
  \begin{align*}
    \nrmb{\hab{\ip{U_kS_mx,x_k^*}}_{k=1}^\infty}_{\ell^q} &=\nrmb{\hab{\ip{U_kS_mx,x_k^*}}_{k=1}^n}_{\ell^q_n}\\&\leq \has{\sum_{k=1}^{n-1}\abs{\ip{U_kx,x_k^*}}^q}^{1/q}+ \absb{\ip{U_nS_mx,x_n^*}}\\
    &\leq \nrmb{\hab{\ip{U_kx,x_k^*}}_{k=1}^\infty}_{\ell^q} + \nrm{x}_{\ell^p}.
  \end{align*}
  Therefore $\nrm{S_m}_{\mc{L}(X)} \leq 2$ for all $m \in \N$.
\end{proof}

\subsection*{The construction of the sectorial operator $A$}
Next we construct the sectorial operator on $X$, for which $A^s|_{H^\alpha_{\theta,A^s}}$ with $s>1$ will not have a bounded $H^\infty$-calculus. Define for $j \in \N$
\begin{equation}\label{eq:deflambdaexample}
  \lambda_j = 2^k \has{2-\frac{1}{s_{k,(j-N_{k-1})}}},\qquad N_{k-1}<j\leq N_k.
\end{equation}
Then $(\lambda_j)_{j=1}^\infty$ is an increasing sequence in $\R_+$, so by Proposition \ref{proposition:schaudersectorial} and Lemma \ref{lemma:examplecanonicalbasis} the operator associated to $(\lambda_j)_{j=1}^\infty$ and the canonical basis $(e_j)_{j=1}^\infty$  is sectorial with $\omega(A) =0$. The following technical lemma will be key in our analysis of this operator.

\begin{lemma}\label{lemma:gintermsoff}
  Let $A$ be the sectorial operator associated to $(e_j)_{j=1}^\infty$ and $(\lambda_j)_{j=1}^\infty$. Let $0<\sigma<\pi$ and suppose that $f,g \in H^1(\Sigma_\sigma)$ such that
  \begin{equation*}
    \nrmb{(g(2^jA)x)_{j \in \Z}}_{\gamma(\Z;X)} \lesssim  \nrmb{(f(2^jA)x)_{j \in \Z}}_{\gamma(\Z;X)}, \qquad x \in c_{00}.
  \end{equation*}
  Then there is a sequence $\mbs{a} \in \ell^2(\Z)$ so that
  \begin{equation*}
    g(z) = \sum_{j \in \Z} a_kf(2^jz), \qquad z \in \Sigma_\sigma.
  \end{equation*}
\end{lemma}

\begin{proof}
   Fix $x \in \ell^p$ with all entries non-zero and $x^* \in \cbrace{x^*_k:k \in \N}$. Let $d_1<d_2<\cdots$ be such that $x^* = x_{d_k}^*$ for all $k \in \N$, which is possible since each element of $\cbrace{x^*_k:k \in \N}$ is repeated infinitely often.
  Define $T:\ell^p \to \ell^p$ by $Te_j = (2-\frac{1}{j})e_j$, then
  $$A = \sum_{k=1}^\infty 2^kV_kTU_k.$$ Fix $n \in \N$ and define
  \begin{equation*}
    y_n = V_{d_1}x + \cdots + V_{d_n}x \in c_{00}.
  \end{equation*}
  Then for all $j\in \N$ we have for $h = f,g$
  \begin{align*}
   h(2^jA)y_n &= \sum_{k=1}^\infty h(2^{j+k}V_{k}TU_{k})y_n = \sum_{k=1}^n  h(2^{j+d_k} T)V_{d_k}x
  \end{align*}
  Noting that the vectors $V_{d_k}x$ are disjointly supported shifts of $x\ind_F$ for $F = F_{d_1} = F_{d_2} = \cdots$, we obtain
  \begin{align*}
    \nrms{\hab{\sum_{j \in \Z} \abs{h(2^jA)y_n}^2}^{1/2}}_{\ell^p} &=  \nrms{\hab{\sum_{k=1}^n \sum_{j \in \Z} \abs{ h(2^{j+d_k}T)V_{d_{k}}x}^2}^{1/2}}_{\ell^p} \\&=
    n^{1/p}\nrms{\hab{\sum_{j \in \Z} \abs{h(2^jT)(x\ind_{F})}^2}^{1/2}}_{\ell^p}
  \end{align*} and similarly
  \begin{align*}
    \nrms{\has{\sum_{j \in \Z} \hab{\abs{\ip{U_kh(2^jA)y_n,x^*_k}}^2}^{1/2}}_{k=1}^\infty}_{\ell^q} &= n^{1/q}\hab{\sum_{j \in \Z} \abs{\ip{h(2^jT)x,x^*}}^2}^{1/2}.
  \end{align*}
  Now since $X$ is a closed subspace of a Banach lattice with finite cotype, we have by Proposition \ref{proposition:compareEuclidean} and our assumption on $f$ and $g$ that there is a $C>0$ such that for all $n \in \N$
  \begin{align*}
    n^{1/p}&\nrms{\hab{\sum_{j \in \Z} \abs{g(2^jT)(x\ind_F)}^2}^{1/2}}_{\ell^p} + n^{1/q}\hab{\sum_{j \in \Z} \abs{\ip{g(2^jT)x,x^*}}^2}^{1/2}\\ &\leq C \has{ n^{1/p} \nrms{\hab{\sum_{j \in \Z} \abs{f(2^jT)(x\ind_F)}^2}^{1/2}}_{\ell^p} + n^{1/q} \hab{\sum_{k \in \Z} \abs{\ip{f(2^jT)x,x^*}}^2}^{1/2}}.
  \end{align*}
  Since $q<p$ we obtain by dividing by $n^{1/q}$ and taking the limit $n \to \infty$ that
  \begin{align*}
    \hab{\sum_{k \in \Z} \abs{\ip{g(2^kT)x,x^*}}^2}^{1/2} &\leq C \hab{\sum_{k \in \Z} \abs{\ip{f(2^kT)x,x^*}}^2}^{1/2}.
  \end{align*}
  In particular we have
   \begin{align}\label{eq:fgT}
     \abs{\ip{g(T)x,x^*}} &\leq C \hab{\sum_{k \in \Z} \abs{\ip{f(2^kT)x,x^*}}^2}^{1/2}.
  \end{align}
  Since $T$ is bounded and invertible, we know by Lemma \ref{lemma:hadamardsum} that $$\sum_{\abs{k}\geq n}\nrm{f(2^kT)x}_{\ell^p} \to 0,\qquad n \to \infty.$$ Therefore \eqref{eq:fgT} extends to all $x^* \in \ell^{p'}$ of norm one by density. Define the closed (compact!) convex set
  \begin{equation*}
    \Gamma:= \cbraces{\sum_{k \in \Z}a_kf(2^kT)x:\nrm{\mbs{a}}_{\ell^2}\leq C}
  \end{equation*}
  and suppose that $g(T)x \notin \Gamma$. Using the Hahn--Banach separation theorem  \cite[Theorem 3.4]{Ru91} on $\Gamma$ and $\cbrace{g(T)x}$, we can find an $x^* \in \ell^{p'}$ such that
  \begin{align*}
    C\, \hab{\sum_{k \in \Z} \abs{\ip{f(2^kT)x,x^*}}^2}^{1/2} &= \sup_{\nrm{\mbs{a}}_{\ell^2}\leq C_1} \re\has{\ips{\sum_{k \in \Z}a_kf(2^kT)x,x^*}}\\
    &<\re\hab{\ip{g(T)x,x^*}}\\ &\leq \abs{\ip{g(T)x,x^*}},
  \end{align*}
  a contradiction with \eqref{eq:fgT}. Thus $g(T)x \in \Gamma$, so there is an $\mbs{a} \in \ell^2$ with $\nrm{\mbs{a}}_{\ell^2} \leq C_1$ such that
  \begin{align*}
    g(T)x=\sum_{k \in \Z}a_kf(2^kT)x.
  \end{align*}
  Since every coordinate of $x$ is non-zero, this implies that
  \begin{equation*}
    g\has{2-\frac{1}{j}} = \sum_{k \in \Z} a_kf\has{2^k\has{1-\frac{1}{j}}}
  \end{equation*}
  for all $j \in \N$. As $\sum_{k \in \Z} a_kf\ha{2^kz}$ converges uniformly to a holomorphic function on compact subsets of $\Sigma_\sigma$, by the uniqueness of analytic continuations this implies that
  \begin{equation*}
    g(z) = \sum_{k \in \Z} a_k f(2^kz), \qquad z \in \Sigma_\sigma,
  \end{equation*}
  which completes the proof.
\end{proof}

Using Lemma \ref{lemma:gintermsoff} we can now prove the main theorem of this section, which concludes our study of Euclidean structures.
\begin{theorem}
  Let $p \in (1,\infty)\setminus\cbrace{2}$ and $\sigma \in (0,\pi)$. There exists a closed subspace $Y$ of $L^p([0,1])$ and a sectorial operator $A$ on $Y$ such that $A$ has a bounded $H^\infty$-calculus, has $\BIP$ and is (almost) $\gamma$-sectorial with $\omega(A) = 0 $ and
  \begin{equation*}
    \omega_{H^\infty}(A) =\omega_{\BIP}(A)= \omega_{\gamma}(A)= \tilde{\omega}_{\gamma}(A) = \sigma.
  \end{equation*}
\end{theorem}

\begin{proof}
  If $p \in (1,2)$ take $X = X_{2,p}$ and if $p \in (2,\infty)$ take $X=X_{p,2}$. Let $A$ be the sectorial operator associated to $(e_j)_{j=1}^\infty$ and $(\lambda_j)_{j=1}^\infty$, with $(\lambda_j)_{j=1}^\infty$ as in \eqref{eq:deflambdaexample} and
   $(e_j)_{j=1}^\infty$ the canonical basis of $\ell^p$. Set $\nu=\pi/\sigma$ and define $B=A^\nu$, which is a sectorial operator with $$\omega(B)=\nu \, \omega(A)=0.$$
  Suppose that the operator $B|_{H^\gamma_{0,B}}$ as in Proposition \ref{proposition:AdefHalpha} has a bounded $H^\infty$-calculus. Then by Theorem \ref{theorem:XABIPHinftyequiv} there is an $1<s'<\infty$ such that for $0<s<s'$ the spaces $H^\gamma_{0,B}(\varphi_s)$ with $\varphi_s(z)=z^{s/2}(1+z^s)^{-1}$ are isomorphic. In particular by \eqref{eq:calculusmultipliercutoff} and a change of variables we have for $x \in c_{00}$
  \begin{equation*}
  \nrm{\varphi_{\nu s}(\cdot A)x}_{\gamma\ha*{\R_+,\frac{\ddn t}{t};X}} = \sqrt{\nu} \nrm{x}_{H^{\gamma}_{0,B}(\varphi_s)} \simeq  \sqrt{\nu} \nrm{x}_{H^{\gamma}_{0,B}} =
    \nrm{\varphi_{\nu}(\cdot A)x}_{\gamma\ha*{\R_+,\frac{\ddn t}{t};X}}
  \end{equation*}
  and thus by Proposition \ref{proposition:equivdiscretecontinuousstable} there is a $1<s<s'$ such that
  \begin{equation*}
    \nrmb{(\varphi_{\nu s}(2^kA)x)_{k \in \Z}}_{\gamma(\Z;X)} \leq C \, \nrmb{(\varphi_{\nu}(2^kA)x)_{k \in \Z}}_{\gamma(\Z;X)}.
  \end{equation*}
  This implies by Lemma \ref{lemma:gintermsoff} that there is a $\mbs{a} \in \ell^2$ such that we have
  \begin{equation}\label{eq:varphicompare}
    \varphi_{\nu s}(z) = \sum_{k \in \Z} a_k \varphi_\nu(2^kz), \qquad z \in \Sigma_{\mu}
  \end{equation}
  for any $0<\mu<\pi/{\nu s}$. Thus  \eqref{eq:varphicompare} holds for all $z \in \Sigma_{\pi/\nu s}$. But $\varphi_{\nu} \in H^1(\Sigma_{\pi/\nu s})$ and $\varphi_{\nu s}$ has a pole of order $1$ on the boundary of $\Sigma_{\pi/\nu s}$, a contradiction. So $B|_{H^\gamma_{0,B}}$ does not have a bounded $H^\infty$-calculus. By Theorem  \ref{theorem:XABIPHinftyequiv} and \eqref{eq:BIPangle} this implies $\omega_{\BIP}(B|_{H^\gamma_{0,B}}) = \pi$.

  Now we have by Proposition \ref{proposition:AdefHalpha} that the operator $A|_{H^\gamma_{0,A}(\varphi_\nu)}$  on $Y = H^\gamma_{0,A}(\varphi_\nu)$ is sectorial with
 \begin{equation*}
  \omega(A|_{H^\gamma_{0,A}(\varphi_\nu)}) = \omega(A) = 0.
\end{equation*}
and by Theorem \ref{theorem:XalphaHinfty} and Corollary \ref{corollary:equalangles} we know that $A|_{H^\gamma_{0,A}(\varphi_\nu)}$ has a bounded $H^\infty$-calculus with
\begin{equation*}
  \omega_{H^\infty}(A|_{H^\gamma_{0,A}(\varphi_\nu)}) = \omega_{\BIP}(A|_{H^\gamma_{0,A}(\varphi_\nu)}) = \frac{1}{\nu} \,\omega_{\BIP}(B|_{H^\gamma_{0,B}}) = \pi \nu = \sigma.
\end{equation*}
Here we used that by Proposition \ref{proposition:AdefHalpha} we have for all $x \in c_{00}$
$$B|_{H^\gamma_{0,B}}x = \hab{A|_{H^\gamma_{0,A}(\varphi_\nu)}}^{1/\nu}x.$$

It remains to observe that $Y$ is a closed subspace of $\gamma(\R_+,\frac{\ddn t}{t};X)$, which is isomorphic to a closed subspace of $L^p(\Omega;\ell^p\oplus \ell^2)$ for some probability space $(\Omega,\P)$, which in turn is isomorphic to a closed subspace of $L^p([0,1])$, see e.g. \cite[Section 6.4]{AK16}. Finally note that $Y$ has Pisier's contraction property and therefore $A|_{H^\gamma_{0,A}(\varphi_\nu)}$ is $\gamma$-sectorial by Theorem \ref{theorem:Hinftyell2bounded} and the angle equalities follow from Proposition \ref{proposition:alphaalmoastalpha} and  Corollary  \ref{corollary:equalangles}.
\end{proof}

\backmatter

\bibliographystyle{alpha}
\bibliography{euclidean}

\end{document}